\documentclass[10pt]{amsart}
\usepackage{fullpage}
\usepackage[utf8]{inputenc}
\usepackage{microtype}
\usepackage{graphics}
\usepackage{amsmath,amssymb,stmaryrd,array,multirow,dsfont,mathrsfs}
\usepackage{amsthm}
\usepackage{enumerate}
\usepackage{xcolor}
\usepackage{color}
\usepackage{hyperref}
\usepackage{indentfirst}
\usepackage{pifont}
\usepackage[T1]{fontenc}
\usepackage{microtype}
\usepackage{graphics}
\usepackage{amsthm}
\usepackage{extarrows}
\usepackage{stmaryrd}
\graphicspath{{Images/}}
\xdefinecolor{darkgreen}{rgb}{0,0.4,0}
  
\usepackage{tikz}
\newcommand*\circled[1]{\tikz[baseline=(char.base)]{
            \node[shape=circle,draw,inner sep=0.5pt] (char) {#1};}}
\newcommand{\beq}{\begin{equation}}
\newcommand{\eeq}{\end{equation}}
\newcommand{\beqs}{\begin{equation*}}
\newcommand{\eeqs}{\end{equation*}}
\newcommand{\beal}{\begin{align}}
\newcommand{\eeal}{\end{align}}
\newcommand{\beals}{\begin{align*}}
\newcommand{\eeals}{\end{align*}}
\newcommand{\ben}{\begin{eqnarray}}
\newcommand{\een}{\end{eqnarray}}
\newcommand{\beno}{\begin{eqnarray*}}
\newcommand{\eeno}{\end{eqnarray*}}

\renewcommand{\div}{{\rm div\,}}

\newcommand{\Id}{{\rm Id}\,}
\newcommand{\Supp}{{\rm Supp}\,}

\newcommand{\Rmnum}[1]{\uppercase\expandafter{\romannumeral #1} }
 \numberwithin{equation}{section}

\allowdisplaybreaks[4]
\newtheorem{thm}{Theorem}[section]
\newtheorem{lem}[thm]{Lemma}
\newtheorem{prop}[thm]{Proposition}
\newtheorem{rmk}[thm]{Remark}
\newtheorem{cor}[thm]{Corollary}

\newtheorem{definition}{Definition}

\def\curl{\mathop{\rm curl}\nolimits}

\def\il{|\!|\!|}
\def \d {\mathrm {d}}
\def\cA{{\mathcal A}}
\def\cB{{\mathcal B}}
\def\cC{{\mathcal C}}

\def\cE{{\mathcal E}}
\def\cF{{\mathcal F}}
\def\cG{{\mathcal G}}
\def\cH{{\mathcal H}}
\def\cI{{\mathcal I}}

\def\cK{{\mathcal K}}
\def\cL{{\mathcal L}}
\def\cM{{\mathcal M}}
\def\cN{{\mathcal N}}
\def\cO{{\mathcal O}}

\def\cQ{{\mathcal Q}}
\def\cR{{\mathcal R}}
\def\cS{{\mathcal S}}
\def\cT{{\mathcal T}}
\def\cU{{\mathcal U}}

\def\cW{{\mathcal W}}

\def\cZ{{\mathcal Z}}

\let\f=\frac
\def \p {\partial}
\def\mR {\mathbb{R}}

\def\ep{\varepsilon}

\def \pt {\partial_{t}}
\def \vr {\varrho}

\def\si {\sigma}
\def\na{\nabla}
\def\izt{\int_0^t}

\def \tta {\tilde{\theta}}

\def \bp {\mathbb{P}}

\def \bn {\textbf{n}}
\def  \bN {\textbf{N}}
\def \bq{\mathbb{Q}}
\def\uhco{L_t^2\underline{H}_{co}}
\def \hcob{L_t^2\tilde{H}}
\def \hco {L_t^2H_{co}}
\def \infco{L_t^{\infty}H_{co}}
\def \uinfco{L_t^{\infty}\underline{H}_{co}}
\def \infcok{L_t^{\infty}H_{co,\sqrt{\kappa}}}
\def \izt{\int_0^t}
\def \iomega{\int_{\Omega}}
\def\izto{\int_0^t\int_{\Omega}}
\def \kpa {\kappa}
\def\lat{\Lambda\big(\f{1}{c_0},\cA_{m,t}\big)}
\def \lab {\Lambda\big(\f{1}{c_0},}
\def \lae{\Lambda\big(\f{1}{c_0},\cN_{m,t}\big)}
\def \hcoch {L_t^2\cH}
\def \infcoch {L_t^{\infty}\cH}
\def \omer {\omega_{r_0}}

\def \tsigma {\tilde{\sigma}}
\def\tlu {\tilde{u}}
\def \tz {\tilde{z}}
\def \tt {\tilde{t}}

\def \tb {\tilde{b}}
\title{Uniform regularity in the low Mach number and inviscid limits for the full Navier-Stokes system in domains with boundaries }
\author{ Changzhen Sun}
\address{%Institut de Math\'ematiques de Toulouse, UMR5219, Universit\'e de Toulouse, CNRS, INSA, F-31077 Toulouse, France
Institut de Mathématiques de Toulouse – UMR 5219, Université de Toulouse; CNRS, Université Paul
Sabatier, 118 route de Narbonne, 31062 Toulouse Cedex 9, France}
\email{changzhen.sun@math.univ-toulouse.fr}

\begin{document}
\maketitle

\begin{abstract}
In the present work, motivated by the studies on the low Mach number limit problem, we establish uniform regularity estimates with respect to the Mach number  for the non-isentropic compressible Navier-Stokes system  in smooth domains with  Navier-slip boundary conditions, in the general case of ill-prepared initial data. The thermal conduction is taken into account and the large variation of temperature is allowed.
Moreover, the obtained regularity estimates are also uniform in the Reynolds number $\text{Re}\in[1,+\infty),$  Péclet  number $\text{Pe}\in [1,+\infty),$ provided 
$$\big|\f{1}{\text{Re}}-\f{\iota_0}{\text{Pe}}\big|\lesssim \f{1}{\text{Pe}^{\f{1}{2}}}\f{1}{\text{Re}},$$ 
where $\iota_0$ is a fixed constant independent of Mach number, Reynolds number and Péclet number. The large temperature variation as well as the interactions of two kinds of boundary layers are the main obstacles of the proof.

%The convergence to the limit system when the Mach number tends to zero is also justified for an exterior domain outside a smooth compact set in $\mR^3$ in the spirit of \cite{MR2106119}.

%The low Mach number limit is then jusitified 
%such an assumption is made essentially due to the interactions between hydrodynamic boundary layers and the thermal boundary layers.
%provided the Péclet  number is propotional to the Reynolds number.
%Péclet  number 
  %This allows  to prove the local existence of a strong solution  on  a time interval   independent of the Mach number and to justify the incompressible limit  through a simple compactness argument.
\end{abstract}
\quad \textbf{Keywords:} uniform regularity, low Mach number limit, fast oscillation, boundary layers
%\begin{keywords}
%\quad \textbf{Keywords:} uniform regularity, low Mach number limit, fast oscillation, boundary layers
%\small
\tableofcontents
\section{Introduction}
 In this manuscript, we consider the following scaled non-isentropic compressible Navier-Stokes system: %$(CNS)_{\varepsilon}$ 
\beq \label{NCNS-S}
 \left\{
\begin{array}{l}
 \displaystyle\pt \rho^{\varepsilon} +\div( \rho^{\varepsilon} u^\varepsilon)=0,\\
 \displaystyle\pt ( \rho^{\varepsilon} u^{\varepsilon})+\div(\rho^{\varepsilon}u^{\varepsilon}\otimes u^{\varepsilon} )+
\f{\nabla P^{\ep}}{\ep^2}-\mu \,\div\mathcal{L}u^{\varepsilon}=0,  \qquad \text{$(t,x)\in \mathbb{R}_{+}\times \Omega $} \\[5pt]
 \displaystyle \pt(\rho^{\ep} e^{\ep})+\div (\rho^{\ep} u^{\ep} e^{\ep})+P^{\ep}\div u^{\ep}=\kappa \Delta \cT^{\ep}+\mu\ep^2\cL u^{\ep}\cdot \mathbb{S}u^{\ep},%\\
% \displaystyle u^{\ep}|_{t=0} =u_0^{\varepsilon} ,\rho|_{t=0}=\rho_0^{\varepsilon},\\
\end{array}
\right.\\
\eeq
where $\Omega\subset \mathbb{R}^3$
is a smooth domain in $\mathbb{R}^3$, $\rho^{\ep}(t,x),$  $u^{\ep}(t,x)$ and $\mathcal{T}^{\ep}$ are respectively the density, the velocity, the temperature
of the fluid.  Representing the  pressure and the total energy
respectively, $P^{\ep}, e^{\ep}$ are given functions of the density and the temperature. We assume that the fluid is  polytropic, %ideal gas,
that is, 
\beqs 
P^{\ep}=R\rho^{\ep}\mathcal{T}^{\ep}, \qquad e^{\ep}=C_v\mathcal{T}^{\ep},
\eeqs
where the two fixed positive constants $R, C_v$ denote respectively the generic gas constant and the heat at constant volume. Moreover,
the viscous stress tensor takes the form:
$$\mathcal{L}u^{\varepsilon}=2\lambda_1\,\mathbb{S}u^{\varepsilon}+\lambda_2 \,\div u^{\varepsilon} \text{Id},\quad \mathbb{S} u^{\varepsilon}=\frac{1}{2}(\nabla u^{\varepsilon}+\nabla^{t}u^{\varepsilon}),$$ 
where $\lambda_1,\lambda_2$ are viscosity parameters that are assumed to be constant and to satisfy the physical condition: 
$$\mu>0, \quad 2\mu+3\lambda>0.$$ Defined as the ratio of characteristic fluid velocity to the sound speed,  the  Mach number $\ep$ is a dimensionless parameter that meassures the compressibility of the fluid.
%and is assumed to be small, that is $\ep\in(0,1]$.
The parameters  $\mu, \kappa \in (0,1]$ are respectively the inverse of the Reynold number and Péclet number that characterize the effects of hydrodynamical dissipation and thermal dissipation. In the current work, they are assumed to be small: 
$$\ep\in (0,1],\quad  \mu \in (0,1], \quad \kappa\in (0,1].$$

As we are considering the system  in a domain with boundaries, 
we shall supplement the system \eqref{NCNS-S}
with the Navier-slip boundary condition on the velocity and the Neumann boundary condition on the temperature
\begin{equation}\label{bdyconditions}
u^{\ep}\cdot \bn=0, \quad
 \Pi(\mathbb{S} u^{\ep} \bn)+a\Pi u^{\ep}=0,\quad \p_{\bn}\cT^{\ep}=0 \quad 
  \text{on }  {\partial{\Omega}}
\end{equation}
where $\bn$ is the unit outward normal vector and  $a$ is a constant related to the slip length
(our analysis  can be easily extended to  the case where $a$ is a smooth function).  We use the notation $\Pi f$ for  the tangential part of a  vector $f,$
$\Pi f^{\ep}=f^{\ep}-(f^{\ep}\cdot \bn)\cdot \bn.$ Let us remark that Navier-slip boundary conditions
  can be expressed as a non-homogeneous Dirichlet condition on $\curl u^{\ep} \times \bn,$ 
  \beq\label{bd-curlun}
\curl u^{\ep}\times \bn
=
2\Pi(-a u^{\ep}+D \bn\cdot u^{\ep} )\quad \text{on }  {\partial{\Omega}}.
\eeq
The special  case $\curl u^{\ep}\times \bn=0$ corresponds to the choice $a = D \bn.$ 

  Taking the smallness of Mach number into account, the scaled system \eqref{NCNS-S} can be derived from the original non-isentropic compressible Navier-Stokes system by introducing suitable change of variable. As a matter of fact, one finds \eqref{NCNS-S} by performing the following scaling:
   $$\rho(t,x)=\rho^{\ep}(\ep t,x),\, u(t,x)=\ep u^{\ep}(\ep t,x),\, \cT=\cT^{\ep}(\ep t,x),\,\tilde{\mu}=\ep \mu,\,\tilde{\kappa}=\ep\kappa,$$
   where $(\rho, u, \cT)$ solve the following system:
   \beq \label{NCNS-O}
 \left\{
\begin{array}{l}
 \displaystyle\pt \rho +\div( \rho u)=0,\\
 \displaystyle\pt ( \rho u)+\div(\rho u\otimes u)-\tilde{\mu} \,\div\mathcal{L}u+
\nabla P=0,  \qquad \text{$(t,x)\in \mathbb{R}_{+}\times \Omega $}, \\
 \displaystyle \pt(\rho e)+\div (\rho u e)+P \div u=\tilde{\kappa} \Delta \cT+\tilde{\mu} \cL u\cdot \mathbb{S}u, 
 %\\
% \displaystyle u^{\ep}|_{t=0} =u_0^{\varepsilon} ,\rho|_{t=0}=\rho_0^{\varepsilon},\\
\end{array}
\right.\\
\eeq
with
\beqs 
P =R\rho\mathcal{T}, \qquad e=C_v\mathcal{T}.
\eeqs

 We are interested in the limit of the solution to \eqref{NCNS-S} as $\ep$ tends to $0.$ Since the pressure depends both on the density and the temperature, %and the quantities appearing in \eqref{NCNS-S} interact with each other in a complicated way, 
 it is not very clear to see what should be the limit system as $\ep\rightarrow 0.$ %Therefore, as in \cite{MR2211706}, we will review the 
 However, as noted in \cite{MR2211706}, it is more convenient to review the system \eqref{NCNS-S} as the equations of the pressure $P^{\ep},$ the velocity $u^{\ep}$ and the temperature $\cT^{\ep}$
 and thus to consider the following equivalent system: 
 \beq \label{NCNS-S1}
 \left\{
\begin{array}{l}
 \displaystyle(\pt%P^{\varepsilon}
 +u^\varepsilon\cdot \nabla) P^{\ep}+\gamma P^{\ep}\div u^{\ep}-(\gamma-1)\kappa\Delta\cT^{\ep}=(\gamma-1)\mu\ep^2\cL u^{\ep}\cdot \mathbb{S}u^{\ep},\\[5pt]
 \displaystyle \rho^{\varepsilon}(\pt+  u^{\varepsilon}\cdot\nabla) u^{\ep}+\f{\nabla P^{\ep}}{\ep^2}-\mu \,\div\mathcal{L}u^{\varepsilon}=0,  \qquad \text{$(t,x)\in \mathbb{R}_{+}\times \Omega $}, \\[5pt]
 \displaystyle \rho^{\ep}C_v(\pt+u^{\ep} \cdot\nabla)\cT^{\ep}+P^{\ep}\div u^{\ep}=\kappa \Delta \cT^{\ep}+\mu\ep^2\cL u^{\ep}\cdot \mathbb{S}u^{\ep},%\\
% \displaystyle u^{\ep}|_{t=0} =u_0^{\varepsilon} ,\rho|_{t=0}=\rho_0^{\varepsilon},\\
\end{array}
\right.\\
\eeq
 where $\gamma=1+C_v$ denotes the ratio of specific heat. We  can thus deduce formally that the limit system takes the form: 
 \beq\label{NINS}
 \left\{
\begin{array}{l}
 \displaystyle \gamma \overline{P}\,\div u^0=(\gamma-1)\kappa\Delta\cT^0,\\
  \displaystyle \rho^0(\pt+u^0\cdot\nabla)u^0+\nabla\pi^0-\mu\div\cL u^0=0,\\
   \displaystyle  C_v \gamma \rho^0(\pt+u^0\cdot\nabla)\cT^0=\kappa\Delta\cT^0,
 \end{array}
\right.\\
 \eeq
 where $\rho^0=\overline{P}/R\cT^0.$
Naturally, if the temperature $\cT^{\ep}$ is assumed to be a constant $\overline{\cT},$ then the non-isentropic system \eqref{NCNS-S1} reduces to the isentropic system and the limit system \eqref{NINS} reduces to the incompressible  Navier-Stokes system. This limit process is therefore frequently referred to as the incompressible limit when the isentropic fluids are considered.

 The current work is motivated by the verification of this limit process for strong solutions to \eqref{NCNS-S1}. 
  Partially due to the appearance of the
   singular term $\f{\na P^{\ep}}{\ep^2}$ 
   in the system \eqref{NCNS-S1}, this limit process is considered to be a singular limit. Since we are dealing with strong solutions, the first (and highly non-trivial) part is to establish some uniform estimates independent of the Mach number $\ep,$ so that the solutions to \eqref{NCNS-S1} for any $\ep\in (0,1]$ exist on a time interval independent of $\ep.$ Once this has been done, the convergence of solutions of \eqref{NCNS-S1} to 
  that of \eqref{NINS} will be shown when $\Omega$ is an exterior domain outside a compact smooth domain in $\mR^3.$ %by using the large time dispersion of wave equation.}

   Due to its great physical interests, %and mathematical challenges,
   the low Mach number limit problem has been investigated widely in several different contexts, depending on the type of solutions (strong solutions or weak solutions),  the generality of the system (isentropic or non-isentropic), the properties of the domain  (whole space, torus or domain with boundaries), as well as the type of the initial data considered. %(well-prepared or ill-prepared). 
   
   We will first review some works concerning the incompressible limit for \textit{isentropic} compressible system.
   The initial works, credited to Ebin \cite{MR431261}, Klainerman-Majda \cite{MR615627,MR668409}, focus on the study of the local strong solution of isentropic fluids occupied in a domain without boundaries ($\mR^3$ or $\mathbb{T}^3$), with well-prepared initial data ($(\div u_0^{\ep},\nabla P_0^{\ep}/\ep)=\cO(\ep)$). Later, the same problem is considered by Ukai \cite{MR849223} in the whole space and Gallagher \cite{MR1794519} on %$\mathbb{T}^3$ 
   the torus for ill-prepared initial data ($(\div u_0^{\ep},\nabla P_0^{\ep}/\ep)=\cO(1)$).
Recently, the author and his collaborators investigate in \cite{MR4403626} the incompressible limit of the strong solution in domain with boundaries for ill-prepared data, by showing the uniform regularity estimates.
%There are also some works concerning the uniform estimates in critical Besov spaces, see \cite{MR1886005, MR2157145, MR3563240}. 
  The incompressbile limit of weak solutions for the viscous system was first investigated by Lions and Masmoudi. In \cite{MR1628173,MR1710123}, they prove that the  weak solutions of the isentropic compressible Navier-Stokes system converge to the solution of the incompressible system.  %and several different domains (whole space, torus and boundeddomain with suitable boundary conditions).
  In general, for ill-prepared data, one can only obtain weak convergence in time, nevertheless, Desjardins and Grenier  establish in \cite{MR1702718} the local strong convergence by using the dispersion of acoustic wave in the whole space. There are also many other related works, one can see for instance \cite{MR1308856,MR1917042,MR3803773,MR1886005,MR3563240,MR1697038,MR2575476,MR2338352,MR918838,MR3916820,MR4134150,MR3240080}.
  For more exhaustive information, one can refer to the well-written survey papers by 
 Alazard \cite{MR2425022}, Danchin \cite{MR2157145}, Feireisl \cite{MR3916821}, Gallagher \cite{MR2167201},
 Jiang-Masmoudi \cite{MR3916820},
 Schochet \cite{MR3929616}.

  Let us focus on the low Mach number limit of the strong solution to non-isentropic system, which is more related to the interest of the current work.  For the non-isentropic %\texit
  {inviscid} system (that is $\mu=\kappa=0$ in system \eqref{NCNS-S}),  Metivier and Schochet \cite{MR1834114} study the low Mach number limit of strong solution in the whole space $\mR^3$ for ill-prepared data. More precisely, they establish uniform estimates %(by using the wave operator) 
  and justify the 
  the convergence of the solution of the non-isentopic Euler system to the solution of the  incompressible inhomogeneous Euler system. 
  In a bounded domain, the local strong solution is established  by Schochet \cite{MR834481} and the low Mach number limit is shown for the well-prepared data. As for the ill-prepared data, Alazard \cite{MR2106119} establishes the uniform estimates in domains with boundaries. Moreover, for an exterior domain, he proves the similar convergence result as in  \cite{MR1834114}. %the whole space 
  With regards to the  non-isentropic Navier-Stokes system \eqref{NCNS-S1}, there are several works depending on the property of the domain and the type of the initial data. In the whole space, 
  Alazard proves \cite{MR2211706} the uniform  estimates and the low Mach number limit for \textit{ill-prepared} data. It is worthy to mention that %that by using the techniques of paradifferential calculus, he obtain estimates
  the estimates obtained in \cite{MR2211706} are uniform not only in Mach number $\ep$, but also in the Reynolds number $\text{Re}=\mu^{-1}$ and the Péclet number $\text{Pe}=\kappa^{-1}$. We also refer to \cite{MR3197662} for the results on the non-isentopic MHD system. 
  However, for the non-isentropic viscous system in a domain $\Omega\subset \mR^3$ with boundaries, there are only a few works and all of them deal with \textit{well-prepared} data. %In \cite{MR3531754}, the authors establish  uniform global (for small data)  $H^2$ estimates under a sufficient {well-prepared} initial data assumption, namely the  second time derivative of the velocity needs to be uniformly bounded initially. 
  In  \cite{MR3531754}, Jiang and Ou obtain uniform estimates 
 for \eqref{NCNS-S} with Dirichlet boundary condition imposed on the velocity ($u^{\ep}|_{\p\Omega}=0$), for {well-prepared} data in the case of vanishing thermal conductivity ($\kappa=0$). The authors in \cite{MR3274763,JU2022131} consider the %more general case with $\kappa>0$ 
 low Mach number limit for \eqref{NCNS-S1} in the case of 
 $\kappa>0$ with the \textit{vorticity-slip} boundary condition $\curl u^{\ep}\times \bn=0.$ %By assuming the data to be well-prepared and the initial temperature to be close to a constant state, 
 Under the well-prepared assumption %together %with small variation of the temperature, 
 they can establish uniform estimates in the Mach number and prove the convergence of 
 \eqref{NCNS-S1} to the isentropic incompressible Navier-Stokes system. %Including large temperature variation raises some technical difficulties, but this 
 %(large temperature variation)
 Nevertheless, %although quite interesting, 
 in the general case of ill-prepared initial data and Navier-slip boundary conditions \eqref{bd-curlun},
 the low Mach number limit for non-isentropic viscous system \eqref{NCNS-S1}
 has not been addressed so far.
 
% Let us remark that the 'vorticity-slip' boundary condition $\curl u^{\ep}\times\bn=0$ is however less general than the boundary condition \eqref{bd-curlun} considered here and the latter would make the analysis much more involved due to the stronger boundary effects.

% As a matter of fact, even for the isentropic viscous fluids, there have been difficulties  due to the simultaneous appearance of the boundary layer and fast oscillating acoustic waves. Indeed, on the one hand, as the viscous system reads formally at large time scale $\tau=t/\ep$ a small viscosity approximation of the hyperbolic system, a boundary layer is created near the boundary.  On the other hand, since the compressible part of the initial data $(\div u_0^{\ep},\nabla P_0^{\ep}/\ep)$ has a size of order one, it is reasonable to expect the appearance of high oscillating acoustic waves $(\div u^{\ep}/\ep,\nabla P^{\ep}/\ep^2)$ in the system. 
   %Together with N. Masmoudi and F. Rousset, the author finds a way \cite{MR4403626} to get over these difficulties and proves uniform in Mach number estimates for isentropic compressible Navier-Stokes system. In \cite{MR4403626}, to neglect extra technicalities, we assume that the Reynolds number is fixed and focus only on the low Mach number limit. Nevertheless, the established estimates can indeed depend on the Reynolds number. 
   
   On the other hand, as in the concrete physics, the dimensionless physical quantities are likely to be related, it is interesting to obtain uniform regularity estimates in all of these parameters. 
   Therefore, in this work, we aim to establish uniform estimates not only in the Mach number, but also in the Reynolds number $\text{Re}=\mu^{-1}$ and the Péclet number $\text{Pe}=\kpa^{-1}$ (here the parameter $\kpa$ appears since we consider non-isentropic system), which will be convenient to  study simultaneously the low Mach number and the inviscid limit for the \textit{strong solutions}.
   As will be explained carefully in Subsection 1.3, there are several difficulties in order to obtain such a result.
 %The objective of the current work are thus two-fold. First, we study the low Mach number limit problem for the more general non-isentropic system \eqref{NCNS-S1}. Second, we shall establish uniform estimates not only in the Mach number, but also in the Reynolds number $\text{Re}=\mu^{-1}$ and the Péclet number $\text{Pe}=\kpa^{-1}$ (here the parameter $\kpa$ appears since we consider non-isentropic system). As will be explained carefully in the subsection 1.3, the simultaneous extensions in these two directions are highly non-trivial and shall raise many new technical issues. %difficulties.
 %shall raise many difficulties.
 %many new difficulties shall arise due to the simultaneous extensions in these two directions.  %shall raise many difficulties.
 %In this aspect, we are only aware of the work \cite{MR2211706} proving the corresponding estimates when the fluid domain $\Omega=\mR^3.$
 %As explained later, both of these two extensions arise new difficulties.   
 %

%The current work can be seen as an analogue result in domain with boundaries of Alazard's work \cite{MR2211706} in whole space.
The similar problem has been studied by Alazard \cite{MR2211706} in the whole space.
 % by comparing with the result for non-isentropic viscous fluids in the whole space.
As noted in \cite{MR1834114}, the non-isentropic system %\eqref{NCNS-S1}
is not uniformly linear stable due to the non-constant temperature, which prevents one from obtaining high regularity estimates by differentiating the equation.  Alazard \cite{MR2211706} can overcome this problem in the whole space by splitting the frequencies and %reduce the matter to the
by employing the para-differential calculus. %More precisely, direct energy estimates can be conducted in the high frequency region ($|\xi|\gtrsim \ep^{-1},$ $\ep$ denotes the Mach number) by using the %large damping property resulting from the diffusion term Friedrichs' lemma to to compensate the singular commutators. 
Nevertheless, it is unclear how could 
this strategy be applied 
 in the  domain with boundaries. Moreover, in the presence of the boundaries, 
 %Let us explain the main differences and the new phenomena due to the presence of the boundaries. 
% Last but not least, %as both the viscosity $\mu$ and thermal conductivity $\kpa$ do not vanish,
due to the non-vanishing  viscosity $\mu$ and thermal conductivity $\kpa,$ two different kinds of boundary layers, namely the viscous boundary layer (due to the non-zero viscosity) and the thermal boundary layer (due to the non-zero thermal conductivity) will show up and interact with each other. This phenomenon can make the analysis rather delicate and some suitable assumption between the thermal conductivity $\kpa$ and the viscosity $\mu$ has to be made.

Due to the appearance of the boundary layers,
 to get uniform high order estimates, we need to use a functional framework based on 
 conormal Sobolev spaces that minimize the use of normal derivatives close to the boundary
in the spirit of \cite{MR2885569,MR3590375, masmoudi2021incompressible, MR4403626}.
We remark that such kind of spaces has
been widely used to study initial boundary value problems for   parabolic
and hyperbolic equations, see for example \cite{MR658735}, \cite{MR233064,MR440182},  \cite{MR1070840,MR2151414,MR2130346}.
 \subsection{Conormal Sobolev spaces and notations}
To define the conormal Sobolev norms,  we take a finite set of generators of vector fields 
that are tangent to the boundary of $\Omega$: $Z_{j}(1\leq j\leq M)$. 
Due to the appearance of the fast oscillations, it is also necessary to involve the scaled  time  derivative
$Z_0=\varepsilon \partial_t.$ 

We set 
$$Z^I=Z_0^{\alpha_0}\cdots Z_{M}^{{\alpha_M}}, \quad I=(\alpha_0,\alpha_1,\cdots \alpha_M)\in \mathbb{N}^{M+1}.$$ Note that $Z^I$ contains not only spatial derivatives but also the scaled time  derivative $\ep\pt.$
 We introduce the following Sobolev conormal spaces: for $p=2$ or $+\infty,$
$$L_t^pH_{co}^{m}=\{f\in L^{p}\big([0,t],L^2(\Omega)\big),Z^{I}f\in L^{p}\big([0,t],L^2(\Omega)\big),|I|\leq m\},$$
equipped with the norm:
\beq\label{not1}
\|f\|_{L_t^pH_{co}^m}=\sum_{|I|\leq m}\|Z^I f\|_{L^{p}([0,t],L^2(\Omega))},
\eeq
where $|I|=\alpha_0+\cdots\alpha_M.$
Since the number of time derivatives and spatial conormal derivatives need sometimes  to be distinguished,   we shall also use  the notation:
\beq
\|f\|_{L_t^p\mathcal{H}^{j,l}}=\sum_{I=(k,\tilde{I}), k\leq j,|\tilde{I}|\leq l}\|Z^I f\|_{L^{p}([0,t],L^2(\Omega))}
\eeq
and to simplify, we will use  $L_t^p\mathcal{H}^{l}=L_t^p\mathcal{H}^{0,l}.$
We will also use sometimes the following  norm that preclude the highest time derivatives:
\beq\label{defuhco} 
\|f\|_{L_t^p\underline{H}_{co}^m}=\sum_{j+l\leq m, j\leq m-1}\|f\|_{L_t^p\mathcal{H}^{j,l}}
\eeq
Moreover, we introduce for convenience the  notation:
\begin{align*}
   \|f\|_{L_t^pH_{co,\nu}^m}=\nu\|f\|_{L_t^pH_{co}^{m}}+\|f\|_{L_t^pH_{co}^{m-1}}.
\end{align*}
For the space modeled on  $L^{\infty}$, we shall use the following notation: %for the norm:
\beq\label{not3}
\interleave f\interleave_{m,\infty,t}%=\colon\|f\|_{L_t^{\infty}W_{co}^{m,\infty}}
=\sum_{|I|\leq m}\|Z^I f\|_{L^{\infty}([0,t]\times\Omega)}.
\eeq
To measure pointwise regularity at a given time $t$ (in particular also with $t=0$), we shall use the semi-norms
 \beq\label{normfixt}
 \begin{aligned}
 \|f(t)\|_{H_{co}^m}=
 \sum_{|I|\leq m}&\|(Z^I f)(t)\|_{L^2(\Omega)}, \quad \|f(t)\|_{\mathcal{H}^{j,l}}=\sum_{I=(k,\tilde{I}),k\leq j,|\tilde{I}|\leq l}\|Z^I f(t)\|_{L^2(\Omega)},\\ & \|f(t)\|_{m,\infty}=\sum_{|I|\leq m}\|(Z^I f)(t)\|_{L^{\infty}(\Omega)}.
 \end{aligned}
 \end{equation}
Finally, to measure regularity  along the boundary, we use 
\beq\label{bdynorm}
  |f|_{L_t^p\tilde{H}^s(\partial\Omega)}=\sum_{j=0}^{[s]}|(\varepsilon\partial_t)^j f|_{L^p([0, t], {H}^{s-j}(\partial\Omega))}. 
  \eeq

Let us recall, how the vector fields $Z_{j}$, $1 \leq j \leq M$ can be defined. We consider  
  $\Omega\in \mathbb{R}^3$   a smooth %simply connected 
  domain whose  %(the following construction and our results are  actually valid as long as the boundary of $\Omega$ 
boundaries can be covered by a finite number of charts. More precisely, there exist sets $\Omega_i, \, i=0,\cdots N$ such that %therefore, there exists a  covering such that :
 \begin{equation}\label{covering1}
     \Omega\subset\Omega_0\cup_{i=1}^{N} \Omega_i,\quad  \Omega_0\Subset \Omega, \quad \Omega_i\cap \partial \Omega\neq \varnothing, 
 \end{equation}
and  $\Omega_{i}\cap \p\Omega$ is the graph of a smooth  function $z=\varphi_i(x_1,x_2)$. Note that any smooth bounded domain or exterior domain outside a smooth compact set in $\mR^3$ admit such property.

In $\Omega_{0}$, we just take the vector fields $\partial_{k}$, $k=1, \, 2, \, 3$.
 To define  appropriate  vector  fields near the boundary, we use the local coordinates  in each  $\Omega_{i}:$ 
 \begin{equation}\label{local coordinates}
 \begin{aligned}
 \Phi_{i}:\quad &
 (- \delta_{i}, \delta _{i})^2 \times (0, \epsilon_{i})
 \rightarrow \Omega_{i}\cap\Omega\\
 &\qquad\quad(y,z)^{t}\rightarrow \Phi_{i}(y,z)=(y, \varphi_i(y)+z)^{t}
 \end{aligned}
 \end{equation}
% We denote the smooth extension of the unit outward normal vector by $\bn$ which is defined in $\Omega_i$:
%  $\bn(\Phi_i(y,z))=\frac{1}{|\textbf{N}|}\textbf{N}$
% with $\textbf{N}(\Phi_i(y,z))=(\partial_1\varphi_i(y),\partial_2\varphi_i(y),-1)^{t}.$ Note that $n$ does not depend on the third variable. Define also the projection $\Pi=\text{Id}-\bn\otimes \bn.$ 
% By this definition, the boundary conditions can be reformulated as:
% \begin{equation}\label{bdryconditionofu}
%    u\cdot \bn|_{\partial\Omega}=0 \quad \Pi (\partial_{\bn} u)=\Pi[ -2a u+(D \bn)u],
% \end{equation}
% where $(D\bn) u=\sum_{j=1}^3  D \bn_j u_j.$
 and we define the vector fields (up to some smooth cut-off functions compactly supported in $\Omega_{i}$) as :
 \begin{equation}\label{spatialvectorinlocal}
Z_{k}^{i}=\partial_{y^k}=\partial_k+\partial_k\varphi_i\partial_3, \quad k=1,2\qquad Z_{3}^{i}=\phi(z)%\p_3.%
(\partial_1\varphi_1\partial_1+\partial_2\varphi_1\partial_2-\partial_3),
\end{equation}
  where $\phi(z)=\frac{z}{1+z},$ and  $\partial_k, k=1,2,3$ are   the derivations with respect to the original  coordinates of $\mathbb{R}^3.$ %We remark that the normal derivative $\p_{\bn}=\partial_1\varphi_1\partial_1+\partial_2\varphi_1\partial_2-\partial_3$ can be spanned by $\p_{y^1},\p_{y^2},\p_3.$
  We remark that if $\Omega=\mR_{+}^3,$ the conormal vector fields can be defined globally due to the flat boundary:
  \beqs
  Z_1=\p_{1}, \, Z_2=\p_{2}, \, Z_3=\f{z}{1+z}\p_z.
  \eeqs
  
  We shall denote by $\bn$ the unit outward normal to the the boundary.
  In each  $\Omega_{i}$, we can extend   it to $\Omega_{i}$   by setting
  $$\bn(\Phi_i(y,z))=\frac{1}{|\textbf{N}|}\textbf{N}, \quad \textbf{N}(\Phi_i(y,z))=(\partial_1\varphi_i(y),\partial_2\varphi_i(y),-1)^{t}.$$
 % Note that $\bn$ does not depend on the third variable {\color{red} unclear ?}. 
 In the same way, the projection on vector fields tangent to the boundary, 
   $$\Pi=\text{Id}-\bn\otimes \bn$$
   can be extended in $\Omega_{i}$ by using the extension of $\bn$.

Let us observe that by identity
\beqs
\Pi(\p_{\bn}u^{\ep})=\Pi((\nabla u^{\ep})\bn)=2\Pi(\mathbb{S}u^{\ep} )-\Pi ((D u^{\ep})\bn)
\eeqs
with $[(\nabla u^{\ep}) \bn]_i=\sum_{j=1}^3 \bn_k \p_k u_i^{\ep}, [(D u^{\ep}) \bn]_i=\sum_{k=1}^3  \p_i u_k^{\ep} \bn_k,$ 
the boundary conditions
\eqref{bdyconditions}
 can be reformulated as:
 \begin{equation}\label{bdryconditionofu}
    u^{\ep}\cdot \bn=0, \quad \Pi (\partial_{\bn} u^{\ep})=\Pi[ -2a u^{\ep}+(D \bn)u^{\ep}], \quad \p_{\bn}\theta^{\ep}=0 \text{ on }  {\partial\Omega},
 \end{equation}
 where $[(D\bn) u^{\ep}]_i=\sum_{k=1}^3  \p_i \bn_k u_k^{\ep}.$ 
\subsection{Main results}

Since the pressure $P^{\epsilon},$ the temperature $\cT^{\ep}$ are strictly positive, we denote 
$$\sigma^{\ep}=\colon\f{\ln P^{\ep}-\ln \overline{P}}{\ep},\quad \theta^{\ep}=\colon \ln \cT^{\ep}-\ln \overline{\cT} $$
or equivalently,
$$P^{\ep}= \overline{P}e^{\ep\sigma^{\ep}}, \quad \cT^{\ep}=\overline{\cT} e^{\theta^{\ep}}$$
where $\overline{P},\overline{\cT}$ are two reference constants. 
The system \eqref{NCNS-S1} can then be written in the following way:
\beq\label{NCNS-S2}
 \left\{
\begin{array}{l}
 \displaystyle\f{1}{\gamma}(\pt+u^\varepsilon\cdot \nabla) \sigma^{\ep}+\f{\div u^{\ep}}{\ep}-\f{\gamma-1}{\gamma\ep}\kappa \Gamma(\ep\sigma^{\ep})\div(\beta(\theta^{\ep})\nabla\theta^{\ep})=\f{\gamma-1}{\gamma}\mathfrak{N}^{\ep} %\mu\ep^2\Gamma(\ep\sigma^{\ep})\cL u^{\ep}\cdot \mathbb{S}u^{\ep}
 ,\\[5pt]
 \displaystyle \f{1}{R\beta(\theta^{\ep})}(\pt+  u^{\varepsilon}\cdot\nabla) u^{\ep}+\f{\nabla \sigma^{\ep}}{\ep}-\mu\Gamma(\ep\sigma^{\ep}) \,\div\mathcal{L}u^{\varepsilon}=0,  \qquad \text{$(t,x)\in \mathbb{R}_{+}\times \Omega $}, \\[5pt]
 \displaystyle \f{C_v}{R}(\pt+u^{\ep} \cdot\nabla)\theta^{\ep}+\div u^{\ep}-\kpa\Gamma(\ep\sigma^{\ep})\div(\beta(\theta^{\ep})\nabla\theta^{\ep})=\ep\mathfrak{N}^{\ep},
\end{array}
\right.\\
\eeq
where we denote:
\beq\label{defgamma-beta}
\Gamma(\ep\sigma^{\ep})=\f{1}{P^{\ep}}(\ep\sigma^{\ep}),\quad \beta(\theta^{\ep})=\cT^{\ep}(\theta^{\ep}), \quad \mathfrak{N}^{\ep}=\mu\ep\Gamma(\ep\sigma^{\ep})\cL u^{\ep}\cdot \mathbb{S}u^{\ep}.
\eeq
From now on, we will work on the above system. 

 In the study of the strong solution of the above hyperbolic-parabolic system with parameters $(\ep,\mu,\kpa)$ small, there shall be two kinds of boundary layers-the viscous boundary layer with width $(\ep\mu)^{\f{1}{2}}$ (incorporated with the fast oscillation effects)  due to the non-zero viscous term in the equation of the velocity $u^{\ep},$
 the thermal boundary layer with width $\kpa^{\f{1}{2}}$ due to the non-vanishing thermal conductivity in the equation of $\theta^{\ep}.$  The interactions of these two kinds of boundary layers can be quite involved, which may indeed prevent one from establishing uniform in $(\ep,\mu,\kpa)$ estimates for \eqref{NCNS-S2}. 
 However, we are able to construct such uniform estimates provided the following assumption on $\mu$ and $\kpa$ holds:
 \begin{align}\label{assumption-mukpa}
      \big|\mu-\f{\kpa}{C_v\gamma \lambda_1}\big|\lesssim \mu\kpa^{\f{1}{2}}.
 \end{align}
 We will explain why this assumption needs to be imposed in the next subsection.
 
 %under the following assumption on 
 %, which, combined with the fast oscillation effects, 
 %due to the appearance and interactions of two different kinds of boundary layers: viscous boundary layer and thermal boundary layer 

In order to establish uniform high regularity estimates, 
we shall use the following quantity:
\beq\label{defcalNmt}
\mathcal{N}^{\mu,\kpa}_{m,T}(\si^{\ep},u^{\ep}, \theta^{\ep})=\mathcal{E}_{m, T}^{\mu,\kpa}(\si^{\ep},u^{\ep}, \theta^{\ep})+\mathcal{A}_{m,T}^{\mu,\kpa}(\si^{\ep},u^{\ep}, \theta^{\ep})
\eeq
where
 \begin{align}\label{defcEmt}
 \mathcal{E}^{\mu,\kpa}_{m,T}(\si^{\ep},u^{\ep},  \theta^{\ep})=
 \mathcal{E}^{\mu,\kpa}_{m,T}(\si^{\ep},u^{\ep})+\mathcal{E}_{m,T}^{\kpa}(\theta^{\ep}),\,\,\, \mathcal{A}^{\mu,\kpa}_{m,T}(\si^{\ep},u^{\ep},  \theta^{\ep})=\mathcal{A}_{m,T}^{\mu, \kpa}(\si^{\ep},u^{\ep})
  + \mathcal{A}^{\kpa}_{m,T}(\theta^{\ep}), 
 \end{align}
 with
 \beq\label{defcEmsigmau}
 \begin{aligned}
 \mathcal{E}^{\mu,\kpa}_{m,T}(\si^{\ep},u^{\ep})
& =\colon \ep\mu^{\f{1}{2}}\big(\|(\sigma^{\ep}, u^{\ep})\|_{L_T^{\infty}\underline{H}_{co}^{m}}+\|\na(\sigma^{\ep}, u^{\ep})\|_{L_T^{\infty}H_{co}^{m-1}}+\kpa^{\f{1}{2}}\|\na u^{\ep}\|_{L_T^{2}\underline{H}_{co}^{m}}\big)\\
&\quad+ 
\ep^{\f{1}{2}}\mu\|\na^2 u^{\ep} \|_{L_T^{2}H_{co,\sqrt{\ep}}^{m-1}}+\ep^{\f{2}{3}}\mu^{\f{1}{2}}\|\na^2 u^{\ep} \|_{L_T^{2}H_{co}^{m-2}}+\ep\mu\|\na^2 (\sigma^{\ep}, u^{\ep})\|_{L_T^{\infty}H_{co}^{m-2}} \\
 %\\
 &\quad  +\|(\sigma^{\ep}, u^{\ep})\|_{L_T^{\infty}H_{co}^{m-1}}+\|\na(\si^{\ep},u^{\ep})\|_{L_T^{\infty}H_{co}^{m-2}}+\mu^{\f{1}{2}}\|\na(\si^{\ep},u^{\ep})\|_{L_T^{2}H_{co}^{m-1}}\\
 &\qquad\qquad\qquad+ \kpa^{\f{1}{2}}\|\na(\sigma^{\ep}, \div u^{\ep})\|_{L_T^{2}H_{co}^{m-2}} +\|\na^2\sigma^{\ep}\|_{L_T^{\infty}H_{co}^{m-3}},
 \end{aligned}
 \eeq
 \beq\label{defcEmtheta}
\begin{aligned}
   \cE_{m,T}^{\kpa}(\theta^{\ep})&= \kpa^{\f{1}{2}}\|\theta^{\ep}\|_{L_T^{\infty}\underline{H}_{co}^{m}}+\|\theta^{\ep}\|_{L_T^{\infty}{H}_{co}^{m-1}}
   +\|\nabla\theta^{\ep}\|_{L_T^{\infty}H_{co,\sqrt{\kpa}}^{m-1}}\\
   &\quad + 
   \kpa^{\f{1}{2}}\|\na^2\theta^{\ep}\|_{L_T^{\infty}H_{co,\sqrt{\kpa}}^{m-2}\cap L_T^2H_{co,\sqrt{\kpa}}^{m-1}}+\kpa \|\na^3\theta^{\ep}\|_{L_T^{2}H_{co,\sqrt{\kappa}}^{m-2}},
\end{aligned}
\eeq
\beq\label{defcAmt}
\begin{aligned}
  &\mathcal{A}_{m,T}^{\mu,\kpa}(\si^{\ep},u^{\ep})
  =\colon \il (\si^{\ep},u^{\ep})\il_{m-3,\infty,T}+ \il (\na\si^{\ep}, (\ep\mu)^{\f{1}{2}}\na u^{\ep})\il_{m-4,\infty,T}\\
 &\qquad\qquad\qquad\qquad\qquad\qquad+ \il (\nabla u^{\ep}, \ep\mu^{\f{1}{2}}\na^2 (\sigma^{\ep}, u^{\ep})) %(\ep\mu)^{\f{1}{2}}\na^2 \sigma^{\ep})
 \il_{m-5,\infty,T},\\
   &\mathcal{A}_{m,T}^{\kpa}(\theta^{\ep})=\colon \il (\theta^{\ep}, \kpa \na\theta^{\ep})\il_{m-3,\infty,T}
 +\kpa^{\f{1}{2}}\il \na\theta^{\ep}\il_{m-4,\infty,T}\\
 &\qquad\qquad\qquad\quad +\il
  (\nabla \theta^{\ep}, \kpa\na^2\theta^{\ep})\il_{m-5,\infty,T}
 + \kpa^{\f{1}{2}}\il\na^2\theta^{\ep}\il_{m-6,\infty,T}.
 \end{aligned}
 \eeq
 Note that the norms involved in the above definitions are defined in \eqref{not1}-\eqref{not3}.
\begin{definition}[Compatibility condition]
We say that $(\sigma_0^{\ep},u_0^{\ep},\theta_0^{\ep})$ satisfy the compatibility conditions up to order $m$ if for any $j=0,1\cdots m-1,$
\beqs
(\varepsilon\partial_t)^{j}(u^{\ep}\cdot\bn)\big|_{t=0}=(\ep\pt)^j (\p_{\bn}\theta^{\ep})|_{t=0}
=0, \quad \Pi \big[\mathbb{S}\big((\varepsilon\partial_t)^{j}u^{\ep}|_{t=0}\big)\bn\big]=-a\Pi \big[(\varepsilon\partial_t)^{j}u^{\ep}|_{t=0}\big] \,\text{ on } \, \partial\Omega.\, 
\eeqs
\end{definition}
Note that the restriction of the time derivatives of the solution at the  initial time can be expressed  inductively by using the equations. For example, we have
$$(\varepsilon\partial_t u^{\varepsilon})(0)=R\beta(\theta^{\ep}_0)%\Gamma(\ep\sigma_0^{\ep})
\big(\ep\mu\Gamma(\ep\sigma_0^{\ep}) \div \mathcal{L}u_{0}^{\ep}-\nabla\sigma_0^{\ep}\big)-\varepsilon u_0^{\ep}\cdot\nabla u_0^{\ep}.$$ 
We thus define the admissible space for the initial data as 
\begin{align*}
Y_{m}=\bigg\{(\sigma^{\ep}_0,u_0^{\ep},\theta^{\ep}_0)\big|%& H^2(\Omega)^4, \quad 
  Y_{m, \mu,\kpa}^{\ep}(0)<+\infty,
 (\sigma_{0}^{\ep},u_0^{\ep},\theta_0^{\ep}) \text{ satisfy  the compatibility  conditions up to order $m$}
\bigg\}
\end{align*}
where
\begin{equation}\label{initialnorm}
\begin{aligned}
 Y_{m, \mu,\kpa}^{\ep}(0)
 &=: \|\mu^{\f{1}{2}}(\ep\sigma^{\ep}, \ep u^{\ep}), \kpa^{\f{1}{2}}\theta^{\ep})(0)\|_{\underline{H}_{co}^m}
 + \|\mu^{\f{1}{2}}\na(\ep\sigma^{\ep}, \ep u^{\ep}), \kpa^{\f{1}{2}}\na\theta^{\ep})(0)\|_{{H}_{co}^{m-1}} \\
&\quad +\|(\si^{\ep},u^{\ep},\theta^{\ep})(0)\|_{H^{m-1}_{co}}+\|\big(\nabla (\si^{\ep},u^{\ep}, \theta^{\ep}), \pt\theta^{\ep}\big)(0)\|_{H^{m-2}_{co}}+\|\ep\mu\na^2\sigma^{\ep}(0)\|_{H_{co}^{m-2}}%+\|\na^2\sigma(0)\|_{\hco^{m-3}}
 \\&\qquad %+\|\na\sigma^{\ep}(0) \|_{m-4,\infty}+
 +\|\big(\na(\theta^{\ep}, u^{\ep}), \kpa^{\f{1}{2}}\na^2 \theta^{\ep}, \ep\mu^{\f{1}{2}}\na^2 (\sigma^{\ep}, u^{\ep})\big)(0)\|_{m-5,\infty}.
 \end{aligned}
\end{equation}
We refer to \eqref{normfixt} for the definition of the above norms. Note that by using inductively the equations to express
the time derivatives, $Y_{m, \mu,\kpa}^{\ep}(0)$ can indeed be expressed in terms of the initial data. 

The main result of this work is the following uniform estimates:

%The first result concerns the uniform regularity estimates in the Mach number $\ep$ when the Reynolds number $\mu^{-1}$ and p\'eclet number $\kpa^{-1}$ are fixed:
%In the next theorem, we state the 
%This theorem is a consequence of the following result:  
\begin{thm}[Uniform estimates %II-varying ($\mu,\kpa$)
]\label{thm1}
Given an integer $m\geq 7$ and a $C^{m+2}$ smooth domain $\Omega$ which admits the property \eqref{covering1}.
Define the set  \begin{align}\label{defsetA}
     A=\big\{ (\mu,\kpa)\in (0,1]^2\big| \,\mu,\kpa \textnormal{ satisfy the relation \eqref{assumption-mukpa}} \big\}.
 \end{align}
 Consider a family of  initial data  depending only on $\ep,$ such that  $(\sigma_0^{\ep},u_0^{\ep}, \theta_0^{\ep})\in Y_{m}, \, u_0^{\ep}|_{\p\Omega}=0$
and  
\beq\label{aption-initial}
\begin{aligned}
&\qquad\quad
%\sup_{(\mu, \kpa) \in A}
\sup_{(\mu, \kpa) \in A,\ep\in(0,1]}Y_{m, \mu,\kpa}^{\ep}(0)<+\infty.
\end{aligned}
\eeq
There exist $\ep_0\in (0, 1]$ and $T_{0}>0,$ such that, for any $0< \ep\leq \ep_0$, any $(\mu,\kpa)\in A,$ the
system \eqref{NCNS-S2}, \eqref{bdryconditionofu} has a unique solution $(\si^{\ep},u^{\ep}, \theta^{\ep})$
which satisfies:
% \beqs -2{\bar{c}}\leq \ep\sigma^{\ep}(t,x),\,  \theta^{\ep}(t,x)\leq 2\bar{c},\quad  \forall (t,x)\in [0,T_0]\times\Omega, \eeqs and 
 \beq\label{uniformes}
 %\sup_{(\mu, \kpa) \in A} \sup_{\ep \in (0, \ep_{0}]}
 \sup_{(\mu, \kpa) \in A,\ep\in(0,1]}\mathcal{N}_{m,T_0}^{\mu,\kpa}(\si^{\ep},u^{\ep}, \theta^{\ep})< +\infty,
 \eeq
where $\mathcal{N}_{m,T_0}^{\mu,\kpa}(\si^{\ep}, u^{\ep}, \theta^{\ep})$ is defined in \eqref{defcalNmt}. 
\end{thm}
\begin{rmk}
Since our main focus is on the low Mach number limit, we only keep the dependence of the solutions on the Mach number $\ep,$ but when $\mu, \kpa$ vary, they do depend on $(\mu, \kpa).$
\end{rmk}
\begin{rmk}
In view of \eqref{uniformes}, 
\eqref{defcEmt} and the third line of \eqref{defcEmsigmau}, 
we have the following  uniform in $\ep$ estimates for 
$(\sigma^{\ep}, u^{\ep}):$ 
\begin{align}\label{prop-uni-sigmau}
    \|(\sigma^{\ep}, u^{\ep})\|_{L_{T_0}^{\infty}H_{co}^{m-1}}+\|\na (\sigma^{\ep}, u^{\ep})\|_{L_{T_0}^2H_{co,\sqrt{\mu}}^{m-1}}\lesssim \cE_{m,{T_0}}^{\mu,\kpa}(\sigma^{\ep}, u^{\ep})<+\infty.
\end{align}
Due to the bad commutation properties of the space conormal derivatives with the singular part of the system, we can  only control $(\sigma^{\ep}, u^{\ep})$ uniformly (in $\ep, \mu, \kpa$)  to the regularity $m-1.$ Moreover, due to the dissipation effects, we can also control the highest regularity in $L_t^2$ type norm with a weight $\mu^{\f{1}{2}}.$
\end{rmk}
\begin{rmk}
In light of \eqref{defcEmtheta}, we have the following uniform in $(\ep, \mu, \kpa)$ estimates for $\theta^{\ep}$ in the $m-1$ regularity
\begin{align}\label{uni-theta-intro}
  % \|\pt\theta^{\ep}\|_{L_{T_0}^{\infty}H_{co}^{m-2}}+
  \|\theta^{\ep}\|_{L_{T_0}^{\infty}H_{co}^{m-1}}+\|(\na, \pt)\theta^{\ep}\|_{L_{T_0}^{\infty}H_{co}^{m-2}}\lesssim \cE_{m,{T_0}}^{\kpa}(\theta^{\ep}),
\end{align}
and $\kappa-$dependent estimates in highest regularity:
\begin{align*}
   \kpa^{\f{1}{2}} \big(\|\theta^{\ep}\|_{L_{T_0}^{\infty}\underline{H}_{co}^{m}}+\|\na\theta^{\ep}\|_{L_{T_0}^{\infty}H_{co}^{m-1}}\big)\lesssim \cE_{m,{T_0}}^{\kpa}(\theta^{\ep}).
\end{align*}
This is because that,  $\theta^{\ep}$ solves $\eqref{NCNS-S2}_3$-the convection-diffusion system, transported by the velocity $u^{\ep},$ which admits the property \eqref{prop-uni-sigmau}.
\end{rmk}
\begin{rmk}\label{rmkmuapproxkpa}
%Due to the interaction between thermal boundary layer and hydrodynamic boundary layer, it is natural that the parameters $\mu$ and $\kappa$ are related. 
%The assumption \eqref{assumption-mukpa} is made mainly  uniform control of. Moreover, 
By the assumption \eqref{assumption-mukpa} on $\mu, \kappa,$ we can find a constant $C_0>1,$ such that
$\f{1}{C_0} \mu \leq \kappa\leq C_0 \mu$ which means that $\kpa$ and $\mu$ are proportional. 
The more restrictive assumption  $\eqref{assumption-mukpa}$ is %useful in the
made to control the $L_t^{\infty}H_{co}^{m-2}$ norm of $\na u^{\ep},$ see Remark \ref{rmk-mukpa-1}. Nevertheless, let us comment that, in the above result, the Mach number $\ep$ and the Reynolds number $\mu^{-1}$ can be completely independent. 
\end{rmk}
\begin{rmk}
One may expect that when $\mu,\ep>0, \kpa=0,$ the uniform estimates in $(\ep,\mu)$ would be much easier to obtain due to the lack of thermal boundary layer.
%, as the thermal boundary layer will not show up.
It turns out not to be the case. Indeed, when $\kpa=0,$ as the temperature is freely transported by the velocity $u^{\ep},$ its second normal derivative %is likely to 
has a size 
$(\ep\mu)^{-\f{1}{2}}.$ This would be a crucial issue when proving the uniform estimates for the (modified) vorticity %$\curl(r_0 u^{\ep})$ 
which solves a transport-diffusion equation with a source term involving second normal derivative of $\theta^{\ep}.$ 
\end{rmk} 
\begin{rmk}\label{rmk-smallvartheta}
Based on the uniform estimate \eqref{uniformes}, if in addition, the temperature admits small variation or in other words $\theta^{\ep}=\ep \theta_1^{\ep},$ where $$\sup_{\ep\in (0, 1]}\big(\|\theta_1^{\ep}|_{t=0}\|_{ H_{co}^{m-1}(\Omega)}+\|\na\theta_1^{\ep}|_{t=0}\|_{ H_{co}^{m-2}(\Omega)\cap L^{\infty}(\Omega)}\big)<+\infty, $$
%the assumption \eqref{} is made,
then by following the calculations in Section 4 and Section 8, this property can be propagated, namely:
\begin{align}
\sup_{(\mu, \kpa) \in A}
 \sup_{\ep \in (0, \ep_{0}]} \big( 
 \|\theta_1\|_{L_{T_0}^{\infty}H_{co}^{m-1}}+\|\na \theta_1\|_{L_{T_0}^{\infty}H_{co}^{m-2}\cap L^{\infty}([0,T_0]\times\Omega)} 
 \big)< +\infty.
\end{align}
\end{rmk}
\begin{rmk}
In order to prove uniform estimates in $\mu,\kpa$, it is assumed, due to some technical reasons (explained at the end of Subsection \ref{subs133}), in Theorem \ref{thm1} that the initial velocity vanishes on the boundary. However, it is somehow not a very stringent assumption in the context of the low Mach number limit problem. Indeed,
the compatibility conditions together with the assumption \eqref{aption-initial}
imply that  the data are  prepared (in the sense that the data can be small compared to $\ep$) on the boundary.  Nevertheless, this assumption can be dropped if one only proves uniform estimates in the Mach number $\ep$, see the following theorem.
\end{rmk}
Following the similar (and easier) arguments as in the proof of Theorem \ref{thm1}, we have also:
\begin{thm}[Uniform estimates fixed $(\mu,\kpa)$]
Assume that $\Omega$ is a $C^{m+2} (m\geq 7)$ smooth domain that admits the property \eqref{covering1}.
Let $\mu,\kpa\in (0,1]^2$ be fixed and 
suppose the initial data $(\sigma_0^{\ep},u_0^{\ep}, \theta_0^{\ep})\in Y_{m}$ satisfying
\beqs
\begin{aligned}
&\qquad\quad
\sup_{\ep\in(0,1]}Y_{m, \mu,\kpa}^{\ep}(0)<+\infty,
\end{aligned}
\eeqs
then the system \eqref{NCNS-S2}, \eqref{bdryconditionofu} admits a unique solution $(\si^{\ep},u^{\ep}, \theta^{\ep})$ which enjoys the uniform in $\ep$ estimate:
\begin{align}\label{unies-fixmukap}
\sup_{\ep \in (0, \ep_{1}]}\mathcal{N}_{m,T_1}^{\mu,\kpa}(\si^{\ep},u^{\ep}, \theta^{\ep})<_{\mu,\kpa}+\infty.
\end{align}
\end{thm}

To prove Theorem \ref{thm1}, the crucial step is to show the following uniform a priori estimate
which is the \textbf{heart} of this paper:
\begin{prop}\label{prop-uniform es}
Assume that  $(\sigma_0^{\ep},u_0^{\ep}, \theta_0^{\ep})$ satisfy the same assumptions as in Theorem \ref{thm1}.
Let $\bar{c}>0$ be a constant such that: 
\beqs -{\bar{c}}\leq \ep\sigma_0^{\ep}(x),\,  \theta_0^{\ep}(x)\leq \bar{c},\quad  \forall x\in \Omega, \ep \in (0,1].
\eeqs 
Denote $r_0(s_1, s_2)=\f{1}{R\beta(s_1)}\exp({\f{Rs_2}{C_v\gamma}}).$
Let $c_0\in (0,1]$ be such that:
\beq\label{preasption1}
\begin{aligned}
\forall s  \in  \big[-3{\bar{c}},  3\bar{c}\big],&  \,\quad  c_0\leq \Gamma(s), \beta(s) \leq 1/{c_0}, \,  \quad |(\Gamma,\beta)|_{C^m\big(\big[-3{\bar{c}},  3\bar{c}\big]\big)}\leq {1}/{c_0}, \\
&\forall\, (s_1, s_2) \in  \big[-3{\bar{c}},  3\bar{c}\big]^2, \,\, c_0\leq r_0(s_1,s_2)\leq {1}/{c_0}.
\end{aligned}
\eeq
Assume that  for some $T \in (0, 1]$ the following assumption holds:
\begin{equation}\label{preasption}
-3{\bar{c}}\leq \ep\sigma^{\ep}(t,x), \ep \tsigma^{\ep}(t,x),  
\theta^{\ep}(t,x)\leq 3\bar{c}, \quad \forall (t,x)\in [0,T]\times\Omega,  \forall \ep\in[0,1],
\end{equation}
where $\tsigma^{\ep}=\sigma^{\ep}-\ep\mu(2\lambda_1+\lambda_2){\Gamma}(\ep\sigma^{\ep})\div u^{\ep}.$
Then, there exist  
two  polynomials $\Lambda_0, \Lambda_1$ (whose coefficients are independent of $\ep, \mu, \kpa$), and a constant $\vartheta_0>0,$
 such that, for any $\ep \in (0, 1], (\mu,\kpa)\in A,$ we have for a smooth enough solution of \eqref{NCNS-S2}  on  $[0, T]$
 the following estimate:
 \beq\label{enerineq}
 \mathcal{N}_{m,T}^{\mu,\kpa}(\si^{\ep}, u^{\ep}, \theta^{\ep})\leq \Lambda_0\big( \f{1}{c_0}, 
 Y_{m, \mu,\kpa}^{\ep}(0)\big)+
 (T+\ep)^{\vartheta_0} \Lambda_1\big(\f{1}{c_0}
,\mathcal{N}_{m,T}^{\mu,\kpa}(\sigma^{\ep},u^{\ep},\theta^{\ep})\big),
 \eeq
 where $
Y_{m, \mu,\kpa}^{\ep}(0),\, \mathcal{N}^{\mu,\kpa}_{m,T}(\si^{\ep}, u^{\ep}, \theta^{\ep})$ are defined in \eqref{initialnorm} \eqref{defcalNmt} respectively.
\end{prop}

Based on the uniform estimates established above, it is possible to have some convergence results in the low Mach number and  inviscid limits. %Since they are essentially not the main contributions of this work, we shall only list some of them in Section 14.   

In light of the uniform estimate \eqref{uni-theta-intro},
up to extraction of subsequences,  $\theta^{\ep}$ converges in $C([0,T_0], H_{loc}^1(\Omega))$ to some $\theta^0$ as $\ep$ tends to $0.$  However, due to the lack of time compactness, one can only conclude from \eqref{prop-uni-sigmau} that
 $(\sigma^{\ep}, u^{\ep})$ converge in $L_w^2([0,T_0], L_{loc}^2(\Omega))$ to some $(\sigma^0, u^0).$ When $\Omega$ is a bounded domain, the weak in time convergence for  $(\sigma^{\ep}, u^{\ep})$ 
 cannot be in general  
 improved, due to the 
 lack of large time dispersion and the interactions of fast oscillating acoustic waves with the physical boundaries. This will prevent one from justifying the limits
 $(u^0, \theta^0)$ to be the solutions of the following limit system which is equivalent to \eqref{NINS}:
 \beq\label{NINS-1}
 \left\{
\begin{array}{l}
 \displaystyle \gamma \overline{P}\,\div u^0=(\gamma-1)\kappa\div(\beta(\theta^0)\na\theta^0),\\
  \displaystyle \f{\overline{P}}{R\beta(\theta^0)}(\pt+u^0\cdot\nabla)u^0+\nabla\pi^0+ %\mu\div\cL u^0
  \mu\lambda_1\curl\curl u^0=0, \\
   \displaystyle  \f{C_v}{R} \gamma\overline{P} (\pt+u^0\cdot\nabla)\theta^0=\kappa\div(\beta(\theta^0)\na\theta^0), 
 \end{array}
\right.\\
 \eeq
 supplemented with the boundary condition
 \begin{align}\label{bc-limit}
     u^0\cdot\bn|_{\p\Omega}=\p_{\bn}\theta^0|_{\p\Omega}=0, \quad (\curl u^{0}\times \bn)|_{\p\Omega}
=2\Pi(-a u^{0}+D \bn\cdot u^{0} )|_{\p\Omega}.%\quad \text{on }  {\partial{\Omega}}.
     %\Pi (\partial_{\bn} u^{0})\big|_{\p\Omega}=\Pi[ -2a u^{0}+(D \bn)u^{0}]\big|_{\p\Omega}. %\text{ on } \p\Omega,
 \end{align}
%Note that the above system 
Nevertheless, when $\Omega$ is an exterior domain, 
due to the (global) dispersion property for the wave equation, 
%as the acoustic waves decay to zero at the infinity, 
the strong convergence in time can be obtained by 
proving local energy decay, either using the defect measure technique \cite{MR2106119} or using the so-called RAGE theorem  \cite{MR2575476, MR3729430}. 
%using semiclassical defect measures 
Consequently, %we are able to prove 
it is possible to show
 the following convergence result. %holds:
 %with our Navier-Slip boundary condition, it is generally not possible to 
%The weak convergence Granted the uniform estimates in Theorem \ref{thm1}, we expect to show that the solutions $(\sigma^{\ep}, u^{\ep}, \theta^{\ep})$ to the system \eqref{NCNS-S2} converge to $(0, u^0, \theta^0)$ which solves the equation:
\begin{thm}[Low mach number limit for fixed $\mu,\kpa $]\label{thm-conv1}
Assume that $\Omega$ is the exterior of a compact smooth domain in $\mR^3$ and  $\mu,\kpa\in (0,1]$ are fixed. Let $(\sigma_0^{\ep}, u_0^{\ep}, \theta_0^{\ep})\in Y_m$ and assume in addition:
\begin{align}\label{additionasp-theta}
   \sup_{\ep\in (0,1]}\| \ep^{\f{1}{2}}\pt\na\theta^{\ep}(0)\|_{H_{co}^1}<+\infty.
\end{align}
 Let $(\sigma^{\ep}, u^{\ep}, \theta^{\ep})\in C([0,T_1], L^2(\Omega))$
be the solutions of the systems \eqref{NCNS-S2} \eqref{bdryconditionofu} with fixed $\mu, \kpa \in (0,1]^2$ satisfying the uniform (in $\ep$) estimate 
\eqref{unies-fixmukap}.
Assume further that
$\theta^{\ep}$ converges in $C([0,T_1], H_{loc}^1(\Omega))\cap C([0,T_1]\times \Omega) $ to $\theta^0,$ which admits the following decay property:
%Suppose that the following decay property holds for $\theta^0$: 
\begin{align}\label{decaytheta0}
    |\theta^0(t,x)|\leq A |x|^{-1-\delta}, \quad
    |\na\theta^0(t,x)|\leq A |x|^{-2-\delta}, \quad\, \forall\, (t,x)\in [0,T_1]\times\Omega.
\end{align}
Then $\sigma^{\ep}\rightarrow 0, \, \div u^{\ep}-\f{\gamma-1}{\gamma}\kappa\Gamma(\ep\sigma^{\ep})\div(\beta(\theta^{\ep})\na\theta^{\ep})\rightarrow 0$ in $L^2([0,T_1], L_{loc}^2(\Omega)).$ Moreover, $u^{\ep}$ converges in $L^2([0,T_1], L_{loc}^2(\Omega))$ to some $u^0,$ and 
$(u^0, \theta^0)$ are the (unique) solutions to the system \eqref{NINS-1} that enjoys the additional regularity:
\begin{align}\label{addition-reg}
  ( u^0 ,\theta^0)\in  L_{T_1}^{\infty}\cH^{0,m-1},\quad (\nabla  u^0, \nabla \theta^0, \pt\theta^0)\in
  L_{T_1}^{\infty}\cH^{0,m-2}
\cap L^{\infty}([0,T_1]\times\Omega).
\end{align}
\end{thm}

%\begin{rmk}
%The extra condition \eqref{additionasp-theta} is used to show the uniform boundedness of $(\ep\kpa)^{\f{1}{2}}\pt\na^2\theta^{\ep}$ in $L_{T_1}^2H_{co}^{1},$ which is needed in the proof of the strong convergence of the penalized terms.\end{rmk}

%\begin{rmk}
%The proof of the convergence of the penalized terms relies on the use of semi-classical defect measures developed in \cite{MR1135919} and on the fast decay properties of the solution to the wave equation  governed by $\sigma^{\ep}$ with coefficients depending on $\theta^{\ep}.$ The latter requires the variable-coefficient wave equation to be a small perturbation at infinity of the constant-coefficient wave equation, which holds true provided the assumption \eqref{decaytheta0} is made. This strategy was first employed by Metivier and Schochet  \cite{MR1834114} in the study of the low mach number limit of the non-isentropic Euler equations in the whole space $\mR^3,$ and was then generalized by Alazard \cite{MR2106119} to the case of exterior domain. %We shall use almost for free the results in \cite{MR2106119}.  \end{rmk}
\begin{rmk}
Due to the lack of control of the second normal derivatives of the limit velocity $u^0,$ 
the solution to the second equations in \eqref{NINS-1} should be interpreted in the weak sense: 
for any $\psi\in C^{\infty}([0,T_1], C_c^{\infty}(\overline{\Omega}))$ with
 $\psi\cdot\bn|_{\p\Omega}=0,$ the following identity holds: for every  $0<t\leq T_1,$
 \beq\label{incom-EI}
 \begin{aligned}
 &\f{\overline{P}}{R}\int_{\Omega}\big(\f{u^0}{\beta(\theta^0)}\cdot\psi\big)(t,\cdot) \ \d x+\mu\lambda_1\izt\int_{\Omega}\curl u^0 \times \curl\psi %+(\lambda_1+\lambda_2)\div u^0\div\psi
 \ \d x\d s\\
 &=\f{\overline{P}}{R}\int_{\Omega}\big(\f{u^0}{\beta(\theta^0)}\cdot\psi\big)(0,\cdot) \ \d x
 -\izt\int_{\Omega}\bigg(\f{\overline{P}}{R}\f{u^0}{\beta(\theta^0)}\cdot\nabla u^0+\na\pi^0\bigg)\cdot\psi  \ \d x\d s\\
 &\qquad+\f{\overline{P}}{R}\izt\int_{\Omega}\f{u^0}{\beta(\theta^0)}\cdot\pt\psi+ \pt\big(\f{1}{\beta(\theta^0)}\big)u^0\cdot\psi \ \d x\d s%\\ &\qquad
 +\mu\lambda_1\int_0^t\int_{\p\Omega}
 2\Pi(-a u^{0}+D \bn\cdot u^{0} )\cdot\psi \ \d S_y \d s.
 \end{aligned}
 \eeq
\end{rmk}

The proof of the above theorem will be presented in Section 14.
% Thanks to the work of Metivier and Schochet  \cite{MR1834114}, and Alazard \cite{MR2106119} for the studies on the low Mach number limit of the non-isentropic Euler equations, the arguments to prove the above convergence result are somehow standard now, we thus skip the proof, the interested reader can refer to  Section 14 of the arxiv version of this manuscript. 
 
Granted by our uniform estimates in Theorem \ref{thm1},
it is also interesting to study the low Mach number and the inviscid limit in the same time.  More precisely, by assuming $\mu=\f{\kpa}{C_v\gamma\lambda_1}=\ep^{\alpha}, \alpha>0$ one expects to show that when $\ep \rightarrow 0,$ $(\theta^{\ep}, u^{\ep})$ will converge to the solution of the following inhomogeneous incompressible Euler system:
\begin{align*}
    \left\{
\begin{array}{l}
   \displaystyle  \f{C_v}{R} \gamma\overline{P}  (\pt+u^0\cdot\nabla)\theta^0=0,\\
  \displaystyle \f{\overline{P}}{R\beta(\theta^0)}(\pt+u^0\cdot\nabla)u^0+\nabla\pi^0=0, \quad \div u^0=0.
 \end{array}
\right.
\end{align*}
Due to some technical reasons, it is difficult to prove such a result when the temperature admits large variation (in other words, $\theta^{\ep}$ is only bounded).
%equipped %with large variation of the temperature. 
Nevertheless, in the spirit of Remark \ref{rmk-smallvartheta}
and the Theorem 3.1  in \cite{MR3070031}, the following convergence result holds:
\begin{thm}
Let $\Omega$ be the exterior of a compact smooth domain in $\mR^3.$ Let $\mu=\f{\kpa}{C_v\gamma\lambda_1}=\ep^{\alpha}, 0<\alpha<\f{10}{3}.$ Given $(\sigma_0^{\ep}, u_0^{\ep}, \theta_0^{\ep})$ satisfying the assumption stated in Theorem \ref{thm1} and also the following: $u_0^{\ep}\rightarrow u_0^0$ in $(L^2(\Omega))^3$ where $\mathbb{P}u_0^0\in (H^s(\Omega))^3 $ with 
$s>\f{5}{2}.$
%$(u_0^{\ep}, \theta_0^{\ep}/\ep)\rightarrow (u_0^0, \theta_1^0)$
Suppose further that
%$\theta^{\ep}|_{t=0}=\ep \theta_1^{\ep}|_{t=0},$ where 
$$\sup_{\ep\in (0, 1]}\big(\|\theta_0^{\ep}/\ep\|_{ H_{co}^{m-1}(\Omega)}+\|\na\theta_0^{\ep}/\ep\|_{ H_{co}^{m-1}(\Omega)\cap L^{\infty}(\Omega)}\big)<+\infty. $$
Then $(u^{\ep}, \theta^{\ep}/\ep )$ converges in $L_{loc}^{\infty}\big((0, \tilde{T}_0], (L_{loc}^2(\Omega))^3\big)\times L_{loc}^{\infty}\big((0, \tilde{T}_0], L_{loc}^q(\Omega)\big), (1\leq q <2)$ 
to $(v^0, \theta_1^0)$ which are the solutions to the 
following Euler-Boussinesq equations:
\begin{align}\label{euler-bou}
    \left\{
\begin{array}{l}
   \displaystyle  \f{C_v}{R} \gamma\overline{P}  (\pt+v^0\cdot\nabla)\theta_1^0=0,\\
  \displaystyle \f{\overline{P}}{R\beta(0)}(\pt+v^0\cdot\nabla)v^0+\nabla\pi^0=0, \quad \div v^0=0,\quad v^0|_{t=0}=\mathbb{P}u_0^0,
 \end{array}
\right.
\end{align}
where we denote $\tilde{T}_0=\min[T_0, T_{max}]$
where $T_{max}$ is the maximal existence time of the solution to $\eqref{euler-bou}_2$ in $H^s(\Omega).$
\end{thm}
The proof is based on the dispersive estimate which is admissible in the case of smooth exterior domain, together
with the relative entropy method which
has been successfully  applied in many singular limit problem  for the weak solutions to the viscous systems. 
The above theorem has been shown in \cite{MR3070031} in the framework of the weak solution, and thus can also be  used for solutions enjoying our additional regularity, we will thus  omit the proof.

\subsection{Difficulties and strategies}
In this subsection, we explain the main difficulties in order to establish uniform in $(\ep,\mu,\kpa)$ estimate
\eqref{enerineq}. In what follows, we will skip the $\ep$-dependence of the solution for notational clarity.
\subsubsection{Uniform  estimates only in the Mach number.}

Indeed, there have already been several  difficulties when one tries to prove  uniform only in Mach number $\ep$ estimates by letting  $\mu$ and $\kpa$ fixed.  
 
%First of all, compared with the isentropic fluid systems, there are some new troubles because of the non-constant temperature. For instance, due to the appearance of the term  $\f{\gamma-1}{\gamma}\kappa \Gamma \div(\beta\nabla\theta)%^{\ep}) $ in the equation $\eqref{NCNS-S2}_1,$ the (linearized) penalized operator associated to the singular terms in the system \eqref{NCNS-S2} 
%\begin{align}\label{penalizedop}  \left(  \begin{array}{ccc}    0 &\div   & \f{\gamma-1}{\gamma}\kappa %\overline{\Gamma} \Gamma(0)\div(\beta(0)\nabla\cdot)\\ \nabla & 0 & 0\\ 0 & 0&\end{array} \right)\end{align}
%is not skew-symmetric. Therefore, in order to control uniformly the $L_t^{\infty}L^2$ norm of the solution as well as its high order weighted time derivatives, the system $\eqref{NCNS-S2}$ is not an appropriate target to perform energy estimates. To circumvent this problem, we consider, inspired by \cite{MR2211706}, the unknown $w=u-\kappa\f{\gamma-1}{\gamma}\Gamma(0) \beta(\theta)\nabla\theta$and work on the system satisfied by $(\sigma, w, \nabla\theta)$ which is amenable to the energy estimates.  
 
 Compared to the non-isentropic system in the whole space $\mR^3,$ some  difficulties arise due to the presence of the boundaries. Indeed, the main feature here is the appearance of the boundary layers with fast time oscillations. This prevents us from applying the strategies employed in the vanishing viscosity problems 
 \cite{MR2885569,MR3509004} where the boundary layer occurs
 without time oscillations. More precisely, in \cite{MR2885569,MR3509004}, one controls the high order tangential derivatives by direct energy estimates, and then uses the vorticity to control the normal derivatives. Nevertheless, for the system with low Mach number, even the tangential derivative estimates are not easy to get, since the spatial tangential derivatives 
 do not commute with  $\nabla,\div,$ defined with the standard derivations in $\mR^3$, and thus creates singular commutators.  Without a priori knowledge on the tangential derivatives, the normal derivatives are hard to control. To get around this difficulty, our general strategy is %to use the Leary projection 
 to split the velocity into a compressible part and an incompressible part, 
  and control the former by the divergence of the velocity and the latter by  direct energy estimates.
  In this way, we can recover all the tangential spatial derivatives and thus control the normal derivatives by careful studies on the vorticity.  Such a strategy has been employed in the previous work \cite{MR4403626} where uniform in (and only in) Mach number regularity estimates are established for \textit{isentropic} compressible Navier-Stokes system. However, as explained below, 
  some crucial modifications need to be made in the non-isentropic setting.
 
 As mentioned above, one crucial step to control  the spatial tangential derivatives is to get a uniform control of the incompressible part of the velocity. This requires
to apply the Leray projection (defined in \eqref{proj-P}) on the equation $\eqref{NCNS-S2}_2$ for the velocity $u$ and to perform a direct energy estimates on the resultant equations:
 \begin{align*}
     \f{1}{R\beta}(\pt+u\cdot\na)(\mathbb{P}u)+\na q-\mu\lambda_1\overline{\Gamma}\Delta(\mathbb{P}u)=-[\mathbb{P}, \f{1}{R\beta}]\pt u+\cdots
 \end{align*}
 where  $\cdots$ denotes some terms that can be uniformly bounded in $L_t^2H_{co}^{m-1}.$ 
 Nevertheless, due to the non-constant temperature $\beta(\theta),$ the commutator $[\mathbb{P}, \f{1}{R\beta}]\pt u$ does not vanish and thus by the ill-prepared assumption ($\pt u|_{t=0}$ is of order $\mathcal{O}(\ep^{-1})$), cannot be uniformly bounded in $\ep.$ Such a problem also happens when one studies the vorticity $\omega=\curl u$ which solves 
 the equations:
 \begin{align*}
     \f{1}{R\beta}(\pt+u\cdot\na)\omega-\mu\lambda_1\overline{\Gamma}\Delta \omega=-\na\big( \f{1}{R\beta}\big)\times \pt u+\cdots .
 \end{align*}
 Therefore, instead of working on the incompressible part and the curl of the velocity $u$, we have to study these counterparts of the modified velocity $\f{1}{R\beta(\theta)} u.$ %On one hand, 
 %On one hand, this will increase the computational complexity. On the other hand, 
  Their estimates would depend on the temperature in an essential way, which not only increase the computational complexity but also
 in turn lead to some technical issues in order to prove uniform in Reynolds number and P\'elect number estimates. 
 We shall explain the latter more carefully in the next subsection.
 %In the situation of the non-isentropic case, besides much computational complexity due to the appearance of the non-constant temperature, there are also several  difficulties in order to obtain the estimates which are uniform in Mach number, Reynolds number and P\'elect number. We will explain them in the following (from now on, we  skip the $\ep$ dependence of the solution for notational simplicity
\subsubsection{New difficulties for the estimates uniform in Reynolds number and P\'eclet number.}
As explained in the last section, we have to study the incompressible part of the quantity: $\f{1}{R\beta}u,$ which satisfies the equations:
\begin{align}\label{eqptu}
     (\pt+u\cdot\na)\big(\mathbb{P}(\f{1}{R\beta}u)\big)+\na q-\mu\lambda_1\overline{\Gamma}\Delta\mathbb{P}\big(\f{1}{R\beta}u\big)=(\f{1}{R\beta})'  u(\pt+u\cdot\na)\theta
  +\cdots
 \end{align}
 In view of the definition \eqref{defcEmt},
 to prove the uniform estimates \eqref{enerineq}, we need to control the $L_t^{\infty}H_{co}^{m-1}$ norm of the above system, which is expected to be accomplished by direct energy estimate. However, by the equations
 $\eqref{NCNS-S2}_3$ for $\theta,$ we have that
 \begin{align*}
     (\pt+u\cdot\na)\theta= \f{R}{C_v}\big(-\div u+\kpa\Gamma(\ep\sigma)\div(\beta(\theta)\nabla\theta)+\ep\mathfrak{N}\big),
 \end{align*}
 which is not uniformly (in $\mu$) bounded in $L_t^2H_{co}^{m-1}.$ Indeed, it turns out that one can  only control $\div u$ in $L_t^2H_{co}^{m-1}$  with a weight $(\mu+\kpa)^{\f{1}{2}}.$ This prevents one from proving the estimates that are uniform in $\mu, \kpa.$ The way to get over this problem is to  introduce the unknown $r_2 u$ where 
 \beqs
 r_2=\f{1}{R\beta(\theta)}\exp (\ep R{\sigma}/C_v\gamma)=\f{1}{R\beta(0)}\exp(-R\tilde{\theta}/{C_v}),\quad \tilde{\theta}=\f{C_v\theta}{R}-\f{\ep\sigma}{\gamma},
 \eeqs
 and then study the incompressible part of this new unknown which solves:
 \begin{align*}
     (\pt+u\cdot\na)\big(\mathbb{P}
    (r_2{u}) \big)+\na q-\mu\lambda_1\overline{\Gamma}\Delta\big(\mathbb{P} (r_2{u})\big)=-\f{R}{C_v} r_2 u(\pt+u\cdot\na)\tilde{\theta}
  +\cdots.
 \end{align*}
 Compared with \eqref{eqptu}, 
 the advantage of this equation is that now 
 $$(\pt+u\cdot\na)\tilde{\theta}=\gamma^{-1}(\kpa\Gamma\div(\beta\na\theta)+\ep\mathfrak{N})$$ which can be uniformly bounded in $\hco^{m-1}.$

Once the (modified) incompressible part $\mathbb{P}(r_2u)$ is controlled, the next
 essential step in order to recover one normal derivative  of the velocity is to study the modified vorticity $\curl(r_2 u)\times \bn=\colon \omega^{\bn}_{r_2},$ where some other difficulties arise.
%Let us now elaborate the difficulties in order to control  which is an essential step in order to recover one normal derivative of the velocity. 
By straightforward computations, we find that $\omega^{\bn}_{r_2}$ solves a transport-diffusion equation
\begin{align}
    \big(\pt+u\cdot \na-\mu\lambda_1\overline{\Gamma}r_2^{-1}\Delta\big)\omega^{\bn}_{r_2}&=\curl \bigg(-\f{R}{C_v}r_2 u(\pt+u\cdot\na)\tilde{\theta}+\mu\lambda_1\overline{\Gamma}
   r_2^{-1} [\curl\curl, r_2]u \bigg)\times \bn+\cdots\notag\\
   &=\colon \curl G_{r_2}\times\bn+\cdots \label{defG2}
\end{align}
with the non-homogeneous Dirichlet boundary condition:
\begin{align}\label{defbdomegan}
   \omega^{\bn}_{r_2}\big|_{\p\Omega}= \f{R\ep}{C_v\gamma}r_2\p_{\bn}\sigma \Pi u+2r_2\Pi\big(-a u+(D\bn)u\big)\big|_{\p\Omega}.
\end{align}
Note that here and in what follows, we denote $\cdots$ as the terms that can be easily controlled in the spaces we are interested in. 
Due to non-homogeneous Dirichlet boundary condition, %the lack of the information of the trace on the boundary, 
one cannot get uniform estimates for  $\omega^{\bn}_{r_2}$ by merely energy estimates. 
%he important feature here is that $\chi_i \omer^j \times \bn$ is governed by a transport-diffusion equation without any singular terms. Nevertheless, due to the non-homogeneous Dirichlet boundary condition \eqref{bcomernj}, the uniform estimates can not be obtained by simply performing energy estimates. 
In order to get the uniform estimate, we split $\omega^{\bn}_{r_2}$ %the system for $\chi_i(\omer^j\times \bn)$
into two parts, one of which solves a (variable coefficient) heat equation with the  nontrivial Dirichlet boundary condition and the vanishing source term and thus can be controlled by using the Green function.
The other satisfies a convection-diffusion equation with homogeneous Dirichlet boundary condition which is amenable to energy estimates.  We comment that such an idea has been successfully applied in our previous works \cite{masmoudi2021incompressible, MR4403626} for the %studies of incompressible limit
uniform in Mach number estimates for \textit{isentropic} system in domains with fixed or free boundaries. Nevertheless, as will be explained in the following, there are some significant differences for non-isentropic system and 
extra difficulties 
for uniform estimates in $\ep,\mu,\kpa.$

By applying the above strategy, we reduce the matter after the change of coordinates to the study of the following two systems in half space: 
    \begin{align}\label{defzeta1-intro}
    \left\{
     \begin{array}{l}
       \big(\pt -\mu \lambda_1 \overline{\Gamma}
       ( \f{{1}}{r_2})^{\Psi}\big|_{z=0}\p_z^2\big)\zeta_{1}=0 , \quad (t,x) \in \mR_{+}\times \mR^3 ,\\[3pt]
       \zeta_{1}|_{t=0}=0, \quad  \zeta_{1}|_{z=0}=
        \omega^{\bn}_{r_2}|_{\p\Omega}\circ \Psi,
     \end{array}   
 \right.
\end{align}
\begin{align}\label{defzeta2-intro}
\left\{
    \begin{array}{l}
   (\pt+ \tilde{u} \cdot \na){\zeta_2}-\mu \lambda_1\overline{\Gamma} ( \f{{1}}{r_2})^{\Psi}(\partial_z^2  +\Delta_{{g}})\zeta_2 =\cS_2,  \quad (t,x) \in \mR_{+}\times \mR^3,  \\[3pt]
       \zeta_{2}|_{t=0}= \omega^{\bn}_{r_2}|_{t=0}\circ \Psi , \quad  \zeta_{2}|_{z=0}=0,
     \end{array}   
 \right.
\end{align}
where $\Psi$ is the change of variable associated to the normal geodesic coordinates \eqref{normal geodesic coordinates}, $(\f{{1}}{r_0})^{\Psi}=\f{1}{r_0}\circ \Psi, 
\tilde{u}=(D\Psi)^{-\star}u,$ 
$\Delta_g$ is defined in \eqref{defdeltag} and involves only tangential derivatives. Moreover, 
\beqs
\cS_2=- (\tilde{u} \cdot \na){\zeta_{1}}+\mu \lambda_1\overline{\Gamma}\bigg(\big( \f{%\tilde{\chi}
1}{r_0}\big)^{\Psi}\Delta_{{g}}\zeta_{1}+ \big(\big( \f{1}{r_0}\big)^{\Psi}-\big( \f{1}{r_0}\big)^{\Psi}\big|_{z=0
}\big)\p_z^2\zeta_{1}\bigg)+ \big(\curl G_{r_2}\times\bn\big)^{\Psi}+\cdots.
\eeqs

In light of the first term in  $\cS_2,$  to bound $\zeta_2$ in the space $L_t^{\infty}H_{co}^{m-2},$ 
we are obliged to control uniformly $\zeta_1$ in $L_t^{2}H_{co}^{m-1}.$ This in turn requires, the boundedness of  $\omega^{\bn}_{r_2}|_{\p\Omega}$ in $L_t^2H^{m-\f{3}{2}}(\p\Omega).$ Nevertheless, such a bound, due to the quantity $\p_{\bn}\sigma$ in the expression of  $\omega^{\bn}_{r_2}$ (see \eqref{defbdomegan}), is unlikely to be uniform in $\kpa$ (remember that we can only expect $\kpa^{\f{1}{2}}\|\na^2\sigma\|_{\hco^{m-2}}$ to be uniformly controlled).  
The way to circumvent this problem is to %introduce
make further correction to the weight function $r_2.$  
More precisely, one can define 
\begin{align}\label{defr0-intro}
    r_0%r_2 \exp(-\ep^2\mu \bar{C}\Gamma\div u)
    =\f{1}{R\beta(\theta)}\exp (\ep R{\tsigma}/C_v\gamma)\quad 
    \text{ with }\, \tilde{\sigma}=\sigma-\ep\mu(2\lambda_1+\lambda_2)\Gamma\div u %\bar{C}=(2\lambda_1+\lambda_2) R/{C_v\gamma}.
\end{align}
and switch to study $\omega_{r_0}^{\bn}=\colon\curl(r_0 u)\times\bn,$ whose boundary condition is changed into 
\begin{align}\label{defbdomeganr0}
   \omega^{\bn}_{r_0}\big|_{\p\Omega}= \f{R\ep}{C_v\gamma}r_0\p_{\bn}\tsigma \Pi u+2r_0\Pi\big(-a u+(D\bn)u\big)\big|_{\p\Omega}.
\end{align}
Let us remark that the instant advantage  is that
by the boundary condition (derived from $\eqref{NCNS-S2}_2$),
\begin{align*}
\p_{\bn}\tsigma=-\big(\f{1}{R\beta}(\ep\pt+\ep u\cdot\na)u+\ep\mu\lambda_1\Gamma\curl\curl u+\ep\mu(2\lambda_1+\lambda_2)(\div u)\na\Gamma\big)\cdot\bn \quad \text{ on } \quad {\p\Omega},
\end{align*}
one can control $|\p_{\bn}\tsigma|_{L_t^2H^{m-\f{3
}{2}}(\p\Omega)}$ uniformly in $(\ep,\mu,\kpa).$ 

The above considerations lead us to write $\omega_{r_0}^{\bn}\circ \Psi=\tilde{\zeta}_1+\tilde{\zeta}_2,$ where  $\tilde{\zeta}_1, \tilde{\zeta}_2$ solve respectively the equation \eqref{defzeta1-intro} \eqref{defzeta2-intro} with $r_2$ changed into $r_0.$ 
Nevertheless, by the property \eqref{appen-5}, we find that to control $\tilde{\zeta}_1$ in the space $L_t^2H_{co}^{m-1},$ it requires the boundedness of 
$(\f{1}{r_0})^{\Psi}|_{z=0}$ in the space $L_t^{\infty}H^{m-1}(\mR^2)$ which is missing. 
%(note that we do not have the control of $\ep^2\mu |\div u|_{L_t^{\infty}H^{m-1}(\mR^2)}$).
To overcome this problem, we introduce further splittings: 
$\omega_{r_0}^{\bn}\circ \Psi={\zeta}_3+{\zeta}_4+{\zeta}_5,$ where 
$\zeta_3-\zeta_5$ solves \eqref{eqzeta3}-\eqref{eqzeta5}. 
Note that the boundary condition is also carefully split for latter use. We comment also that $\zeta_5$ solves the similar equation as $\tilde{\zeta}_2$ with small variations. 

To prove uniform estimates for ${\zeta}_5,$ there are still some (essential) troubles which force us to make the assumption \eqref{assumption-mukpa} on $\mu$ and $\kpa.$ We will explain this point more carefully in the next subsection. %Nevertheless, as explained in the next subsection, there are still some (essential) troubles in order to prove uniform estimates for $\tilde{\zeta}_2,$ 

\subsubsection{Remarks on the assumption \eqref{assumption-mukpa}}\label{subs133}
As mentioned in the last subsection, 
the relation \eqref{assumption-mukpa} between $\mu$ and $\kpa$ is made mainly to control the modified vorticity
$\omega_{r_0}^{\bn}$ or more precisely $\zeta_5$ defined as the solution to the following (rough) system: 
\begin{equation}\label{eqzeta5-intro}
\left\{
    \begin{array}{l}
   (\pt+ \tilde{u} \cdot \na){\zeta_5}-\mu \lambda_1\overline{\Gamma} ( \f{1}{r_0})^{\Psi}(\partial_z^2  +\Delta_{{g}})\zeta_5 =\big(\curl G_{r_0} \times \bn\big)^{\Psi}+\cdots,  \quad (t,x) \in \mR_{+}\times \mR_{+}^3,  \\[2.5pt]
\zeta_5|_{z=0}= 0,\quad
\zeta_5|_{t=0}=\zeta|_{t=0}, 
    \end{array}
\right.
\end{equation}
where 
\begin{align}\label{defGr0}
G_{r_0}=\colon  u(\pt+u\cdot\na)r_0+\mu\lambda_1\overline{\Gamma}
   r_0^{-1} [\curl\curl, r_0]u.
 \end{align}
 Since $\zeta_5$ vanishes on the boundary and the equation \eqref{eqzeta5-intro} does not involve any 
 singular terms, we expect to perform a direct energy estimate to control the $L_t^{\infty}H_{co}^{m-2}$ norm of $\zeta_5.$  By doing so, we arrive at the following energy estimate:
 \begin{align}
   &\|\zeta_5(t)\|_{H_{co}^{m-2}}^2+ \mu \iota \|\na \zeta_5\|_{\hco^{m-2}}^2 \leq \|\zeta_5(0)\|_{H_{co}^{m-2}}^2\notag\\
  & +\sum_{|I|\leq m-2}\bigg|\izt\int_{\mR_{+}^3}Z^I(\curl G_{r_0} \times \bn)^{\Psi}\cdot Z^I\zeta_5\, \d x \d s\bigg|+\bigg|\mu\izt\int_{\mR_{+}^3}
   Z^I\big(\big(\f{1}{r_0}\big)^{\Psi}\big) \p_z\zeta_5 \cdot \p_z Z^I \zeta_5 \d x \d s\bigg|+\cdots \label{zeta5-badterm}
 \end{align}
 where $\cdots$ denotes some quantities that are relatively easy to control and $\iota$ stands for some constant independent of $\ep,\mu,\kpa$. Let us comment that the interactions between the thermal boundary layer and the viscous boundary layer are mainly reflected in the first two terms in \eqref{zeta5-badterm}. Indeed,
 we first observe that as (see \eqref{timeder-r0})
 \begin{align*}
     (\pt+u\cdot\na)r_0=-\f{R\kpa}{C_v\gamma}r_0 \Gamma \div(\beta\na\theta)+\cdots,
 \end{align*}
 the term $\curl G_{r_0}\times\bn$ may involve three normal derivatives of $\theta$.
 However, after some calculations, we can only expect $\kpa^{\f{3}{2}}\na^3\theta$ to be uniformly bounded in $\hco^{m-2}.$ The natural idea is to perform an integration by parts in space to obtain (see the definition of $G_{r_0}$ in \eqref{defGr0})
 \begin{align*}
     \bigg|\izt\int_{\mR_{+}^3}Z^I(\curl G_{r_0} \times \bn)\cdot Z^I\zeta_5 \d x \d s\bigg|\lesssim 
     \|\mu^{\f{1}{2}}\na\zeta_5\|_{\hco^{m-2}} \big(\kpa\mu^{-\f{1}{2}}+\mu^{\f{1}{2}}\big)\|\na^2\theta\|_{\hco^{m-2}}\lat+\cdots
 \end{align*}
 where $\Lambda$ represents a polynomial function with respect to its arguments whose coefficients are independent of $(\ep,\mu,\kpa)$. In view of the definition of 
 $\cE_{m,t}(\theta)$ in \eqref{defcEmtheta}, we have the control of $\kpa^{\f{1}{2}}\|\na^2\theta\|_{\hco^{m-2}}.$ This 
 leads us to assume that:
 \begin{align*}
     \mu^{\f{1}{2}}\lesssim \kpa^{\f{1}{2}},\, \kpa\mu^{-\f{1}{2}}\lesssim \kpa^{\f{1}{2}}, \text{ or equivalently }  \, \kappa\sim \mu.
 \end{align*}
 As for the control of the second term in \eqref{zeta5-badterm}, the main problem we encounter is the estimate of 
 \begin{align*}
    \mu^{\f{1}{2}} \|Z^I\big(\big(\f{1}{r_0}\big)^{\Psi}\big) \p_z\zeta_5\|_{L_t^2L^2(\mR_{+}^3)}, \quad \text{ when } |I|=m-2. 
 \end{align*}
 Indeed, 
 due to the viscous boundary layer, the term
 $\mu^{\f{1}{2}}\|\p_z\zeta_5\|_{L_t^qL^P(\mR_{+}^3)}$
is unlikely to be uniformly (in $\ep, \mu$) bounded for any $2\leq q\leq +\infty, \, 2<p\leq +\infty$ (see Remark \ref{rmk-mukpa-1} for more explanations).  
We are thus forced to attribute the $
L_t^{\infty}L^{\infty}(\mR_{+}^3)$ norm on $Z^I\big(\f{1}{r_0}\big)^{\Psi}$ %(1/r_{0}^{\Psi}),$ 
which is not uniformly (in $\kappa$) bounded. As a matter of fact, by the Sobolev embedding and the definition \eqref{defcEmtheta} for $\cE_{m,t}(\theta),$ we have at best:
\begin{align*}
    \kpa^{\f{1}{2}}\il\theta \il_{m-2,\infty,t}\lesssim  \kpa^{\f{1}{2}}(\|\na\theta\|_{\infco^{m-1}}+\|\theta\|_{\infco^m})\lesssim \cE_{m,t}(\theta).
\end{align*}
To deal with this problem, we %expect to 
split again $\zeta_5$ into two unknowns: $\zeta_5=\zeta_{51}+\zeta_{52},$ and try to prove an estimate like:
\begin{align}\label{zeta51-52}
    \|\p_z \zeta_{51}\|_{L_t^2L^2}%\lesssim %(T^{\f{1}{2}}+\ep)
   % \lae, \quad 
   +\mu^{\f{1}{2}}\il \p_z \zeta_{52} \il_{0,\infty,t}\lesssim \lab Y_m(0)\big)+
\lae.
\end{align}
In this way, the term $ \sum_{|I|=m-2} \mu^{\f{1}{2}} \|Z^I\big(\big(\f{1}{r_0}\big)^{\Psi}\big) \p_z\zeta_5\|_{L_t^2L^2(\mR_{+}^3)}$ can be controlled as:
\begin{align*}
   \mu^{\f{1}{2}}\il r_0^{-1}\il_{m-2,\infty,t}\|\p_z \zeta_{51}\|_{L_t^2L^2}+\|Z (r_0^{-1})\|_{\hco^{m-3}}
\il  \mu^{\f{1}{2}}\p_z \zeta_{52} \il_{0,\infty,t}.
\end{align*}
In order to meet the requirement \eqref{zeta51-52}, we design carefully two systems (see \eqref{eqzeta51} satisfied by $\zeta_{51}, \zeta_{52}).$ One of which is,
roughly speaking,
\begin{align*}
 \left\{  
 \begin{array}{l}
   \big(\p_{t}-\mu\lambda_1\overline{\Gamma}%(\f{1}{r_0})^{\Psi}
   b_0\p_z^2\big)  \zeta_{51}=\big(\curl G_{r_0}\times\bn\big)^{\Psi}+\cdots,    \\
 \zeta_{51}|_{z=0}=0,\quad   \zeta_{51}|_{t=0}=0.
    \end{array}
    \right.
\end{align*}
Therefore, to prove the boundedness of $\|\p_z\zeta_{51}\|_{L_t^2L^2},$ one needs to control more or less $\|\mu^{-1}G_{r_0}\|_{L_t^2L^2}.$ 
Nevertheless, in view of the definition \eqref{defGr0}, 
we get after some computations that:
\begin{align*}
    G_{r_0}=\lambda_1\overline{\Gamma} \big(\mu-\f{\kpa}{C_v\gamma \lambda_1} \big)u\Delta\theta +\cdots
\end{align*}
which, combined with the boundedness of 
$\kpa^{\f{1}{2}}\|\Delta\theta\|_{L_t^2L^2},$ leads us to assume the relation \eqref{assumption-mukpa}. At this stage, let us also mention that in order to prove $\mu^{\f{1}{2}}\p_z\zeta_{52}$ to be bounded in $L_{t,x}^{\infty},$ we need to impose some compatibility condition on the data, which is satisfied once the initial velocity vanishes on the boundary. One can refer to 
the proof of Lemma \ref{lemzeta5-prop} for more details.

%$\curl G_{r_0}\times \bn$ involves three normal derivatives of $r_0=r_0(\theta, \tsigma).$
%are forced to make further assumption on $\mu,\kpa.$
%As a matter of fact, in the original variable, the tangential component of the velocity $u$ is expected to enjoy the form:
%\begin{align*}    \Pi u(t,x) =\Pi u^I \big(\f{t}{\ep},t,x\big)+\sqrt{\ep\mu}  u_{\tau}^B\big(\f{t}{\ep},t, \Pi x, \f{x\cdot\bn}{\sqrt{\ep\mu}}\big)+\cO(\ep\mu).\end{align*}
%\begin{align*}
 % \left(\begin{array}{l}
  %     \Pi u    \\      u\cdot\bn   \end{array} \right)(t,x)= \left(\begin{array}{l}  \Pi u^I    \\  u^I\cdot\bn  \end{array}  \right)\big(\f{t}{\ep},t,x\big)+\sqrt{\ep\mu}\left( \begin{array}{l}     u_1^B(\f{t}{\ep},t, \Pi x, \f{x\cdot\bn}{\sqrt{\ep\mu}})  \\ 0  \end{array}\right). \end{align*}

\subsection{Sketch of the uniform estimates} 
As the uniform regularity estimates are somehow delicate to establish, we outline in this subsection the main steps toward this goal.
 
  $\bullet$ \textbf{Step 1. Highest-order $\ep$-dependent estimate}: 
  $$(\mu+\kappa)^{\f{1}{2}}(\|(\ep\sigma,\ep u,\theta)\|_{L_t^{\infty}\underline{H}_{co}^m},\|(\ep\nabla\sigma,\ep\div u,\nabla\theta)\|_{L_t^{\infty}H_{co}^{m-1}}).$$ 
  As the singular terms appear in the equations of $\sigma$ and $u,$ but not in the equation of $\theta,$ we expect to get the highest-order estimates for $(\sigma,u)$ with a weight $\ep$ and for $\theta$ without weight. The parameter 
  $(\mu+\kappa)^{\f{1}{2}}$ is also added in order to obtain
  the estimates independent of $\mu,\kappa.$ This is consistent with \cite{MR2211706} when the physical domain is the whole space $\mR^3.$
 
 $\bullet$ \textbf{Step 2. Uniform estimates for time derivatives}: $\|(\sigma,u,\kappa\nabla\theta)\|_{L_{t}^{\infty}\cH^{m-1,0}}.$ Since the time derivatives commute with the singular terms, one expects to get the uniform estimates for time derivatives directly by energy estimates. However,
 the
 penalized operator 
 associated to singular terms in the system
 \eqref{NCNS-S2}
 is not skew-symmetric. Inspired by \cite{MR2211706},
 we need to work on the unknown
 $w=u-\kappa\f{\gamma-1}{\gamma}
\Gamma(0) \beta(\theta)\nabla\theta$ and perform energy estimates for $(\sigma,w,\kappa\nabla\theta).$

  $\bullet$ \textbf{Step 3. Uniform estimates for $\theta.$}
  Under suitable a priori knowledge on the properties of $u,$ we can %perform some direct energy 
  obtain some uniform estimates (in $\ep,\mu$) for $\theta,$ namely the quantities appearing in $\cE_{m,t}(\theta)$ (defined in \eqref{defcEmtheta}). 
%  \beq \|\theta\|_{\infco^{m-1}}, \quad \|\nabla\theta\|_{\infcok^{m-1}},\quad  \sqrt{\kappa}\|\nabla^2\theta\|_{\infcok^{m-2}} \eeqs
Such estimates will be useful in the following steps.

$\bullet$ \textbf{Step 4 (Non)-uniform second normal derivatives.}  In this step, we aim to control the quantities that involve the second normal derivatives of $(u,\sigma)$ in the definition of $\cE_{m,t}(\sigma,u)$ (see \eqref{defcEmsigmau}). Such estimates are achieved by working on $\omega_{\bn}=\curl u\times\bn-2 \Pi (-a u+D\bn\cdot u)$ which satisfies a convection-diffusion equation with vanishing boundary condition and on $\Delta\sigma$ which solves a damped transport equation.

$\bullet$ \textbf{Step 5. Uniform control for the compressible part-I}:
$\kappa^{\f{1}{2}}\|(\nabla\sigma,\div u)\|_{L_t^2H_{co}^{m-1}}.$\\
 In this step, we will use recursively the equation for $(\sigma,u)$ to recover the higher order space conormal derivatives for the compressible part $(\nabla\sigma,\div u).$

$\bullet$  \textbf{Step 6. Uniform estimates for the incompressible part of the modified velocity}:  
Denote $v=\bp (r_0u)$ the incompressible part of the velocity, where $\bp$ is the Leray projection defined in \eqref{proj-P} and $r_0$ is defined in \eqref{defr0-intro}. After lengthy but straightforward computations, one finds that $v$ solves the equations:
\beq\label{eqofv-intro}
\left\{
\begin{array}{l}
(\partial_t+u\cdot\nabla) v - \lambda_1 \mu\f{\overline{\Gamma}}{r_0} \Delta v+\nabla q=G-[\bp,u\cdot\nabla](r_0 u) \quad \text{in}\quad
     {\Omega} \\[7pt]
     v\cdot\bn=0, \quad \Pi(\p_{\bn}v)=\Pi\big(-2a r_0u+(D \bn)  (\nabla\Psi+ r_0u)\big)+%\f{R\ep}{C_v\gamma} r_0 u\p_n\sigma
      \f{R\ep}{C_v\gamma}r_0\p_{\bn}\tsigma \Pi u%+r_0\Pi\big(-2a u+(D\bn)u\big)+ \Pi \big((D\bn)\na\Psi\big)
     \quad \text{on}\quad
     {\p\Omega}%\\
   %  v|_{t=0}=\bp((ru)|_{t=0})
\end{array}
\right.
\eeq
where $\nabla q=\colon\bq (G+\lambda_1 \mu\f{\overline{\Gamma}}{r_0} \Delta v)$ and $G$ is defined in \eqref{def-G}.
%\begin{align*}G&=u(\pt+u\cdot\na)r_0+ \ep\mu r_1\div\cL u+\mu(\Gamma-\overline{\Gamma}) \div\cL u\notag\\&\qquad\qquad\qquad-\mu\overline{\Gamma}r_0^{-1}[\div\cL, r_0]u-\mu(2\lambda_1+\lambda_2)\overline{\Gamma}\na(r_0^{-1})\div(r_0 u).\end{align*}
%\beqsG=\colon u(\pt+u\cdot\nabla)r_0+\ep\mu r_1\sigma\div\cL u-\mu\f{\Gamma-\overline{\Gamma}}{r_0}\div\cL u-\mu\f{ \Gamma}{r_0}[\div\cL, r_0]u+(2\lambda_1+\lambda_2)\mu \nabla(\f{\bar{\Gamma}}{r_0})\div (r_0 u).\eeqs
We remark that in the derivation of $v,$ we can not naively
consider the term $\bp (u\cdot\na (r_0 u))$ as the source term, as it is not uniformly (in $\mu$) bounded in $L_t^2H_{co}^{m-1}.$ However, we find that it is the case for the commutator $[\bp, u\cdot\nabla](r_0 u),$  see \eqref{es-bpcom}. 

As no singular terms involved in the equation \eqref{eqofv-intro}, one expects to perform direct energy estimates to obtain in the same time:
$$\|v\|_{\infco^{m-1}},\quad \mu^{\f{1}{2}}\|\nabla v\|_{L_t^2H_{co}^{m-1}}.$$ Due to the appearance of the typical term like  $\ep\mu\na^2\sigma,$ the source term $G$ can not be bounded uniformly in
$\hco^{m-1}.$ We will thus split the term $G$ into two parts, one of which can be uniformly controlled in $\hco^{m-1},$ the other is bounded only in $\hco^{m-2}$ but can gain the extra weight $\ep\mu^{\f{1}{2}}$ which is very useful to take benefits of the boundedness of $\mu^{\f{1}{2}}\|\nabla v\|_{L_t^2H_{co}^{m-1}}.$
The pressure $\na q$ is dealt with in the same fashion. We can thus close the estimates for $v.$ 

$\bullet$ \textbf{Step 7. Uniform estimates for the  vorticity.} In this step, we aim to control $\|\nabla u\|_{L_t^{\infty}H_{co}^{m-2}}$ or alternatively $\|\nabla(r_0 u)\|_{L_t^{\infty}H_{co}^{m-2}}.$
It is the most involved step and the main place where the relation \eqref{assumption-mukpa} for $\mu$ and $\kpa$ is used. 
 As the difficulty lies in the estimate near the boundary, we will work on a local charts $\Omega_i.$ 
By the virtue of the identities:
\beqs
\partial_{\bn} (r_0u\cdot \bn) =\div(r_0 u)-(\Pi \partial_{y^1}(r_0 u))^1-(\Pi \partial_{y^2}(r_0 u))^2,\quad \Pi(\partial_{\bn} (r_0 u))=\Pi(\curl(r_0 u)\times \bn)-\Pi[(D \bn) (r_0 u)],
\eeqs
 %where $\bn$ is an extension of the unit normal and $\Pi$ projects on $(\bn)^\perp$, 
it suffices  to control $\|\curl(r_0 u)\times\bn\|_{L_t^{\infty}H_{co}^{m-2}}.$ We refer to the last subsection for the strategies of this estimate. 
%The main strategies towards this are explained in the last subsection. 

$\bullet$ \textbf{Step 8.  Control of $L_{t,x}^{\infty}$ norm $\cA_{m,t}(\theta).$} We aim to control in this step the  $L_{t,x}^{\infty}$ norms appearing in $\cA_{m,t}(\theta),$ which will be useful in the next step. The most non-trivial part is the control of $\il\na\theta\il_{m-5,\infty,t},$ or more precisely $\il\p_{\bn}\theta\il_{m-5,\infty,t}.$ To this end, we
take benefit of the Green function %\eqref{green-heat} 
for the transport-diffusion equation \eqref{eqetai}  in half space  with variable coefficients  solved by $\p_{\bn}\theta\circ \Psi,$ where $\Psi$ is the transformation associated to the normal geodesic coordinates
\eqref{normal geodesic coordinates}. 
%normal geodesic 
%use the normal geodesic coordinates to change 

%which, after some reductions,   

$\bullet$ \textbf{Step 9. Uniform control for the compressible part-II}:
$\|(\nabla\sigma,\div u)\|_{L_t^{\infty}H_{co}^{m-2}}.$ \\
In this step, we aim to get $L_t^{\infty}$ type energy norms. Such an estimate is achieved again by iteration. Nevertheless, %as we have only  the estimate of $\bp (r_0 u),$ we need to deal with the term 
in this process, it requires the a-priori knowledge of $\il u \il_{[\f{m}{2}]-1,\infty,t},$ which is hopeful only after  $\div u$ is controlled in
$L_t^{\infty}H_{co}^{[\f{m}{2}]}.$ Therefore, we will perform a two-tier iteration. That is, we first control $(\na\sigma,\div u)$ in $L_t^{\infty}H_{co}^{m-4}$ norm which only requires the $L_{t,x}^{\infty}$ type norms for $\theta, \na\theta,$ and then proceed to get the higher order estimates in $L_t^{\infty}H_{co}^{m-2}$ by using the previous estimate.

$\bullet$ \textbf{Step 10. Control of $\cA_{m,t}(\sigma,u)$ } To finish the uniform estimates, it remains to control the $L_{t,x}^{\infty}$ norms for $(\sigma, u )$ defined in $\cA_{m,t}(\sigma,u).$ By the virtue 
of  the Sobolev embedding \eqref{sobebd},
most of them can be bounded directly by previously controlled quantities. After some arguments, we reduce the matter to the control of $\il (\curl(r_0 u)\times\bn, \ep \mu^{\f{1}{2}}(\p_{\bn}\omega_{\bn}, \Delta\sigma))\il_{m-5,\infty,t}$ which bears some resemblance as the estimate of $\il\na\theta\il_{m-5,\infty,t}$ done in Step 8. \\

\textbf{Organization of the paper.} 
We carry out the estimates outlined in Step 1--Step 10 from  Section 2 to Section 11. In Section 12, we summarize the estimates obtained in the previous sections to show 
Proposition \ref{prop-uniform es}. The proof of Theorem \ref{thm1} is then presented in Section 13. In Section 14, we sketch the proof of the convergence result stated in Theorem \ref{thm-conv1}.
Finally, in appendix, we gather some useful estimates and prove some technical lemmas which are used throughout this work. \\

%In Section 13, we prove
%The estimates outlined in Step 1- Step 10 are carried out  
%The Section 12 are delicated  We will  the estimates outlined in Step 1- Step 10  in Section 10

\textbf{Further notations}

$\bullet$ We denote $\Lambda(\cdot,\cdot)$
a polynomial that may differ from line to line but independent of $\ep\in(0,1], \mu,\kpa \in A.$

$\bullet$ We use the notation  $\lesssim$ for
$\leq C(1/c_0)$ for some number
$C(1/c_0)$ that depends only on $1/c_0$ (in particular, independent of $(\ep,\mu,\kpa)).$

$\bullet$  We use the notation $L_t^p L^2=L^p([0,t], L^2(\Omega)), \, (p=2,+\infty).$

%$\bullet$ We denote $f^b$ as the trace of $f$ on the boundary $\p\Omega.$

$\bullet$ We use frequently the lighten notation 
$\cN_{m,T},\, \cE_{m,T},$  $\cA_{m,T}$  and
$Y_m(0)$
for $$\cN_{m,T}^{\mu,\kpa}(\sigma^{\ep}, u^{\ep},\theta^{\ep}),\,\, \cE_{m,T}^{\mu,\kpa}(\sigma^{\ep}, u^{\ep},\theta^{\ep}),\,\, \cA_{m,T}^{\mu,\kpa}(\sigma^{\ep}, u^{\ep},\theta^{\ep}), \,\, Y^{\ep}_{m,\mu,\kpa}(0).$$  \\

\textbf{Caution of notations:}

$\bullet$ As there are many parameters/quantities in the system \eqref{NCNS-S2}, it is useful to keep in mind that 
$\lambda_1,\lambda_2, R, C_v, \gamma$ are some physical constants while $\beta(\theta^{\ep}), \Gamma(\ep\sigma^{\ep})$ are some variable functions.

$\bullet$ For notational convenience, we skip the $\ep-$dependence of the solutions when proving the a-priori estimates conducted in Section 2--Section 12.

\section{$\ep-$dependent highest order estimates.}
In this section, we aim to prove some $\ep-$dependent highest order energy estimates, which will be useful to obtain the desired uniform estimates afterwards. The main result in this section is the following: 
\begin{prop}\label{prop-highest}
Assume that \eqref{preasption} is satisfied for some $T\in (0,1]$ and 
suppose  $m\geq 7,$
then for any $\ep \in (0,1], (\mu,\kpa)\in A,$  any $0<t\leq T,$
the following estimate holds:
\beq\label{EE-highest}
\begin{aligned}
  & \kappa \|(\ep \sigma,\ep u,\theta)\|_{\uinfco^m}^2+\kappa\|(\nabla(\ep\sigma,\theta), \ep \div u)\|_{\infco^{m-1}}^2\\
  &\,+\kappa\mu \big(\|\ep\nabla u\|_{\uhco^m}^2+\|\ep\nabla\div u\|_{\hco^{m-1}}^2\big)+\kappa^2\big(\|\nabla\theta\|_{\uhco^m}^2+\|
    \nabla^2\theta\|_{\hco^{m-1}}^2\big)\\
   &\lesssim Y_m^2(0)+(T^{\f{1}{2}}+\ep)\Lambda\big(\f{1}{c_0},\cA_{m,t}\big)\cE_{m,t}^2.
\end{aligned}
\eeq
\end{prop}
As the first attempt, one may expect to prove 
the above proposition by performing
direct energy estimates on the system \eqref{NCNS-S2}. Nevertheless, as \eqref{NCNS-S2} is not skew-symmetric,  there shall be some difficulties when one
deals with the problematic terms $\f{\gamma-1}{\gamma}\kappa \Gamma \div(\beta\nabla\theta)$ (in the equation for $\sigma$) and $\div u$ (in the equation for $\theta$). Therefore, 
as in \cite{MR2211706}, we introduce the new unknown 
\beq\label{defvr}
\vr=\f{\ep}{\gamma}\sigma-\f{(\gamma-1)C_v}{\gamma R}\theta,
\eeq
%and work on the unknowns $(\vr, \ep u,\theta).$
which plays the same role as $\ep\rho.$  
Denote $u_{\ep}=\ep u$ for simplicity, we find after some computations that  $(\vr, u_{\ep},\theta)$ solves the equations which is amenable to the energy estimates: %skew-symmetric:
\beq\label{newsys}
\left\{
\begin{array}{l}
 (\pt+u\cdot\nabla)\vr+\f{1}{\gamma\ep}\div u_{\ep}=0,      \\[5pt]
 \f{1}{R\beta(\theta)}(\pt+u\cdot\nabla)u_{\ep}  +\f{\gamma}{\ep}\nabla\vr+\f{(\gamma-1)C_v}{R\ep}\nabla\theta-\mu\div\cL u_{\ep}=0,                  \\[5pt]
 \f{C_v}{R}(\pt+u\cdot\nabla)\theta+\f{\div u_{\ep}}{\ep}-\Gamma(\ep\sigma)\div(\beta\nabla\theta)=\mu\ep^2\Gamma(\ep\sigma)\cL u\cdot \mathbb{S}u. 
\end{array}
\right.
\eeq
We shall prove Proposition \ref{prop-highest}  by showing Propositions \ref{prop-highest-1} and \ref{prop-highest-2} which are presented respectively in the following two subsections.

\subsection{$\ep-$dependent  estimates-I}
\begin{prop}\label{prop-highest-1}
Under the same assumption as in Proposition \ref{prop-highest}, for any $\ep \in (0,1], (\mu,\kpa)\in A,$  any $0<t\leq T,$ it holds that:
\begin{align}\label{EE-highest1}
  & \kappa \|(\vr,u_{\ep},\theta)\|_{\uinfco^m}^2+\kappa\mu \|\nabla u_{\ep}\|_{\uhco^m}^2+\kappa^2\|\nabla\theta\|_{\uhco^m}^2\lesssim Y^2_m(0)+(T^{\f{1}{2}}+\ep)\Lambda\big(\f{1}{c_0},\cA_{m,t}\big)\cE_{m,t}^2.
\end{align}
\end{prop}
\begin{proof}

Applying the vector field $Z^I, (|I|\leq m-1,$ or $Z^I=Z_{j}Z^{\tilde{I}},|\tilde{I}|\leq m-1, j=1,2\cdots M$) on the system \eqref{newsys} and denoting $(\vr^I,u_{\ep}^I,\theta^I)=Z^I(\vr, u_{\ep},\theta),$ we find the following system:
\beq\label{sec2:eq1}
\left\{
\begin{array}{l}
 (\pt+u\cdot\nabla)\vr^I+\f{1}{\gamma\ep}\div u_{\ep}^I=\cC_{\vr}^I,      \\[5pt]
 \f{1}{R\beta(\theta)}(\pt+u\cdot\nabla)u_{\ep}^I +\f{\gamma}{\ep}\nabla\vr^I+\f{(\gamma-1)C_v}{R\ep}\nabla\theta^I-\mu \Gamma(\lambda_1\div Z^I\nabla u_{\ep}+(\lambda_1+\lambda_2)\na Z^I\div u_{\ep})
 %\div\cL u_{\ep}^I
 =\cC_{u_{\ep}}^I,                  
 \\[5pt]
 \f{C_v}{R}(\pt+u\cdot\nabla)\theta^I+\f{\div u_{\ep}^I}{\ep}-\Gamma(\ep\sigma)\div(\beta Z^I\nabla\theta)=\cC_{\theta}^I,
\end{array}
\right.
\eeq
where $(\cC_{\vr}^I,\cC_{u_{\ep}}^I,\cC_{\theta}^I)$ are defined as: 
\beq\label{com-vr}
\cC_{\vr}^I=-[Z^I,u\cdot\nabla]\vr-\f{1}{\gamma}[Z^{I},\div]u,
\eeq
\beq\label{com-uep}
\begin{aligned}
\cC_{u_{\ep}}^I=& \underbrace{-\f{1}{R\beta(\theta)}[Z^{I},u\cdot\nabla]u_{\ep}}_{\cC_{u_{\ep,1}}^I}\underbrace{-\f{1}{\beta(\theta)}[Z^{I},\beta(\theta)\nabla]\sigma+ \f{\mu}{\beta(\theta)}[Z^I,\beta\Gamma]\div\cL u_{\ep}}_{\cC_{u_{\ep,2}}^I}\\
&+\underbrace{\mu \Gamma\big(\lambda_1[Z^I,\div]\nabla u_{\ep}+(\lambda_1+\lambda_2)[Z^I,\nabla]\div u_{\ep}\big)}_{\cC_{u_{\ep,3}}^I},
\end{aligned}
\eeq
\beq\label{com-theta}
\begin{aligned}
\cC_{\theta}^I=& -\f{C_v}{R}[Z^{I},u\cdot\nabla]\theta-[Z^{I},\div]u+\mu\ep^2 Z^I(\Gamma
\cL u\cdot \mathbb{S}u)\\
&+\kappa [Z^I,\Gamma\beta]\Delta\theta+\kappa\Gamma\beta[Z^I,\div]\nabla \theta+\kappa Z^I(\Gamma\nabla\beta\cdot\nabla\theta)-\kappa\Gamma\nabla\beta\cdot Z^I\nabla\theta. 
\end{aligned}
\eeq
Multiplying  the equation \eqref{sec2:eq1} by $\kappa(\gamma^2\vr^I,u_{\ep}^I,\f{(\gamma-1)C_v}{R}\theta^I)$ and integrating in space and time, one finds the identity: 
\begin{align}\label{sec2:eq2}
 &  \f{\kappa}{2}\int_{\Omega} \big(\gamma^2|\vr^I|^2+\f{1}{R\beta(\theta)}|u_{\ep}^I|^2+\f{(\gamma-1)C_v^2}{R^2}|\theta^I|^2\big)(t)\d x+\f{(\gamma-1)C_v}{R}\kappa^2 \izt\iomega\beta\Gamma |Z^I\nabla\theta|^2\d x\d s  \notag\\
 & \qquad\qquad\qquad\qquad\qquad\qquad\qquad +\mu\kappa\izt\iomega \lambda_1\Gamma|Z^I\nabla u|^2+(\lambda_1+\lambda_2)\Gamma|Z^I\div u|^2\,\d x\d s \notag\\
 &=\cK_0+\cK_1+\cdots+\cK_6,
\end{align}
where
\begin{align}
 \cK_0&=\f{\kappa}{2}\int_{\Omega} \big(\gamma^2|\vr^I|^2+\f{1}{R\beta(\theta)}|u_{\ep}^I|^2+\f{(\gamma-1)C_v^2}{R^2}|\theta^I|^2\big)(0)\,\d x,\notag\\
 \cK_1&=\f{\kappa}{2}\izt\int_{\Omega}\big(\pt(\f{1}{R\beta})+\div(\f{1}{R\beta}u)\big)|u_{\ep}^{I}|^2+\div u\big(\gamma^2|\vr^I|^2+\f{(\gamma-1)C_v^2}{R^2}|\theta^I|^2\big)\,\d x\d s,\notag\\
\cK_2&= \kappa\mu\izt\int_{\p\Omega}(\lambda_1+\lambda_2)\Gamma Z^I\div u_{\ep} (u_{\ep}^I\cdot\bn)+\lambda_1\Gamma (Z^I\nabla u_{\ep}\cdot\bn)\cdot u_{\ep}^I\, \d S_y\d s=\colon\cK_{21}+\cK_{22},\notag\\
\cK_3&=\kappa^2\f{(\gamma-1)C_v}{R}\int_0^t\int_{\p\Omega}\beta\Gamma (Z^I\nabla\theta\cdot\bn)\theta^I\,\d S_y \d s,\notag\\
\cK_4&=-\kappa^2\f{(\gamma-1)C_v}{R}\izt\int_{\Omega}\beta \theta^I\nabla\Gamma\cdot Z^I\nabla\theta+\beta\Gamma Z^I\nabla\theta\cdot[\nabla,Z^I]\theta \,\d x\d s,\notag\\
\cK_5&=-\kappa\mu\izt\int_{\Omega} (\lambda_1+\lambda_2) (u_{\ep}^I\cdot \nabla\Gamma+\Gamma[\div, Z^I]u_{\ep}) Z^I\div u_{\ep} \notag\\
&\qquad\qquad\qquad\qquad\qquad
+\lambda_1 \big((\na\Gamma\cdot Z^I \na u_{\ep})\cdot u_{\ep}^I+\Gamma Z^I\na u_{\ep}\cdot [\na, Z^I] u_{\ep}\big)
 \,\d x\d s,\notag\\
%\cK_4&=\mu\kappa\int_0^t\int_{\p\Omega}\lambda_1 (Z^I\nabla u\cdot\bn)\cdot u^I+(\lambda_1+\lambda_2)(Z^I\div u)(u^I\cdot\bn) \d S_y\d s  \notag\\\\
\cK_6&=\kappa \izt\int_{\Omega} \gamma^2 \cC_{\vr}^I\vr^I+\f{(\gamma-1)C_v}{R}\cC_{\theta}^{I}\theta^I+
\cC_{u_{\ep}}^I\cdot u_{\ep}^I\,\d x\d s.
%=\colon \cK_{51}+\cK_{52}+\cK_{53},
\label{defk5}
\end{align}
Note that hereafter, $\d S_y$ denotes the surface measure of $\Omega.$ 
We will control $\cK_0-\cK_6$ term by term. First of all,
it follows from the assumption \eqref{preasption} and the property \eqref{preasption1} that:
\begin{align}\label{Gamma-Beta-infty}
 c_0 \leq  \beta(\theta)(t,x),\, \Gamma(\ep\sigma)(t,x)\leq {1}/{c_0}, \quad \,\forall (t,x)\in [0,T]\times \Omega.
\end{align}
Consequently, we have by the definition of $Y_m(0)$ in \eqref{initialnorm} and Remark \ref{rmkmuapproxkpa} that, for any $(\mu,\kpa)\in A, \ep\in (0,1],$
\beq\label{ck0}
\cK_0\lesssim \|\kpa^{\f{1}{2}}(\ep u,\ep\sigma, \theta)(0)\|_{H_{co}^m}^2 \lesssim  Y^2_m(0).
\eeq
Thanks to \eqref{Gamma-Beta-infty}, 
%using again assumption \eqref{preasption} and the property \eqref{preasption1},
it holds also that 
\beqs
\il(\pt(\f{1}{R\beta}),\div(\f{1}{R\beta}u),\div u)\il_{0,\infty,t}\lesssim \Lambda\big(\f{1}{c_0},%\il(\theta,u)\il_{1,\infty,t}+
\il(\theta,u,\pt\theta,\nabla\theta,\div u)\il_{0,\infty,t}\big)\lesssim \lat. %\Lambda_{1,\infty,t}.
\eeqs
The term  $\cK_1$ can thus be controlled as:
\beq\label{ck1}
\cK_1\lesssim T\kappa\|(\vr^I,u_{\ep}^I,\theta^I)\|_{L_t^{\infty}L^2}^2\lat\lesssim T \lat\cE_{m,t}^2.
\eeq
In order to control the term $\cK_2,$ we estimate successively $\cK_{21}, \cK_{22}.$ First, for $\cK_{21},$  in light of the fact $(\ep\pt)^ k (u\cdot\bn)|_{\p\Omega}=0$ for $k\leq m-1,$ we may assume 
$Z^I$ contains at least one spatial tangential derivative, say $Z^I=\p_y Z^{\tilde{I}}.$ By using the duality between $H^{-\f{1}{2}}(\p\Omega)$ and $H^{\f{1}{2}}(\p\Omega)$ and the trace inequality \eqref{normaltraceineq},
one controls $\cK_{21}$ as: %{\color{red}need to be seen, here we used $\kappa\lesssim \mu^{\f{1}{2}}$}
\begin{align}\label{k21}
  \cK_{21}&=\kappa\mu(\lambda_1+\lambda_2)\izt\int_{\p\Omega} \Gamma Z^I\div u_{\ep} [Z^I,\bn\cdot]u_{\ep} \d S_y\d s\notag\\
  &\lesssim \kappa\mu |\div u_{\ep}|_{L_t^2\tilde{H}^{m-\f{1}{2}}}|[Z^I,\bn\cdot]u_{\ep}|_{L_t^2{H}^{\f{1}{2}}}\\
  &\lesssim \ep \|(\kpa\mu)^{\f{1}{2}}(\nabla\div u_{\ep},\div u_{\ep})\|_{\hco^{m-1}}  \|(\kpa\mu)^{\f{1}{2}}(\nabla u,u)\|_{\hco^{m-1}}\lesssim \ep \cE_{m,t}^2. \notag
\end{align}
We refer to  \eqref{bdynorm} for the definition of the norm $\hcob^{m-\f{1}{2}}.$
%Note that in the last inequality, we have used the fact $\mu^{\f{1}{2}}\|\nabla u\|_{\hco^{m-1}}\lesssim\cE_{m,t}.$
Second, for $\cK_{22},$ we split it further into two terms:
\begin{align*}
   \cK_{22}&=\kappa\mu\lambda_1 \izt\int_{\p\Omega}  \Gamma(Z^I\nabla u_{\ep}\cdot\bn)\cdot u_{\ep}^I\, \d S_y\d s\\
   &=\kappa\mu\lambda_1 \izt\int_{\p\Omega}\Gamma\big([\bn\cdot,Z^I]\nabla u_{\ep}\cdot u_{\ep}^I +Z^I\Pi \p_{\bn}u_{\ep}\Pi u_{\ep}^I +[\Pi,Z^I]\p_{\bn}u_{\ep}\cdot\Pi u_{\ep}^I\big)\,\d S_y\d s\\
   &\quad+\kappa\mu\lambda_1 \izt\int_{\p\Omega}\Gamma Z^I \p_{\bn}u_{\ep}\cdot \bn (u_{\ep}^I\cdot\bn) \,\d S_y\d s=\colon \cK_{221}+\cK_{222}.
\end{align*}
Thanks to the Cauchy-Schwarz inequality, the term $\cK_{221}$ can be bounded as: 
\beqs
\cK_{221}\lesssim \kappa\mu |u_{\ep}^I|_{L_t^2L_y^2}\big(|\nabla u_{\ep}|_{\hcob^{m-1}}+|Z^I\Pi \p_{\bn}u_{\ep}|_{L_t^2L_y^2}\big).
%|(\Pi \p_{\bn}u_{\ep})^b|_{\hcob^{m}}
\eeqs
By the virtue of the identity
\beq\label{normalofnormalder}
\partial_{\bn} u\cdot \bn=\div u-(\Pi\partial_{y^1}u)^1-(\Pi\partial_{y^2}u)^2,
\eeq
and the boundary condition \eqref{bdryconditionofu}, one finds that:
\beqs
|\nabla u_{\ep}|_{\hcob^{m-1}}+|Z^I\Pi \p_{\bn}u_{\ep}|_{L_t^2L^2}%|(\Pi \p_{\bn}u_{\ep})^b|_{\hcob^{m}}
\lesssim | (u_{\ep},\p_y u_{\ep})|_{\hcob^{m-1}}+|
\div u_{\ep}|_{\hcob^{m-1}}.
\eeqs
Application of the trace inequality \eqref{traceL2} then yields:
\begin{align*}
\cK_{221}&\lesssim\kappa\mu |u_{\ep}^I|_{L_t^2L_y^2}\big(|(u_{\ep},\p_y u_{\ep})|_{\hcob^{m-1}}+|\div u_{\ep}|_{\hcob^{m-1}}\big)\\
&\lesssim \ep (\kappa\mu)^{\f{1}{2}}\big(\|\nabla u_{\ep}\|_{\uhco^m}+ \|\nabla (\div u_{\ep}, u)\|_{\hco^{m-1}} \big)
(\kappa\mu)^{\f{1}{2}}\big(\|u\|_{\uhco^m}+\|\div u\|_{\hco^{m-1}}\big)\\
&%\qquad + \ep\kpa\mu \|u\|_{\uhco^m}^2
\lesssim \ep \cE_{m,t}^2,
\end{align*}
where $\|\cdot\|_{\uhco^m}$ is defined in \eqref{defuhco} that precludes the highest time derivatives $\|(\ep\pt)^m\cdot\|_{L_t^2L^2}.$ 
%{\color{red},if it is $\kappa, \hco$ should be changed as $\uhco$ in the norm of $\nabla u$}
By using the identity \eqref{normalofnormalder}, the boundary condition \eqref{bdryconditionofu} and the
trace inequality \eqref{normaltraceineq}, 
the term $\cK_{222}$ can be treated in the similar way as  $\cK_{21}:$
\begin{align*}
  \cK_{222}&\lesssim   \kappa \mu |\p_{\bn} u_{\ep}|_{_{L_t^2\tilde{H}^{m-\f{1}{2}}}} %|\kappa\mu |\div u_{\ep}|
  |[Z^I,\bn\cdot]u_{\ep}|_{L_t^2{H}^{\f{1}{2}}}\notag\\
  &\lesssim \ep  |(\kappa \mu)^{\f{1}{2}}(u_{\ep},\p_y u_{\ep},\div u_{\ep})|_{_{L_t^2\tilde{H}^{m-\f{1}{2}}}}|(\kappa \mu)^{\f{1}{2}} u|_{\hcob^{m-\f{1}{2}}}\\
  &\lesssim 
  \ep \big(\|(\kappa\mu)^{\f{1}{2}}\nabla u_{\ep}\|_{\uhco^m}^2+\|(\kappa\mu)^{\f{1}{2}}(\na\div u_{\ep}, \na u, u)\|_{\hco^{m-1}}^2\big) %\|(\kappa\mu)^{\f{1}{2}}(\nabla u,u)\|_{\hco^{m-1}}+\ep^2\kpa\mu\|(\na u, u)\|_{\hco^{m-1}}^2\\
\lesssim \ep\cE_{m,t}^2.
\end{align*}
The above two inequalities lead to that
%\beqs \eeqs
$\cK_{22} \lesssim\ep\cE_{m,t}^2,$
which, combined with \eqref{k21}, yields:
\beq\label{ck2}
\cK_2\lesssim \ep\cE_{m,t}^2.
\eeq

Let us proceed to estimate $\cK_3.$ In view of the Neumann boundary condition $\p_{\bn}\theta|_{\p\Omega}=0$, we can assume that $Z^I$ contains at least one spatial tangential derivative $\p_y$ (say $Z^I=\p_y Z^{\tilde{I}},|\tilde{I}|\lesssim |I|-1$), otherwise it vanishes. Therefore,
\begin{align*}
    \cK_3&\lesssim \kappa^2 \big|[Z^{I},\bn]\nabla\theta\big|_{L_t^2L^2}|\p_y \theta|_{\hcob^{m-1}}
    \lesssim |\kappa\nabla\theta|_{\hcob^{m-1}}|\kappa \p_y \theta|_{\hcob^{m-1}}.
\end{align*}
Applying again the trace inequality \eqref{traceL2}, one can control $\cK_3$ in the following way:
\begin{align}
     \cK_3 &\lesssim \big(\|\kappa\nabla^2\theta\|_{\hco^{m-1}}\|\kappa \nabla\theta\|_{\hco^{m-1}}+\|\kappa\nabla\theta\|_{\hco^{m-1}}^2\big)^{\f{1}{2}}\cdot\notag\\
     &\qquad \big(\|\kappa\nabla \p_y \theta\|_{\hco^{m-1}}\|\kappa\p_y\theta\|_{\hco^{m-1}}+\|\kappa\p_y\theta\|_{\hco^{m-1}}^2\big)^{\f{1}{2}}\lesssim T^{\f{1}{2}}\cE_{m,t}^2.\notag
\end{align}
Note that in the last inequality, we have used the fact: $$\|\kappa\na\theta\|_{\hco^{m-1}}\lesssim T^{\f{1}{2}}\|\kpa\na \theta\|_{L_t^{\infty}H_{co}^{m-1}}\lesssim T^{\f{1}{2}}\cE_{m,t}.$$
Next, for the term $\cK_4,$ we use the Cauchy-Schwarz inequality to find:
\begin{align}\label{ck4}
\cK_4&\lesssim (\| \kappa \theta^I\|_{L_t^{2}L^2}+\|\kappa[\nabla,Z^I]\theta\|_{L_{t}^{2}L^2})\|\kappa Z^I\nabla\theta\|_{L_t^2L^2}
\Lambda\big(\f{1}{c_0}, \il \ep\sigma\il_{1,\infty,t} \big)
%\Lambda_{1,\infty,t}
\notag\\
&\lesssim
T^{\f{1}{2}}(\| \kappa \theta^I\|_{L_t^{\infty}L^2}+\|\kappa\nabla\theta\|_{\infco^{m-1}})\|\kappa Z^I\nabla\theta\|_{L_t^2L^2}\Lambda\big(\f{1}{c_0}, \il \ep\sigma\il_{1,\infty,t} \big)\\
&\lesssim T^{\f{1}{2}}\lat\cE_{m,t}^2.\notag
\end{align}
Let us remark that when controlling the  term
$[\nabla,Z^I]\theta,$ we have used the following identity 
which can be shown by induction:
\beq\label{comu}
[Z^{I},\partial_i]=\sum_{j=1}^3\sum_{|{J}|\leq|I|-1} c_{I,{J}} Z^{{J}}\partial_j=\sum_{j=1}^3\sum_{|{{J}}|\leq|I|-1} d_{I,{J}}\partial_j Z^{{J}}
\eeq
where $J$ is a  $(M+1)$ multi-index and $c_{I,J}, d_{I,J}$ are smooth functions that depend on $I$, $J$,  $i$ and the derivatives (up to order $|I|$) of $\nabla \phi$,   $\p_i$ is the derivation in the standard Euclidean coordinates. 

Following the similar arguments as in the estimate of $\cK_5$, we can control the next term $\cK_5$ as: 
\begin{align}\label{ck5}
    \cK_5\lesssim \ep\|(\kpa\mu)^{\f{1}{2}}\na u_{\ep}\|_{\uhco^{m}}\|(\kpa\mu)^{\f{1}{2}}(u, \na u)\|_{\hco^{m-1}}\lat\lesssim \ep\lat\cE_{m,t}^2.
\end{align}
Finally, for the term $\cK_6$ which is defined in \eqref{defk5},  we use again the Cauchy-Schwarz inequality to get:
\begin{align*}
    \cK_{6}\lesssim T^{\f{1}{2}}\|\kpa^{\f{1}{2}}(\vr^I,u_{\ep}^I,\theta^I)\|_{L_t^{\infty}L^2}\|(\cC_{\vr}^I,\cC_{u_{\ep}}^I,\cC_{\theta}^I)\|_{L_t^{2}L^2}.
\end{align*}
It will be shown very soon in the next lemma that:
%We will prove in the following lemma that 
$$\|\kappa^{\f{1}{2}}(\cC_{\vr}^I,\cC_{u_{\ep}}^I,\cC_{\theta}^I)\|_{L_t^{2}L^2}\lesssim \Lambda\big(\f{1}{c_0},\cA_{m,t}\big)\cE_{m,t},$$
which, together with the previous inequality, yields:
\begin{align}\label{ck6}
    \cK_{6}\lesssim T^{\f{1}{2}}\Lambda\big(\f{1}{c_0},\cA_{m,t}\big)\cE_{m,t}^2.
\end{align}
Collecting the estimates \eqref{ck0}-\eqref{ck1}, \eqref{ck2}-\eqref{ck4} and \eqref{ck5}, \eqref{ck6},
we find \eqref{EE-highest1}  by using the energy identity \eqref{sec2:eq2}.
%\begin{align}%\label{EE-Highest}  & \kappa \|(\vr,u_{\ep},\theta)\|_{\uinfco^m}^2+\kappa\mu \|\nabla u\|_{\uhco^m}^2+\kappa\|\nabla\theta\|_{\uhco^m}^2\\  &\lesssim Y_m(0)+(T^{\f{1}{2}}+\ep)\cE_{m,t}^2\Lambda\big(\f{1}{c_0},\cA_{m,t}\big).\end{align}
%This leads to \eqref{} by recalling the definition of $\vr, u_{\ep}.$
%which is the sum of $\cK_{51}-\cK_{53}$ (defined in \eqref{defk5}).
\end{proof}
\begin{lem}
Under the same assumptions %\eqref{preasption}
as in the the Proposition \ref{prop-highest}
 and given $Z^I=(\ep\pt)^{k},k\leq m-1$ or $Z^I=Z_{j}Z^{\tilde{I}},|I|\leq m-1, j=1,\cdots M,$  the following estimate holds:
\beqs
\kappa^{\f{1}{2}}\|(\cC_{\vr}^I,\cC_{u_{\ep}}^I,\cC_{\theta}^I)\|_{L_t^{2}L^2}\lesssim \Lambda\big(\f{1}{c_0},\cA_{m,t}\big)\cE_{m,t},
\eeqs
where $\cC_{\vr}^I,\cC_{u_{\ep}}^I,\cC_{\theta}^I$ are defined in \eqref{com-vr}-\eqref{com-theta}.
\end{lem}
\begin{proof}
We begin with the estimate $\cC_{\vr}^I.$ For the first term $[Z^I,u\cdot\nabla]\vr,$
we write:
$$[Z^I,u\cdot\nabla]\vr=[Z^I,u,\nabla\vr]+u\cdot[Z^I,\nabla]\vr+\nabla\vr\cdot Z^I u,$$ where $$[Z^I,u,\nabla\vr]=Z^I(u\cdot\nabla\vr)-Z^I u\cdot\nabla\vr+u\cdot Z^I\nabla\vr.$$
One can thus use %commutator estimate \eqref{roughcom} and
the identities \eqref{comu} and \eqref{id-product-com} to obtain that:
\beq\label{sec2:eq100}
\begin{aligned}
&\kappa^{\f{1}{2}}\|[Z^I,u,\nabla\vr]+u\cdot[Z^I,\nabla]\vr\|_{L_t^2L^2}\\
&\lesssim \kappa^{\f{1}{2}}
(\|u\|_{\hco^{m-1}}+\|\nabla\vr\|_{\hco^{m-1}})\Lambda\big(\f{1}{c_0},\il\nabla\vr\il_{[\f{m}{2}]-1,\infty,t}+\il u\il_{[\f{m+1}{2}],\infty,t}\big)\lesssim \lat\cE_{m,t}.
\end{aligned}
\eeq
Note that  in view of the definition of $\cA_{m,t}=\cA_{m,t}(\sigma, u,\theta)$ in \eqref{defcEmt}, it holds by the assumption $m\geq 7$ that:
\beq\label{useful-inftyhalf}
\begin{aligned}
    \il(\na(\theta, u), \na^2(\kpa\theta, \ep\mu^{\f{1}{2}}u))\il&_{[\f{m}{2}]-1,\infty,t}+ \il\na(\sigma, \kpa^{\f{1}{2}}\theta, \ep\mu^{\f{1}{2}}u)\il_{[\f{m}{2}],\infty,t}%\\&\quad 
    +\il (\sigma, u, \theta,\kpa\na\theta) \il_{[\f{m+1}{2}],\infty,t}\lesssim \cA_{m,t}.
\end{aligned}
\eeq
Moreover, it is direct to see that $\|\kappa^{\f{1}{2}}\nabla\vr \cdot Z^I u\|_{L_t^2L^2}\lesssim \|\kappa^{\f{1}{2}}u\|_{\uhco^{m}}\il\nabla\vr\il_{1,\infty,t},$ which, combined with \eqref{sec2:eq100},
yields that:
\beqs
\|\kpa^{\f{1}{2}}[Z^I,u\cdot\nabla]\vr\|_{L_t^2L^2}\lesssim \lat\cE_{m,t}.
\eeqs
Next, by using again the identity \eqref{comu}, we find that:
\begin{align*}
\|\kpa^{\f{1}{2}}[Z^{I},\div] u\|_{\hco^{m-1}}\lesssim \|\kpa^{\f{1}{2}}\nabla u\|_{\hco^{m-1}}\lesssim \cE_{m,t}.
\end{align*}
Combining the above two inequalities, we finish the estimate of $\cC_{\vr}^{I}$ and find that:
\beq\label{sec2:eq5}
\|\kpa^{\f{1}{2}}\cC_{\vr}^{I}\|_{L_t^2L^2}\lesssim \lat\cE_{m,t}.
\eeq

We now proceed to control the commutator $\cC_{u_{\ep}}^I=\cC_{u_{\ep,1}}^I+\cC_{u_{\ep,2}}^I+\cC_{u_{\ep,3}}^I$ defined in \eqref{com-uep}.
Simply replacing $\vr$ by $u_{\ep}$, the term $\cC_{u_{\ep,1}}^I$ can be controlled in a similar way as the first term of $\cC_{\vr}^I:$ 
\beq\label{sec2:eq4}
\|\kpa^{\f{1}{2}}\cC_{u_{\ep,1}}^I\|_{L_t^2L^2}\lesssim \lat\cE_{m,t}.
\eeq
For $\cC_{u_{\ep,1}}^I,$ we use the commutator estimate \eqref{roughcom} and the identity \eqref{comu}
to get that:
\begin{align*}
    \|\kpa^{\f{1}{2}}\cC_{u_{\ep,2}}^I\|_{L_t^2L^2}&\lesssim \kpa^{\f{1}{2}}\big(\|(Z\beta, Z(\Gamma\beta))\|_{\uhco^{m-1}}+
    \|(\nabla\sigma,\mu\div\cL u_{\ep})\|_{\hco^{m-1}}\big)\\
   & \quad \cdot\Lambda\big(\f{1}{c_0}, \il (Z\beta, Z(\Gamma\beta))\il_{[\f{m-1}{2}],\infty,t}+\il(\nabla\sigma,\mu\div\cL u_{\ep})\il_{[\f{m}{2}]-1,\infty,t}\big).
\end{align*}
%It follows from the equation of the velocity $\eqref{NCNS-S2}_2$ and the fact \eqref{useful-inftyhalf} that: 
%\beqs\il\mu\div\cL u_{\ep}\il_{[\f{m}{2}]-1,\infty,t}\lesssim \Lambda\big(\f{1}{c_0},\il (u,\nabla\sigma,\theta)\il_{[\f{m}{2}],\infty,t}\big)\lesssim \lat.\eeqs
Applying the estimates \eqref{esofGamma-1}, \eqref{esofGamma-2} and using \eqref{useful-inftyhalf}, we then derive that:
\begin{align*}
    \|\kpa^{\f{1}{2}}\cC_{u_{\ep,2}}^I\|_{L_t^2L^2}&\lesssim \kpa^{\f{1}{2}}\big(\|(\theta, \ep\sigma)\|_{\uhco^m}+\|(\nabla\sigma,\ep\mu\nabla^2 u)\|_{\hco^{m-1}}\big)\lat\lesssim \lat\cE_{m,t}.
\end{align*}
Next, in view of the identity \eqref{comu}, the term $\cC_{u_{\ep,3}}^I$ can be controlled as:
 \begin{align*}
     \|\kpa^{\f{1}{2}}\cC_{u_{\ep,3}}^I\|_{L_t^2L^2}\lesssim \kappa^{\f{1}{2}}\|\ep\mu\nabla^2 u\|_{\hco^{m-1}}\lesssim \kappa^{\f{1}{2}}\cE_{m,t}.
 \end{align*}
 The previous two inequalities, together with \eqref{sec2:eq4}, yield:
 \beq\label{sec2:eq6}
  \|\kpa^{\f{1}{2}}\cC_{u_{\ep}}^I\|_{L_t^2L^2}\lesssim \lat\cE_{m,t}.
 \eeq 
 
 It remains to deal with the term $\cC_{\theta}^I$ defined in \eqref{com-theta}. The first two terms can be controlled in a similar way as $\cC_{\vr}^I,$ the other terms can be handled  with the help of product estimate \eqref{roughproduct1}, commutator estimate \eqref{roughcom} as well as the properties for $\Gamma(\ep\sigma),\beta(\theta)$ in Corollary \ref{cor-gb}. %Let us detail the estimate for one term:
 For instance, the term
 $\kappa[Z^I,\Gamma\beta]\Delta\theta$ can be bounded in the following way:
 \begin{align*}
   \kpa^{\f{1}{2}}  \|\kappa[Z^I,\Gamma\beta]\Delta\theta\|_{L_t^2L^2}&\lesssim 
  \big( \kappa^{\f{1}{2}}\|Z(\Gamma\beta)\|_{\uhco^{m-1}}+\|\kappa^{\f{3}{2}}\Delta\theta\|_{\hco^{m-1}}\big)\\
  &\qquad \cdot \Lambda\big(\f{1}{c_0}, \il Z(\Gamma\beta)\il_{[\f{m-1}{2}],\infty,t}+\il\kappa\Delta\theta\il_{[\f{m}{2}]-1,\infty,t}\big)\\
  &\lesssim \lat\cE_{m,t}.
 \end{align*}
 After some  (slightly lengthy) further computations, we would find:
 \beq\label{sec2:eq7}
  \|\kpa^{\f{1}{2}}\cC_{\theta}^I\|_{L_t^2L^2}\lesssim \lat\cE_{m,t}.
 \eeq 
 The proof is thus complete in view of estimates \eqref{sec2:eq5}, \eqref{sec2:eq6}, \eqref{sec2:eq7}.
\end{proof}

\subsection{$\ep-$dependent estimate-II}
In this subsection, we aim to control $\kappa^{\f{1}{2}}\|(\ep\nabla\sigma,\ep\div u,\nabla\theta)\|_{L_t^{\infty}H_{co}^{m-1}}.$ More precisely, the following result will be shown:
\begin{prop}\label{prop-highest-2}
Under the same assumption as in Proposition \ref{prop-highest}, we have that for any $\ep \in (0,1], (\mu,\kpa)\in A,$  any $0<t\leq T,$
\beq\label{EE-2}
\begin{aligned}
    &\kappa\|(\nabla(\vr,\theta),\div u_{\ep})\|_{\infco^{m-1}}^2+\kpa\mu\|\nabla\div u_{\ep}\|_{\hco^{m-1}}^2+\kpa^2\|%\div(\beta\nabla\theta)
    \nabla^2\theta\|_{\hco^{m-1}}^2\\
    &\lesssim Y^2_{m}(0)+(T+\ep)^{\f{1}{2}}\lat \cE_{m,t}^2.
\end{aligned}
\eeq
\end{prop}
\begin{proof}
Denote $h_{\ep}=\f{1}{R\beta}u\cdot\nabla u_{\ep}+\nabla\sigma-\mu\Gamma\div\cL u_{\ep},$ and rewrite the equation of $u_{\ep}$ in \eqref{newsys} as:
\beqs
\f{1}{R}\pt u_{\ep}+\beta(\theta) h_{\ep}=0.
\eeqs
Taking a vector field $Z^J (|J|\leq m-1),$ applying $\div Z^J$ on the above equations,
multiplying by $\f{\kappa }{\beta(\theta)}Z^J\div u_{\ep},$ and integrating in space and time, we find that:
\begin{align}\label{sec2:eq8}
&\f{\kappa}{2} \iomega \f{1}{R\beta} |Z^J\div u_{\ep}|^2(t)\d x-\kpa \izto Z^J(\beta h_{\ep})\nabla\big(\f{1}{\beta}Z^J\div u_{\ep}\big)\d x\d s\notag\\
&= \f{\kappa}{2} \iomega \f{1}{R\beta} |Z^J\div u_{\ep}|^2(0)\d x-\kpa\izto \f{1}{R\beta}\ep\pt[\div, Z^J] u \cdot Z^J \div u_{\ep}\, \d x\d s\\
&\quad +\kpa\izt\int_{\p\Omega} %Z^J(\beta h_{\ep})
\f{1}{R\beta} Z^J\div u_{\ep}  (Z^J \ep\pt u)\cdot\bn\,\d S_y\d s.\notag
\end{align}
Let us write further the second term  in the L.H.S. of the above identity as:
\begin{align}\label{sec2:eq9}
-\kpa \izto Z^J(\beta h_{\ep})\nabla&\big(\f{1}{\beta}Z^J\div u_{\ep}\big)\d x\d s=-\kpa\izto Z^J\nabla\sigma \cdot \nabla Z^J \div u_{\ep}\,\d x\d s\notag\\
%& \qquad+\kappa\mu(2\lambda_1+\lambda_2) \izto |Z^J(\nabla\div u)|^2\,\d x \d s+\sum_{j=1}^5\cI_j,\notag
&+\kappa\mu(2\lambda_1+\lambda_2) \izto\Gamma |Z^J(\nabla\div u_{\ep})|^2\,\d x \d s-\sum_{j=3}^7\cI_j,
\end{align}
where 
\begin{align}\label{sec2:eq9.5}
&\cI_3=\f{\kappa}{R} \izto  Z^J(u\cdot\nabla u_{\ep})\nabla\big(\f{1}{\beta}Z^J\div u_{\ep}\big)\, \d x\d s,\notag\\
&\cI_4= -\kpa\izto \f{\nabla\beta}{\beta}\cdot (Z^J\nabla\sigma-\mu\Gamma Z^J\div\cL u_{\ep} )Z^J\div u_{\ep}\,\d x\d s,\notag\\
&\cI_5=\kappa \izto\big([Z^J,\beta]\na\sigma -\mu [Z^J, \Gamma\beta]\div\cL u_{\ep}\big)\nabla\big(\f{1}{\beta}Z^J\div u_{\ep}\big) \,\d x\d s,\\
&\cI_6=-\kpa\mu (2\lambda_1+\lambda_2)\izto  %\div\cL u_{\ep}
 \Gamma Z^J( \nabla \div u_{\ep}) [\nabla,Z^J]\div u_{\ep}\,\d x\d s,\notag\\
&\cI_7=\kpa\mu \lambda_1 \izto \Gamma Z^J \curl\curl u_{\ep} \cdot\nabla Z^J\div u_{\ep}\,\d x\d s.\notag
\end{align}

Next, applying $ Z^J\nabla $ and $\nabla Z^J$ to the first and the third equation of \eqref{newsys} respectively, we find that:
\beq\label{sec2:eq9.75}
\left\{
\begin{array}{l}
   (\pt+u\cdot\nabla)(Z^J \nabla\vr)+\f{1}{\gamma\ep}\nabla Z^J\div u_{\ep}= \mathscr{C}_{\vr}^J,  \\[5pt]
   \f{C_v}{R}  (\pt+u\cdot\nabla) Z^J \nabla\theta+\f{1}{\ep}\nabla Z^J \div u_{\ep}-\kappa \nabla Z^J[\Gamma\div(\beta\nabla\theta)]
     =\mathscr{C}_{\theta}^J,
\end{array}
\right.
\eeq
where 
\beq\label{defscrrhotheta}
\begin{aligned}
   &\mathscr{C}_{\vr}^J=\colon-Z^J(\nabla u\cdot\nabla\vr)-[Z^J,u\cdot\nabla]\nabla \vr -\f{1}{\gamma}[Z^J,\nabla]\div u,\\
  & \mathscr{C}_{\theta}^J=\colon -\f{C_v}{R}\big([\nabla,Z^J]\pt\theta+[\nabla Z^J, u\cdot\nabla]\theta\big).
\end{aligned}
\eeq
Multiplying the equations \eqref{sec2:eq9.75}  by $\kpa(\gamma^2Z^J\nabla\vr,\f{C_v(\gamma-1)}{R}Z^J\nabla\theta)$, integrating in space and time, we find after suitable integration by parts that:
\begin{align}\label{sec2:eq10}
    &\quad\f{\kpa}{2}\iomega \gamma^2|Z^J\nabla\vr|^2+\f{C_v^2(\gamma-1)}{R^2}|Z^J\nabla\theta|^2(t)\,\d x+ \f{C_v(\gamma-1)}{R}\kpa^2\izto \Gamma\beta|Z^J\Delta\theta|^2\d x\d s\notag\\
    &= \f{\kappa}{2}\iomega \gamma^2|Z^J\nabla\vr|^2+\f{C_v^2(\gamma-1)}{R^2}|Z^J\nabla\theta|^2(0)\,\d x
    -\kpa\izto Z^J\nabla\sigma \cdot \nabla Z^J \div u_{\ep}\,\d x\d s
    +\sum_{j=8}^{12}\cI_j,
    \end{align}
    where 
    \begin{align}\label{sec2:eq11}
   &\cI_{8}=\f{C_v(\gamma-1)}{R}\kpa^2\izt\int_{\p\Omega}Z^J[\Gamma\div(\beta\nabla\theta)]Z^J\nabla\theta\cdot\bn\,\d S_y\d s,\notag\\
   & \cI_9=- \f{C_v(\gamma-1)}{R}\kpa^2\izto\Gamma\beta Z^J\Delta\theta [\div,Z^J]\nabla\theta\,\d x\d s,\notag\\
    &\cI_{10}=- \f{C_v(\gamma-1)}{R}\kpa^2\izto \big([Z^J,\Gamma\beta]\Delta\theta+ Z^J(\Gamma\nabla\beta\cdot\nabla\theta)\big) \div (Z^J \nabla\theta)\,\d x\d s,\\
    &\cI_{11}=\f{\kappa}{2}\izto \div u\big(\gamma^2 |Z^J\nabla\vr|^2+\f{C_v^2(\gamma-1)}{R^2}|Z^J\nabla\theta|^2\big) \,\d x\d s,\notag\\
    & \cI_{12}=\kappa \izto \gamma^2 \mathscr{C}_{\vr}^J \cdot Z^J\nabla\vr +\f{C_v(\gamma-1)}{R} \mathscr{C}_{\theta}^J\cdot Z^J\nabla\theta \,\d x\d s.\notag
\end{align}
Note that the relation $\gamma\nabla\vr+\f{C_v(\gamma-1)}{R}\nabla\theta=\ep\nabla\sigma$ has been used in the derivation of \eqref{sec2:eq10}.
Collecting \eqref{sec2:eq8}, \eqref{sec2:eq9} and \eqref{sec2:eq10}, we find the following identity:
\begin{align}
&\f{\kappa}{2} \iomega \big(\f{1}{R\beta} |Z^J\div u_{\ep}|^2+\gamma^2|Z^J\nabla\vr|^2+\f{C_v^2(\gamma-1)}{R^2}|Z^J\nabla\theta|^2\big)(t)\,\d x\notag\\
&+\kappa \izto \mu(2\lambda_1+\lambda_2)|Z^J(\nabla\div u)|^2+\f{C_v(\gamma-1)}{R} \Gamma\beta|Z^J\Delta\theta|^2\,\d x \d s= \sum_{j=0}^{12}\cI_j,%\notag\\&
\end{align}
where $\cT_{3}-\cI_{7}$ and $\cT_{8}-\cT_{11}$ are defined in \eqref{sec2:eq9.5}-\eqref{sec2:eq11}
and
\begin{align*}
  & \cI_0= \f{\kappa}{2} \iomega \big(\f{1}{R\beta} |Z^J\div u_{\ep}|^2+\gamma^2|Z^J\nabla\vr|^2+\f{C_v^2(\gamma-1)}{R^2}|Z^J\nabla\theta|^2\big)(0)\,\d x, \\
   & \cI_1=-\kpa\izto \f{1}{R\beta}\ep\pt[\div, Z^J] u \cdot Z^J \div u_{\ep}\, \d x\d s,\\
& \cI_2=\kpa\izt\int_{\p\Omega} %Z^J(\beta h_{\ep})
\f{1}{R\beta} Z^J\div u_{\ep}  (Z^J \ep\pt u)\cdot\bn\,\d S_y\d s.
\end{align*}
To get the desired energy estimates, it suffices to control the terms $\cT_0-\cT_{12}$ which is the task of the following proof.
By the definition of $Y_m(0)$ in \eqref{initialnorm}
and the property \eqref{Gamma-Beta-infty}, 
one can find that:
\beq\label{i0}
\cI_0\lesssim \|\kpa^{\f{1}{2}}\na(\ep\sigma, \ep u, \theta)(0)\|_{H_{co}^{m-1}}^2\lesssim  Y_m^2(0).
\eeq
For the term $\cI_1,$ we use the Cauchy-Schwarz inequality, the assumption \eqref{preasption} and the identity \eqref{comu} to get that:
\begin{align}
    \cI_1&\lesssim \kpa \big\|Z^J\div u_{\ep}\big\|_{L_t^2L^2}\big\|\ep\pt [\div,Z^J]u \big\|_{L_t^2L^2}\notag\\
    &\lesssim T^{\f{1}{2}}\|\kpa^{\f{1}{2}}\div u_{\ep}\|_{L_t^{\infty}H_{co}^{m-1}}
    \|\kappa^{\f{1}{2}}\nabla u\|_{\hco^{m-1}}\lesssim T^{\f{1}{2}}\cE_{m,t}^2.
\end{align}
For the next term $\cI_2,$ one can assume that $Z^J$ contains at least one spatial tangential derivative $\p_{y^i} \,(i=1,2),$ otherwise it vanishes identically. By 
the fact $$Z^J(\ep\pt u)\cdot\bn= [Z^J\ep\pt, \bn]u$$
and by the duality between $H^{-\f{1}{2}}(\p\Omega)$ and  $H^{\f{1}{2}}(\p\Omega),$ one finds that:
\begin{align*}
    \cI_{2}\lesssim \kpa \big|\div u_{\ep}%)^b
    \big|_{\hcob^{m-\f{3}{2}}} \big|[Z^J\ep\pt, \bn]u\big|_{L_t^2H^{\f{1}{2}}}.
\end{align*}
The trace inequality \eqref{normaltraceineq} then yields:
\begin{align}\label{i2}
     \cI_{2}\lesssim \ep \|\kpa^{\f{1}{2}}(\nabla \div u,\div u)\|_{\hco^{m-2}}\|\kpa^{\f{1}{2}}(\nabla u, u)\|_{\hco^{m-1}}\lesssim \ep \cE_{m,t}^2.
     %\notag\\ &
\end{align}

We now focus on the estimates of $\cI_3-\cI_{12}$ which are defined in \eqref{sec2:eq9.5}, \eqref{sec2:eq11}. Let us begin 
with the term $\cI_3.$\\[3pt]
$\underline{\text{Estimate of } \cI_3}.$ To avoid losing derivatives, we integrate by parts in  space to write:
$\cI_3= \cI_{3,1}+\cI_{3,2}+\cI_{3,3},$
where
\beq\label{defcI31-I33}
\begin{aligned}
  &\cI_{3,1}=  \f{\kpa}{2}\izto \div(\f{u}{R\beta}) |Z^J\div u_{\ep}|^2\,\d x\d s, \quad   \cI_{3,2}=\kpa\izt\int_{\p\Omega}\f{1}{R\beta}Z^J\div u_{\ep}Z^J(u\cdot\nabla u_{\ep})\cdot\bn\, \d S_y\d s,\\
  &\cI_{3,3}=-\kpa\izto \big([\div,Z^J](u\cdot\nabla u_{\ep})+Z^J(\nabla u\cdot\nabla u_{\ep})+[Z^J, u\cdot\na]\div u_{\ep}\big) \f{1} {R\beta} Z^J  \div u_{\ep}\d x\d s.
\end{aligned}
\eeq
Thanks to \eqref{Gamma-Beta-infty},  the first term $\cI_{3,1}$ can be controlled as: 
\beq\label{i32}
\cI_{3,1}\lesssim T\|\kpa^{\f{1}{2}}Z^J\div u_{\ep}\|_{L^{\infty}L^2}^2\lab \il(\div u, \na\theta)\il_{0,\infty,t}\big)\lesssim T\lat\cE_{m,t}^2.
\eeq
Let us now bound $\cI_{3,2}$ as follows:
\begin{align}
    \cI_{3,2}&\lesssim \kpa^{\f{3}{4}}|\div u_{\ep}|_{\hcob^{m-1}} \kpa^{\f{1}{4}}|Z^J(u\cdot\nabla u_{\ep}) \cdot\bn|_{L_t^2L^2(\p\Omega)}.\notag
\end{align}
Since $(u\cdot\nabla)|_{\p\Omega}=(u_1\p_{y^1}+u_2\p_{y^2})|_{\p\Omega}$ is a tangential vector and $(u\cdot\bn)|_{\p\Omega}=0,$
one can write 
\beqs
Z^J(u\cdot\nabla u_{\ep})\cdot\bn= [Z^J, \bn\cdot](u_1\p_{y_1}u_{\ep}+u_2\p_{y_2} u_{\ep})-Z^J\big((u\cdot \nabla \bn)\cdot u_{\ep}\big) \text{ on } \p\Omega,
\eeqs
which together with the product estimate \eqref{product-bd} and the trace inequality \eqref{traceL2}, gives rise to:
\begin{align*}
    \kpa^{\f{1}{4}}|Z^J(u\cdot\nabla u_{\ep}) \cdot\bn|_{L_t^2L^2(\p\Omega)}&\lesssim 
\kpa^{\f{1}{4}}|u_{\ep}|_{\hcob^{m-1}}\il u \il_{[\f{m}{2}],\infty,t}\\
&\lesssim \ep \|(\kpa^{\f{1}{2}}\nabla u,u)\|_{\hco^{m-1}}\lat.
\end{align*}
This, combined with the following estimate which can be derived again from the trace inequality \eqref{traceL2}
\beqs
\kpa^{\f{3}{4}}|\div u_{\ep}|_{\hcob^{m-1}}\lesssim  \|(\kpa \nabla\div u_{\ep},\kpa^{\f{1}{2}}\div u_{\ep})\|_{\hco^{m-1}}\lesssim \cE_{m,t},
\eeqs
yields that,
\beq\label{i31}
  \cI_{3,2}\lesssim \ep \lat\cE_{m,t}^2.
\eeq 
We now proceed to deal with the term $\cI_{3,3}$ defined in \eqref{defcI31-I33}.
It follows from the identity \eqref{comu} that :
\beq\label{cI33-0}
\cI_{3,3}\lesssim T^{\f{1}{2}}\|\kpa^{\f{1}{2}} Z^J\div u_{\ep}\|_{L^{\infty}L^2}\kpa^{\f{1}{2}} \big(\|(\nabla u\cdot\nabla u_{\ep})\|_{\hco^{m-1}}+\|u\cdot\nabla\na u_{\ep}\|_{\hco^{m-2}} %\|\kpa^{\f{1}{2}}(\nabla u\cdot\nabla u_{\ep}, u\cdot\na\nabla u_{\ep}) \|_{\hco^{m-2}}
+\|[Z^J, u\cdot \na ]\div u_{\ep} \|_{L_t^2L^2} \big). %\big(\|\kpa^{\f{1}{2}}\nabla(u\cdot\nabla u_{\ep})\|_{\hco^{m-2}}+\|\kpa^{\f{1}{2}}(\nabla u\cdot\nabla u_{\ep})\|_{\hco^{m-1}}\big).
\eeq
By counting the derivatives hitting on each term and using \eqref{useful-inftyhalf}, one finds  
that: %the product estimate \eqref{roughproduct-2}, 
\beq\label{cI33-1}
\begin{aligned}
    \|\kpa^{\f{1}{2}}(\nabla u\cdot\nabla u_{\ep})\|_{\hco^{m-1}}&\lesssim \ep \|\kpa^{\f{1}{2}}\nabla u\|_{\hco^{m-1}}\il \nabla u\il_{[\f{m}{2}]-1,\infty,t}+\|\nabla u\|_{\hco^{m-2}}\il\ep\kpa^{\f{1}{2}}\nabla u\il_{[\f{m-1}{2}],\infty,t}\\
    &\lesssim \cA_{m,t} \cE_{m,t}.
\end{aligned}
\eeq
%Note that by the assumption $m\geq 7,$ 
%\begin{align}\label{inftynorm-2}    \il \nabla u\il_{[\f{m}{2}]-1,\infty,t}+\il\ep\kpa^{\f{1}{2}}\nabla u\il_{[\f{m-1}{2}],\infty,t}\lesssim  \il \nabla u\il_{m-5,\infty,t}+\il\ep\kpa^{\f{1}{2}}\nabla u\il_{m-4,\infty,t}\lesssim \cA_{m,t}.\end{align}
The term $\kpa^{\f{1}{2}}u\cdot\nabla\nabla u_{\ep}$ is easier to control in the interior domain $\Omega_0$, where the conormal spaces are equivalent to the usual Sobolev norm, we thus
focus on the case near the boundary. Let $\chi_i$ be the cut-off function associated to the chart $\Omega_i.$ 
We have by using the local coordinates that 
\beq\label{id-convetion}
\chi_i u\cdot\nabla =\chi_i \big(u_1\p_{y_1}+u_2\p_{y_2}+\f{u\cdot\bN}{\phi}\phi\p_z\big).
\eeq
Moreover, since $u\cdot\bN|_{\p\Omega}=0,$ it follows from the Hardy inequality and the fundamental theorem of calculus that
\begin{align}\label{hardy-calculus}
    \|\chi_i\f{u\cdot\bN}{\phi}\|_{\hco^l}\lesssim\|(u,\na u)\|_{\hco^l},\quad  \il \f{u\cdot\bN}{\phi}\il_{l,\infty,t}\lesssim \il (u, \na u)\il_{l,\infty,t}.
\end{align}
As a result, one obtains by noticing \eqref{useful-inftyhalf} that
%Therefore, we control the term 
\begin{align*}
\kpa ^{\f{1}{2}}\|\chi_i  u\cdot\nabla \nabla u_{\ep}\|_{\hco^{m-2}}&\lesssim \|(u, \nabla u)\|_{\hco^{m-2}} \big(\il(u, \nabla u)\il_{[\f{m}{2}]-1,\infty,t}+\il \ep\kpa^{\f{1}{2}}\nabla u\il_{[\f{m-1}{2}],\infty,t}\big)\\
&\qquad\qquad\qquad +\ep\|\kpa ^{\f{1}{2}}\nabla u\|_{\hco^{m-1}}\il(u,\nabla u)\il_{0,\infty,t}\\
&\lesssim \lat\cE_{m,t}.
\end{align*}
To summarize, we have found that:
\begin{align}\label{cI33-2}
    \kpa ^{\f{1}{2}}\|  u\cdot\nabla \nabla u_{\ep}\|_{\hco^{m-2}}\lesssim \lat\cE_{m,t}.
\end{align}
Finally, as $[Z^J, u\cdot \na ]\div u_{\ep}=[Z^J, u]\cdot\na\div u_{\ep} +u\cdot [Z^J, \na]\div u_{\ep},$
it follows from  the identity \eqref{comu}, the commutator estimate \eqref{roughcom1} and the fact \eqref{useful-inftyhalf} that:
\beq\label{cI33-3}
\begin{aligned}
&\kpa^{\f{1}{2}}\|[Z^J, u\cdot \na ]\div u_{\ep}\|_{L_t^2L^2}\\
&\lesssim 
\ep\kpa^{\f{1}{2}}\|\na\div u\|_{\hco^{m-2}}\il u\il_{[\f{m}{2}],\infty,t}+\|u\|_{\hco^{m-1}}\il\ep\kpa^{\f{1}{2}}\na\div u\il_{[\f{m-1}{2}]-1,\infty,t}\\
&\lesssim \lat\cE_{m,t}.
\end{aligned}
\eeq
%Again, we have used the assumption $m\geq 7$ to find that:
%\begin{align*} \il\ep\kpa^{\f{1}{2}}\na\div u\il_{[\f{m-1}{2}]-1,\infty,t}\lesssim \il\ep \mu^{\f{1}{2}} \na^2 u\il_{m-5,\infty,t}\lesssim \cA_{m,t}.\end{align*}
Plugging the estimates \eqref{cI33-1}, \eqref{cI33-2}, \eqref{cI33-3} into \eqref{cI33-0}, we obtain that:
\beqs
\cI_{3,3}\lesssim T^{\f{1}{2}}\lat\cE_{m,t}^2,
\eeqs
which, combined with \eqref{i31} and \eqref{i32}, yields:
\beq\label{i3}
\cI_{3}\lesssim (T^{\f{1}{2}}+\ep)\lat\cE_{m,t}^2.
\eeq
$\underline{\text{Estimate of } \cI_4}.$
%We proceed to control $\cI_4.$ 
In view of the very definition in \eqref{sec2:eq9.5}, we have by the Cauchy-Schwarz inequality and the inequality \eqref{Gamma-Beta-infty} that:
\begin{align}
\cI_4&\lesssim T^{\f{1}{2}}\|\kpa^{\f{1}{2}}\div u_{\ep}\|_{L_t^{\infty}H_{co}^{m-1}}\|\kpa^{\f{1}{2}}(\nabla\sigma, \mu \div\cL u_{\ep})\|_{\hco^{m-1}}\Lambda\big(\f{1}{c_0},\il\nabla\theta\il_{0,\infty,t}\big)\notag\\
&\lesssim T^{\f{1}{2}} \lat\cE_{m,t}^2.
\end{align}
$\underline{\text{Estimate of } \cI_5}.$ Thanks to the  Cauchy-Schwarz inequality and the commutator estimate \eqref{roughcom}, one has that:
\begin{align*}
    \cI_5&=\kappa \izto\big([Z^J,\beta]\na\sigma -\mu [Z^J, \Gamma\beta]\div\cL u_{\ep}\big)\nabla\big(\f{1}{\beta}Z^J\div u_{\ep}\big) \,\d x\d s %-\kappa \izto \mu [Z^J,\beta]\div\cL u_{\ep}\nabla\big(\f{1}{\beta}Z^J\div u_{\ep}\big) \,\d x\d s
    \notag\\
    &\lesssim \kpa \big\|\big( [Z^J,\beta]\na\sigma,  \mu[Z^J, \Gamma\beta]\div\cL u_{\ep}\big)\big\|_{L_t^2L^2}\|\nabla\big(\f{1}{\beta}Z^J\div u_{\ep})\|_{L_t^2L^2}\notag\\
    &\lesssim T^{\f{1}{2}}\|(Z\beta, Z(\Gamma\beta), \na\sigma, \mu\div\cL u_{\ep})\|_{\infco^{m-2}}\|\kpa(\nabla\div u_{\ep}, \div u_{\ep}, \na\theta)\|_{\hco^{m-1}}\lat .
\end{align*}
Moreover, it follows from the estimates \eqref{esofbeta}, \eqref{esofGamma-2} that
%and the definition of $\cE_{m,t},$ 
\begin{align*}
\|(Z\beta, Z(\Gamma\beta))\|_{\infco^{m-2}}\lesssim 
\il(\theta,\ep\sigma)\il_{\infco^{m-1}}\lat\lesssim \lat\cE_{m,t}.
\end{align*}
The above two estimates, together with the definition of  $\cE_{m,t}$ (see \eqref{defcEmt}) 
then lead to:
\begin{align}
   \cI_{5} \lesssim  T^{\f{1}{2}}\lat\cE_{m,t}^2.
\end{align}
$\underline{\text{Estimate of } \cI_6}.$ In view of the identity \eqref{comu}, this one can be bounded  as:
\begin{align}
  \cI_6 \lesssim \ep\|(\kpa\mu)^{\f{1}{2}}\nabla\div u_{\ep}\|_{\hco^{m-1}}\|\kpa^{\f{1}{2}}\nabla\div u\|_{\hco^{m-2}}  \lesssim \ep\cE_{m,t}^2.  
\end{align}
$\underline{\text{Estimate of } \cI_7}.$ We split it into two terms:
$\cI_{7}=\cI_{7,1}+\cI_{7,2}$ where 
\begin{align*}
    \cI_{7,1}=\kpa\mu \lambda_1 \izto  \Gamma[Z^J ,\curl]\curl u_{\ep} \cdot\nabla Z^J\div u_{\ep}\,\d x\d s, \\
    \cI_{7,2}=\kpa\mu \lambda_1 \izto \Gamma \curl (Z^J\curl u_{\ep}) \cdot\nabla Z^J\div u_{\ep}\,\d x\d s.
\end{align*}
 Using the identity \eqref{comu}, the fact \eqref{Gamma-Beta-infty}, 
 one can estimate $ \cI_{7,1}$ as:
 \beq\label{i71}
  \cI_{7,1}\lesssim \ep^{\f{1}{2}}\|(\kpa\mu)^{\f{1}{2}}\nabla\div u_{\ep}\|_{\hco^{m-1}} \| (\ep\kpa\mu)^{\f{1}{2}} \nabla^2 u\|_{\hco^{m-2}}\lesssim \ep^{\f{1}{2}}\cE_{m,t}^2.
 \eeq
 Moreover, integrating by parts in space, one controls $\cI_{7,2}$ as follows:
 \begin{align*}
     \cI_{7,2}&=-\kpa\mu\lambda_1 \izt\int_{\p\Omega}\big(Z^J \curl u_{\ep}\times\bn\big)\cdot \Pi\nabla Z^J\div u_{\ep}\,\d S_y\d s\\
     &\qquad +\kpa\mu\lambda_1 \izto Z^J \curl u_{\ep}\cdot( \na\Gamma \times \na Z^J \div u_{\ep})\,\d x\d s\\
     &\lesssim \ep (\kpa\mu)^{\f{1}{2}} |Z^J\curl u\times\bn|_{L_t^2H^{\f{1}{2}}(\p\Omega)}  (\kpa\mu)^{\f{1}{2}}\big|\Pi\nabla Z^J\div u_{\ep}\big|_{L_t^2H^{-\f{1}{2}}(\p\Omega)} \\
     &\qquad+ \ep^2 \|(\kpa\mu)^{\f{1}{2}}\nabla(\div u_{\ep}, u)\|_{\hco^{m-1}}^2 \lab \il\na\sigma\il_{0,\infty,t}\big).
 \end{align*}
 As $\Pi\nabla$ contains only the tangential derivatives, we have by the trace inequality  \eqref{normaltraceineq} that:
 \beqs
  (\kpa\mu)^{\f{1}{2}}\big|\Pi\nabla Z^J\div u_{\ep}\big|_{L_t^2H^{-\f{1}{2}}(\p\Omega)}\lesssim 
  (\kappa\mu)^{\f{1}{2}}\|(\nabla\div u_{\ep},\div u_{\ep})\|_{\hco^{m-1}}.
 \eeqs
 Furthermore, it follows from the  identity \eqref{normalofnormalder}, the identity
 $$Z^J(\curl u)\times\bn=Z^J(\curl u\times\bn)+[Z^J,\bn\times]\curl u,$$
as well as the boundary conditions \eqref{bd-curlun}, \eqref{bdryconditionofu} that  
 \begin{align*}
    & (\kpa\mu)^{\f{1}{2}} |Z^J\curl u\times\bn|_{L_t^2H^{\f{1}{2}}(\p\Omega)} \lesssim 
     (\kpa\mu)^{\f{1}{2}}\big( |u|_{\hcob^{m-\f{1}{2}}}+|\curl u|_{\hcob^{m-\f{3}{2}}}\big)\\
     &\lesssim (\kpa\mu)^{\f{1}{2}}\big( |u|_{\hcob^{m-\f{1}{2}}}+|\div u|_{\hcob^{m-\f{3}{2}}}\big)\lesssim (\kpa\mu)^{\f{1}{2}}\big(\|\nabla u\|_{\hco^{m-1}}+\|\nabla\div u\|_{\hco^{m-2}}\big).
 \end{align*}
 The previous two inequalities then yield $\cI_{7,2}\lesssim \ep\lat\cE_{m,t}^2,$ which, combined with \eqref{i71}, leads to that:
 \beq
  \cI_{7}\lesssim \ep^{\f{1}{2}}\lat\cE_{m,t}^2.
 \eeq
 We now begin to control the terms $\cI_8-\cI_{12}$ defined in $\eqref{sec2:eq11}.$\\[3pt]
 $\underline{\text{Estimate of } \cI_8}.$ Due to the Neumann boundary condition of $\theta,$
 this term vanishes if $Z^J$ contains only the time derivatives, we thus can assume that $Z^J=\p_{y^i} Z^{\tilde{J}}, (i=1,2)$ where $|\tilde{J}|\lesssim m-2.$ Therefore, by integration by parts along the physical boundary $\p\Omega,$ %and by product estimate \eqref{product-bd} 
 we control $\cI_8$ as follows:
 \begin{align}
   \cI_8&=-\f{C_v(\gamma-1)}{R}\kpa^2\izt\int_{\p\Omega}Z^J\big(\Gamma\div(\beta\nabla\theta)\big)[Z^J,\bn\cdot]\nabla\theta\,\d S_y\d s\notag\\
   &=\f{C_v(\gamma-1)}{R}\kpa^2\izt\int_{\p\Omega}Z^{\tilde{J}}\big(\Gamma\div(\beta\nabla\theta)\big)\p_{y^i}\big([Z^J,\bn\cdot]\nabla\theta\big)\,\d S_y\d s\notag\\
   &\lesssim \kpa^2 \big|\Gamma\div(\beta\nabla\theta)\big|_{\hcob^{m-2}}\big|\nabla\theta\big|_{\hcob^{m-1}}.\notag
 \end{align}
 By the product estimate on the boundary \eqref{product-bd}, the assumption $m\geq 7,$ %and Corollary \ref{cor-gb},
 \begin{align*}
    &\kpa^{\f{5}{4}}|\Gamma\div(\beta\nabla\theta)\big|_{\hcob^{m-2}}\\
    &\lesssim  \big(|Z \Gamma|_{\hcob^{m-3}}+\kpa^{\f{5}{4}}|\div(\beta\nabla\theta)|_{\hcob^{m-2}}\big)  \lab \il(\sigma,\kpa \div(\beta\nabla\theta )\il_{[\f{m}{2}]-1,\infty,t} %\il \il_{[\f{m}{2}]-1,\infty,t} 
    \big)\\
    & \lesssim  \big|\big(\ep\sigma,\theta,\kpa^{\f{5}{4}}\nabla\theta,\kpa^{\f{5}{4}}\Delta\theta\big)\big|_{\hcob^{m-2}}\lat.%+|()^b|_{\hcob^{m-2}}+|()^b|_{{\hcob^{m-2}}}
 \end{align*}
Thanks to the trace inequality \eqref{traceL2}:
 \begin{align*}
 & |\ep\sigma,\theta|_{\hco^{m-2}}\lesssim \|(\nabla,\text{Id})(\ep\sigma,\theta)\|_{\hco^{m-2}}\lesssim T^{\f{1}{2}}\cE_{m,t},\\
    & \kpa^{\f{5}{4}}|\Delta\theta|_{\hco^{m-2}}\lesssim \|(\kpa^{\f{3}{2}}\nabla\Delta\theta,\kpa\Delta\theta)\|_{\hco^{m-2}}^{\f{1}{2}}\|\kpa\Delta\theta\|_{\hco^{m-2}}^{\f{1}{2}}\lesssim T^{\f{1}{4}} \cE_{m,t},\\
    &\kpa^{\f{3}{4}}|\nabla\theta|_{\hcob^{m-1}}\lesssim  \|(\kpa \nabla^2\theta, \kpa^{\f{1}{2}}\nabla\theta)\|_{\hco^{m-1}}^{\f{1}{2}}\|\kpa^{\f{1}{2}}\nabla\theta\|_{\hco^{m-1}}^{\f{1}{2}}\lesssim T^{\f{1}{4}}\cE_{m,t}.
 \end{align*}
 Note that in light of the definition \eqref{defcEmt}
 for $\cE_{m,t}=\cE_{m,t}(\sigma, u,\theta)$
 $$ \|(\nabla,\text{Id})(\ep\sigma,\theta)\|_{L_t^{\infty}H_{co}^{m-2}}+ \|\kpa\Delta\theta\|_{L_t^{\infty}H_{co}^{m-2}}+\|\kpa^{\f{1}{2}}\nabla\theta\|_{L_t^{\infty}H_{co}^{m-1}}+\|\kpa^{\f{3}{2}}\na\Delta\theta\|_{\hco^{m-2}}\lesssim \cE_{m,t}.$$
 We thus obtain by collecting the previous three estimates that:
 \begin{align}
     \cI_8\lesssim T^{\f{1}{2}}\lat\cE_{m,t}^2.
 \end{align}
 $\underline{\text{Estimate of } \cI_9}.$ Simply using the Cauchy-Schwarz inequality and identity \eqref{comu}, one can get that:
 \beq
 \cI_9\lesssim \|\kpa\Delta\theta\|_{\hco^{m-1}}\|\kpa\nabla^2\theta\|_{\hco^{m-2}}\lesssim 
 T^{\f{1}{2}}\cE_{m,t}^2.
 \eeq
 $\underline{\text{Estimate of } \cI_{10}}.$ 
 We estimate this term by the Cauchy-Schwarz inequality, the commutator estimate \eqref{roughcom} and the product estimate \eqref{roughproduct1}:
 \begin{align}
   \cI_{10}&\lesssim \|\kpa  \div (Z^J \nabla\theta)\|_{L_t^2L^2}  \big(\|\kpa [Z^J,\Gamma\beta]\Delta\theta\|_{L_t^2L^2}+ \|\kpa\Gamma\nabla\beta\cdot\nabla\theta\|_{\hco^{m-1}}\big) \notag \\
   &\lesssim \|\kpa\Delta\theta\|_{\hco^{m-1}}\big(\|\kpa \Delta\theta\|_{\hco^{m-2}}+\|(\kpa^{\f{1}{2}}\nabla\theta,\theta,\sigma)\|_{\hco^{m-1}}\big)\lat\\
   &\lesssim T^{\f{1}{2}}\lat\cE_{m,t}^2.\notag
 \end{align}
$\underline{\text{Estimate of } \cI_{11}}.$ It is direct to see that:
\begin{align}
    \cI_{11}\lesssim T\kpa\|(\nabla\vr,\nabla\theta)\|_{L_t^{\infty}H_{co}^{m-1}}^2\lab\il
    \div u\il_{0,\infty,t}\big)\lesssim T\lat \cE_{m,t}^2.
\end{align}
$\underline{\text{Estimate of } \cI_{12}}.$ Let us bound this term as:
\beqs
 \cI_{12}\lesssim T^{\f{1}{2}}\|\kpa^{\f{1}{2}}\nabla (\vr,\theta)\|_{L_t^{\infty}H_{co}^{m-1}}\|\kpa^{\f{1}{2}}(\mathscr{C}_{\vr}^J, \mathscr{C}_{\theta}^J)\|_{\hco^{m-1}},
\eeqs
%By using Lemma \ref{prop-prdcom}, one can check that:
which, together with the following estimate that will be shown in the next lemma, 
\beqs
\|\kpa^{\f{1}{2}}(\mathscr{C}_{\vr}^J, \mathscr{C}_{\theta}^J)\|_{L_t^2L^2}\lesssim \lat\cE_{m,t},
\eeqs
leads to the estimate
\beq\label{i12}
 \cI_{12}\lesssim T^{\f{1}{2}}\lat\cE_{m,t}^2.
\eeq

  Finally, inserting the estimates \eqref{i0}-\eqref{i2}, \eqref{i3}-\eqref{i12} into \eqref{sec2:eq10}, we achieve the  desired estimate \eqref{EE-2}.
\end{proof}
\begin{lem}
Let $\mathscr{C}_{\vr}^J, \mathscr{C}_{\theta}^J$ be defined in \eqref{defscrrhotheta}, we have,  
under the same assumption as in Proposition \ref{prop-highest-2} that, for any $|J|\leq m-1,$ any $0<t\leq T,$
\begin{align*}
    \|\kpa^{\f{1}{2}}(\mathscr{C}_{\vr}^J, \mathscr{C}_{\theta}^J)\|_{L_t^2L^2}\lesssim \lat\cE_{m,t}.
\end{align*}
\end{lem}
\begin{proof}
The above estimate can be shown by applying the product and  commutator estimates in Proposition \ref{prop-prdcom}, %\eqref{roughcom}, \eqref{roughcom1}, 
let us only sketch the estimate of 
$\mathscr{C}_{\theta}^J=-\f{C_v}{R}\big([\nabla,Z^J]\pt\theta+[\nabla Z^J, u\cdot\nabla]\theta\big).$ 
It follows from the fact \eqref{comu} and the equation $\eqref{NCNS-S2}_3$
that, 
\begin{align*}
 \kpa^{\f{1}{2}}\| [\nabla,Z^J]\pt\theta \|_{L_t^2L^2}\lesssim \|\na\pt\theta\|_{\hco^{m-2}}&\lesssim \kpa^{\f{1}{2}}\|\na \big((u\cdot\na)\theta, \div u, \Gamma \div(\beta\na\theta)\big)\|_{\hco^{m-2}} \\
 &\lesssim \lat\cE_{m,t}.
\end{align*}
Moreover, we can bound $[\nabla Z^J, u\cdot\nabla]\theta$ as:
\beqs 
\begin{aligned}
    &\kpa^{\f{1}{2}}\|[\nabla Z^J, u\cdot\nabla]\theta\|_{L_t^2L^2}\lesssim \kpa^{\f{1}{2}}\|\na (u\cdot\na)\theta\|_{\hco^{m-2}}\\
    &+\kpa^{\f{1}{2}}
    \big(\|[Z^J, \na u]\na\theta\|_{L_t^2L^2}+\|[Z^J, u\cdot]\na\na\theta\|_{L_t^2L^2}+\|\na^2\theta\|_{\hco^{m-2}}\il u \il_{0,\infty,t} \big)\\
    &\lesssim \lat\cE_{m,t}.
\end{aligned}
\eeqs
\end{proof}
\section{Uniform estimates for the time derivatives}\label{sec-timeder}
In this section, we prove the uniform bound for the (weighted) time derivatives to the order $m-1,$ namely $\|(\sigma,u,\theta)\|_{L_t^2\cH^{m-1,0}}.$ As the time derivatives commute with the singular terms  in the system \eqref{NCNS-S2}, this can be done by energy estimates. Nevertheless, due to the appearance of the term  $\f{\gamma-1}{\gamma}\kappa \Gamma \div(\beta\nabla\theta)$ in the equation $\eqref{NCNS-S2}_1,$ the (linearized) penalized operator associated to the singular terms in the system \eqref{NCNS-S2} 
\begin{align*}
  \left(  \begin{array}{ccc}
        0 &\div   & \f{\gamma-1}{\gamma}\kappa \overline{\Gamma} \div(\beta(0)\nabla\cdot)\\
        \nabla & 0 & 0\\
        0 & 0&0
    \end{array}
    \right)
\end{align*}
is not skew-symmetric. 
Moreover, the term $\div u$ in the equation $\eqref{NCNS-S2}_3$ is troublesome due to the absence of the a-priori control of $\|\div u\|_{L_t^2\cH^{m-1,0}}.$ 
Therefore, %such the direct energy estimates cannot be performed directly on 
the system $\eqref{NCNS-S2}$ is not an appropriate target to perform energy estimates. Inspired by \cite{MR2211706}, we introduce the unknown 
\beq\label{def-w}
w=u-\f{\gamma-1}{\gamma}\overline{\Gamma}\kappa \beta(\theta)\nabla\theta
\eeq
and consider the system satisfied by $(\sigma, w, \theta)$ as follows:
\beq\label{sys-uni}
\left\{
\begin{array}{l}
\f{1}{\gamma} (\pt+u\cdot\nabla)\sigma+\f{1}{\ep}\div w=F_{\sigma},      \\[5pt]
 \f{1}{R\beta(\theta)}(\pt+u\cdot\nabla)w  +\f{\nabla\sigma}{\ep}-\mu\overline{\Gamma}\div\cL u+\f{(\gamma-1)}{C_v\gamma}\overline{\Gamma}\kpa\na\big(-\div w+\f{\kpa}{\gamma}\overline{\Gamma} \div(\beta\nabla\theta)\big)\\[5pt]
  \qquad  \qquad  \qquad  \qquad  \qquad  \qquad \qquad  \qquad  \qquad =F_{w}-\ep\f{\kpa}{C_v}\overline{\Gamma}\na F_{\sigma},\\                  
 \f{C_v}{R}(\pt+u\cdot\nabla)\theta+\div w-\kpa\f{\overline{\Gamma}}{\gamma}\div(\beta\nabla\theta)=F_{\theta},
\end{array}
\right.
\eeq
with
\begin{align}
&F_{\sigma}=\f{\gamma-1}{\gamma}\big(\f{\Gamma-\overline{\Gamma}}{\ep}\kappa \div(\beta(\theta)\nabla\theta)+\mu\ep\Gamma(\ep\sigma)\cL u\cdot \mathbb{S}u\big),\notag\\
&F_{\theta}=\kpa (\Gamma-\overline{\Gamma})\div(\beta\nabla\theta) +\mu\ep^2\Gamma(\ep\sigma)\cL u\cdot \mathbb{S}u ,\\
& F_w=\kpa \overline{\Gamma} \f{\gamma-1}{R\gamma}%\kappa\grad \big((\Gamma-\overline{\Gamma}) \div(\beta\nabla\theta)\big)
\nabla u\cdot\nabla\theta-\kpa \overline{\Gamma}\f{\gamma-1}{R\gamma}(\nabla \ln \beta) (\pt+u\cdot\nabla)\theta+\ep\mu\f{\Gamma-\overline{\Gamma}}{\ep}\div\cL u.\notag
\end{align}
\begin{prop}
Suppose that the assumption \eqref{preasption} holds and $m\geq 7,$ then for any $\ep\in(0,1], (\mu,\kpa)\in A,$
any $0<t \leq T,$ we have  the following estimate: 
\begin{align}\label{EE-3}
&\|(\sigma,w,\kpa \nabla\theta)\|^2_{L_t^{\infty}\cH^{m-1,0}}+\kpa\|\div w\|_{L_t^2\cH^{m-1,0}}^2+\mu \|\nabla u\|_{L_t^2\cH^{m-1,0}}^2\notag\\
&\lesssim 
Y_m^2(0)+(T+\ep)^{\f{1}{2}}\lat\cE_{m,t}^2.
\end{align}
\end{prop}
\begin{proof}
Applying $(\ep\pt)^j$ to the equation \eqref{sys-uni} and denoting 
$(\sigma^j, w^j, \theta^j, u^j)=(\ep\pt)^j(\sigma, w,\theta,u),$ we find that 
\beq\label{sys-uni-j}
\left\{
\begin{array}{l}
\f{1}{\gamma} (\pt+u\cdot\nabla)\sigma^j+\f{1}{\ep}\div w^j=F_{\sigma}^j,      \\[5pt]
 \f{1}{R\beta(\theta)}(\pt+u\cdot\nabla)w^j +\f{\nabla\sigma^j}{\ep}-\mu\overline{\Gamma}\div\cL u^j+\f{(\gamma-1)}{C_v\gamma}\overline{\Gamma}
 \kpa\nabla\big(-\div w^j+\f{\kpa}{\gamma}\overline{\Gamma}\div(\ep\pt)^j(\beta\nabla\theta)\big)\\
 \qquad\qquad\qquad\qquad \qquad\qquad\qquad\qquad\qquad\qquad\qquad\qquad =F_{w}^j-\ep\f{\kpa}{C_v}\overline{\Gamma}\na (\ep\pt)^j F_{\sigma},  \\[5pt]
 \f{C_v}{R}(\pt+u\cdot\nabla)\theta^j+\div w^j-\kpa\f{\overline{\Gamma}}{\gamma}\div(\beta\nabla\theta^j)=F_{\theta}^j,
\end{array}
\right.
\eeq
where 
\begin{align}\label{sec3:eq0}
&  F_{\sigma}^j= (\ep\pt)^j F_{\sigma}+\f{1}{\gamma}[(\ep\pt)^j,u\cdot\nabla]\sigma,\notag\\
 &F_w^j=(\ep\pt)^j F_w-[(\ep\pt)^j, \f{1}{R\beta}]\pt w-[(\ep\pt)^j, \f{1}{R\beta}u\cdot\nabla]w, \\%+\ep\mu[(\ep\pt)^j, \Gamma]\div\cL u^{\ep},\\%-\kpa^2 (\overline{\Gamma}/{\gamma})^2 \f{(\gamma-1)}{C_v}\nabla\div \big([(\ep\pt)^j,\beta]\nabla\theta\big),\\
 & F_{\theta}^j=(\ep\pt)^j F_{\theta}-\f{C_v}{R}[(\ep\pt)^j, u\cdot\nabla]\theta+ (\kpa\overline{\Gamma}/{\gamma})\div\big([(\ep\pt)^j,\beta]\nabla\theta\big).\notag
\end{align}
Multiplying the equation \eqref{sys-uni-j} by $\big(\sigma^j, w^j, -\kpa^2(\overline{\Gamma}/{\gamma})^2 \f{(\gamma-1)}{C_v}\div(\beta\nabla\theta^j)\big)$ and integrating in space and time, one gets by using the boundary condition $u^j\cdot\bn|_{\p\Omega}=\p_{\bn}\theta^j|_{\p\Omega}=0$ the following identity:
\begin{align}\label{sec3:eq1}
&\f{1}{2}\int_{\Omega}\big(\f{1}{\gamma}|\sigma^j|^2+\f{1}{R\beta}|w^j|^2+\kpa^2 
(\overline{\Gamma}/{\gamma})^2 \f{(\gamma-1)}{R}\beta|\nabla\theta^j|^2\big)(t)\d x\notag\\
&+
\mu\izto \lambda_1\overline{\Gamma}|\nabla u^j|^2+(\lambda_1+\lambda_2)\overline{\Gamma}|\div u^j|^2 +\f{(\gamma-1)\overline{\Gamma}}{C_v\gamma}\kpa \big|\div w^j-(\overline{\Gamma}/{\gamma})\kpa \div(\beta\nabla\theta^j)\big|^2 \d x\d s\\
&=\sum_{k=0}^{6}\cG_{k}, \notag
\end{align}
where
\begin{align}
&\cG_{0}=\f{1}{2}\int_{\Omega}\big(\f{1}{\gamma}|\sigma^j|^2+\f{1}{R\beta}|w^j|^2+\kpa^2 
(\overline{\Gamma}/{\gamma})^2 \f{(\gamma-1)}{R}\beta|\nabla\theta^j|^2\big)(0)\,\d x,\notag\\
&\cG_{1}=\lambda_1\mu\izt\int_{\p\Omega} \p_{\bn}u^j\cdot w^j\d S_y\d s,\notag\\
&\cG_2=\f{1}{2}\izto \f{1}{\gamma}\div u|\sigma^j|^2+\f{1}{R}\big(\pt(\f{1}{\beta})+\div(\f{1}{\beta}u)\big)|w^j|^2\,\d x\d s,\notag\\
&\cG_3=\kpa^2 
(\overline{\Gamma}/{\gamma})^2\f{\gamma-1}{R}\izto\f{1}{2}(\pt\beta+\div(\beta u))|\nabla\theta^j|^2-\beta\nabla\theta^j\cdot(\nabla u\cdot\nabla\theta^j)\,\d x\d s,\\
&\cG_4=\mu\kpa \f{(\gamma-1)\overline{\Gamma}^2}{\gamma}\izto\lambda_1\nabla u^j \cdot \nabla(\ep\pt)^j(\beta \nabla\theta)+(\lambda_1+\lambda_2)\div u^j(\ep\pt)^j\div(\beta \nabla\theta) \,\d x\d s,\notag\\
&\cG_5= \kpa \overline{\Gamma} \f{(\gamma-1)}{C_v\gamma^2}\izto\bigg(
\overline{\Gamma}\kpa\div \big([(\ep\pt)^j,\beta]\nabla\theta\big)+ %(\ep\pt)^j\big((\Gamma-\overline{\Gamma})\div(\beta\nabla\theta)\big)
\f{\gamma^2}{\gamma-1}\ep (\ep\pt)^j F_{\sigma}\bigg)\div w^j\,\d x\d s, \notag\\
&\cG_6=\izto \sigma^j F_{\sigma}^j+w^j \cdot F_{w}^j- \kpa^2(\overline{\Gamma}/{\gamma})^2 \f{(\gamma-1)}{C_v}\div(\beta\nabla\theta^j) F_{\theta}^j\,\d x\d s.\notag
\end{align}
Therefore, it is just a matter to control $\cG_0-\cG_6$ defined above.
By the definition of $Y_m(0)$ in \eqref{initialnorm}:
\beq\label{cg0}
\cG_1\lesssim \|(\sigma,\kpa\na\theta, u)(0)\|_{H_{co}^{m-1}}^2 \lesssim  Y_m^2(0).
\eeq
For the term $\cG_1,$ we get  with the help of the the  boundary conditions $\eqref{bdryconditionofu}$ and
$(w^j\cdot\bn)|_{\p\Omega}=0,$
the trace inequality \eqref{traceL2} that: %, we bound the term $\cG_1$ in the following way: 
\begin{align}
    &\cG_1=\lambda_1\mu\izt\int_{\p\Omega}\Pi \p_{\bn}u^j\cdot \Pi w^j\d S_y\d s \lesssim \mu |(u^j,w^j)|_{L_t^2L^2(\p\Omega)}^2\notag\\
  &\lesssim (T\mu)^{\f{1}{2}}\|(u^j, w^j)\|_{L_t^{\infty}L^2} \|\mu^{\f{1}{2}}(u^j, w^j, \nabla u^j, \nabla w^j)\|_{L_t^2L^2}\lesssim T^{\f{1}{2}}\cE_{m,t}^2.\notag
\end{align}
Next, $\cG_2+\cG_3$ can be estimated directly as: 
\beq 
\cG_2+\cG_3 \lesssim \lat \|(\sigma^j, w^j, \kappa \nabla\theta^j)\|_{L_t^2L^2}^2\lesssim T\lat\cE_{m,t}^2.
\eeq
Moreover, the term $\cG_4$ is bounded by using the identity \eqref{roughproduct-2}, the estimates \eqref{esofbeta}, \eqref{esofGamma-2} for $\beta, \Gamma,$
 and the Young's inequality: 
\begin{align}
    \cG_4&\lesssim \mu^{\f{1}{2}} \|\mu^{\f{1}{2}}\nabla u^j\|_{L_t^2L^2} 
    \| \kpa\nabla (\beta\nabla\theta)\|_{L_t^2\cH^{m-1,0}}%\notag\\&
    \lesssim \mu^{\f{1}{2}} \|\mu^{\f{1}{2}}\nabla u^j\|_{L_t^2L^2}\cdot\notag\\
    &\qquad \qquad  \bigg(\|\kpa\nabla^2\theta\|_{\hco^{m-1}}+\big(\|(\theta,\kappa^{\f{1}{2}}\nabla\theta)\|_{%L_t^2\cH^{m-1,0}
    \hco^{m-1}}+\|\kpa\na^2\theta\|_{\hco^{m-2}}\big)\lat\bigg)\notag\\
    &\lesssim  \mu^{\f{1}{2}} \|\mu^{\f{1}{2}}\nabla u^j\|_{L_t^2L^2}\|\kpa\nabla^2\theta\|_{\hco^{m-1}}+T^{\f{1}{2}}\lat\cE_{m,t}^2\\
    &\leq \f{1}{2}\lambda_1 \mu \|\nabla u^j\|_{L_t^2L^2}^2+C(\lambda_1, c_0)\mu\|\kpa \nabla^2\theta\|_{\hco^{m-1}}^2+T^{\f{1}{2}}\lat\cE_{m,t}^2.\notag
\end{align}
As for the term $\cG_5$, one first derives from  the commutator and product estimates 
\eqref{roughcom}, \eqref{roughproduct1} and Corollary \ref{cor-gb} that:
\begin{align*}
\kpa\|\div \big([(\ep\pt)^j,\beta]\nabla\theta\big)\|_{L_t^2L^2}&\lesssim 
\kpa \|\big([(\ep\pt)^j,\beta]\Delta\theta, [(\ep\pt)^j,\nabla\beta\cdot]\nabla\theta\big)\|_{{L_t^2L^2}}\notag\\
&\lesssim\big( \kappa\|\Delta\theta\|_{\hco^{m-2}}+\|(\theta,\kpa^{\f{1}{2}}\nabla\theta)\|_{\hco^{m-1}}\big)\lat\notag\lesssim T^{\f{1}{2}}\lat\cE_{m,t},\notag\\
 %\kpa\|(\Gamma-\overline{\Gamma})\div(\beta\nabla\theta)\|_{L_t^2\cH^{m-1,0}}
\ep \| F_{\sigma}\|_{L_t^2\cH^{m-1,0}}
    &\lesssim \ep \|\ep^{-1}(\Gamma-\overline{\Gamma}), \mu \na u, \kpa\div (\beta\nabla\theta)\|_{\hco^{m-1}}\lat\lesssim \ep \lat\cE_{m,t}.
\end{align*}

The above two estimates, together with the Cauchy-Schwarz inequality, leads to
\begin{align}
    \cG_5\lesssim %\kpa\|\div \big([(\ep\pt)^j,\beta]\nabla\theta\big)\|_{L_t^2L^2} \|\kpa^{\f{1}{2}}\div w^j\|_{L_t^2L^2}
(T^{\f{1}{2}}+\ep)\lat\cE_{m,t}^2.
\end{align}
Finally, let us see the estimate of $\cG_6,$ which is a bit lengthy since it contains many terms. %Let us first bounded it 
We first use the Cauchy-Schwarz inequality to get:
\begin{align*}
    \cG_6\lesssim T^{\f{1}{2}}(\|(\sigma^j, w^j)\|_{L_t^{\infty}L^2})\|(F_{\sigma}^j,F_{w}^j)\|_{L_t^2L^2}+\|\kpa  \div(\beta\nabla\theta^j) \|_{L_{t}^2L^2} \|\kpa  F_{\theta}^j\|_{L_{t}^2L^2}.
\end{align*}
 It can be verified that:
\beqs
\|\kpa  \div(\beta\nabla\theta^j) \|_{L_{t}^2L^2}\lesssim \|\kappa \Delta\theta\|_{\hco^{m-1}}+T^{\f{1}{2}}\lat\cE_{m,t},
\eeqs
which, combined with the estimates of $\|(F_{\sigma}^j,F_{w}^j, F_{\theta}^j)\|_{L_t^2L^2}$
in Lemma \ref{lem-source-time},
enables us to find:
\beq\label{cg6}
 \cG_6\lesssim (T^{\f{1}{2}}+\ep)\lat\cE_{m,t}^2.%+\|\kpa\nabla^2\theta\|_{\hco^{m-1}}^2.
\eeq

Gluing \eqref{cg0}-\eqref{cg6} into \eqref{sec3:eq1}, one finds that:
\begin{align*}
    &\|(\sigma,w,\kpa \nabla\theta)\|^2_{L_t^{\infty}\cH^{m-1,0}}+(\kpa+\mu)\|\div w\|_{L_t^2\cH^{m-1,0}}^2+\mu \|\nabla u\|_{L_t^2\cH^{m-1,0}}^2\\
    &\lesssim 
Y_m^2(0)+(T^{\f{1}{2}}+\ep)\lat\cE_{m,t}^2+\|\kpa\nabla^2\theta\|_{\hco^{m-1}}^2,
\end{align*}
which, together with the estimate \eqref{EE-2}, gives rise to \eqref{EE-3}.
\end{proof}
\begin{lem}[Estimates on $(F_{\sigma}^j,F_{w}^j, F_{\theta}^j)$]\label{lem-source-time}
Recall that $(F_{\sigma}^j,F_{w}^j, F_{\theta}^j)$ are defined in \eqref{sec3:eq0}, the following estimates hold:
\begin{align}
    &\|(F_{\sigma}^j,F_{w}^j)\|_{L_t^2L^2}\lesssim \lat\cE_{m,t},\\
    & \| F_{\theta}^j\|_{L_t^2L^2}\lesssim %\|\kpa\nabla^2\theta\|_{\hco^{m-1}}+
    (T^{\f{1}{2}}+\ep)\lat \cE_{m,t}.
\end{align}
\end{lem}
\begin{proof}
We will focus on the estimate of $F_{w}^j,$ since the other two can be controlled in a similar way and some of the terms appearing in $(F_{\sigma}^j,  F_{\theta}^j)$ have essentially been dealt with during the analysis of
$\cG_4-\cG_5$ in the above proposition. Let us begin with the term $(\ep\pt)^j F_{w},$ 
where 
\beqs 
F_w=\kpa \overline{\Gamma} \f{\gamma-1}{R\gamma}
\nabla u\cdot\nabla\theta-\kpa \overline{\Gamma}\f{\gamma-1}{R\gamma}(\nabla \ln \beta) (\pt+u\cdot\nabla)\theta+\ep\mu\f{\Gamma-\overline{\Gamma}}{\ep}\div\cL u.\notag
\eeqs
By the product estimate \eqref{roughproduct1}, the estimates \eqref{esofbeta}, \eqref{esofGamma-2} and the assumption $m\geq 7,$ 
\begin{align*}
    \kpa \|\nabla u\cdot\nabla\theta\|_{L_t^2\cH^{m-1,0}}&\lesssim \|\kpa^{\f{1}{2}}(\nabla u,\nabla\theta)\|_{\hco^{m-1}}\lab \il \nabla u\il_{[\f{m}{2}]-1,\infty,t}+\il\kpa^{\f{1}{2}}\nabla\theta\il_{[\f{m}{2}],\infty,t}\big)\\
    &\lesssim \lat \cE_{m,t},\\
   \ep\mu \|\ep^{-1}(\Gamma-\overline{\Gamma}) \div \cL u \|_{L_t^2\cH^{m-1,0}}&\lesssim \|(\sigma, \ep\mu\na^2 u)\|_{\hco^{m-1}}\lab \il \sigma \il_{[\f{m}{2}],\infty,t}+\il\ep\mu\na^2 u\il_{[\f{m}{2}]-1,\infty,t}\big)\\
    &\lesssim \ep\kpa \lat \cE_{m,t},\\
    \kpa \|(\nabla \ln \beta) (\pt+u\cdot\nabla)\theta\|_{L_t^2\cH^{m-1,0}}&\lesssim \|\kpa^{\f{1}{2}}(\nabla u,\pt\theta, u\cdot\nabla\theta)\|_{\hco^{m-1}}
    \lab \il (\pt\theta,\nabla\theta,u)\il_{[\f{m}{2}]-1,\infty,t}+\il(\theta,\kpa^{\f{1}{2}}\nabla\theta)\il_{[\f{m}{2}],\infty,t}\big)\\
    & \lesssim \lat \cE_{m,t}.
\end{align*}
We get from the previous three estimates that, for any $j\leq m-1,$
\beqs
\|(\ep\pt)^j F_{w}\|_{L_t^2L^2}\lesssim \lat \cE_{m,t}.
\eeqs
To finish the estimate of $\|F_{w}^j\|_{L_t^2L^2},$ it remains to control the $L_t^2L^2$ norm of 
$$[(\ep\pt)^j, \f{1}{R\beta}u\cdot\nabla]w,\quad [(\ep\pt)^j, \f{1}{R\beta}]\pt w.$$ 
Let us give the detail for the second one, the former one being easier. %We can write explicitly that
It follows from the formula
\beq\label{timecom-id}
 [(\ep\pt)^j, \f{1}{R\beta}]\pt f=\sum_{l=0}^{j-1} C_j^{l+1} (\ep\pt)^l\pt(\f{1}{R\beta})(\ep\pt)^{j-l}f
\eeq
 that:
%We thus can get from the above identity that:
\begin{align*}
\|[(\ep\pt)^j, \f{1}{R\beta}]\pt w\|_{L_t^2L^2}&\lesssim \|(\theta,\pt\theta,\ep\pt w)\|_{L_t^2\cH^{m-2,0}}\lab \il \pt\theta\il_{[\f{m}{2}]-1,\infty,t}+\il w \il_{[\f{m}{2}]+1,\infty,t}\big)\\
&\lesssim T^{\f{1}{2}}\lat\cE_{m,t}.
\end{align*}
Note that since $m\geq 7,$
\begin{align*}
   \il w \il_{[\f{m}{2}]+1,\infty,t}\lesssim \lab \il (u,\theta, \kpa\na\theta) \il_{m-3,\infty,t}\big)\lesssim \lat.
\end{align*}
The estimate of $\|F_{w}^j\|_{L_t^2L^2}$ is now complete.
\end{proof}

\section{Uniform estimates for $\theta$}
In this section, we aim to get various 
uniform control of energy norms of $\theta$ summarized in $\cE_{m,t}(\theta):$
\begin{align*}
   \cE_{m,t}(\theta)&=\sqrt{\kpa}\|\theta\|_{L_t^{\infty}\underline{H}_{co}^{m}}+ \|\theta\|_{L_t^{\infty}H_{co}^{m-1}}+\|\nabla\theta\|_{\infcok^{m-1}}\\
& \qquad + \sqrt{\kappa}\|\na^2\theta\|_{\infcok^{m-2}\cap L_t^{2}H_{co,\sqrt{\kappa}}^{m-1}}+\kpa \|\na^3\theta\|_{L_t^{2}H_{co,\sqrt{\kappa}}^{m-2}}.
\end{align*}
%$$\|\theta\|_{L_t^{\infty}H_{co,\sqrt{\kappa}}^{m}},  \quad \|\nabla\theta\|_{\infcok^{m-1}}, \quad \sqrt{\kappa}\|\Delta\theta\|_{\infcok^{m-2}},\quad   \sqrt{\kappa}\|\Delta\theta\|_{\hco^{m-2}}, \quad \kpa \|\na\Delta\theta\|_{L_t^{2}H_{co,\sqrt{\kappa}}^{m-2}}$$
\begin{prop}\label{prop-theta}
Suppose that the assumption \eqref{preasption} holds and $m\geq 7,$ we have the following estimate for $\theta:$ there exists a constant 
$\tilde{\vartheta}_0>0,$ %which is independent of 
such that, for any $\ep\in(0,1], (\mu,\kpa)\in A,$
 any $0<t\leq T:$
\begin{align}\label{EE-theta}
   \cE_{m,t}^2(\theta)\lesssim \lab Y_m(0) \big) + %Y_m^2(0)+
   (T+\ep)^{\tilde{\vartheta}_0}
 \Lambda\big(\f{1}{c_0},\cN_{m,t}\big)+\|\kpa^{\f{1}{2}}\nabla\sigma\|_{\hco^{m-1}}^2.
\end{align}
\end{prop}
\begin{proof}
In view of estimate \eqref{EE-highest} %and \eqref{EE-2}, 
it remains for us to control the following quantities:
$$\|\theta\|_{L_t^{\infty}H_{co}^{m-1}},\quad \|\nabla\theta\|_{L_t^{\infty}H_{co}^{m-2}}, 
\quad \sqrt{\kappa}\|\na^2\theta\|_{\infcok^{m-2}\cap \hco^{m-2}}, \quad \kpa \|\na^3\theta\|_{L_t^{2}H_{co,\sqrt{\kappa}}^{m-2}},$$
which  are done in the following three propositions. 
\end{proof}
%In this section, we aim to obtain some auxiliary estimates for $\theta,$ namely,$$\|\theta\|_{L_t^{\infty}H_{co}^{m-1}},\quad \|(\nabla\theta,\kpa\Delta\theta)\|_{L_t^{\infty}H_{co}^{m-2}},\quad \|\kpa^{\f{1}{2}}\Delta\theta\|_{L_t^{\infty}H_{co}^{m-3}}.$$
%\begin{rmk}The estimate \eqref{EE-theta-0} , together with \eqref{EE-highest} and \eqref{EE-2}, leads to the complete control of the energy norms of $\theta,$\end{rmk}
\begin{prop}
Under the same assumption as in Proposition \ref{prop-theta}, we have that for any $\ep \in (0,1], (\mu,\kpa)\in A,$ any $0<t\leq T,$ 
\begin{align}\label{EE-theta-0}
\|\theta\|_{L_t^{\infty}H_{co}^{m-1}}^2+\|\kpa^{\f{1}{2}}\nabla\theta\|_{\hco^{m-1}}^2
\lesssim Y_m^2(0)+(T^{\f{1}{2}}+\ep)\lat\cE_{m,t}^2.
\end{align}
\end{prop}
\begin{proof}
We could have expected to get this inequality by performing direct energy estimates on $\theta.$ However, the presence of the term 
$\div u$ in the equation of $\theta$ prevents us from doing so, since otherwise the estimate cannot be uniform in $\kpa.$ The strategy is to introduce  and work on the new unknown 
\beqs\label{deftta}
\tilde{\theta}=\colon \f{C_v}{R}\theta-\f{\ep}{\gamma}\sigma.
\eeqs
By the equation $\eqref{NCNS-S2}_1, \eqref{NCNS-S2}_3,$ we find that $\tta$ solves the equation:
\begin{equation}\label{sec4:eq0}
   ( \pt+u\cdot\nabla)\tta-%\f{\kpa}{\gamma}
   \kpa\gamma^{-1}\Gamma\div(\beta\nabla\theta)=\f{\ep}{\gamma}\mathfrak{N},
\end{equation}
where $\mathfrak{N}$ is defined in \eqref{defgamma-beta}.
Applying $Z^{I}, (|I|\leq m-1)$ on the above equation and multiplying it by $Z^{I}\tta,$ one obtains the following estimate:
\begin{align}\label{sec4:eq1}
    &\f{1}{2}\iomega |Z^{I}\tta|^2(t)\, \d x-\f{C_v\kpa}{R\gamma}\izto Z^{I}(\Gamma\div(\beta\nabla\theta))Z^{I}{\theta}\,\d x\d s\notag\\
    &= \f{1}{2}\iomega |Z^{I}\tta|^2(0)\, \d x +\f{1}{2}\izto \div u |Z^{I}\tta|^2 \d x \d s
    +\izto \big([Z^I,u\cdot\nabla]\tta+\f{\ep}{\gamma}Z^I \mathfrak{N}\big) Z^I \tta\,\d x\d s\\
    &\qquad\qquad\qquad\qquad\qquad\qquad-\f{\ep\kpa}{\gamma^2}\izto Z^I(\Gamma\div(\beta\nabla\theta))Z^I\sigma\,\d x\d s\notag\\
    &\lesssim Y_m^2(0)+(T+\ep)\lat\cE_{m,t}^2.\notag
\end{align}
Note that in the derivation of the last inequality, we have used the Cauchy-Schwarz inequality and the facts:
\begin{align*}
   & \|Z^I \tta\|_{L_t^{2}L^2}\lesssim T^{\f{1}{2}}\|Z^I \tta\|_{L_t^{\infty}L^2}\lesssim T^{\f{1}{2}}\cE_{m,t},\\
    & \|\big([Z^I,u\cdot\nabla]\tta, Z^I(\Gamma\div(\beta\nabla\theta)), Z^I\sigma , Z^I \mathfrak{N} \big)\|_{L_t^{2}L^2}\lesssim \lat\cE_{m,t},
\end{align*}
which result from the product and commutator estimates in Proposition \ref{prop-prdcom}.
%\begin{align*}\|[Z^I,u\cdot\nabla]\tta Z^I \tta\|_{L^1([0,t]\times\Omega)}&\lesssim \|Z^I \tta\|_{L_t^{2}L^2}\|[Z^I,u\cdot\nabla]\tta\|_{L_t^2L^2}\\&\lesssim T^{\f{1}{2}}\|Z^I \tta\|_{L_t^{\infty}L^2} (\|\nabla\tta\|_{\hco^{m-2}}+\|u\|_{\hco^{m-1}})\lat\\&\lesssim T\lat\cE_{m,t}^2,\end{align*}as well as:
%\begin{align*}   \ep\kpa  \|Z^I(\Gamma\div(\beta\nabla\theta))Z^I\sigma\|_{L^1([0,t]\times\Omega)}&\lesssim \ep \|\kpa \Gamma\div(\beta\nabla\theta) \|_{\hco^{m-1}}\|\sigma\|_{\hco^{m-1}}\lesssim \ep \lat\cE_{m,t}^2.\end{align*}
It suffices to see the second term in the left hand side of \eqref{sec4:eq1}. 
Let us write:
\begin{align*}
   & -\f{C_v\kpa}{R\gamma}
  \izto Z^{I}(\Gamma\div(\beta\nabla\theta))Z^{I}{\theta}\,\d x\d s\\
   &= \f{C_v\kpa}{R\gamma} \izto \Gamma \beta |Z^I\nabla\theta|^2\,\d x\d s-
    \f{C_v\kpa}{R\gamma}\izt\int_{\p\Omega}  Z^I(\Gamma \beta\nabla\theta)\cdot\bn Z^I\theta \,\d S_y\d s \\
   &\quad +\f{C_v\kpa}{R\gamma} \izto \big([Z^I, \Gamma \beta ]\nabla\theta \cdot\nabla Z^I\theta+ \Gamma \beta Z^I(\nabla\theta)\cdot [\nabla, Z^I]\theta\big)-\big([Z^I, \div](\Gamma \beta \nabla\theta )-Z^I(\beta\nabla \Gamma\cdot\nabla\theta) \big)Z^I\theta\,\d x\d s.
 % & \quad +\f{C_v\kpa}{R\gamma} \izto \big([Z^I, \div](\Gamma \beta \nabla\theta )-Z^I(\beta\nabla \Gamma\cdot\nabla\theta) \big)Z^I\theta\,\d x\d t
\end{align*}
One can verify that the last line in the above identity can be bounded as:
$$T^{\f{1}{2}}\lat\cE_{m,t}^2.$$
Such an estimate is achieved by the direct (although a bit lengthy) application of Proposition \ref{prop-prdcom} and Corollary \ref{cor-gb}, we thus omit the detail.
Let us now control the boundary term appearing in the above identity by taking the benefits of the Neumann boundary condition for $\theta$ and the trace inequality \eqref{traceL2}:
\begin{align*}
& \f{C_v\kpa}{R\gamma}\bigg|\izt\int_{\p\Omega}  Z^I(\Gamma \beta\nabla\theta)\cdot\bn Z^I\theta \,\d S_y\d s \bigg|& \\
 &= \f{C_v\kpa}{R\gamma}\bigg|\izt\int_{\p\Omega}  [Z^I,\bn](\Gamma \beta\nabla\theta) Z^I\theta \,\d S_y\d s\bigg|\lesssim |\kpa^{\f{3}{4}}\Gamma\beta\nabla\theta|_{\hcob^{m-2}}|\kpa^{\f{1}{4}}\theta|_{\hcob^{m-1}}\\
 &\lesssim \|\big(\kpa \nabla^2\theta, \kpa^{\f{1}{2}}\nabla\theta, (\Id,\na)(\theta,\ep\sigma)\big)\|_{\hco^{m-2}}\lat \|(\theta,\kpa^{\f{1}{2}}\nabla\theta)\|_{\hco^{m-1}}\lesssim T\lat\cE_{m,t}^2.
\end{align*}
To summarize, we have found that:
\begin{align*}
    -\f{C_v\kpa}{R\gamma}
  \izto Z^{I}(\Gamma\div(\beta\nabla\theta))Z^{I}{\theta}\,\d x\d t\geq 
\f{C_v\kpa}{R\gamma} \izto \Gamma \beta |Z^I\nabla\theta|^2\,\d x\d t-T^{\f{1}{2}}\lat\cE_{m,t}^2.
 \end{align*}
Substituting this estimate into \eqref{sec4:eq1}, we obtain \eqref{EE-theta-0} 
by using \eqref{Gamma-Beta-infty}. %the a-priori assumption $\Gamma,\beta\geq c_0$ for any $(t,x)\in [0,T]\times\Omega.$
\end{proof}

\begin{prop}
Under the same assumption as in Proposition \ref{prop-theta}, for any $\ep \in (0,1], (\mu,\kpa)\in A,$ any $0<t\leq T,$ it holds that:
\begin{align}\label{EE-theta-1}
 \|\nabla\theta\|_{L_t^{\infty}H_{co}^{m-2}}^2+\kpa\|\nabla^2\theta\|_{\hco^{m-2}}^2 \lesssim Y_m^2(0)+(T^{\f{1}{2}}+\ep)\lae.
\end{align}
\end{prop}
\begin{proof}
Taking a vector field $Z^J, |J|\leq m-2,$ applying $Z^J\nabla$ on the equation 
\eqref{sec4:eq0}, we get that:
\begin{align}\label{sec4:eq2}
    (\pt+u\cdot\nabla)Z^J\nabla\tta-\kpa\gamma^{-1}\nabla Z^J(\Gamma\div(\beta\nabla\theta))=%\text{R.H.S of \eqref{sec4:eq2}}
   \eqref{sec4:eq2}_1+\eqref{sec4:eq2}_2+\eqref{sec4:eq2}_3 ,
\end{align}
where
$$ \eqref{sec4:eq2}_1=\kpa\gamma^{-1}[Z^J,\nabla](\Gamma\div(\beta\nabla\theta)),\quad 
  \eqref{sec4:eq2}_2=Z^J\big(\f{\ep}{\gamma}\mathfrak{N} -\nabla u\cdot\nabla\tta\big),\quad \eqref{sec4:eq2}_3=-[Z^J,u\cdot\nabla]\nabla\tta.$$
  For the first term, we get in light of the identity \eqref{comu} that:
\begin{align*}
     \|\eqref{sec4:eq2}_3\|_{L_t^2L^2}\lesssim \kpa \|\nabla(\Gamma\div(\beta\nabla\theta))\|_{\hco^{m-3}}.
\end{align*}
By using the equation \eqref{sec4:eq0}, the product estimate \eqref{roughproduct1} and the estimates \eqref{esofGamma-2}, \eqref{esofbeta} for $\Gamma,\beta$.  %in Corollary \ref{cor-gb}, 
we control the right hand further as follows:
\begin{align*}
& \kpa \|\nabla(\Gamma\div(\beta\nabla\theta))\|_{\hco^{m-3}}\\
 &\leq \kpa \overline{\Gamma}\|\nabla\div(\beta\nabla\theta)\|_{\hco^{m-3}}+\kpa\big\|\nabla\big(\f{\Gamma-\overline{\Gamma}}{\ep}\Gamma^{-1}(\ep\pt+\ep u\cdot\nabla)\tta\big)\big\|_{\hco^{m-3}}\\
  &\lesssim \kpa \|\beta \nabla\Delta\theta\|_{\hco^{m-3}}+\lat\cE_{m,t}\\
  &\lesssim \lab\il\theta\il_{m-3,\infty,t}\big)\|\kpa\nabla\Delta\theta\|_{\hco^{m-3}}+\lat\cE_{m,t}\lesssim \lae. %\lat\cE_{m,t}.
\end{align*}
Next, thanks to the product estimate \eqref{roughproduct2},
%by counting derivatives hitting on each term, one can see that:
%Moreover,  We can use the commutator estimate \eqref{roughcom}, the  product estimate \eqref{roughproduct1} respectively to control the $L_t^2L^2$ norms of the last two terms:
  \begin{align*}
         \|\eqref{sec4:eq2}_2\|_{L_t^2L^2}&\lesssim \|\sigma, \nabla(\sigma, u,\tta, \ep\mu\na u)\|_{\hco^{m-2}}\il\sigma, \nabla( u, \tta, \ep\mu\na u)\il_{[\f{m}{2}]-1,\infty,t}\lesssim \lat\cE_{m,t}.
  \end{align*}
  Moreover, it results from the commutator estimate \eqref{roughcom} and the identity \eqref{id-convetion} that:
  \begin{align*}
        \|\eqref{sec4:eq2}_3\|_{L_t^2L^2}&\lesssim \|(u, \nabla u,\nabla\tta)\|_{\infco^{m-2}}\big(\il\nabla\tta\il_{[\f{m}{2}]-1,\infty,t}+\il u \il_{[\f{m+1}{2}],\infty,t}+\|\div u\|_{L_t^2W_{co}^{[\f{m-1}{2}],\infty}}\big)\\
        &\lesssim \lae.
  \end{align*}
 Note that by the Sobolev embedding \eqref{sobebd} and the assumption $m\geq 7,$
 \begin{align*}
     \|\div u\|_{L_t^2W_{co}^{[\f{m-1}{2}],\infty}}
     \lesssim \|\na\div u\|_{\hco^{m-3}}+\|\div u\|_{\hco^{m-2}}\lesssim \lae.
 \end{align*}
To summarize, we have obtained so far 
\beq\label{sec4:eq2.5}
\|(\eqref{sec4:eq2}_1, \eqref{sec4:eq2}_2, \eqref{sec4:eq2}_3)\|_{L_t^2L^2}\lesssim \lae.
\eeq
Now, we take the inner product of the equation \eqref{sec4:eq2} by $Z^J\na\tta$
and integrate in space and time to obtain that:
\begin{align}\label{sec4:eq3}
&\f{1}{2}\int_{\Omega} |Z^{J}\nabla\tta|^2(t)\,\d x+\kpa\gamma^{-1}\izto %Z^J(\Gamma\div(\beta\nabla\theta))\div Z^J \nabla\theta
\Gamma\beta|Z^J\Delta\theta|^2
\, \d x\d s \notag\\
&=\f{1}{2}\int_{\Omega} |Z^{J}\nabla\tta|^2(0)\,\d x+\underbrace{\izto\f{1}{2}\div u|Z^J\nabla\tta|^2+ 
Z^J\na\tta\cdot \big(\eqref{sec4:eq2}_1+\eqref{sec4:eq2}_2+\eqref{sec4:eq2}_3\big) \d x\d s}_{\eqref{sec4:eq3}_1}\notag\\
&\quad\underbrace{-\kpa\gamma^{-1}\izto Z^J(\Gamma\div(\beta\nabla\theta))[\div, Z^J] \nabla\theta+\big(Z^J(\Gamma\na\beta\cdot\na\theta)+[Z^J,\Gamma\beta]\Delta\theta\big)Z^{J}\Delta\theta\,\d x\d s}_{\eqref{sec4:eq3}_2}\\
&\quad\underbrace{-\ep\kpa\gamma^{-2} \izto \nabla Z^J(\Gamma\div(\beta\nabla\theta))\cdot Z^J \nabla\sigma \,\d x\d s}_{\eqref{sec4:eq3}_3}.\notag
\end{align}
Thanks to \eqref{sec4:eq2.5}, the term $\eqref{sec4:eq3}_1$ can be controlled by using the Cauchy-Schwarz inequality: 
\beq\label{sec4:eq10}
\eqref{sec4:eq3}_1\lesssim T^{\f{1}{2}}\|Z^J\na\tta\|_{L_t^{\infty}L^2}
\lae \lesssim T^{\f{1}{2}}\lae.
\eeq
Next, one can show by using repeatedly the identity \eqref{comu}, the Proposition \ref{prop-prdcom} and Corollary \ref{cor-gb} that for any $|J|\leq m-2,$
\begin{align*}
 &  \| \kpa^{\f{1}{2}}[\div, Z^J] \nabla\theta \|_{L_t^{\infty}L^2}\lesssim  \|\kpa^{\f{1}{2}}\nabla^2\theta\|_{\infco^{m-3}},              \\
 & \|\kpa^{\f{1}{2}}[Z^J,\Gamma\beta]\nabla\theta\|_{L_t^{\infty}L^2}\lesssim (\|\kpa^{\f{1}{2}}\Delta\theta\|_{\infco^{m-3}}+\|(\ep\sigma,\theta)\|_{\infco^{m-2}})\lat\lesssim \lat\cE_{m,t},\\
   & \|\kpa^{\f{1}{2}} \Gamma\div(\beta\nabla\theta)\|_{\hco^{m-2}}+\|\Gamma \na\beta\na\theta\|_{\hco^{m-2}}\lesssim \lat\cE_{m,t}. 
\end{align*}
The term $\eqref{sec4:eq3}_2$ can thus be controlled as:
\begin{align}\label{sec4:eq11}
 |\eqref{sec4:eq3}_2|&\lesssim T^{\f{1}{2}} \big( \kpa^{\f{1}{2}}\|\big([\div, Z^J] \nabla\theta , [Z^J,\Gamma\beta]\nabla\theta\big)\|_{L_t^{\infty}L^2}+\|\kpa \Delta\theta\|_{\infco^{m-2}}\big)\lat\cE_{m,t}\notag\\
 &\lesssim T^{\f{1}{2}} \lae.
\end{align}

It now remains to handle the term $\eqref{sec4:eq3}_3,$ which cannot be done directly by the Cauchy-Schwarz inequality, since we are not able to control 
$\ep\kpa \|\nabla\Delta\theta\|_{\hco^{m-2}}$ uniformly in $\kpa.$ Therefore, integration by parts is needed. Before doing so, 
we introduce $\tilde{\sigma}=\sigma-\ep\mu(2\lambda_1+\lambda_2)\Gamma\div u$ and write:
\begin{align*}
    \eqref{sec4:eq3}_3=\ep\kpa\gamma^{-2}\izto \nabla Z^J(\Gamma\div(\beta\nabla\theta))\cdot Z^J \nabla(\tilde{\sigma}+\ep\mu (2\lambda_1+\lambda_2)\Gamma\div u) \,\d x\d s.
  %  +\ep\kpa\mu\gamma^{-2}(2\lambda_1+\lambda_2)%\izto \nabla Z^J(\Gamma\div(\beta\nabla\theta))\cdot Z^J \nabla\div u \,\d x\d s
\end{align*}
By Remark \ref{rmkmuapproxkpa}, for any $(\mu,\kpa)\in A,$ it holds that  $\mu\sim \kappa,$ we thus can have that
\begin{align}\label{sec4:eq8}
&\ep^2\kpa\mu\gamma^{-2}(2\lambda_1+\lambda_2)\izto \nabla Z^J(\Gamma\div(\beta\nabla\theta))\cdot Z^J \nabla(\Gamma\div u) \,\d x\d s\notag\\
 &\lesssim \ep^2 \|\kpa^{\f{3}{2}}\nabla(\Gamma\div(\beta\nabla\theta)) \|_{\hco^{m-2}} \|\kappa^{\f{1}{2}}\na\div u\|_{\hco^{m-2}}\lesssim \ep^2 \lat \cE_{m,t}^2.
\end{align}
Moreover, integration by parts in space gives rise to:
\begin{align}\label{sec4:eq4}
  & \ep\kpa\gamma^{-2}\izto \nabla Z^J(\Gamma\div(\beta\nabla\theta))\cdot Z^J \nabla\tilde{\sigma} \,\d x\d s\notag\\
   &=-\ep\kpa\gamma^{-2}\izto Z^J\big(\Gamma\div(\beta\nabla\theta)\big) \big(Z^J\Delta\tilde{\sigma}+[\div, Z^J]\na \tilde{\sigma}\big) \,\d x\d s\\
   &\quad +\ep\kpa\gamma^{-2}\izt\int_{\p\Omega} Z^J\big(\Gamma\div(\beta\nabla\theta)\big) Z^J\na\tilde{\sigma}\cdot\bn\,\d S_y\d s\notag \\
  & =\colon \eqref{sec4:eq4}_1+ \eqref{sec4:eq4}_2.\notag
\end{align}
One can bound the term $\eqref{sec4:eq4}_1$ by the Cauchy-Schwarz inequality:
\begin{align}\label{sec4:eq5}
\eqref{sec4:eq4}_1&\lesssim \ep \|\kpa^{\f{1}{2}}\Gamma\div(\beta\nabla\theta)\|_{\hco^{m-2}}\big(\|\kpa^{\f{1}{2}}\Delta\tilde{\sigma}\|_{\hco^{m-2}}+\|\kpa^{\f{1}{2}}\nabla^2\tilde{\sigma}\|_{\hco^{m-3}}\big),
\end{align}
which, combined with the estimate \eqref{Es-tsigma1} for $\tsigma$ proved in the next lemma, yields:
%The estimate \eqref{} for $\tsigma$ proved in the next lemma then enable us to obtain that:
\begin{align}\label{sec4:eq7}
\eqref{sec4:eq4}_1&\lesssim\ep \lat\cE_{m,t}^2.
\end{align}
For the boundary term $\eqref{sec4:eq4}_2,$ we control it as:
\beq\label{sec4:eq7.5}
\eqref{sec4:eq4}_2\lesssim \ep |\kpa\Gamma\div(\beta\nabla\theta)|_{\hcob^{m-2}(\p\Omega)} |Z^J\na \tilde{\sigma}\cdot\bn|_{L_t^2L^2(\p\Omega)}.
\eeq
Thanks to the trace inequality \eqref{traceL2} and the estimates \eqref{Es-tsigma2}, \eqref{Es-tsigma3}, we can have that:
\begin{align}\label{sec4:eq6.5}
    |Z^J\na \tilde{\sigma}\cdot\bn|_{L_t^2L^2(\p\Omega)}&\lesssim |\p_{\bn}\tilde{\sigma}|_{\hcob^{m-2}}+|\nabla\tilde{\sigma}|_{\hcob^{m-3}}\notag\\
    %&\lesssim |\big((u\cdot\nabla)\bn\cdot u, \ep\mu \p_{\bn}\Gamma\div u,\ep\mu \Gamma \curl\curl u\cdot\bn\big)|_{\hcob^{m-2}}+
    &\lesssim  |\p_{\bn}\tilde{\sigma}|_{\hcob^{m-2}}+ \|(\nabla \tilde{\sigma},\nabla^2\tilde{\sigma})\|_{\hco^{m-3}}\\
    &\lesssim T^{\f{1}{2}} \lat\cE_{m,t}.\notag
\end{align}
Moreover, by the product estimate \eqref{product-bd} on the boundary,
\begin{align*}
   \kpa |\Gamma\div(\beta\nabla\theta)|_{\hcob^{m-2}}&\lesssim 
  \big|\big(\sigma, \theta,\kappa^{\f{1}{2}}\nabla\theta,\kpa \Delta\theta\big)\big|_{\hcob^{m-2}}\lat.
\end{align*}
It results from the trace inequality \eqref{traceL2} that:
\beqs
\big|\big(\sigma, \theta,\kappa^{\f{1}{2}}\nabla\theta,\kpa \Delta\theta\big)\big|_{\hcob^{m-2}}\lesssim \cE_{m,t}.
\eeqs
For instance,
\begin{align*}
   \big| \kpa \Delta\theta\big|_{\hcob^{m-2}}\lesssim \|(\kpa^{\f{3}{2}}\nabla\Delta\theta, \kpa^{\f{1}{2}}\Delta\theta)\|_{\hco^{m-2}}\lesssim \cE_{m,t}.
\end{align*}
We thus obtain that:
\beqs
 \kpa |\Gamma\div(\beta\nabla\theta)|_{\hcob^{m-2}} \lesssim\lat\cE_{m,t},
\eeqs
which, together with  \eqref{sec4:eq7.5}, \eqref{sec4:eq6.5}, leads to that 
\beqs
\eqref{sec4:eq4}_2\lesssim \ep \lat\cE_{m,t}^2.
\eeqs
Combined with \eqref{sec4:eq7}, one derives:
\beqs
 \ep\kpa\gamma^{-2}\izto \nabla Z^J(\Gamma\div(\beta\nabla\theta))\cdot Z^J \nabla\tilde{\sigma} \,\d x\d s\lesssim \ep \lat\cE_{m,t}^2.
\eeqs
This inequality and \eqref{sec4:eq8} then enable us to conclude:
\beq\label{sec4:eq9}
\eqref{sec4:eq3}_3\lesssim \ep \lat\cE_{m,t}^2.
\eeq
Inserting \eqref{sec4:eq10}, \eqref{sec4:eq11}, \eqref{sec4:eq9} into \eqref{sec4:eq3}, we eventually arrive at \eqref{EE-theta-1}.
\end{proof}
\begin{rmk}\label{rmkptna2theta}
Following the similar arguments as in the proof of \eqref{EE-theta-1}, we can show that:
\begin{align*}
    \|\ep^{\f{1}{2}}\pt\na\theta\|_{\infco^{1}}^2+\|(\ep\kpa)^{\f{1}{2}}\pt\na^2\theta\|_{\infco^{1}}^2\lesssim  \|\ep^{\f{1}{2}}\pt\na\theta(0)\|_{\infco^{1}}^2+ \lat\cE_{m,t}^2.
\end{align*}
Such an estimate  plays no role in the uniform estimates but will be  useful in the proof of convergence result stated in Theorem \ref{thm-conv1}.
\end{rmk}

\begin{lem}\label{lemtsigma}
Let $\tsigma=\sigma-\ep\mu(2\lambda_1+\lambda_2)\Gamma\div u.$ Under the assumption \eqref{preasption},  the following estimates hold:
\begin{align}
& \kpa^{\f{1}{2}}(\|\na^2\tsigma\|_{\hco^{m-2}}+\|\na\tsigma\|_{\hco^{m-1}})\lesssim   \lat\cE_{m,t},\label{Es-tsigma1}\\
 &  \|\na^2\tsigma\|_{\infco^{m-3}}+\|\na\tsigma\|_{\infco^{m-2}} \lesssim   \lat\cE_{m,t},     \label{Es-tsigma2}\\
 &\kpa^{\f{1}{2}} |\p_{\bn}\tsigma|_{L_t^2\tilde{H}^{m-\f{3}{2}}}+\ep^{-\f{1}{2}}|\p_{\bn}\tsigma|_{L_t^{\infty}\tilde{H}^{m-2}}\lesssim   \lat\cE_{m,t}, \label{Es-tsigma3}\\
&\il(\na\tsigma, (\ep\mu)^{\f{1}{2}}\na^2\tsigma )\il_{m-5,\infty,t}+ \il \na^2\tsigma\il_{m-6,\infty,t}\lesssim \lat.  \label{Es-tsigma4}
\end{align}
\end{lem}
\begin{proof}
%By using the equation of the velocity 
Taking the divergence of the equations $\eqref{NCNS-S2}_2$ for the velocity,  we find that $\tilde{\sigma}$ solves the elliptic problem 
\beq\label{eq-tiltasigma}
\left\{
\begin{array}{l}
      \Delta\tilde{\sigma}=\div f  \text{ in } \Omega, \\[2pt]
      \p_{\bn} \tilde{\sigma}=f\cdot\bn, \text{ on } \p\Omega, %-\f{1}{R\beta}(\ep\pt+\ep u\cdot\na)u\cdot\bn+\ep\mu\lambda_1\curl\curl u\cdot\bn 
\end{array}
   \right.
\eeq 
where 
\beq\label{def-f}
  f=-\f{1}{R\beta}(\ep\pt+\ep u\cdot\na)u-\ep\mu\lambda_1\Gamma\curl\curl u-\ep\mu(2\lambda_1+\lambda_2)(\div u)\na\Gamma.
\eeq
Applying the elliptic estimates \eqref{highconormal}, \eqref{secderelliptic}, we obtain that 
\begin{align*}
 \kpa^{\f{1}{2}}(\|\na^2\tilde{\sigma}\|_{\hco^{m-2}}+\|\na\tilde{\sigma}\|_{\hco^{m-1}})\lesssim   \kpa^{\f{1}{2}}(\|f\|_{\hco^{m-1}}+\|\div f\|_{\hco^{m-2}}),\\
  \|\na^2\tsigma\|_{\infco^{m-3}}+\|\na\tsigma\|_{\infco^{m-2}}\lesssim\|f\|_{\infco^{m-2}}+\|\div f\|_{\infco^{m-3}}.
\end{align*}
The proof of \eqref{Es-tsigma1}, \eqref{Es-tsigma2} is finished once we have controlled the right hand sides of the above two inequalities by $\lat\cE_{m,t},$ which are shown in the next lemma. 

 We now prove the estimate \eqref{Es-tsigma3}. The first quantity can be bounded by the trace inequality \eqref{traceL2}:
 \beqs
 \kpa^{\f{1}{2}}|\p_{\bn}\tsigma|_{L_t^2\tilde{H}^{m-\f{3}{2}}}\lesssim  \kpa^{\f{1}{2}} \|(\na^2\tsigma, \na\tsigma)\|_{\hco^{m-2}}\lesssim \lat\cE_{m,t}.
 \eeqs
 The second quantity can be controlled by using the explicit expression of $\p_{\bn}\tsigma$ in $\eqref{eq-tiltasigma}_2.$ Thanks to the boundary condition $u\cdot\bn|_{\p\Omega}=0$ as well as the fact
 $u\cdot \na|_{\p\Omega}=(u_1 \p_{y^1}+u_2\p_{y^2})|_{\p\Omega},$ one derives that:
\begin{align}\label{sec4:eq6}
    |\p_{\bn}\tilde{\sigma}|_{L_t^{\infty}\tilde{H}^{m-2}}&=|f\cdot\bn|_{L_t^{\infty}\tilde{H}^{m-2}}\notag\\
    &\lesssim \ep\big|\big((u\cdot\nabla)\bn\cdot u, \mu \p_{\bn}\Gamma\div u,\mu \Gamma \curl\curl u\cdot\bn\big)\big|_{L_t^{\infty}\tilde{H}^{m-2}}.
\end{align}
It follows from %successively 
the product estimate \eqref{product-bd}, the trace inequalities \eqref{traceL2},
\eqref{traceLinfty} that: 
\begin{align*}
    |(u\cdot\nabla)\bn\cdot u|_{L_t^{\infty}\tilde{H}^{m-2}}\lesssim |u|_{L_t^{\infty}\tilde{H}^{m-2}}\lab \il u\il_{[\f{m}{2}]-1,\infty,t}\big) \lesssim \|(u,\na u)\|_{\infco^{m-2}}\lat,
\end{align*}
\begin{align*}
    |\mu \p_{\bn}\Gamma\div u|_{L_t^{\infty}\tilde{H}^{m-2}}&\lesssim |(\ep^{-1}\mu)^{\f{1}{2}}\p_{\bn}\Gamma, (\ep\mu)^{\f{1}{2}} \div u|_{L_t^{\infty}\tilde{H}^{m-2}}\lab \il (\sigma,\na\sigma,\div u)\il_{[\f{m}{2}]-1,\infty,t}\big)\\
    &\lesssim \|(\Id, \ep\mu \na) (\na\sigma, \div u)\|_{\infco^{m-2}}\lat,
\end{align*}
from which we see that:
\beqs 
|\big((u\cdot\nabla)\bn\cdot u, \mu \p_{\bn}\Gamma\div u\big)|_{L_t^{\infty}\tilde{H}^{m-2}}\lesssim \lat\cE_{m,t}.
\eeqs
This, combined with the estimate \eqref{sec4:eq12} proved later in Lemma \ref{lemcurlcurl}, yields that:
\beqs 
|\p_{\bn}\tilde{\sigma}|_{L_t^{\infty}\tilde{H}^{m-2}}\lesssim \ep^{\f{1}{2}} \lat\cE_{m,t}.
\eeqs
We thus finish the proof of \eqref{Es-tsigma3}.

Finally, let us show the $L_{t,x}^{\infty}$ estimate \eqref{Es-tsigma4}. First, by using the definition of 
$\tsigma,$ we get  that:
\begin{align}\label{tsigmaLinfty}
\il\na\tsigma\il_{m-5,\infty,t}&\lesssim \il\na\sigma\il_{m-5,\infty,t}+\lab \il (\div u,\ep\mu \na\div u, \ep\sigma, \ep\na\sigma)  \il_{m-5,\infty,t} \big)\notag\\
&\lesssim \lat.
\end{align}
 Moreover, in the local coordinate, there are some coefficients $a_{ij}$ that depends smoothly  on $\bn,$ such that (we use the convention $\p_{y^3}=\p_{\bn}$):
 \beq\label{Laplace-local}
 \Delta=\p_{\bn}^2+\sum_{0\leq i,j\leq 3, 
(i,j)\neq (3,3).
} \p_{y^i}(a_{ij}\p_{y^j}\cdot).
\eeq
Consequently, 
 \beq\label{tsigmaLinfty-sec}
 \il\nabla^2 \tsigma\il_{k,\infty,t}\lesssim \il \Delta\tsigma\il_{k,\infty,t}+\il\nabla\tsigma\il_{k+1,\infty,t}, \,\forall k\geq 0.
 \eeq
However, in view of the equation $\eqref{eq-tiltasigma}$ and 
the definition \eqref{def-f}, we find that
\begin{align*}
&\il\Delta\tsigma\il_{m-6,\infty,t}+(\ep\mu)^{\f{1}{2}} \il\Delta\tsigma\il_{m-5,\infty,t}=\il\div f\il_{m-6,\infty,t}+(\ep\mu)^{\f{1}{2}}\il\div f\il_{m-5,\infty,t} \\
&\lesssim \lab \il\na(\sigma, u, \theta), \ep\mu\na^2(\sigma, u)\il_{m-5,\infty,t}+\il(\ep\mu)^{\f{1}{2}}\na u \il_{m-4,\infty,t}\big) \lesssim \lat,
\end{align*}
which, together with \eqref{tsigmaLinfty}, \eqref{tsigmaLinfty-sec}, yields \eqref{Es-tsigma4}.
\end{proof}
In the following two lemmas, we give some complementary estimates on $f$ and $(\curl\curl u\cdot\bn)|_{\p\Omega},$ which are used in the proof of Lemma \ref{lemtsigma}.
\begin{lem}\label{lem-f}
Suppose that \eqref{preasption} is satisfied and let $f$ be defined in  \eqref{def-f}. For any $\ep\in(0,1], (\mu,\kpa)\in A,$ any $0<t\leq T,$ it holds that:
\beqs
\kpa^{\f{1}{2}}(\|\div f\|_{\hco^{m-2}}+\|f\|_{\hco^{m-1}})+\|\div f\|_{\infco^{m-3}}+\|f\|_{\infco^{m-2}}\lesssim \lat\cE_{m,t}.
\eeqs
\end{lem}
\begin{proof}
%As $f$ contains a few terms, 
These estimates can be derived by using Proposition \ref{prop-prdcom}, Corollary \ref{cor-gb} as well as the definitions of $\cE_{m,t},\cA_{m,t}.$ 
To prove this property,
it requires the estimates for four norms of four terms in the definition of $f.$ 
We will only detail two  of these sixteen estimates, the others can be treated similarly. Let us focus on the control of $\kpa^{\f{1}{2}} \|\div\big(\f{1}{R\beta}\ep\pt u\big)\|_{\hco^{m-2}}, $ $\kpa^{\f{1}{2}}\|\div (\ep\mu\Gamma\curl\curl u)\|_{\hco^{m-2}}.$ It follows from the  product estimate \eqref{roughproduct1} and  the estimate \eqref{esofbeta} for $\beta$ that:
\begin{align*}
 & \kpa^{\f{1}{2}}\|\div\big(\f{1}{R\beta}\big)\|_{\hco^{m-2}}\lesssim \kpa^{\f{1}{2}}\|(\f{1}{R\beta}\ep\pt\div u, \ep\pt u\cdot\na (\f{1}{R\beta}\ep\pt u)) \|_{\hco^{m-2}}\\
&\lesssim \big(\kpa^{\f{1}{2}}\|(u, \div u)\|_{\hco^{m-1}}+\|(\theta,\na\theta)\|_{\hco^{m-2}}\big)\Lambda\big(\f{1}{c_0}, \il\na(u,\theta)\il_{[\f{m}{2}]-1,\infty,t}+\il(u,\theta)\il_{[\f{m+1}{2}],\infty,t}\big)\\
&\lesssim \lat\cE_{m,t}.
\end{align*}
Moreover, using %\eqref{roughproduct1} and
\eqref{esofGamma-2}, one can find that:
\begin{align*}
   &\kpa^{\f{1}{2}}\|\div (\ep\mu\Gamma\curl\curl u)\|_{\hco^{m-2}}=
   \kpa^{\f{1}{2}}\|\na\Gamma\cdot \ep\mu\curl\curl u\|_{\hco^{m-2}}\\
   &\lesssim  \kpa^{\f{1}{2}}\ep \|(\na \sigma, \ep\mu \curl\curl u)\|_{\hco^{m-2}}\Lambda\big(\f{1}{c_0}, \il(\na\sigma ,\ep\mu \curl\curl u)\il_{[\f{m}{2}]-1,\infty,t}\big)\\
   &\lesssim \kpa^{\f{1}{2}}\ep \lat\cE_{m,t}.
\end{align*}
\end{proof}
\begin{lem}\label{lemcurlcurl}
The following estimates hold:
\begin{align}
& |\curl\curl u\cdot\bn|_{[\f{m}{2}]-1,\infty,t}\lesssim \lat,\label{curlcurl-infty0} \\
&|\mu^{\f{1}{2}} \curl\curl u\cdot\bn|_{\hcob^{m-\f{3}{2}}}+|(\ep\mu)^{\f{1}{2}} \curl\curl u\cdot\bn|_{L_t^{\infty}\tilde{H}^{m-2}}\lesssim \cE_{m,t}, \label{sec4:eq12}\\
&\ep \sum_{j=0}^{m-2}|(\ep\pt)^j \curl\curl u\cdot\bn|_{L_t^{\infty}{H}^{m-j-\f{5}{2}}}\lesssim \ep \|\na\div u\|_{\infco^{m-2}}+\ep \cE_{m,t}. \label{curlcurlu-infty}
\end{align}
\end{lem}
\begin{proof}
Denote $\omega=\curl u.$
We begin with the identity
\beq\label{curlcurludotn}
-\curl\curl u \cdot \bn=\div(\omega\times \bn)+\omega\cdot \curl \bn.
\eeq
Near the boundary, it follows from \eqref{normalofnormalder} that:
\begin{equation}\label{omegatimesn}
\begin{aligned}
\div (\omega\times \bn)&=\partial_{\bn}(\omega\times \bn)\cdot \bn+(\Pi\partial_{y^1}(\omega\times \bn))^1+(\Pi\partial_{y^2}(\omega\times \bn))^2\\
&=-(\omega\times \bn)\cdot\partial_{\bn} \bn+(\Pi\partial_{y^1}(\omega\times \bn))^1+(\Pi\partial_{y^2}(\omega\times \bn))^2.
\end{aligned}   
\end{equation}
Therefore,  by noticing the  boundary conditions \eqref{bd-curlun}, \eqref{bdryconditionofu} and the identity
\eqref{normalofnormalder}, one finds after the use of the trace inequalities \eqref{normaltraceineq}, \eqref{traceLinfty} that:
\begin{align}
|\curl\curl u\cdot\bn|_{[\f{m}{2}]-1,\infty,t}&\lesssim |(\div u, \p_y u, u)|_{[\f{m}{2}]-1,\infty,t}\lesssim \lat,\notag\\
   \mu^{\f{1}{2}} | \curl\curl u\cdot\bn|_{\hcob^{m-\f{3}{2}}}&\lesssim \mu^{\f{1}{2}}|(\div u,\p_y u, u)|_{\hcob^{m-\f{3}{2}}}\notag\\
    &\lesssim  \mu^{\f{1}{2}} \big(\|\nabla\div u\|_{\hco^{m-2}}+\|\nabla u\|_{\hco^{m-1}}\big)\lesssim \cE_{m,t},\notag
\end{align}
\begin{align}
   &(\ep\mu)^{\f{1}{2}} | \curl\curl u\cdot\bn|_{L_t^{\infty}\tilde{H}^{m-2}}\lesssim (\ep\mu)^{\f{1}{2}} | (\div u,\p_y u, u)|_{L_t^{\infty}\tilde{H}^{m-2}}\notag\\
    &\qquad\qquad\qquad\qquad\qquad\lesssim  \|(\ep \mu\nabla\div u, \div u)\|_{\infco^{m-2}}+\|(\ep\mu \nabla u, u)\|_{\infco^{m-1}}\lesssim \cE_{m,t}. \label{sec4:eq12.5}
   % & \ep\mu \sum_{j=0}^{m-2}|(\ep\pt)^j \curl\curl u\cdot\bn|_{L_t^{\infty}{H}^{m-j-\f{5}{2}}}\\&\lesssim \ep \mu | (\div u,\p_y u,u)|_{L_t^{\infty}\tilde{H}^{m-\f{5}{2}}}+\ep\mu \sum_{j=0}^{m-2}|(\ep\pt)^j\div u|_{L_t^{\infty}L^2}\notag\\   &\lesssim \ep \mu \big(\|\nabla\div u\|_{\infco^{m-2}}+\|(u,\nabla u)\|_{\infco^{m-2}}\big)\lesssim  \ep \mu \|\nabla\div u\|_{\infco^{m-2}}+\ep\cE_{m,t}.\notag
\end{align}
We thus finish the proof of \eqref{sec4:eq12}. Moreover, using the fact
$$\ep  | u|_{L_t^{\infty}\tilde{H}^{m-\f{3}{2}}}\lesssim \ep \|(u,\na u)\|_{\infco^{m-2}}\lesssim \ep \cE_{m,t},$$ 
one can show \eqref{curlcurlu-infty} by adapting the proof made in \eqref{sec4:eq12.5}.
%The estimate \eqref{curlcurlu-infty} can be shown by adapting the 
\end{proof}

\begin{prop}
Under the same assumption as in Proposition \ref{prop-theta}, we can find some $\tilde{\vartheta}_0>0$ %which is independent of 
such that, for any $\ep\in(0,1], (\mu,\kpa)\in A,$
 any $0<t\leq T:$ 
\begin{align}%\label{EE-theta-2}
 &\|\kpa^{\f{1}{2}}\Delta\theta\|_{L_t^{\infty}H_{co}^{m-3}}^2+ \|\kpa \nabla\Delta\theta\|_{\hco^{m-3}}^2\lesssim\Lambda\big(\f{1}{c_0}, Y_m(0)\big)+(T+\ep)^{\tilde{\vartheta}_0} \Lambda\big(\f{1}{c_0},\cN_{m,t}\big),\label{EE-theta-2-1}\\
 &\|\kpa\Delta\theta\|_{L_t^{\infty}H_{co}^{m-2}}^2+\|\kpa^{\f{3}{2}}\nabla\Delta\theta\|_{\hco^{m-2}}^2 \notag\\
 &\lesssim \Lambda\big(\f{1}{c_0}, Y_m(0)\big)+(T+\ep)^{\tilde{\vartheta}_0} \Lambda\big(\f{1}{c_0},\cN_{m,t}\big)+\|\kpa^{\f{1}{2}}\nabla\sigma\|_{\hco^{m-1}}^2.\label{EE-theta-2-2}
\end{align}
\end{prop}
\begin{proof}
By the equation \eqref{sec4:eq0} satisfied by $\tilde{\theta}=\colon \f{C_v}{R}\theta-\f{\ep}{\gamma}\sigma,$ we can rewrite the equation for $\theta$ as:
\beqs
\f{C_v}{R}\pt\theta-\kpa\gamma^{-1}\Gamma\div(\beta\nabla\theta)=-u\cdot\nabla\tta+\gamma^{-1}\ep(\pt\sigma+\mathfrak{N}).
\eeqs
Taking a vector field $Z^J,$ $|J|\leq m-2$ and 
applying $\kpa^{\f{1}{2}}Z^J\nabla$ on the above equation, 
we get that:
\beq\label{rew-theta}
\f{C_v}{R}\kpa^{\f{1}{2}}\pt Z^J\na\theta-\gamma^{-1} \kpa^{\f{3}{2}}Z^J(\Gamma\beta\nabla\Delta\theta)=
\eqref{rew-theta}_1+\eqref{rew-theta}_2, %+\eqref{rew-theta}_3
\eeq
where 
\begin{align*}
\eqref{rew-theta}_1 = \kpa^{\f{1}{2}}Z^J\nabla\big(-u\cdot\nabla\tta+\gamma^{-1}\ep(\pt\sigma+\mathfrak{N})\big), \quad 
\eqref{rew-theta}_2=\gamma^{-1} \kpa^{\f{3}{2}}Z^J\big(\na(\Gamma\na\beta\cdot\na\theta)+\na(\Gamma\beta)\Delta\theta\big).
\end{align*}
Taking the inner product of the equation \eqref{rew-theta} with 
$-\kpa^{\f{1}{2}}\nabla\div(Z^J\na\theta),$ we get after integration in space and time the following identity:
\begin{align}\label{sec4:eq13}
   &\f{C_v}{2 R} \kpa\iomega |\div Z^J\na\theta|^2(t)\,\d x+\gamma^{-1}\kpa^2\izto \Gamma\beta |Z^J\nabla\Delta\theta|^2\d x \d s\notag\\
   &=\f{C_v}{2 R} \kpa \iomega |\div Z^J\na\theta|^2(0)\,\d x\underbrace{-\izto \kpa^{-\f{1}{2}}(\eqref{rew-theta}_1+\eqref{rew-theta}_2)
   \kpa\nabla\div(Z^J\na\theta)\,\d x\d s}_{\eqref{sec4:eq13}_1}\\
   &\quad-\underbrace{\gamma^{-1} \kpa^{2}\izto [Z^J,\Gamma\beta]\na\Delta\theta\cdot \na\div Z^J\na\theta +\Gamma\beta Z^J\nabla\Delta\theta [\nabla\div, Z^J]\nabla\theta  \,\d x\d s}_{\eqref{sec4:eq13}_2}. \notag
\end{align}
We first prove \eqref{EE-theta-2-1} and thus assume $|J|=k\leq m-3.$ Let us begin with the the estimate of  $\eqref{sec4:eq13}_2.$
By using the identity \eqref{comu} twice, %and also the elliptic estimate \eqref{elliptic-order3}
one can get that:
\beqs
\kpa\|[\nabla\div, Z^J]\nabla\theta\|_{L_t^2L^2}\lesssim \kpa\|\nabla^3 \theta\|_{\hco^{k-1}}\lesssim\kpa \|\na \Delta \theta\|_{\hco^{k-1}}
\eeqs
and thus 
\beq\label{sec4:eq14}
\kpa\| \na\div Z^J\na\theta\|_{L_t^2L^2}\lesssim \kpa \|\na \Delta \theta\|_{\hco^{k}}.
\eeq
Moreover, since $k\leq m-3,$ we bound the commutator term $[Z^J,\Gamma\beta]\na\Delta\theta$ simply by 
\begin{align*}
\kpa\|[Z^J,\Gamma\beta]\na\Delta\theta\|_{L_t^2L^2}&\lesssim \il Z(\Gamma\beta)\il_{k-1,\infty,t}\|\kpa \na \Delta \theta\|_{\hco^{k-1}}\\
&\lesssim \lab \il\theta\il_{m-3,\infty,t}\big)\|\kpa \na \Delta \theta\|_{\hco^{k-1}}+\ep\lat\cE_{m,t}.
\end{align*}
In view of the previous three estimates, we can now control %the last line in 
  $\eqref{sec4:eq13}_2$ as:
\begin{align}\label{sec4:eq15}
    \eqref{sec4:eq13}_2& \lesssim \kpa \|\na \Delta \theta\|_{\hco^{k}} \lab \il\theta\il_{m-3,\infty,t}\big) \|\kpa \na \Delta \theta\|_{\hco^{k-1}}+\ep\lat\cE_{m,t}^2\notag\\
    &\leq \delta \|\kpa \na \Delta \theta\|_{\hco^{k}}^2+C_{\delta}  \lab \il\theta\il_{m-3,\infty,t}\big)\|\kpa \na \Delta \theta\|_{\hco^{k-1}}^2+\ep\lat\cE_{m,t}^2.
\end{align}
Let us now see the term $\eqref{sec4:eq13}_1.$ It can be verified without much trouble that for $|J|\leq m-3,$
\begin{align*}
    \kpa^{-\f{1}{2}}\|\eqref{rew-theta}_1\|_{L_t^2L^2}&\lesssim \|(\sigma, u, \nabla( \sigma, u, \theta, \ep\mu\na u))\|_{\hco^{m-2}}\lat \lesssim T^{\f{1}{2}}\lat\cE_{m,t},\\
    \kpa^{-\f{1}{2}}\|\eqref{rew-theta}_2\|_{L_t^2L^2}&\lesssim \|(\ep\sigma, \theta,  \nabla(\ep \sigma, \theta), \kpa^{\f{1}{2}}\na^2\theta)\|_{\hco^{m-3}}\lat \lesssim T^{\f{1}{2}}\lat\cE_{m,t}.
\end{align*}
We thus derive from the above two estimates and \eqref{sec4:eq14} that:
\beqs
\eqref{sec4:eq13}_1\lesssim T^{\f{1}{2}}\lat \cE_{m,t},
\eeqs
 which, together with \eqref{sec4:eq15}, allow us to find, by choosing $\delta$ small enough that for any $0\leq k\leq m-3,$
\begin{align}
   \| \kpa^{\f{1}{2}}\Delta\theta\|_{\infco^{k}}^2+\|\kpa \na\Delta\theta\|_{\hco^k}^2 &\lesssim Y_m^2(0)+(T^{\f{1}{2}}+\ep)\lat\cE_{m,t}^2\notag\\
  & +\lab \il\theta\il_{m-3,\infty,t}^2\big)\|\kpa \na \Delta \theta\|_{\hco^{k-1}}^2+\| \kpa^{\f{1}{2}}\Delta\theta\|_{\infco^{k-1}}^2.
\end{align}
Note that in the above,  we have used the convention  $\|\cdot\|_{\infco^l}=\|\cdot\|_{\hco^l}=0,$ if $l=-1.$
It thus then follows from  induction that:
\begin{align}\label{sec4:eq16}
     \| \kpa^{\f{1}{2}}\Delta\theta\|_{\infco^{m-3}}^2+\|\kpa \na\Delta\theta\|_{\hco^{m-3}}^2 &\lesssim\lab \il\theta\il_{m-3,\infty,t}^2\big) Y_m^2(0)+(T^{\f{1}{2}}+\ep)\Lambda\big(\f{1}{c_0},\cN_{m,t}\big).
\end{align}
Note that we denote $\Lambda$ a polynomial that may differ from line to line. 
By the Sobolev embedding \eqref{sobebd} and the estimates \eqref{EE-theta-0}, \eqref{EE-theta-1},
\beq\label{thetainfty}
\il\theta\il_{m-3,\infty,t}^2\lesssim \|\na\theta\|_{L_t^{\infty}H_{co}^{m-2}}^2+
\|\theta\|_{L_t^{\infty}H_{co}^{m-1}}^2\lesssim Y_m^2(0)+(T^{\f{1}{2}}+\ep)\lae,
\eeq
%from which, one gets that:
%\beqs(1+ \il\theta\il_{m-3,\infty,t}^{2(m-3)})Y_m^2(0)\lesssim  \Lambda\big(\f{1}{c_0}, Y_m(0)\big)+\big((T+\ep)^{\f{1}{2}}+(T+\ep)^{\f{m-3}{2}}\big) \Lambda\big(\f{1}{c_0},\cN_{m,t}\big).\eeqs
plugging which into \eqref{sec4:eq16}, we find a constant $\tilde{\vartheta}_0>0$ such that 
\eqref{EE-theta-2-1} holds.

We now comment on the proof of \eqref{EE-theta-2-2}. Assume $|J|=m-2$ in \eqref{sec4:eq13} and multiply it  by $\kappa.$ Following the same path as what we just did for $|J|\leq m-3,$ we find the following estimate:
\begin{align}\label{sec4:eq17}
 \| \kpa\Delta\theta\|_{\infco^{m-2}}^2+\|\kpa^{\f{3}{2}}\na\Delta\theta\|_{\hco^{m-2}}^2 &\lesssim Y_m^2(0)+\| \kpa\Delta\theta\|_{\infco^{m-3}}^2
 +\| \kpa^{\f{3}{2}}\nabla\Delta\theta\|_{\hco^{m-3}}^2\\
&+\|(\eqref{rew-theta}_1,\eqref{rew-theta}_2)\|_{L_t^2L^2}^2+\| \kpa^{\f{3}{2}}[Z^J,\Gamma\beta]\na\Delta\theta\|_{L_t^2L^2}^2.\notag
\end{align}
One can then verify that for $|J|=m-2,$
\begin{align}
     &\|\eqref{rew-theta}_1\|_{L_t^2L^2}\lesssim \big(\|\nabla (u, \theta,\ep\sigma)\|_{\hco^{m-2}}+\|\kpa^{\f{1}{2}}\na(\theta,\ep\sigma)\|_{\hco^{m-1}}\big)\lat+\|\kpa^{\f{1}{2}}\nabla\sigma\|_{\hco^{m-1}}\notag\\
     &\qquad\qquad\qquad \lesssim T^{\f{1}{2}}\lat\cE_{m,t}+\|\kpa^{\f{1}{2}}\nabla\sigma\|_{\hco^{m-1}},\notag\\
     &\|\eqref{rew-theta}_2\|_{L_t^2L^2}\lesssim \|(\ep\sigma, \theta,  \nabla(\ep \sigma, \theta), \kpa\na^2\theta)\|_{\hco^{m-2}}\lat\lesssim T^{\f{1}{2}}\lat\cE_{m,t},\notag\\
 &\| \kpa^{\f{3}{2}}[Z^J,\Gamma\beta]\na\Delta\theta\|_{L_t^2L^2}\lesssim 
  \lab  \kpa^{\f{1}{2}}\il(\theta, \ep\sigma)\il_{m-2,\infty,t}\big) \|\kpa\na\Delta\theta\|_{\hco^{m-3}}+\ep\lat\cE_{m,t}.\notag
\end{align}
Moreover, by the Sobolev embedding \eqref{sobebd} and estimates \eqref{EE-highest}, \eqref{EE-2},
\beq\label{thetasigmainfty-m-2}
 \begin{aligned}
 \kpa\il(\theta,\ep\sigma)\il_{m-2,\infty,t}^2&\lesssim \kpa\|(\theta,\ep\sigma)\|_{L_t^{\infty}\underline{H}_{co}^{m}}^2+\kpa\|\na(\theta,\ep\sigma)\|_{L_t^{\infty}{H}_{co}^{m-1}}^2\\
 &\lesssim Y_m^2(0)+(T+\ep)^{\f{1}{2}}\lae.
 \end{aligned}
 \eeq
%\begin{align*} \kpa^{\f{1}{2}}\il(\theta, \ep\sigma)\il_{m-2,\infty,t}^2\lesssim   \kpa^{\f{1}{2}}(\|\na(\theta,\ep\sigma)\|_{L_t^{\infty}H_{co}^{m-1}}^2+\|(\theta,\ep\sigma)\|_{L_t^{\infty}H_{co}^{m-1}}^2)\lesssim Y_m^2(0)+(T^{\f{1}{2}}+\ep)\lae.  \end{align*}
Plugging the previous four estimates into \eqref{sec4:eq17}, we obtain \eqref{EE-theta-2-2} by using  \eqref{EE-theta-2-1} and by choosing $\tilde{\vartheta}_0$ smaller if necessary.
\end{proof}
\section{Estimates for the second normal derivatives} 
In this section, we prove some (non-uniform)  estimates for the second normal derivatives of the velocity $u$
which are included  in the energy norm $\cE_{m,T}(\sigma, u)$  defined in 
\eqref{defcEmsigmau}.
\begin{prop}
Under the assumption \eqref{preasption}, there is a constant $\vartheta_1>0,$ such that for any $\ep \in (0,1], (\mu,\kpa)\in A,$  any $0<t\leq T,$ 
\begin{align}
\ep\mu^2\|\na^2 u\|_{\hco^{m-2}}^2+\ep\mu \|\na^2 u\|_{\hco^{m-3}}^2\lesssim \lab Y_m(0)\big)+(T+\ep)^{\vartheta_1}\lae.
\end{align}
Moreover, the following estimate also holds:
\begin{align}\label{na2uhigh}
\ep^2\mu \big(\|\na u\|_{L_t^{\infty}H_{co}^{m-1}}^2
+\mu\|\na^2 u\|_{L_t^{2}H_{co}^{m-1}}^2\big)
+\ep^{\f{4}{3}}\mu\|\na^2 u\|_{L_t^{2}H_{co}^{m-2}}^2
\lesssim Y_m^2(0) +(T+\ep)^{\f{1}{6}}\lae. 
\end{align}
\end{prop}
\begin{proof}
By the virtue of the equivalence of the 
conormal spaces and the usual Sobolev spaces in  $\Omega_0$ (interior of the domain), the identities
\eqref{normalofnormalder} and 
\beqs%\label{tan-nor-u}
\Pi(\partial_{\bn} u)=\Pi(\curl u \times \bn)+\Pi \na (u\cdot\bn) +\Pi(-(D \bn)u)
\eeqs
as well as the established estimate \eqref{EE-highest}, %\eqref{nauL2-uniform}, 
it suffices to prove the corresponding estimates on each chart $\Omega_i (i=1,\cdots N)$ near the boundary for $\curl u\times(\chi_i\bn)$ where $ \chi_i$ is the cut-off function associated to $\Omega_i.$ 
%Moreover, by the boundary condition \eqref{bd-curlun}
Let us define 
\beqs%\label{defomegani}
\omega_{\bn i}=\colon \curl u\times\chi_i \bn -2 \Pi (\chi_i(-a u+D\bn\cdot u)),
\eeqs
then again by \eqref{EE-highest}, the matter is reduced to proving the following estimate: for any $i=1,\cdots N,$
\begin{align}
 & \ep\mu^2\|\na \omega_{\bn i}\|_{\hco^{m-2}}^2+\ep\mu \|\na\omega_{\bn i}\|_{\hco^{m-3}}^2\lesssim \lab Y_m(0)\big)+(T+\ep)^{\vartheta_1}\lae,\label{es-omegan-m-3}\\
 & \ep^2 \mu \big(\|\omega_{\bn i}\|_{\infco^{m-1}}^2
   +\mu\|\na\omega_{\bn i} \|_{L_t^{2}H_{co}^{m-1}}^2\big)+\ep^{\f{4}{3}}\mu
   \|\na\omega_{\bn i} \|_{\hco^{m-2}}^2\notag\\
  &\quad \lesssim 
   Y_m^2(0) +(T+\ep)^{\f{1}{6}}\lae.\label{es-omegan-m-1}
\end{align}

The reason to work on $\omega_{\bn i}$ instead of $ \curl u\times\chi_i \bn$ is that by the boundary condition \eqref{bd-curlun}, $\omega_{\bn i}$ vanishes on the boundary:
\beq\label{bdryomegan}
\omega_{\bn i}|_{\p\Omega}=0,
\eeq
which is more favorable for energy estimates. 
By lengthy but direct computations, we find that $\omega_{\bn i}$ is governed by the following transport-diffusion equation:
\begin{align}\label{omegan}
    (\pt +u\cdot\na)\omega_{\bn i}-\lambda_1 R\mu (\beta\Gamma)\Delta \omega_{\bn i}=F_i^{\omega}+L_i
\end{align}
where $F_i^{\omega}$ and $L$ are the source terms respectively of the equations for $\omega\times(\chi_i\bn), (\omega=\colon \curl u),$ and 
$-2 \Pi (\chi_i(-a u+D\bn\cdot u)):$
\beq
\begin{aligned}\label{defFomegaL}
  &  F_i^{\omega}=F_{1,i}^{\omega}\times (\chi_i\bn)+\omega\times(u\cdot \na(\chi_i\bn))-\lambda_1\mu R (\beta\Gamma) \big(\omega\times \Delta(\chi_i\bn)+2\na\omega\times\na(\chi_i\bn)\big),\\
 & L_i=2R\beta\Pi L_{1,i}+[\Pi, u\cdot\na]\big( 2\chi_i (-a u+D\bn \cdot u)\big)-\lambda_1\mu R (\beta\Gamma)[\Pi,\Delta]\big(2\chi_i(-a u+D\bn \cdot u)\big),
\end{aligned}
\eeq
with 
\begin{align*}
  & F_{1,i}^{\omega}= \omega\cdot\na u-\omega\div u-\na(R\beta)\times\f{\na\sigma}{\ep}+\na (R\beta\Gamma)\times\mu\div\cL u,\\
     &L_{1,i}=\f{1}{ R\beta}\big(-u\cdot \na (\chi_i D\bn\cdot)u + u\cdot\na(a\chi_i)u\big)
+\chi_i(D\bn\cdot{\na\sigma}+a\na\sigma)/\ep\\ &+\mu\lambda_1\Gamma \big(2\na(\chi_i D\bn)\cdot\na u+\Delta(\chi_i D\bn)\cdot u-2\na(a\chi_i)\nabla u+\Delta(a\chi_i)u\big)+\mu\Gamma\chi_i (\lambda_1+\lambda_2)
(a -D\bn\cdot)\na\div u.
\end{align*}

 In view of the above
expression of $F_i^{\omega}, L_i,$ the following fact can be verified by using the Proposition \ref{prop-prdcom} and the definition of $\cE_{m,t}:$
\begin{lem}
For any $1\leq i\leq N,$
\begin{align}\label{Property-FL}
 \ep\mu^{\f{1}{2}}\|F_i^{\omega}+L_i\|_{\hco^{m-1}}+\ep\|F_i^{\omega}+L_i\|_{\hco^{m-2}}\lesssim \lat\cE_{m,t}.
\end{align}
\end{lem}
%\begin{rmk}
%Note that the $\ep\mu\|\na^2 u\|_{L_t^2H_{co}^{m-1}}$ is not included in the definition of $\cE_{m,t},$ as we are not able to control $\mu\|u\|_{L_t^2\cH^{m,0}},$ that is the exact reason why the term $G_i$ need to be considered separately.\end{rmk}
We will skip the proof of the above lemma and proceed to 
prove \eqref{es-omegan-m-1}, \eqref{es-omegan-m-3}. %which relies on the energy estimates for the system \eqref{omegan} and \eqref{bdryomegan}. 
Performing standard energy estimates for the system \eqref{omegan} and \eqref{bdryomegan}, one gets that (we omit the subscript $i$ for simplicity)
\begin{align*}
&\f{\ep}{2}\iomega |\omega_{\bn}|^2(t)\,\d x+\ep \mu R\lambda_1  \izto  \beta\Gamma |\na \omega_{\bn}|^2 \,\d x\d s-\f{\ep}{2}\izto \div u |\omega_{\bn}|^2 \d x \d s\\
&=\f{\ep}{2}\iomega |\omega_{\bn}|^2(0)\,\d x
+\ep\izto (F^{\omega}+L) \cdot \omega_{\bn}\d x\d s
-\ep\mu R\lambda_1 \izto (\na(\Gamma \beta)\cdot \na \omega_{\bn}) \cdot \omega_{\bn}\d x\d s.
\end{align*}
Therefore, it follows from  the estimate \eqref{Property-FL}, Young's inequality and the 
fact $\|\omega_{\bn}\|_{L_t^{2}L^2}\lesssim T^{\f{1}{2}}\cE_{m,t}$ that:
\begin{align}\label{omegan-0}
\ep\|\omega_{\bn}\|_{L_t^{\infty}L^2}^2+\ep\mu\|\na\omega_{\bn}\|_{L_t^2L^2}^2\lesssim Y_m^2(0)+ T^{\f{1}{2}}\lat\cE_{m,t}^2.
%(\|\omega_{\bn}\|_{L_t^{\infty}L^2}^2 + )
\end{align}
Let us now focus on the higher regularity estimates. 
For any vector field $Z^J, |J|=k\leq m-1,$ we have the following energy identity:
\begin{align}\label{sec9:eq0}
\f{\ep}{2}\iomega |Z^J\omega_{\bn}|^2(t)\,\d x+ \ep\mu R\lambda_1  \izto  \beta\Gamma |Z^J\na \omega_{\bn}|^2 \,\d x\d s=\f{\ep}{2}\iomega |Z^J\omega_{\bn}|^2(0)\,\d x +\ep\sum_{l=1}^4\mathfrak{F}^J_l, 
\end{align}
where
\beq
\begin{aligned}
\mathfrak{F}^J_1&= \f{1}{2}\izto |Z^J \omega_{\bn}|^2\div u \, \d x\d s,\,
\mathfrak{F}^J_2=\izto \big(Z^J (F^{\omega}+L) -[Z^J, u\cdot\na]\omega_{\bn}\big) \cdot Z^J \omega_{\bn}\d x\d s,\\
\mathfrak{F}^J_3&= -R\lambda_1\mu \izto Z^J(\na(\Gamma \beta)\cdot \na \omega_{\bn}) \cdot Z^J \omega_{\bn}\d x\d s,\\
\mathfrak{F}^J_4&=-R\lambda_1\mu \izto  [Z^J, \Gamma \beta]\na\omega_{\bn}\cdot\na Z^J \omega_{\bn}
+\Gamma \beta  Z^J \na \omega_{\bn}\cdot [\na, Z^J]\omega_{\bn}\\
&\qquad\qquad\qquad\qquad\qquad \qquad\qquad \qquad\qquad -[Z^J, \div](\Gamma\beta \na \omega_{\bn})\cdot Z^J \omega_{\bn}\d x\d s.
\end{aligned}
\eeq
Let us first consider the case when $|J|=k\leq m-3.$ The term $\mathfrak{F}^J_1$ is controlled easily 
as:
\begin{align}\label{sec9:eq10}
    \ep \mathfrak{F}^J_1\lesssim \ep \| \omega_{\bn}\|_{\hco^{m-3}}\il \div u\il_{0,\infty,t}\lesssim \ep\lat\cE_{m,t}^2.
\end{align}
Next, by the Cauchy-Schwarz inequality, 
\begin{align*}
    \ep \mathfrak{F}^J_2 \lesssim T^{\f{1}{2}}\|\omega_{\bn}\|_{\infco^{m-3}}\|\ep\big(Z^J(F^{\omega}+L), [Z^J, u\cdot\na]\omega_{\bn}\big)\|_{L_t^2L^2}. %\lesssim T^{\f{1}{2}}\lae.
\end{align*}
Thanks to the identity \eqref{id-convetion}, the commutator estimate \eqref{roughcom} and the estimate \eqref{hardy-calculus}, we get
for any $k\leq m-2$ that
\begin{align*}
   \| [Z^J, u\cdot\na] \omega_{\bn}\|_{L_t^2L^2}&\lesssim
   \big(\|  \omega_{\bn} \|_{L_t^2H_{co}^k}+\|(u,\na u)\|_{\hco^k}\big)\Lambda\big( \f{1}{c_0}, \il (u,\na u)\il_{[\f{k}{2}],\infty,t} \big)\\
   &\lesssim \lat\cE_{m,t},
\end{align*}
which, combined with the property \eqref{Property-FL}, 
yields that:
\begin{align}\label{sec9:eq11}
     \ep \mathfrak{F}^J_2 \lesssim T^{\f{1}{2}}\lat\cE_{m,t}^2.
\end{align}
Moreover, applying the product estimate \eqref{roughproduct1},
we can proceed to control $ \ep \mathfrak{F}^J_3$ as:
\beq\label{sec09:eq11}
\begin{aligned}
    \ep \mathfrak{F}^J_3 &\lesssim T^{\f{1}{2}}\|\omega_{\bn}\|_{\infco^{m-3}} 
    \big(\|\na(\Gamma\beta)\|_{\hco^{m-3}}+ \|\ep \mu\na \omega_{\bn}\|_{\hco^{m-3}}\big)\lat\\
    &\lesssim  T^{\f{1}{2}}\lat\cE_{m,t}^2.
\end{aligned}
\eeq
Finally, it follows from the identity \eqref{comu}, the integration by parts and Young's inequality that:
\begin{align}\label{sec9:eq12}
  \ep \mathfrak{F}^J_4\leq  \delta \ep\mu \|Z^J \na\omega_{\bn} \|_{L^2}^2+C\big({\delta}, \f{1}{c_0}\big)\ep\mu\|\na\omega_{\bn} \|_{\hco^{k-1}}^2\big(1+\il\Gamma\beta\il_{m-3,\infty,t}^2\big). %
\end{align}
Inserting \eqref{sec09:eq11}-\eqref{sec9:eq12} into \eqref{sec9:eq0}, we find by choosing $\delta$ sufficiently small 
that for any $0\leq k\leq m-3,$
\begin{align*}
    \ep \|\omega_{\bn}\|_{\infco^{k}}^2+ \ep \mu
    \|\omega_{\bn}\|_{\hco^{k}}^2\lesssim Y_m^2(0)
    + \ep \mu
    \|\omega_{\bn}\|_{\hco^{k-1}}^2\big(1+\il\Gamma\beta\il_{m-3,\infty,t}^2\big)+(T+\ep)^{\f{1}{2}}\lae,
\end{align*}
which, combined with \eqref{omegan-0} and the induction arguments, leads to that
\begin{align*}
    \ep \|\omega_{\bn}\|_{\infco^{m-3}}^2+ \ep \mu
    \|\omega_{\bn}\|_{\infco^{m-3}}^2\lesssim 
    Y_m^2(0)\lab \il(\theta, \ep\sigma)\il_{m-3,\infty,t} \big) + (T+\ep)^{\f{1}{2}}\lae.
\end{align*}
Moreover, by using the estimate \eqref{thetainfty} and the fact $\il\ep\sigma\il_{m-3,\infty,t}\lesssim \ep\lat,
$ we can find a constant $\vartheta_1>0,$ such that:
\begin{align}\label{omeganm-3}
    \ep \|\omega_{\bn}\|_{\infco^{m-3}}^2+ \ep \mu
    \|\omega_{\bn}\|_{\infco^{m-3}}^2\lesssim \lab Y_m(0)\big)+ (T+\ep)^{\vartheta_1}\lae.
\end{align}

Let us now focus on the case $|J|=m-2.$
%$m-2\leq |J|\leq m-1.$ 
Due to the difficulty arising from the control of  $\mathfrak{F}^J_4,$ we need to multiply \eqref{sec9:eq0} by $\ep^{\f{1}{3}}$ or $\mu$ when $|J|= m-2.$ The following paragraph is thus devoted to the control of $\ep^{\f{4}{3}}\sum_{l=1}^4 \mathfrak{F}^J_l $ and $\ep \mu \sum_{l=1}^4 \mathfrak{F}^J_l$  when $|J|= m-2.$ 

Analogues to \eqref{sec9:eq10}-\eqref{sec09:eq11}, 
one can show that if $|J|=m-2,$
\begin{align}\label{sec9:eq1-m-2}
 \ep (\mathfrak{F}^J_1+\mathfrak{F}^J_2+\mathfrak{F}^J_3)  \lesssim (T^{\f{1}{2}}+\ep)\lae. 
\end{align}
It remains to estimate $\ep^{\f{4}{3}}\mathfrak{F}^J_4 $ and $\ep\mu\mathfrak{F}^J_4.$ %when $|J|=m-2$ %and $\ep^2 \mu \mathfrak{F}^J_4$ when $|J|=m-1.$ If $|J|=m-2,$
 We have by the property \eqref{comu}, the integration by parts and the Young's inequality that:
\begin{align*}
\ep\nu\mathfrak{F}^J_4\lesssim \ep \mu\nu\|\na\omega_{\bn}\|_{\hco^{m-2}}  \big(\|\na\omega_{\bn} \|_{\hco^{m-3}}+\|[Z^J, \Gamma\beta]\na\omega_{\bn}\|_{L_t^2L^2}+\|\Gamma\beta \na\omega_{\bn}\|_{\hco^{m-3}}\big)
\end{align*}
where $\nu$ is a placeholder of $\ep^{\f{1}{3}}$ or $\mu.$ If $\nu=\ep^{\f{1}{3}},$
the term $[Z^J, \Gamma\beta]\na\omega_{\bn}, \ep\mu^{\f{1}{2}} \Gamma\beta \na\omega_{\bn}$ 
can be bounded as follows:
\begin{align*}
   & \ep^{\f{2}{3}}\mu^{\f{1}{2}}\big(\big\| [Z^J, \Gamma\beta]\na\omega_{\bn}\big\|_{L_t^2L^2}+\| \Gamma\beta \na\omega_{\bn}\|_{\hco^{m-3}}\big)\\
   & \lesssim \ep^{\f{1}{6}} \|(\ep\mu)^{\f{1}{2}}\na\omega_{\bn}\|_{\hco^{m-3}} \il \Gamma\beta\il_{m-3,\infty,t}+T^{\f{1}{6}}
    \|Z(\Gamma\beta)\|_{L_t^{\infty}W_{co}^{m-3,6}}\| \ep^{\f{2}{3}}\mu^{\f{1}{2}}\na\omega_{\bn}\|_{L_t^3L^3}.
 \end{align*}
 By the usual Sobolev embedding and interpolation as well as the estimates \eqref{esofGamma-2} and \eqref{esofbeta} for $\Gamma$ and $\beta,$
 \begin{align*}
 & \|Z(\Gamma\beta)\|_{L_t^{\infty}W_{co}^{m-3,6}}\lesssim \|(\Id, \na) (\Gamma,\beta)\|_{\infco^{m-2}} \lesssim \|(\Id, \na)(\sigma,\theta)\|_{\infco^{m-2}}  \lat, \\
 & \| \ep^{\f{2}{3}}\mu^{\f{1}{2}}\na\omega_{\bn}\|_{L_t^3L^3}\lesssim \il\ep\mu^{\f{1}{2}}\na\omega_{\bn} \il_{0,\infty,t}^{\f{1}{3}} \|(\ep\mu)^{\f{1}{2}}\na\omega_{\bn}\|_{L_t^2L^2}^{\f{2}{3}}\lesssim \lae,
 \end{align*}
 which, combined with the previous estimate, leads to that:
\begin{align*}
  \ep^{\f{2}{3}}\mu^{\f{1}{2}}\big(\big\| [Z^J, \Gamma\beta]\na\omega_{\bn}\big\|_{L_t^2L^2}+\| \Gamma\beta \na\omega_{\bn}\|_{\hco^{m-3}}\big)\lesssim     (T+\ep)^{\f{1}{6}}\lae. 
\end{align*}
%\begin{align*}
  %&\lesssim  \ep^{\f{1}{2}}\big(\| (\ep\mu)^{\f{1}{2}}\na\omega_{\bn}\|_{\hco^{m-3}}+\|Z(\Gamma\beta)\|_{\hco^{m-4}}\big)\lat\lesssim  \ep^{\f{1}{2}}\lat\cE_{m,t}.\end{align*}
Consequently, we obtain that, for $|J|= m-2,$ 
\beqs%\label{f3-m-2}
\begin{aligned}
 \ep^{\f{4}{3}}\mathfrak{F}^J_4\leq \delta \ep^{\f{4}{3}} \mu\|\na\omega_{\bn}\|_{\hco^{m-2}}^2+(T+\ep)^{\f{1}{6}}\lae, 
\end{aligned}
\eeqs
which, together with \eqref{sec9:eq1-m-2}, yields that
\beq\label{es-omeganm-2}
\ep^{\f{4}{3}} \| \omega_{\bn}\|_{\infco^{m-2}}^2+\ep^{\f{4}{3}}\mu \|\na \omega_{\bn}\|_{\hco^{m-2}}^2\lesssim Y_m^2(0)+(T+\ep)^{\f{1}{6}}\lae.
\eeq
Moreover, if $\nu=\mu,$
\begin{align*}
    &\ep^{\f{1}{2}}\mu\| [Z^J, \Gamma\beta]\na\omega_{\bn}\|_{L_t^2L^2}+\ep^{\f{1}{2}}\mu\| \Gamma\beta \na\omega_{\bn}\|_{\hco^{m-3}} \lesssim (\ep\mu)^{\f{1}{2}} \|\na\omega_{\bn}\|_{\hco^{m-3}}\il\mu^{\f{1}{2}}\Gamma\beta\il_{m-2,\infty,t},
\end{align*}
and thus:
\begin{align*}
   \ep\mu \mathfrak{F}^J_4\leq \delta \ep \mu^2\|\na\omega_{\bn}\|_{\hco^{m-2}}^2 +  \|(\ep\mu)^{\f{1}{2}}\na\omega_{\bn}\|_{\hco^{m-3}}^2 \il\mu^{\f{1}{2}}\Gamma\beta\il_{m-2,\infty,t},
\end{align*}
which, combined with \eqref{sec9:eq1-m-2}, yields that, by choosing $\delta$ small enough,
\beqs
\begin{aligned}
&\ep\mu\| \omega_{\bn}\|_{\infco^{m-2}}^2+\ep\mu^2 \|\na \omega_{\bn}\|_{\hco^{m-2}}^2\\
&\lesssim Y_m^2(0)+(T^{\f{1}{2}}+\ep)\lae+\ep\mu\|\na\omega_{\bn}\|_{\hco^{m-3}}^2\lab \il\mu^{\f{1}{2}}(\theta,\ep\sigma)\il_{m-2,\infty,t}\big).
\end{aligned}
\eeqs
Therefore, by noticing \eqref{thetasigmainfty-m-2},
and \eqref{omeganm-3}, one obtains by choosing $\vartheta_1$ smaller if necessary that:
\begin{align*}
    \ep\mu \|\omega_{\bn}\|_{\infco^{m-2}}^2+\ep\mu^2\|\na\omega_{\bn}\|_{\hco^{m-2}}^2\lesssim \lab Y_m(0)\big) + (T+\ep)^{\vartheta_1}\lae. 
\end{align*}
The desired estimate \eqref{es-omegan-m-3} is now complete in light of the above estimate and \eqref{omeganm-3}.

In view of \eqref{es-omeganm-2}, to finish the proof of \eqref{es-omegan-m-1}, it remains to control 
$\ep\mu^{\f{1}{2}}(\|\omega_{\bn}\|_{\infco^{m-1}}+\mu^{\f{1}{2}} \|\na \omega_{\bn}\|_{\hco^{m-1}}),$ we thus  multiply the both sides of \eqref{sec9:eq0}
by $\ep\mu$ and control the terms $\ep^2\mu\sum_{l=1}^4 \mathfrak{F}^J_l.$ Similar to \eqref{sec9:eq10}-\eqref{sec9:eq11}, one can verify that,
when $|J|=m-1,$
\begin{align}\label{sec9:eq1-m-1}
 \ep^2\mu (\mathfrak{F}^J_1+\mathfrak{F}^J_2+\mathfrak{F}^J_3)  \lesssim (T^{\f{1}{2}}+\ep)\lae.
\end{align}
For instance, thanks to the estimate \eqref{Property-FL} and the identity \eqref{id-convetion}, 
we can bound $\mathfrak{F}^J_2, \mathfrak{F}^J_3 $ %when $|J|=m-1$ 
in the following way: 
\begin{align*}%\label{f2-m-1}
   \ep^2 \mu \mathfrak{F}^J_2&\lesssim \ep\|\mu^{\f{1}{2}}\omega_{\bn}\|_{\hco^{m-1}} \|\ep\mu^{\f{1}{2}}\big(Z^J(F^{\omega}+L), [Z^J, u\cdot\na]\omega_{\bn}\big)\|_{L_t^2L^2} \lesssim \ep\lae, \\
   \ep^2\mu \mathfrak{F}^J_3 &\lesssim \ep \|\mu^{\f{1}{2}}\omega_{\bn}\|_{\hco^{m-1}} \|\mu^{\f{1}{2}}\na(\Gamma\beta) \cdot \na (\ep\mu \omega_{\bn})\|_{\hco^{m-1}}\lesssim \ep \lae.
\end{align*}
Moreover, it results again from the property \eqref{comu}, the integration by parts and the Young's inequality that, 
\begin{align*}
\ep^2\mu\mathfrak{F}^J_4\lesssim %\delta(\ep\mu)^2 \|\na\omega_{\bn}\|_{\hco^{m-1}}^2+ C_{\delta}
\ep\mu \|\na\omega_{\bn}\|_{\hco^{m-1}} \ep\mu
\big(\|\na\omega_{\bn}\|_{\hco^{m-2}}+\|[Z^J, \Gamma\beta]\na\omega_{\bn}\|_{L_t^2L^2}+\|\Gamma\beta \na\omega_{\bn}\|_{\hco^{m-2}}\big).
\end{align*}
We thus find by using the following estimate
\begin{align*}
   &  \ep\mu\big(\| [Z^J, \Gamma\beta]\na\omega_{\bn}\|_{L_t^2L^2}+ \|\Gamma\beta \na\omega_{\bn}\|_{\hco^{m-2}}\big)\\
    &\lesssim \ep\mu \|\na\omega_{\bn}\|_{\hco^{m-2}} \il\Gamma\beta\il_{2,\infty,t}+
    \il Z(\Gamma\beta)\il_{\hco^{m-2}}\il\ep\mu\na\omega_{\bn}\il_{m-4,\infty,t}\\
    &\lesssim (T+\ep)^{\f{1}{2}}\lae
\end{align*}
 that $$\ep^2\mu\mathfrak{F}^J_4\lesssim  (T+\ep)^{\f{1}{2}}\lat\cE_{m,t}^2. $$
This, together with \eqref{sec9:eq1-m-1}, leads to the estimate:
\begin{align*}
   \ep^2\mu\|\omega_{\bn}\|_{\infco^{m-1}}^2+ \ep^2\mu^2\|\na \omega_{\bn}\|_{\hco^{m-1}}^2\lesssim Y_m^2(0)+(T+\ep)^{\f{1}{2}}\lae.
\end{align*}
The proof of \eqref{es-omegan-m-1} is thus finished.
\end{proof}
\begin{prop}
 Under the assumption \eqref{preasption}, it holds that, for any $\ep \in (0,1], (\mu,\kpa)\in A,$  any $0<t\leq T,$ 
 \begin{align}\label{secnoru-LinftyL2}
     \ep\mu \|\na^2 u \|_{\infco^{m-2}}\lesssim \| \mu \na(\ep u, \theta) \|_{\infco^{m-1}}+\big\|\big(\f{1}{R\beta}\ep\pt u, \na\sigma\big)\big\|_{\infco^{m-2}}+\ep\lae.
 \end{align}
\end{prop}
\begin{proof}
It stems from the equation $\eqref{newsys}_1$ that:
\beq\label{nadivuLinftyL2}
\begin{aligned}
    \ep\mu^{\f{1}{2}}\|\na\div u \|_{\infco^{m-2}}&\lesssim  \ep\mu^{\f{1}{2}} \|\na(\pt+u\cdot\na)(\theta,\ep\sigma)\|_{\infco^{m-2}}\\
    &\lesssim \mu^{\f{1}{2}} \|\na\theta\|_{\infco^{m-1}}+\ep\lae.
\end{aligned}
\eeq
Moreover, we use $\eqref{NCNS-S2}_2$ to write:
 \begin{align}\label{rewrite-equ}
    \ep \mu \lambda_1 \Gamma \Delta u= -\ep\mu(\lambda_1+\lambda_2)\Gamma\na\div u+\f{\ep}{R\beta}(\pt+u\cdot\na)u+\na\sigma,
 \end{align}
 which, combined with the local expression of the Laplacian \eqref{Laplace-local}, leads to that:
 \begin{align*}
     \ep \mu\| \na^2 u\|_{\infco^{m-2}}\lesssim \|\ep\mu\na u\|_{\infco^{m-1}}+
     \|( \ep\mu \na\div u, \f{1}{R\beta}\ep\pt u, \na\sigma)\|_{\infco^{m-2}}+\ep\lae.
 \end{align*}
 The desired estimate \eqref{secnoru-LinftyL2} then follows from the above inequality and \eqref{nadivuLinftyL2}.
\end{proof}
\begin{rmk}\label{rmk-thirdordervelocity}
It follows from the similar arguments as above that:
\begin{align*}
    \ep\mu \|\na^3 u\|_{\hco^{m-3}}\lesssim \lae.
\end{align*}
\end{rmk}
\section{Uniform (in Mach number) estimates for the compressible part-I}
In this section, we aim to recover the higher order spatial derivatives for the compressible part $(\na\sigma,\div u),$ more precisely, we will control $\kpa^{\f{1}{2}}\|(\na\sigma,\div u)\|_{\hco^{m-1}}, \kpa^{\f{1}{2}}\|\na(\na\sigma, \div u)\|_{\hco^{m-2}}.$
%which can be done by using recursively the equation $\eqref{NCNS-S2}_1, \eqref{NCNS-S2}_2.$ %To define  the incompressible part and the compressible and of the $r_0 u$, we shall use  the Leray projection whose definition is recalled here.
 Before stating and prove the main results in this section, we recall the definition of the Leray projection.
 
\textbf{Definition of the Leary projection.}
We define the projection operator
$\mathbb{Q}$:
\begin{equation}\label{proj-Q}
   \begin{aligned}
   \mathbb{Q}: \quad &
L^2(\Omega)^3\rightarrow L^2(\Omega)^3\\
&\qquad f \rightarrow \mathbb{Q} f=\nabla\Psi
   \end{aligned} 
\end{equation}
where $\Psi$ is defined as the unique
solution of
\begin{equation}\label{Neumann}
\left\{
  \begin{array}{l}
  \displaystyle \Delta\Psi=\div f  \quad \text{in} \quad \Omega, \\
      \displaystyle\partial_{\bn} \Psi
      =f\cdot {\bn} \quad \text{on} \quad \partial\Omega,\\
 \int_{\Omega} \Psi \,\d x=0 \text{ if } \Omega \text{ is bounded}, \Psi(x)\xlongrightarrow{|x|\rightarrow\infty} 0 \text{ if } \Omega \text{ is an  exterior domain.}
  \end{array}
  \right.
\end{equation}
 Note that the solvability of the Neumann problem \eqref{Neumann} in $H^{1}(\Omega)$ is well-known as an application of the Lax-Milgram theorem.
 The Leary projection is then defined as:
\begin{align}\label{proj-P}
 \mathbb{P} =\Id -\mathbb{Q}.
\end{align}
 %Moreover, by Proposition \eqref{propneumann}, one has that for a $C^{k+1}$ bounded domain,
 %\begin{equation}\|\nabla\Psi(t)\|_{H_{co}^k}\lesssim \|f(t)\|_{H_{co}^k}, \qquad \|\nabla^2\Psi(t)\|_{H_{co}^{k-1}}\lesssim \|\div f(t)\|_{H_{co}^{k-1}}+\|f(t)\|_{H_{co}^{k}}. \end{equation}
%Note that in these estimates, the time variable is just an external parameter. 
%The following is the first result in this section:
We are now ready to prove the following result:
\begin{prop}\label{lem-com1}
Under the assumptions \eqref{preasption}, for any $\ep \in (0,1], (\mu,\kpa)\in A,$ any $0<t\leq T,$  we have the following estimate: 
\begin{align}\label{com-uniL2}
    \kpa\|(\na\sigma,\div u)\|_{\hco^{m-1}}^2\lesssim Y_m^2(0)+ (T+\ep)^{\f{1}{6}}\lae.
\end{align}
\end{prop}
\begin{proof}
Let us first prove that 
\beq\label{sec5:eq2}
 \kpa\|(\na\sigma,\div u)\|_{\hcoch^{m-1,0}}^2 \lesssim Y_m^2(0)+ (T+\ep)^{\f{1}{6}}\lae.
\eeq
However, as it has been shown in \eqref{EE-3} that:
\beq\label{sec5:eq3}
 \kpa\|\div u\|_{\hcoch^{m-1,0}}^2 \lesssim Y_m^2(0)+ (T+\ep)^{\f{1}{2}}\lae,
\eeq
we only need to control  $\kpa\|\na\sigma\|_{\hcoch^{m-1,0}}^2.$ Applying $(\ep\pt)^{j}\, (0\leq j\leq m-1)\,$ on $(\eqref{NCNS-S2}_1, \eqref{NCNS-S2}_2)$ and multiplying them by $\ep\kpa\big(\f{\gamma }{R\beta}(\ep\pt)^{j}\div u, - (\ep\pt)^{j}\na\sigma\big),$ one finds by denoting $f^j=(\ep\pt)^j f$ that:
\begin{align*}
     \kpa\|(\nabla\sigma)^j\|_{L_t^2L^2}^2=\sum_{l=1}^4\cB_l^j,
\end{align*}
where
\begin{align*}
   &\cB_1^j=-\ep\kpa \izt\pt\iomega  \f{1 }{R\beta} %(\ep\pt)^{j}
   (\na\sigma)^j\cdot  u^j\, \d x\d s +\kappa\gamma \izto \f{1}{R\beta}|(\div u)^j|^2\d x \d s,\notag\\
  &  \cB_2^j=
  -\kpa \izto \ep\pt\sigma^j\na(\f{1}{R\beta})\cdot u^j \d x\d s +\ep\kpa \izto \pt (\f{1 }{R\beta}) (\na\sigma)^j\cdot u^j\d x\d s,\notag\\
   &  \cB_3^j=\ep\kpa\izto\f{1 }{R\beta} \big( u\cdot\na\sigma -\ep\mu(\gamma-1)\Gamma \cL u\cdot \mathbb{S}u\big)^j(\div u)^j\d x\d s,\\
  &  \cB_4^j=\ep\kpa\izto\big( \big(-\f{u}{R\beta}\cdot\na u+\mu\Gamma\div\cL u\big)^j-[\ep\pt)^{j},\f{1}{R\beta}]\pt u\big) \cdot(\na\sigma)^j\d x\d s.\notag%\\
 % &\lesssim Y_m^2(0)+\|(\ep\kpa \nabla\sigma, u)\|_{L_t^{\infty}\cH^{m-1,0}}^2+\kpa\|\div u\|_{\hcoch^{m-1,0}}^2+(T^{\f{1}{2}}+\ep)\lae.\notag
\end{align*}
In view of the estimate \eqref{Gamma-Beta-infty}, the first term $\cB_1^j$ can be controlled directly as:
\begin{align*}
   \cB_1^j\lesssim  Y_m^2(0)+\|(\ep\kpa \nabla\sigma, u)\|_{L_t^{\infty}\cH^{m-1,0}}^2+\kpa\|\div u\|_{\hcoch^{m-1,0}}^2.
\end{align*}
Next, thanks to the Cauchy-Schwarz inequality, 
\begin{align*}
     \cB_2^j \lesssim T^{\f{1}{2}}\|u\|_{L_t^{\infty}\cH^{m-1,0}} \|\kpa^{\f{1}{2}}(\ep\pt,\na)\sigma^j\|_{L_t^2L^2}\lat\lesssim T^{\f{1}{2}}\lae, \\
     \cB_3^j \lesssim \ep \|\kpa^{\f{1}{2}} \div u^j\|_{L_t^2L^2}\|\kpa^{\f{1}{2}}(u, \na\sigma, \na u)\|_{\hco^{m-1}}\lat\lesssim \ep\lae.
\end{align*}
Note that by using the equation $\eqref{NCNS-S2}_1,$
\begin{align*}
   \kpa^{\f{1}{2}} \|\ep\pt \sigma^j\|_{L_t^2L^2}\lesssim  \|(\kpa^{\f{1}{2}} \div u, \kpa^{\f{3}{2}}\Gamma\div(\beta\na\theta))\|_{\hco^{m-1}}+\ep\lae\lesssim \lae.
\end{align*}
Finally, by the Young's inequality, the term $ \cB_4^j $ can be controlled as:
\begin{align*}
     \cB_4^j&\lesssim \|\kpa^{\f{1}{2}} \na \sigma^j\|_{L_t^2L^2}\|\ep\mu \na^2 u\|_{\hco^{m-1}}+\ep\lae \\
     &\leq \f{\kpa}{2} \|\na \sigma^j\|_{L_t^2L^2}^2+ C\big(\f{1}{c_0}\big)\|\ep\mu \na^2 u\|_{\hco^{m-1}}^2+\ep\lae.
\end{align*}
Collecting the previous estimates for $ \cB_1^j- \cB_4^j ,$
we find that:
\begin{align*}
    \kpa  \|\na \sigma\|_{L_t^2\cH^{m-1,0}}^2&\lesssim Y_m^2(0)+\|(\ep\kpa \nabla\sigma, u)\|_{L_t^{\infty}\cH^{m-1,0}}^2+\kpa\|\div u\|_{\hcoch^{m-1,0}}^2\\
    &\qquad+\|\ep\mu \na^2 u\|_{\hco^{m-1}}^2+(T^{\f{1}{2}}+\ep)\lae.
\end{align*}
 By using the estimates \eqref{EE-highest} for $\|\ep\kpa^{\f{1}{2}}\na\sigma\|_{\infco^{m-1}},$ \eqref{EE-3} for $\|u\|_{\infcoch{}^{m-1,0}}$, \eqref{sec5:eq3} for $\kpa^{\f{1}{2}}\|\div u\|_{\hcoch^{m-1,0}},$ \eqref{na2uhigh} for $\|\ep\mu \na^2 u\|_{\hco^{m-1}},$
 we see that $\|\nabla\sigma\|_{\hcoch^{m-1,0}}$ can be controlled by the right hand side of \eqref{sec5:eq2}.

We will show in below that for any $j,l$ with $0\leq j+l\leq m-1, l\geq 1,$ the following estimate holds:
\beq\label{sec5:eq9}
\begin{aligned}
\kpa^{\f{1}{2}}\|(\na{\sigma},\div u)\|_{\hcoch^{j,l}}&\lesssim \kpa \|(\na{\sigma}, \div u)\|_{\hcoch^{j+1,l-1}}\\
&\quad+\|(\kpa\Delta\theta, %\kpa^{\f{1}{2}}
    \ep\mu\na\div u)\|_{\hco^{m-1}}+(T+\ep)^{\f{1}{2}}\lae.
   \end{aligned}
\eeq
This estimate in hand, one can show 
by induction that:
\begin{align*}
    \kpa^{\f{1}{2}}\|(\na\sigma, \div u)\|_{\hco^{m-1}}&\lesssim  \kpa^{\f{1}{2}}\|(\na\sigma, \div u)\|_{L_t^2\cH^{m-1,0}}\\
    &\quad +\|(\kpa\Delta\theta, %\kpa^{\f{1}{2}}
    \ep\mu\na\div u)\|_{\hco^{m-1}}+(T+\ep)^{\f{1}{2}}\lae,
\end{align*}
which, together with \eqref{sec5:eq2}, \eqref{EE-highest}, \eqref{na2uhigh}, yields \eqref{com-uniL2}.

Let us now prove \eqref{sec5:eq9}.
First,  the equation $\eqref{NCNS-S2}_1$ can be rewritten as: 
\beq \label{rew-sigma}
\div u= -\f{1}{\gamma}(\ep\pt+\ep u\cdot\nabla)\sigma+\kpa\f{\gamma-1}{\gamma} \Gamma \div(\beta\nabla\theta)+\f{\gamma-1}{\gamma}\mathfrak{N},
\eeq
from which we can deduce that:
\begin{align*}%\label{sec5:eq4}
  \kpa^{\f{1}{2}}\|\div u\|_{\hcoch^{j,l}}&\lesssim    \kpa^{\f{1}{2}}\|\na\sigma\|_{ \hcoch^{j+1,l-1}}+ \kpa^{\f{1}{2}}\|\sigma\|_{L_t^2\cH^{j+1,0}}\notag\\
 &\quad +\ep\kpa^{\f{1}{2}} \|(u\cdot\na\sigma, \mu \Gamma \cL u\cdot \mathbb{S}u)\|_{\hco^{m-1}}+\kpa^{\f{3}{2}}\|\Gamma \div(\beta\nabla\theta)\|_{\hco^{m-1}}.
\end{align*}
Moreover, it follows from the product estimate \eqref{roughproduct1} and the commutator estimate \eqref{roughcom} that the 
last line in the above inequality can be bounded by:
\begin{align*}
    \kpa\|\Delta\theta\|_{\hco^{m-1}}+(T+\ep)^{\f{1}{2}}\lae.
\end{align*}
%\begin{align*}&\kpa^{\f{1}{2}} \|u\cdot\na\sigma\|_{\hco^{m-1}}\lesssim \|(u,\kpa^{\f{1}{2}}\nabla\sigma)\|_{\hco^{m-1}}\lat\lesssim \lae,\\& \kpa\|\Gamma \div(\beta\nabla\theta)\|_{\hco^{m-1}}\lesssim \kpa^{\f{3}{2}}\|(\Gamma\beta\Delta\theta,\Gamma\na\beta\na\theta)\|_{\hco^{m-1}}\\&\lesssim \kpa\|\Delta\theta\|_{\hco^{m-1}}+\big(\|\kpa\Delta\theta\|_{\hco^{m-2}}+\|(\ep\sigma,\theta, \kpa^{\f{1}{2}}\na\theta)\|_{\hco^{m-1}}\big)\lat\\&\lesssim \kpa\|\Delta\theta\|_{\hco^{m-1}}+T^{\f{1}{2}}\lae.\end{align*}
%Inserting these two estimates into \eqref{sec5:eq4}, 
Consequently, we obtain that:
\begin{align}\label{sec5:eq5}
    \kpa^{\f{1}{2}} \|\div u\|_{\hcoch^{j,l}}\lesssim \kpa^{\f{1}{2}}\|\na\sigma\|_{\hcoch^{j+1,l-1}}+\kpa\|\Delta\theta\|_{\hco^{m-1}}+(T+\ep)^{\f{1}{2}}\lae.
\end{align}

We now switch to control $\kpa^{\f{1}{2}} \|\na\sigma\|_{\hcoch^{j,l}}.$ Denote $\bq (\f{1}{R\beta} u)$ and $ \bp(\f{1}{R\beta} u)$  as the compressible part and incompressible part of $\f{1}{R\beta} u.$ By the definition of $\bq, \bp$ in \eqref{proj-Q}, \eqref{proj-P} as well as the boundary condition $\f{1}{R\beta} u\cdot\bn|_{\p\Omega}=0,$ 
we have that 
$$\div\big( \f{1}{R\beta} u\big)=\div \bq \big(\f{1}{R\beta} u\big),\quad  \bp\big(\f{1}{R\beta} u\big)\cdot\bn|_{\p\Omega}=\bq \big(\f{1}{R\beta} u\big)\cdot\bn|_{\p\Omega}=0. $$
Therefore, the equation \eqref{eq-tiltasigma} for $\tilde{\sigma}=\sigma-\ep\mu(2\lambda_1+\lambda_2){\Gamma}\div u$  %solves the elliptic equation \eqref{eq-tiltasigma} 
can be   rewritten as:
\beq\label{eq-tiltasigma-1}
\left\{
\begin{array}{l}
      \Delta\tilde{\sigma}=\div \big(f+ \ep\pt\bp \big(\f{1}{R\beta} u\big)+\mu\ep\lambda_1\overline{\Gamma}\curl\curl u\big) \text{ in } \Omega,\\[4pt]
      \p_{\bn} \tilde{\sigma}=\big(f+ \ep\pt\bp \big(\f{1}{R\beta} u\big)\big)\cdot\bn \text{ on } \p\Omega,
\end{array}
   \right.
\eeq 
where $f$ is defined in \eqref{def-f}. 
Thanks to the elliptic estimate \eqref{highconormal}:
\beq\label{sec5:eq6}
\begin{aligned}
    \kpa^{\f{1}{2}}\|\na\tilde{\sigma}\|_{\hcoch^{j,l}}&\lesssim 
   \kpa^{\f{1}{2}} \|f+ \ep\pt\bp \big(\f{1}{R\beta} u\big)+\mu\ep\lambda_1\overline{\Gamma}\curl\curl u\|_{\hcoch^{j,l}}\\
   &\qquad \qquad\quad+\ep\mu\lambda_1\overline{\Gamma} |\curl\curl u\cdot\bn|_{\hcob^{m-\f{3}{2}}}.
\end{aligned}
\eeq
In view of \eqref{def-f}, it holds that:
\begin{align*}
&f+ \ep\pt\bp \big(\f{1}{R\beta} u\big)+\mu\ep\lambda_1\overline{\Gamma}\curl\curl u\\
&=-\ep\pt  \bq \big(\f{1}{R\beta} u\big)+\ep\pt (\f{1}{R\beta} )u - \f{\ep}{R\beta} u\cdot\na u
-\mu\ep\lambda_1 (\Gamma-\overline{\Gamma})\curl\curl u-\ep\mu(2\lambda_1+\lambda_2)(\div u)\na\Gamma,
\end{align*}
from which we can conclude that:
\begin{align*}
     \kpa^{\f{1}{2}}  \|f+ \ep\pt\bp \big(\f{1}{R\beta} u\big)+\mu\ep\lambda_1\overline{\Gamma}\curl\curl u\|_{\hcoch^{j,l}}
      \lesssim  \kpa^{\f{1}{2}}\| \bq \big(\f{1}{R\beta} u\big)\|_{\hcoch^{j+1,l}}+\ep \lat\cE_{m,t}.
\end{align*}
However, by the elliptic estimates \eqref{highconormal}, \eqref{secderelliptic-1}:
\begin{align*}
    \kpa^{\f{1}{2}} \| \bq \big(\f{1}{R\beta} u\big)\|_{\hcoch^{j+1,l}}&\lesssim \kpa^{\f{1}{2}}\| \bq \big(\f{1}{R\beta} u\big)\|_{L_t^2\cH^{j+1}} + \|\na \bq \big(\f{1}{R\beta} u\big)\|_{\hcoch^{j+1,l-1}}\\
    &\lesssim  \kpa^{\f{1}{2}}\|\f{1}{R\beta} u\|_{L_t^2\cH^{j+1}}+ \kpa^{\f{1}{2}}\|\div(\f{1}{R\beta} u)\|_{\hcoch^{j+1,l-1}}   \\
   & \lesssim \kpa^{\f{1}{2}}\|\f{1}{R\beta}\div u\|_{\hcoch^{j+1,l-1}}+T^{\f{1}{2}}
   \|\kpa^{\f{1}{2}} \big(\f{1}{R\beta} u, \na(\f{1}{R\beta})\cdot u \big) \|_{L_t^{\infty}H_{co}^{m-1}}.
\end{align*}
With the help of product estimates \eqref{roughproduct-2} and  \eqref{roughproduct1},  one can conclude that: 
\begin{align*}
& \kpa^{\f{1}{2}}\|\f{1}{R\beta}\div u\|_{\hcoch^{j+1,l-1}}\lesssim 
  \kpa^{\f{1}{2}}\|\div u\|_{\hcoch^{j+1,l-1}}+ T^{\f{1}{2}}(\|\div u\|_{\infco^{m-2}}+\|\theta\|_{\infco^{m-1}})\lat\\
  &\qquad\qquad\qquad\qquad\qquad\lesssim  \kpa^{\f{1}{2}}\|\div u\|_{\hcoch^{j+1,l-1}}+T^{\f{1}{2}}\lat\cE_{m,t},\\
   &\|\kpa^{\f{1}{2}} \big(\f{1}{R\beta} u, \na(\f{1}{R\beta})\cdot u \big)\|_{L_t^{\infty}H_{co}^{m-1}}\lesssim \|\kpa^{\f{1}{2}}(\theta, u,\na\theta)\|_{\infco^{m-1}}\lat\lesssim \lat\cE_{m,t}.
\end{align*}
Consequently, it is found that:
\beqs
\kpa^{\f{1}{2}}\|f+ \ep\pt\bp \big(\f{1}{R\beta} u\big)-\mu\ep\lambda_1\overline{\Gamma}\curl\curl u\|_{\hcoch^{j,l}}% \| \bq \big(\f{1}{R\beta} u\big)\|_{\hcoch^{j+1,l}}
\lesssim \kpa^{\f{1}{2}}\|\div u\|_{\hcoch^{j+1,l-1}}+(T^{\f{1}{2}}+\ep)\lat\cE_{m,t},
\eeqs
which, together with  \eqref{sec4:eq12}, \eqref{sec5:eq6}, yields: 
\begin{align}\label{sec5:eq8}
    \kpa^{\f{1}{2}}\|\na\tilde{\sigma}\|_{\hcoch^{j,l}}\lesssim \kpa^{\f{1}{2}} \|\div u\|_{\hcoch^{j+1,l-1}}+(T+\ep)^{\f{1}{2}}\lae.
\end{align}
  
     We can then obtain \eqref{sec5:eq9} from the estimates \eqref{sec5:eq5} and \eqref{sec5:eq8},
  and thus finish the proof. 
\end{proof}
\begin{rmk}
By \eqref{com-uniL2} and \eqref{EE-theta},  it holds  that by choosing $\tilde{\vartheta}_0$ (found in Proposition \ref{prop-theta}) smaller if necessary, %, by assuming $\f{m-3}{2}\geq 2$ and
%assuming $T+\ep\leq 2,$
\begin{align}\label{EE-theta-final}
% &\|\theta\|_{L_t^{\infty}H_{co,\sqrt{\kpa}}^{m}}^2+\|\nabla\theta\|_{\infcok^{m-1}}^2+ \sqrt{\kappa}\|\Delta\theta\|_{\infcok^{m-2}\cap \hco^{m-2}}^2+\kpa \|\na\Delta\theta\|_{L_t^{2}H_{co,\sqrt{\kappa}}^{m-2}}\notag\\
%\|\kpa\Delta\theta\|_{L_t^{\infty}H_{co}^{m-2}}^2+\|\kpa^{\f{3}{2}}\nabla\Delta\theta\|_{\hco^{m-2}}^2 
\cE_{m,T}^2(\theta)\lesssim \Lambda\big(\f{1}{c_0}, Y_m(0)\big)+%\big((T+\ep)^{\f{1}{2}}+(T+\ep)^{\f{m-3}{2}}\big)
(T+\ep)^{\tilde{\vartheta}_0}\Lambda\big(\f{1}{c_0},\cN_{m,t}\big).
\end{align}
\end{rmk}

\begin{rmk}
 One can find some $\vartheta_2>0$ such that:
 \begin{align}\label{divbetanatta}
\| \kpa\div(\beta\na\theta)\|_{\infco^{m-2}}\lesssim \Lambda\big(\f{1}{c_0}, Y_m(0)\big)+ (T+\ep)^{\vartheta_2}\lae.
 \end{align}
Such a fact can be shown in the following way.
%Indeed, %this estimate can be derived from the previous estimates on $\theta.$ 
First, by using the product estimate \eqref{roughproduct1} and the estimate \eqref{esofbeta} for $\na\beta,$
\begin{align*}
  &\kpa  \|\na\beta\cdot\na\theta\|_{\infco^{m-2}} \lesssim \|(\theta,\na\theta)\|_{\infco^{m-2}}\lab \il\theta\il_{[\f{m}{2}],\infty,t}+\kpa^{\f{1}{2}}\il\na\theta\il_{[\f{m}{2}]-1,\infty,t} \big).
 \end{align*}
 By the Sobolev embedding \eqref{sobebd}, 
 \begin{align}
%&\il\theta\il_{m-3,\infty,t}\lesssim \|\theta\|_{\infco^{m-1}}+\|\na\theta\|_{\infco^{m-2}},\label{thetainfty}\\ %\|\theta\|_{\infco^{[\f{m}{2}]+2}}+\|\na\theta\|_{\infco^{[\f{m}{2}]+1}}\\
\kpa^{\f{1}{2}}\il\na\theta\il_{m-4,\infty,t}\lesssim \label{nathetainfty-2} \|\na\theta\|_{\infco^{m-2}}+\kpa\| \na^2\theta\|_{\infco^{m-3}}.
 \end{align}
 As $m\geq 7,$ it holds that $[\f{m}{2}]\leq m-4,$ we thus conclude from the above estimate and \eqref{thetainfty} that:
 %The previous two estimates then lead to that:
 \beq\label{sec4:eq30}
 \kpa  \|\na\beta\cdot\na\theta\|_{\infco^{m-2}}\lesssim \lab  \|\theta\|_{\infco^{m-1}}+
 \|(\na\theta,\kpa\na^2\theta)\|_{\infco^{m-2}} \big).
 \eeq
Next, we control $\kpa\beta\Delta\theta$ as follows:
\begin{align*}
   \kpa \|\beta\Delta\theta\|_{\infco^{m-2}} &\lesssim \|\kpa\Delta\theta\|_{\infco^{m-2}}\il\beta\il_{m-3,\infty,t}+\|\kpa^{\f{1}{2}}\Delta\theta\|_{L_t^{\infty}L^2} \|\kpa^{\f{1}{2}}Z\beta\|_{m-3,\infty,t}\\
   &\lesssim( \|\kpa\Delta\theta\|_{\infco^{m-2}}+ \kpa^{\f{1}{2}}\|\Delta\theta\|_{L_t^{\infty}L^2})\Lambda\big(\f{1}{c_0}, \|\kpa^{\f{1}{2}}\theta\|_{m-2,\infty,t}+\il\theta\il_{m-3,\infty,t} \big).
\end{align*}
It follows again from the Sobolev embedding \eqref{sobebd} that:
\begin{align*}\|\kpa^{\f{1}{2}}\theta\|_{m-2,\infty,t}\lesssim \kpa^{\f{1}{2}}(\|\theta\|_{L_t^{\infty}\underline{H}_{co}^m}+\|\na\theta\|_{\infco^{m-1}}),\end{align*}
which, combined with \eqref{thetainfty}, yields:
\beq\label{sec4:eq31}
 \kpa \|\beta\Delta\theta\|_{\infco^{m-2}}\lesssim \lab \kpa^{\f{1}{2}}\big(\|\theta\|_{L_t^{\infty}\underline{H}_{co}^m}+ \|\na\theta\|_{\infco^{m-1}}+\|\Delta\theta\|_{L_t^{\infty}L^2}\big)+\|\kpa\Delta\theta\|_{\infco^{m-2}}\big).
\eeq
In view of estimates \eqref{sec4:eq30} and \eqref{sec4:eq31},
we have found that:
\begin{align*}
    \|\kpa\div(\beta\na\theta)\|_{\infco^{m-2}} \lesssim \lab \cE_{m,t}(\theta) \big),
\end{align*}
which, combined with \eqref{EE-theta-final}, allows us to 
 find some $\vartheta_2>0,$ such that \eqref{divbetanatta} holds.
\end{rmk}
The next proposition is devoted to the control of the second normal derivatives for $\sigma, \mathbb{Q}u:$
\begin{prop}
Under the assumption \eqref{preasption},  the following estimate holds: for any $\ep \in (0,1], (\mu,\kpa)\in A,$ any $0<t\leq T,$ 
\begin{align}  
&\kpa \|\na\div u\|_{\hco^{m-2}}^2\lesssim \lab Y_m(0)\big) +(T+\ep)^{\tilde{\vartheta}_0}\lae,\label{nadivuL2} \\
&\ep^2\mu\kappa\|\na^2\sigma\|_{\infco^{m-2}}^2+
 \kpa \|\na^2\sigma\|_{\hco^{m-2}}^2%+\kpa \|\na^2\div u\|_{\hco^{m-2}}^2 
 \lesssim 
 Y_m^2(0)+ (T+\ep)^{\f{1}{6}}\lae.\label{na2sigma}
\end{align}
\end{prop}
\begin{proof}
Let us first control $\kpa \|\na\div u\|_{\hco^{m-2}}^2.$ It follows from \eqref{rew-sigma} that:
\begin{align*}
    \kpa\|\na\div u\|_{\hco^{m-2}}^2&\lesssim  \kpa\|\na\big(\ep\pt+\ep u\cdot\na)\sigma, \kappa \Gamma\div(\beta\na\theta), \mathfrak{N} \big) \|_{\hco^{m-2}}^2 \\
 &\lesssim  \kpa \|\na\sigma\|_{\hco^{m-1}}^2+ \|\kpa^{\f{3}{2}}\na\Delta\theta\|_{\hco^{m-2}}^2+(T+\ep)\lae,
\end{align*}
which, combined with \eqref{EE-theta-2-2},  \eqref{com-uniL2}, leads to \eqref{nadivuL2}. 

We now switch to the proof of \eqref{na2sigma}. 
Thanks to the estimates \eqref{EE-2}, \eqref{com-uniL2} and the fact \eqref{Laplace-local}, it  suffices for us to 
control $\ep^2\mu\kappa\|\Delta\sigma\|_{\infco^{m-2}}^2+\kpa \|\Delta\sigma\|_{\hco^{m-2}}^2$ by the right hand side of \eqref{na2sigma}. Let us begin with the derivation of the equation satisfied by $\Delta\sigma.$ %By the definition of \eqref{defvr} and \eqref{}, 
Applying $\ep\mu\gamma(2\lambda_1+\lambda_2)\Gamma \nabla$ on the equation $\eqref{newsys}_1$ solved by $\vr=\f{\ep}{\gamma}\sigma-\f{C_v}{R}\f{\gamma-1}{\gamma}\theta$, we find that $\na\sigma$ is governed by: %the following damped transport equation:
\begin{align*}
\ep\mu(2\lambda_1+\lambda_2)\Gamma\bigg((\ep\pt+\ep u\cdot\na)\na\sigma+ \na\div u\bigg)=\mathfrak{H}_1%
\end{align*}
with 
\beq\label{deffrakh1}
\mathfrak{H}_1=\ep\mu(2\lambda_1+\lambda_2)\Gamma \bigg(\f{C_v(\gamma-1)}{R}(\pt+u\cdot\na)\na\theta -\gamma\na u\cdot\na\vr\bigg).
\eeq
By using  $\eqref{NCNS-S2}_2,$ we can
substitute the quantity $\ep\mu(2\lambda_1+\lambda_2)\Gamma \na\div u$ in the above equation by 
$$\f{1}{R\beta}(\ep\pt+\ep u\cdot\na)u-\ep\lambda_1\mu\Gamma \curl\curl u+\na\sigma$$
 to find that:
 \begin{align}\label{eq-nasigma}
     \ep^2\mu(2\lambda_1+\lambda_2)\Gamma(\pt+u\cdot\na)\na\sigma+\na\sigma=\mathfrak{H}_2,
 \end{align}
 where 
 \begin{align}\label{deffrakh12}
     \mathfrak{H}_2=-\f{1}{R\beta}(\ep\pt+\ep u\cdot\na)u-\ep\lambda_1\mu\Gamma \curl\curl u+  \mathfrak{H}_1. 
 \end{align}
 Taking the divergence of the equation \eqref{eq-nasigma}, one derives the equation satisfied by $\Delta\sigma:$
 \begin{align}\label{eqDeltasigma}
      \ep^2\mu(2\lambda_1+\lambda_2)\Gamma(\pt+u\cdot\na)\Delta\sigma+\Delta\sigma=\div \mathfrak{H}_2-\ep^2\mu(2\lambda_1+\lambda_2) \big(\na\Gamma\cdot\pt\na\sigma+\na(\Gamma u)\cdot\na\sigma\big).
 \end{align}
It then follows from the direct energy estimates on \eqref{eqDeltasigma} and the Young's inequality that
 \beqs
 \begin{aligned}
   &\ep^2\mu\kpa \|\Delta\sigma\|_{\infco^{m-2}}^2+\kpa \|\Delta\sigma\|_{\hco^{m-2}}^2\\
   &\lesssim Y_m^2(0)+ \ep^2\mu\kpa \il\div u\il_{0,\infty,t}\|\Delta\sigma\|_{\hco^{m-2}}^2+ \|\text{ R.H.S of } \eqref{eqDeltasigma}\|_{\hco^{m-2}}^2,
 \end{aligned}
 \eeqs
 which, combined with the estimates \eqref{sec9:eq2}, \eqref{sec9:eq3} in Lemma \ref{lem9.4}, yields that:
 \beqs
 \begin{aligned}
     &\ep^2\mu\kpa \|\Delta\sigma\|_{\infco^{m-2}}^2+\kpa \|\Delta\sigma\|_{\hco^{m-2}}^2\\
   &\lesssim Y_m^2(0)+(T+\ep)\lat\cE_{m,t}^2+ \kpa\|(\div u, \mu \Delta\theta )\|_{\hco^{m-1}}^2.
 \end{aligned}
 \eeqs
 We thus finish the proof of \eqref{na2sigma}
 by noticing \eqref{com-uniL2} and \eqref{EE-2}.
 \end{proof}
  We list in the next lemma some estimates used in the previous proposition for the right hand side of \eqref{eqDeltasigma} .
\begin{lem}\label{lem9.4}
The following estimates hold true:
\begin{align}
    \ep\mu \|\big(\na\Gamma\cdot\pt\na\sigma+\na(\Gamma u)\cdot\na\sigma\big)\|_{\hco^{m-2}\cap \infco^{m-3}}\lesssim \lat\cE_{m,t},\label{sec9:eq2}\\
     \kpa^{\f{1}{2}} \|\div \mathfrak{H}_2\|_{\hco^{m-2}}\lesssim \kpa^{\f{1}{2}}\|(\div u, \mu \Delta\theta )\|_{\hco^{m-1}}+(T^{\f{1}{2}}+\ep)\lat\cE_{m,t},\label{sec9:eq3}\\
      \|\div \mathfrak{H}_2\|_{\infco^{m-3}}\lesssim \|(\div u, \mu \Delta\theta )\|_{\infco^{m-2}}+\ep\lat\cE_{m,t}. \label{sec9:eq4}
\end{align}
\end{lem}
 These estimates stem from the very expression of each target term, the product estimates in the Proposition
\ref{prop-prdcom} as well as the definitions of $\cA_{m,t}, \cE_{m,t},$ we thus skip the details. 
 \begin{rmk}
 By the equation \eqref{eqDeltasigma}, the estimates \eqref{na2sigma}, \eqref{sec9:eq2}, \eqref{sec9:eq4} and the identity \eqref{Laplace-local}, we have also that, for any $\ep \in (0,1], (\mu,\kpa)\in A,$ any $0<t\leq T,$ 
 \begin{align}\label{na2sigmauni-0} 
     \|\na^2\sigma\|_{\infco^{m-3}}^2\lesssim Y_m^2(0)+ \|(\na\sigma,\div u, \mu\Delta\theta)\|_{\infco^{m-2}}^2+(T+\ep)^{\f{1}{6}}\lae.
 \end{align}
\end{rmk}
\section{Uniform estimates for the incompressible part}
In this section, we introduce the  incompressible part for the modified velocity-- $r_0 u,$ where 
\beq\label{defr0}
r_0=\f{1}{R\beta}\exp (\ep R\tilde{\sigma}/C_v\gamma)=\f{1}{R\beta(0)}\exp(-R\tilde{\theta}/{C_v})
\exp(\mu(2\lambda_1+\lambda_2)R\ep^2 \Gamma\div u /C_v\gamma),
%r_0=\f{1}{R\beta(0)}\exp(-R\tilde{\theta}/{C_v})\exp{\ep^2\mu(2\lambda_1+\lambda_2) R/C_v\gamma},\quad \tilde{\theta}=\f{C_v}{R}\theta-\f{\ep}{\gamma}\sigma.
\eeq
with
$$ \tilde{\sigma}=\sigma-\ep\mu(2\lambda_1+\lambda_2)\Gamma\div u, \quad \tilde{\theta}=\f{C_v}{R}\theta-\f{\ep}{\gamma}\sigma.$$
Denote $r_1= \ep^{-1}(\exp (\ep R\tilde{\sigma}/C_v\gamma)-1).$ 
The system $\eqref{NCNS-S2}_2$ can be rewritten as: 
\beq\label{sec7:eq1}
\begin{aligned}
&(\pt+u\cdot\na)(r_0 u)+\f{\na\tilde{\sigma}}{\ep}+\mu\lambda_1\Gamma\curl\curl u\\ 
&=u(\pt+u\cdot\na)r_0- r_1(\na\tilde{\sigma}+\ep\mu\lambda_1\Gamma\curl\curl u)-\mu(2\lambda_1+\lambda_2)(1+\ep r_1)\div u \na\Gamma.
\end{aligned}
\eeq
By writing further:
$$\mu \lambda_1\Gamma \curl\curl u=\mu\lambda_1(\Gamma-\overline{\Gamma})\curl\curl u-\mu\lambda_1\overline{\Gamma}r_0^{-1}[\curl\curl, r_0]u+\mu\lambda_1\overline{\Gamma}r_0^{-1}\curl\curl(r_0 u),$$
the system \eqref{sec7:eq1} can thus be reformulated as:
\beq\label{eqr0u}
(\pt+u\cdot\na)(r_0 u)+\f{\na\tilde{\sigma}}{\ep}
+\mu\lambda_1 \overline{\Gamma}r_0^{-1}\curl\curl(r_0 u)
=-r_1\na\tilde{\sigma}+G,
\eeq 
where 
\begin{align}\label{def-G}
G&=u(\pt+u\cdot\na)r_0+\mu\lambda_1\overline{\Gamma}r_0^{-1}[\curl\curl, r_0]u\notag\\
&\quad-\ep\mu \lambda_1 r_1\Gamma\curl\curl u-\mu\lambda_1(\Gamma-\overline{\Gamma})\curl\curl u-\mu(2\lambda_1+\lambda_2)(1+\ep r_1)\div u \na\Gamma.
\end{align}
Let us now derive the equations for the incompressible part $v=\bp (r_0 u)$ of the quantity $r_0 u.$
 Taking the Leray projection $\bp$ (defined in \eqref{proj-P}) on the equations \eqref{eqr0u} and
noticing that $r_1 \na \tsigma$ can be written as a gradient 
$$r_1\na\tilde{\sigma} =\ep^{-1}\na \bigg( {C_v\gamma}\big(\exp(\ep R\tsigma/C_v\gamma)-1\big)/{\ep R} -\tsigma\bigg)$$
and also the identity $\curl\curl(r_0u)=-\Delta v,$ 
we find that $v$ solves the equation: 
\begin{equation*}
    (\pt+u\cdot\na)v-\mu\lambda_1\overline{\Gamma}\bp (r_0^{-1}\Delta v)=-[\bp,u\cdot\na](r_0 u)+\bp G.
\end{equation*}
Let us define
%\beq\label{defnaq}
$\na q=\bq (G+\mu\lambda_1\overline{\Gamma}r_0^{-1}\Delta v),$ the above equation can thus be rewritten as:
\beq\label{eq-v}
(\pt+u\cdot\na)v-\mu\lambda_1\overline{\Gamma}r_0^{-1}\Delta v+\na q=-[\bp,u\cdot\na](r_0 u)+G.
\eeq
Moreover, by the boundary condition \eqref{bdryconditionofu}, the definition  
\eqref{defr0},
%and the definition of $\bq, \bp,$
we find that, %by denoting $\na\Psi=\bq u,$
\beq\label{bc-v-1}
v\cdot\bn=0 \text{ on } \p\Omega,
\eeq
\begin{align}\label{bc-v-2}
     \Pi(\p_{\bn}v)=&\Pi\p_{\bn}(r_0 u)-\Pi(\p_{\bn}\bq(r_0 u) )=\p_{\bn}r_0\Pi u
     + r_0\Pi \p_{\bn}u+\Pi ((D\bn)\bq (r_0 u))\notag\\
     &= \f{R\ep}{C_v\gamma}r_0\p_{\bn}\tsigma \Pi u+r_0\Pi\big(-2a u+(D\bn)u\big)+ \Pi \big((D\bn)\bq (r_0 u)\big) \text{ on } \p\Omega.
\end{align}
In the next several lemmas, we give some preliminary estimates for $r_0, r_1,$ $G, \na q$  and $[\bp,u\cdot\na](r_0 u),$ %under the assumption \eqref{preasption} and
when $m\geq 7,$ $\ep \in (0,1], (\mu,\kpa)\in A,$  which will be quite useful for the estimate of $v.$ 
Let us begin with the estimates for $r_0:$
\begin{lem}
Suppose that the assumption \eqref{preasption} holds for some $T>0,$ then
for any $0<t\leq T,$ we have the following estimates: 
\begin{align}
\|(\pt,\na) r_0)\|_{\infco^{m-2}}&\lesssim \lat\cE_{m,t}, \label{esr0-2}\\
   \mu^{\f{1}{2}}\|\na (r_0, r_0^{-1})\|_{\hco^{m-1}}
   &\lesssim (T^{\f{1}{2}}+\ep) \lat\cE_{m,t}. \label{esr0-3}
\end{align}
\end{lem}
\begin{proof}
Since $(\pt, \na) r_0=-r_0(\pt, \na) (\theta-\ep R C_v\tsigma/\gamma), $ we can apply the commutator estimate \eqref{roughcom1} and the estimate \eqref{esr0} to obtain that, by denoting 
 $\Theta=(\theta, \ep\sigma, \ep^2\mu\Gamma\div u),$
\begin{align*}
   \|(\pt,\na) r_0)\|_{\infco^{m-2}}
& \lesssim \|(\Id ,\pt, \na)\Theta\|_{\infco^{m-2}}\lab \il(\Id ,\pt, \na)\Theta\il_{[\f{m}{2}]-1,\infty,t}\big)\\
    &\lesssim \lat\cE_{m,t}.
\end{align*}
We thus finish the proof of \eqref{esr0-2}. 
To prove \eqref{esr0-3}, we use commutator estimate \eqref{roughcom} and the estimate \eqref{esr0} to get:
\begin{align*}
  \mu^{\f{1}{2}} \|\na(r_0, r_0^{-1})\|_{\hco^{m-1}}
& \lesssim  \mu^{\f{1}{2}} \|(\Id, \na)\Theta\|_{\hco^{m-1}}\lab \il \na \Theta\il_{[\f{m}{2}]-1,\infty,t}+\il \Theta\il_{[\f{m}{2}],\infty,t}\big)\\
    &\lesssim \big(T^{\f{1}{2}} \|\mu^{\f{1}{2}} (\theta,\na\theta)\|_{\infco^{m-1}}+\ep \|\mu^{\f{1}{2}}(\Id,\na)(\sigma, \ep\mu \Gamma\div u)\|_{\hco^{m-1}}\big)
    \lat\\
   &\lesssim (T^{\f{1}{2}}+\ep)\lat \cE_{m,t}.
\end{align*}

%Similarly, we can

\end{proof}

In the following lemma, we present some estimates for the source term $G:$
\begin{lem}
Let 
\beq\label{defG1}
G_1=\mu\lambda_1\overline{\Gamma}(R\beta r_0)^{-1}[\curl\curl, R\beta r_0]u, 
\eeq
Under the assumption \eqref{preasption},
the following estimates hold: for any $0<t\leq T,$ 
\begin{align}
\|G_1\|_{\hco^{m-2}}\lesssim \ep\mu^{\f{1}{2}}\lat\cE_{m,t},\label{es-G-1}\\
\quad \|G-G_1\|_{\uhco^{m-1}}\lesssim \lat\cE_{m,t}. \label{es-G-2} 
\end{align}
\end{lem}
\begin{rmk}
We can only control the derivatives of $(R\beta r_0)^{-1}[\curl\curl, R\beta r_0]u$ to the order $m-2,$ as it may involve two derivatives of $\tsigma.$ 
But due to the smoothing effect, $\kpa \na^2 \theta$ can be bounded in the regularity norm $L_t^2H_{co}^{m-1}.$
\end{rmk}
\begin{proof}
 The above estimates rely on the use the product and commutator estimates in Proposition  \ref{prop-prdcom}, the estimates for $\Gamma$ in Corollary \ref{cor-gb}
 as well as the facts: 
\beqs
\|\na (u,\sigma, \mu^{\f{1}{2}}\na\tsigma)\|_{\hco^{m-2}}\lesssim \lat\cE_{m,t},
\eeqs
\beqs
\|(u,\theta,\tsigma, \mu^{\f{1}{2}}(\na u, \na\theta, \na\tsigma), \ep\mu \na^2 u, \kpa \na^2 \theta)\|_{\hco^{m-1}}\lesssim \cE_{m,t}.
\eeqs
Let us first prove  \eqref{es-G-1}. 
By the definition \eqref{defr0} for $r_0,$ $$ R\beta r_0=\exp(\ep R\tsigma/{C_v\gamma})=1+\ep r_1. $$ 
Moreover, direct calculation yields:
\begin{align}\label{sec7:eq-1}
    [\curl\curl, f] u=\curl (\na f\times u)+\na f\times \curl u. % \lambda_1(\Delta f\cdot u+2(\na f\cdot\na )u)+(\lambda_1+\lambda_2)(\na f\div u+\na (u\cdot\na f))
\end{align}
We can use successively the product estimate \eqref{roughproduct1},
%commutator estimate \eqref{roughcom}
the estimate \eqref{esr1} %and also the estimate 
 for $r_1,$ the estimate \eqref{Es-tsigma1} for $\tsigma$
 to obtain that:
\begin{align*}
    \|\mu\lambda_1\overline{\Gamma}(R\beta r_0)^{-1}[\curl\curl, R\beta r_0]u\|_{\hco^{m-2}}&\lesssim \ep\mu^{\f{1}{2}} \|(u, \tsigma, \na (u,\tsigma,  \mu^{\f{1}{2}}\na\tsigma)\|_{\hco^{m-2}}\lat\\
    &\lesssim \ep\mu^{\f{1}{2}}\lat\cE_{m,t}.
\end{align*}
We now show the proof of \eqref{es-G-2}. First, as 
$$\mu\lambda_1\overline{\Gamma}\big(r_0^{-1}[\curl\curl, r_0]u-(R\beta r_0)^{-1}[\curl\curl, R\beta r_0]u\big)=-\mu\lambda_1\overline{\Gamma} (R\beta r_0)^{-1} [\curl\curl, R\beta](r_0 u),$$
we can control this term by using the identity \eqref{sec7:eq-1}, the product estimate \eqref{roughproduct1}, the estimate \eqref{esofbeta}, \eqref{esbeta-sec} for $\beta$
the estimate \eqref{esr0} for $r_0$ that: 
\begin{align}\label{sec7:eq15}
   & \mu%\lambda_1\overline{\Gamma}
   \big\| (R\beta r_0)^{-1} [\curl\curl, R\beta]( r_0 u)\big\|_{\hco^{m-1}}\notag\\
    &\lesssim \|( u, \theta, \ep\tsigma,  \mu^{\f{1}{2}}\na (u, \theta, \ep\tsigma),\mu\na^2\theta)\|_{\hco^{m-1}}\lat\\
    &\lesssim \lat\cE_{m,t}. \notag
\end{align}
To finish the proof of \eqref{es-G-2}, it still remains for us to control several terms including 
$$u(\pt+u\cdot\na)r_0, \quad -(\ep\mu\lambda_1 r_1\Gamma\curl\curl u+\mu\lambda_1(\Gamma-\overline{\Gamma}) \curl\curl u),\quad -\mu(2\lambda_1+\lambda_2)(1+\ep r_1)\div u \na\Gamma. $$ %-\mu(2\lambda_1+\lambda_2)\overline{\Gamma}\na(r_0^{-1})\div(r_0 u).$$ 
Let us begin with the estimate of the first term.  %and skip the other two, which can be controlled in a similar way. 
It follows from the definition \eqref{defr0} for $r_0$ and the equation \eqref{sec4:eq0} for $\tilde{\theta}$ that: 
\beq\label{timeder-r0}
\begin{aligned}
(\pt+u\cdot\na)r_0&=-\f{R}{C_v}r_0 (\pt+u\cdot\na)(\tilde{\theta}-\mu(2\lambda_1+\lambda_2)\ep^2 \Gamma\div u/\gamma)\\
&=-\f{R r_0}{C_v\gamma}\big( \kpa\Gamma \div(\beta\na\theta)+\ep\mathfrak{N} \big)
-\f{R}{C_v\gamma}(2\lambda_1+\lambda_2) r_0 (\pt+u\cdot\na)(\mu\ep^2 \Gamma\div u).
\end{aligned}
\eeq
We can thus derive from the product estimate \eqref{roughproduct1} and the estimate \eqref{esr0} for $r_0$
that: 
\begin{align}
   \| u(\pt+u\cdot\na)r_0\|_{\uhco^{m-1}}&\lesssim \|(u,\theta, \ep\sigma, \kpa^{\f{1}{2}}\na\theta, \mu\ep(\ep\pt,\nabla)\div u, \mu\ep\na u, \kpa \na^2\theta)\|_{\uhco^{m-1}}\lat\notag\\
   &\lesssim \lat\cE_{m,t}.
\end{align}
Next, using the product estimate \eqref{roughproduct1} and the estimate \eqref{esr1} for $r_1,$
the last two terms can be estimated in a similar way: 
\begin{align}
 &\| \ep\mu\lambda_1 r_1\Gamma\curl\curl u+\mu\lambda_1(\Gamma-\overline{\Gamma}) \curl\curl u\|_{\hco^{m-1}}\notag\\
 &\lesssim \|(\sigma, \tsigma, \ep\mu \curl\curl u)\|_{\hco^{m-1}}\lat \lesssim \lat\cE_{m,t},
\end{align}
\begin{align}\label{sec7:eq16}
    \|\mu(1+\ep r_1)\div u \na\Gamma\|_{\hco^{m-1}}&\lesssim \ep\|\mu^{\f{1}{2}}
    (\sigma, \div u, \na\sigma)%(\text{Id} , \na)(u, \theta, \ep\sigma)
    \|_{\hco^{m-1}}\lat\notag\\
    &\lesssim \ep\lat\cE_{m,t}.
\end{align}
We thus finish the proof of \eqref{es-G-2} by collecting \eqref{sec7:eq15}-\eqref{sec7:eq16}.
\end{proof}
In the next lemma, we give the estimate for $\na q:$
\begin{lem}\label{prop-naq}
 There exist $\na q_1, \na q_2\in L_t^2L^2,$ such that 
 $\na q= \na q_1+\na q_2$ and 
 \begin{align}\label{naq12}
    \|\na q_1\|_{\hco^{m-2}}\lesssim \ep \mu^{\f{1}{2}}\lat\cE_{m,t}, \quad \|\na q_2\|_{\uhco^{m-1}}\lesssim \lat\cE_{m,t}.
 \end{align}
\end{lem}
\begin{proof}
By the definition: $\na q=\bq  (G+\mu\lambda_1\overline{\Gamma}r_0^{-1}\Delta v )$ and the identity 
\beqs
r_0^{-1}\Delta v=-r_0^{-1}\curl\curl(r_0 u)=-\curl(r_0^{-1}\curl(r_0 u))+\na(r_0^{-1})\times \curl(r_0 u),
\eeqs
we find that $\na q$ solves the following elliptic equation (assume might as well that $\Omega$ is bounded)
\beq
\left\{
\begin{array}{l}
    \displaystyle \Delta q=\div \big(G +\mu\lambda_1\overline{\Gamma}\na(r_0^{-1})\times \curl(r_0 u)\big) \qquad\qquad\qquad\qquad\qquad\qquad\text{ in } \Omega, \\[4pt]
     \displaystyle \p_{\bn} q=\big(G +\mu\lambda_1\overline{\Gamma}\na(r_0^{-1})\times \curl(r_0 u)\big)\cdot\bn - \mu\lambda_1\overline{\Gamma}\curl (r_0^{-1}\curl(r_0 u))\cdot\bn \text{ on } \p \Omega,\\[4pt]
     \iomega q \,\d x =0 .
\end{array}
\right.
\eeq
Note that when $\Omega$ is an exterior domain, the zero mean condition should be replaced as $q(\cdot,x)\rightarrow 0 $ as $|x|\rightarrow+\infty.$
 The same changes should be made in the following. 
We define $\na q_1, \na q_2$ as the solutions to the following elliptic equations:
\beqs
\left\{
\begin{array}{l}
    \displaystyle \Delta q=\div G_1 \qquad\quad\text{ in } \Omega, \\[4pt]
     \displaystyle \p_{\bn} q_2=G_1\cdot\bn-g \quad \text{ on } \p \Omega, \\[4pt]
     \iomega q_1 \,\d x =0 ,
\end{array}
\right. \quad
\left\{
\begin{array}{l}
    \displaystyle \Delta q_2=\div G_2 \qquad\qquad\qquad\qquad\qquad\qquad\text{ in } \Omega, \\[4pt]
     \displaystyle \p_{\bn} q_2=G_2\cdot\bn - \mu\lambda_1\overline{\Gamma}\curl (r_0^{-1}\curl(r_0 u))\cdot\bn +g\text{ on } \p \Omega,\\[4pt]
     \iomega q_2 \,\d x =0 .
\end{array}
\right.
\eeqs
where $G_1$ is defined in \eqref{defG1} and
\beqs
G_2=G-G_1 +\mu\lambda_1\overline{\Gamma}\na(r_0^{-1})\times \curl(r_0 u),
\eeqs
\beq
   g=\mu\lambda_1\overline{\Gamma} \big(  (\Pi\partial_{y^1}(r_0^{-1}(\na r_0\times u)\times \bn))^1+(\Pi\partial_{y^2}(r_0^{-1}(\na r_0\times u)\times \bn))^2\big).
\eeq
It follows from the elliptic estimate \eqref{highconormal}, the estimates \eqref{es-G-1} for $G_1,$ \eqref{es-G-2} for $G-G_1,$ as well as estimates \eqref{es-g}, \eqref{r0curlcurlu}  for the boundary term that:
\begin{align*}
    \|\na q_1\|_{\hco^{m-2}}\lesssim \|G_1\|_{\hco^{m-2}}+ \sum_{j=0}^{m-2}|(\ep\pt)^j g|_{L_t^2{H}^{m-j-\f{5}{2}}}
    \lesssim  \ep \mu^{\f{1}{2}}\lat\cE_{m,t},
\end{align*}
\begin{align*}
    \|\na q_2\|_{\uhco^{m-1}}&\lesssim \big\|(G-G_1, \mu\na(r_0^{-1})\times \curl(r_0 u)\big\|_{\uhco^{m-1}}+ \mu |\curl (r_0^{-1}\curl(r_0 u))\cdot\bn-g|_{L_t^2\tilde{H}^{m-\f{3}{2}}}\\
    &\lesssim \lat\cE_{m,t}.
\end{align*}
\end{proof}
\begin{lem}
Similar to \eqref{sec4:eq12}, it can be verified that: 
\begin{equation}\label{es-g}
   \sum_{j=0}^{m-2}|(\ep\pt)^j g|_{L_t^2\tilde{H}^{m-j-\f{5}{2}}}
    \lesssim  \ep \mu^{\f{1}{2}}\lat\cE_{m,t},
\end{equation}
 \begin{equation}\label{r0curlcurlu}
     \mu |\curl (r_0^{-1}\curl(r_0 u))\cdot\bn-g|_{L_t^2\tilde{H}^{m-\f{3}{2}}}\lesssim\mu^{\f{1}{2}} \lat\cE_{m,t}. 
 \end{equation}
\end{lem}
\begin{proof}
Taking the benefits of the boundary conditions  $u\cdot \bn|_{\p\Omega}=\p_{\bn}\theta|_{\p\Omega}=0$, we find that:
\begin{align}
r_0^{-1}(\na r_0\times u)\times \bn=r_0^{-1}u\p_{\bn}r_0 = \f{R\ep}{C_v\gamma}u \p_{\bn}\tilde{\sigma} \, \text{ on } \p\Omega.
\end{align}
Therefore, by the product estimate \eqref{product-bd}, the trace inequality 
\eqref{traceL2} and the estimate \eqref{Es-tsigma1}:
\begin{align*}
    \sum_{j=0}^{m-2}|(\ep\pt)^j g|_{L_t^2{H}^{m-j-\f{5}{2}}(\p\Omega)}
   & \lesssim \ep\mu |(u,\p_{\bn}\tsigma)|_{L_t^2\tilde{H}^{m-\f{3}{2}}}\lat\\
   &\lesssim \ep\mu^{\f{1}{2}} \|(u,\na (u, \tsigma, \mu^{\f{1}{2}} \na \tsigma))\|_{\hco^{m-2}}\lat \\
   & \lesssim  \ep\mu^{\f{1}{2}} \lat\cE_{m,t}.
\end{align*}
We now proceed to prove \eqref{r0curlcurlu}. Let us write 
\beqs
\curl (r_0^{-1}\curl(r_0 u))\cdot \bn-g=\curl\curl u\cdot\bn+\big(\curl(r_0^{-1}\na r_0\times u)\cdot\bn-g\big).
\eeqs
The first term $\curl\curl u\cdot\bn$ has been controlled in \eqref{sec4:eq12}, 
we thus focus on the second one. By using the identity 
$$(\curl  a)\cdot b=\div (a\times b)+ a\cdot \curl b, $$ 
one finds that:
\beqs 
\curl (r_0^{-1}\na r_0\times u)\cdot\bn=\div(r_0^{-1}(\na r_0\times u)\times \bn)+ r_0^{-1}(\na r_0\times u)\cdot\curl \bn.
\eeqs
 It follows from \eqref{normalofnormalder} that: 
\begin{align*}
\div(r_0^{-1}(\na r_0\times u)\times \bn)&=(r_0^{-1}(\na r_0\times u)\times\p_{\bn}\bn) \cdot\bn+g.
\end{align*}
The above two equalities, together with the trace inequality \eqref{traceL2} and estimates  \eqref{esr0}, \eqref{esr0-sec}  enable us to find that: 
\begin{align*}
   &\mu |\curl(r_0^{-1}\na r_0\times u)\cdot\bn-g|_{L_t^2\tilde{H}^{m-\f{3}{2}}}\lesssim \mu |r_0^{-1}\na r_0\times u|_{L_t^2\tilde{H}^{m-\f{3}{2}}}\\
   &\lesssim\mu^{\f{1}{2}} \|(\text{Id},\na) \big(u, (\Id, \mu^{\f{1}{2}}\na)(\theta,\ep\tsigma) \big)\|_{\hco^{m-2}}\lat\\
   &\lesssim \mu^{\f{1}{2}}\lat\cE_{m,t}.
\end{align*}
\end{proof}
In the next lemma, we present the estimate of  $[\bp, u\cdot\na](r_0 u),$ which is needed later for the estimate of $v.$ 
\begin{lem}
It holds that, 
for any $0<t\leq T,$ 
\begin{align}\label{es-bpcom}
   \big\|[\bp, u\cdot\na](r_0 u)\big\|_{\uhco^{m-1}}\lesssim T^{\f{1}{2}} \lat\cE_{m,t}.
\end{align}
\end{lem}
\begin{proof}
First, it follows from the product estimate \eqref{roughproduct1}, the elliptic estimate \eqref{highconormal} that:
\begin{align}\label{bpcomlow}
\big\|[\bp, u\cdot\na](r_0 u)\big\|_{\hco^{m-2}}&\lesssim \| u\cdot \na (r_0 u, v)\|_{\hco^{m-2}}\notag\\
&\lesssim \|(\text{Id},\na)(u,\theta, \ep\tsigma)\|_{\hco^{m-2}}\lat\lesssim  T^{\f{1}{2}} \lat\cE_{m,t}.
\end{align}
Next, one gets from the fact $\div \bp=0$ that:
\begin{align}\label{divbpcom}
 \big\|\div \big([\bp, u\cdot\na](r_0 u)\big)\big\|_{\hco^{m-2}}=\|\na u\cdot\na v
 \|_{\hco^{m-2}}\lesssim T^{\f{1}{2}} \lat\cE_{m,t}.
\end{align}
By writing $[\bp, u\cdot\na](r_0 u)=u\cdot\na \bq(r_0 u)-\bq (u\cdot\na(r_0 u)) $ and
using the fact $\curl \na=\curl \bq =0,$ we have that:
\begin{align}\label{curlbpcom}
 \big\|\curl \big([\bp, u\cdot\na](r_0 u)\big)\big\|_{\hco^{m-2}}=\|\na u\times \na\bq (r_0 u)\|_{\hco^{m-2}}\lesssim T^{\f{1}{2}} \lat\cE_{m,t}.
\end{align}
Moreover, it can be derived from
$$(u\cdot\na)|_{\p\Omega}=(u_{y^1}\p_{y^1}+u_{y^2}\p_{y^2})|_{\p\Omega},\quad (v\cdot\bn)|_{\p\Omega}=0, $$
that
$$[\bp, u\cdot\na](r_0 u)\cdot\bn|_{\p\Omega}=- (u\cdot\na)v\cdot\bn|_{\p\Omega}=(u\cdot\na )\bn \cdot v|_{\p\Omega},$$ which gives rise to
\beqs
|[\bp, u\cdot\na](r_0 u)\cdot\bn|_{L_t^2\tilde{H}^{m-\f{3}{2}}}\lesssim \|(u,\na u)\|_{\hco^{m-2}}\lat\lesssim T^{\f{1}{2}}\lat \cE_{m,t}.
\eeqs
This, together with \eqref{divbpcom} and \eqref{curlbpcom}, yields that:
\beq\label{nabpcom}
\big\|\na  \big([\bp, u\cdot\na](r_0 u)\big)\big\|_{\hco^{m-2}}\lesssim T^{\f{1}{2}} \lat\cE_{m,t}.
\eeq
In light of \eqref{bpcomlow} and \eqref{nabpcom}, 
we finish the proof of \eqref{es-bpcom}.
\end{proof}
The above lemmas in hand, we are now able to prove the uniform estimates for $v=\bp(r_0 u):$
\begin{prop}\label{lem-v-l2}
Assume that assumption \eqref{preasption} holds and 
$m\geq 7,$ then for any $\ep \in (0,1], (\mu,\kpa)\in A,$ for any $0<t\leq T,$  we have that 
\begin{align}\label{EE-v}
  \|v\|_{\infco^{m-1}}^2+\mu\|\na v\|_{\hco^{m-1}}^2\lesssim 
  Y_m^2(0)+(T+\ep)^{\f{1}{2}} \lae. %\lat\cE_{m,t}^2.
\end{align}
\end{prop}
\begin{proof}
We first show that: 
\beq\label{timederv}
  \|v\|_{\infcoch^{m-1,0}}^2+\mu\|\na v\|_{\hcoch^{m-1,0}}^2\lesssim 
  Y_m^2(0)+(T+\ep)^{\f{1}{2}}\lae.
\eeq
Indeed, for any $j\leq m-1,$ it follows from \eqref{esr0-2}, \eqref{esr0-3} that
\begin{align*}
 & \big\|[(\ep\pt)^j, r_0] u\big\|_{L_t^{\infty}L^2}\lesssim \ep \|(\pt r_0, u)\|_{\infcoch^{m-2,0}}\lat\lesssim \ep \lat \cE_{m,t},\\
& \big\|[(\ep\pt)^j, r_0] \na u\big\|_{L_t^{2}L^2}\lesssim  \ep \|(\pt r_0, \na u)\|_{\hcoch^{m-2,0}}\lat\lesssim 
 %T^{\f{1}{2}}
  \ep\lat\cE_{m,t}, \\
 &\mu^{\f{1}{2}} \|\na r_0 \otimes u\|_{\hcoch^{m-1,0}}\lesssim \mu^{\f{1}{2}}  \|(\na r_0, u)\|_{\hcoch^{m-1,0}}\lat
 \lesssim (T^{\f{1}{2}}+\ep)\lat\cE_{m,t}.
\end{align*}
Consequently, we obtain the following estimate
\begin{align*}
  \|r_0 u\|_{\infcoch^{m-1,0}}^2+ \mu\|\na (r_0 u)\|_{\hcoch^{m-1,0}}^2\lesssim \| u \|_{\infcoch^{m-1,0}}^2+ \mu\|\na u\|_{\hcoch^{m-1,0}}^2+ (T+\ep^2)\lat\cE_{m,t}^2,
\end{align*}
which, combined with estimate \eqref{EE-3},
elliptic estimates \eqref{secderelliptic}, \eqref{highconormal} and also the definition  $v=r_0 u-\bq (r_0 u),$ yields \eqref{timederv}.

 We now focus on the control of $\|v\|_{\uinfco^{m-1}}$ and $\mu^{\f{1}{2}}\|\na v\|_{\uhco^{m-1}}.$  Let $Z^I$ be a vector field with $|I|\leq m-1, I_0\leq m-2.$ Applying $Z^I$ on the equation 
\eqref{eq-v}, we find that: 
\begin{align}
  (\pt+u\cdot\na)Z^I v + Z^I \na q-\mu\lambda_1 \overline{\Gamma}\div Z^I(r_0^{-1}\na v)=\mu\lambda_1 \overline{\Gamma}[Z^I, \div](r_0^{-1}\na v)+H^{I}
\end{align}
where 
$$H^{I}=-[Z^I, u\cdot\na]v-\mu\lambda_1\overline{\Gamma}Z^I(\na r_0^{-1}\cdot\na v)+Z^I(G-[\bp,u\cdot\na](r_0 u)).$$

Taking the vector product of the resulting equation with $Z^I v,$ we find that:
\begin{align}\label{sec7:eq0}
&\f{1}{2} \iomega |Z^I v|^2(t)\,\d x+\mu\lambda_1 \overline{\Gamma}\izto r_0^{-1}|Z^I \na v|^2 \, \d x\d s\notag\\
& =\f{1}{2}\iomega |Z^I v|^2(0)\,\d x+\f{1}{2}\izto \div u |Z^I v|^2 \,\d x\d s\notag\\
&\quad-\mu\lambda_1 \overline{\Gamma}\izto Z^I( r_0^{-1}\na v)\cdot [\na,Z^I]v+[Z^I, r_0^{-1}]\na v\cdot Z^I\na v \,\d x \d s\\
&\quad+\izto  Z^I \na q \cdot Z^I v\,\d x\d s+\mu\lambda_1 \overline{\Gamma}\izt\int_{\p\Omega} \bn \cdot Z^I(r_0^{-1}\na v)\cdot Z^I v \, \d S_y \d s\notag\\ 
&\quad+ \mu\lambda_1\overline{\Gamma} \izto [ Z^I, \div](r_0^{-1}\na v)\cdot Z^I v\,\d x\d s+\izto Z^I v\cdot H^I \, \d x \d s\notag\\
& =\colon \cH_0+\cdots \cH_6. \notag
\end{align}
We now control $\cH_0-\cH_6.$ Firstly, let us bound $\cH_0-\cH_1$ as:
\begin{align}\label{ch01}
    \cH_0+\cH_1&\lesssim Y_m^2(0)+T\|Z^I v\|_{L_t^{\infty}L^2}^2\il\div u\il_{0,\infty,t}\lesssim Y_m^2(0)+T^{\f{1}{2}}\lat\cE_{m,t}^2.
\end{align}
Next, by using the identity  \eqref{comu} and the Cauchy-Schwarz inequality, we control $\cH_2$ as follows:
\beq\label{ch2}
\begin{aligned}
    \cH_2&= \mu\lambda_1 \overline{\Gamma}\izto Z^I( r_0^{-1}\na v)\cdot [\na,Z^I]v+[Z^I, r_0^{-1}]\na v\cdot Z^I\na v \,\d x \d s\notag\\
&\lesssim  \mu \|r_0^{-1}\na v\|_{\hco^{m-1}}\|\na v\|_{\hco^{m-2}}+\mu\|\na v\|_{\hco^{m-1}} \|[Z^I, r_0^{-1}]\na v\|_{L_t^2L^2}\notag\\
&\lesssim \|\mu^{\f{1}{2}}\na v\|_{\hco^{m-1}} \|(\mu^{\f{1}{2}}\na v, Z r_0)\|_{\hco^{m-2}}\lat 
%T^{\f{1}{2}} \|\mu \na(r_0 u)\|_{\hco^{m-1}}\|\na(r_0 u)\|_{\infco^{m-2}}
\lesssim  T^{\f{1}{2}}\lae.
\end{aligned}
\eeq
Moreover, for the term $\cH_3=\izto Z^I \na q \cdot Z^I v \,\d x\d s,$ we apply Proposition \ref{prop-naq} to write: 
\beqs
\cH_3 = \izto Z^I\cdot  \na q_1 Z^I v \,\d x\d s+ \izto Z^I \na q_2\cdot  Z^I v \,\d x\d s.
\eeqs
By assumption, $Z^I$ contains at least one spatial tangential derivative, 
we thus estimate the first term by integration by parts in space and using the  estimate for $\na q_1$ in \eqref{naq12}:
\beqs
\izto Z^I \na q_1\cdot  Z^I v \,\d x\d s\lesssim \mu^{-\f{1}{2}}\|\na q_1\|_{\hco^{m-2}}
\|\mu^{\f{1}{2}} v\|_{\uhco^{m}}\lesssim \ep \lat\cE_{m,t}^2.
\eeqs
Thanks to the estimate for $\na q_2$ in \eqref{naq12}, we can control the second one directly by the Cauchy-Schwarz inequality: 
\begin{align*}
    \izto Z^I \na q_1 \cdot Z^I v \,\d x\d s\lesssim T^{\f{1}{2}} \|v\|_{\infco^{m-1}}\|\na q_2\|_{\uhco^{m-1}}\lesssim T^{\f{1}{2}} \lat\cE_{m,t}^2.
\end{align*}
The above two estimates then yield that:
\beq
\cH_3\lesssim (T^{\f{1}{2}}+\ep) \lat\cE_{m,t}^2.
\eeq
We now deal with the boundary term $\cH_4=\mu\lambda_1 \overline{\Gamma}\izt\int_{\p\Omega} \bn \cdot Z^I(r_0^{-1}\na v)\cdot Z^I v \, \d S_y \d s,$ which can be split further into three terms $\cH_4=\cH_{41}+\cH_{42}+\cH_{43}$ with 
\begin{align*}
    \cH_{41}&= \mu\lambda_1 \overline{\Gamma}\izt\int_{\p\Omega}Z^I(r_0^{-1}\p_{\bn} v\cdot\bn)(Z^I v\cdot\bn)\d S_y \d s,\\
     \cH_{42}&=\mu\lambda_1 \overline{\Gamma}\izt\int_{\p\Omega}Z^I(\Pi(r_0^{-1}\p_{\bn} v))\cdot \Pi Z^I v\, \d S_y \d s,\\
     \cH_{43}&=-\mu\lambda_1 \overline{\Gamma} \izt\int_{\p\Omega}\big([Z^I,\bn](r_0^{-1}\p_{\bn} v)\cdot (Z^I v\cdot \bn)+ 
    [Z^I,\Pi](r_0^{-1}\p_{\bn}v)\cdot \Pi Z^I v,\\
   & \qquad\qquad \qquad\qquad \qquad\qquad\qquad\qquad\qquad+[Z^I, \bn](r_0^{-1}\na v)\cdot Z^I v\big) \, \d S_y \d s.
\end{align*}
Let us begin with the treatment of $\cH_{41}.$ One may assume that $Z^I$ contains at least one spatial tangential derivative, that is $Z^I=\p_{y^j}Z^{\tilde{I}}, j=1,2.$ 
By the identity \eqref{normalofnormalder}, 
\beq\label{normalofnormal-v}
\partial_{\bn} v\cdot \bn=-(\Pi\partial_{y^1}v)^1-(\Pi\partial_{y^2}v)^2.
\eeq
Moreover, by the boundary condition \eqref{bc-v-1},  $$(Z^I v\cdot\bn)|_{\p\Omega}=-([Z^I,\bn]v)|_{\p\Omega},$$
we thus can use the duality between $H^{-1}(\p\Omega)$ and $H^{1}(\p\Omega)$, the trace inequality \eqref{traceL2} to get that:
%One thus can use the duality between $H^{-1}(\p\Omega)$ and $H^{1}(\p\Omega)$ to get that:%We thus bound $\cH_{41}$ by using the fact $Z^I v \cdot\bn=$
\begin{align}\label{ch31}
    \cH_{41}&\lesssim \mu |r_0^{-1}\p_{\bn}v\cdot\bn|_{L_t^2\tilde{H}^{m-2}} \big|[Z^I ,\bn] v\big|_{L_t^2H^1}\notag\\
    &\lesssim \mu |v|_{L_t^2\tilde{H}^{m-1}}(|v|_{L_t^2\tilde{H}^{m-1}}+|(\theta,\ep\tsigma)|_{L_t^2\tilde{H}^{m-2}})\lat \notag\\
    &\lesssim  \mu^{\f{1}{2}}\|\na v\|_{\hco^{m-1}}\|v\|_{\hco^{m-1}}+\|v\|_{\hco^{m-1}}^2+ \|(\Id, \na)(\theta,\ep\tsigma)\|_{\hco^{m-2}}^2\lat\\
    &\lesssim (T^{\f{1}{2}}+\ep)\lat\cE_{m,t}^2. \notag
\end{align}
Note that in the derivation of the last inequality,  we have used the facts:
\begin{align}\label{sec7:eq5}
&\|v\|_{\hco^{m-1}}\lesssim T^{\f{1}{2}}\|v\|_{\infco^{m-1}}
\lesssim T^{\f{1}{2}}\|u\|_{\infco^{m-1}}\lesssim T^{\f{1}{2}}\cE_{m,t},\\ &\|(\theta,\na\theta)\|_{\infco^{m-2}}+\|(\Id,\na)(\sigma,\mu\div u)\|_{\hco^{m-2}}\lesssim \cE_{m,t}.\notag%\label{sec7:eq5.5}
\end{align}
Next, for $\cH_{42},$ we use the boundary condition \eqref{bc-v-2} to obtain that:
\begin{align*}
\cH_{42}&\lesssim\mu\ep |\p_{\bn}\tsigma \Pi u|_{L_t^2\tilde{H}^{m-\f{3}{2}}}|v|_{L_t^2\tilde{H}^{m-\f{1}{2}}}+ \mu\big|r_0^{-1}\Pi(\p_{\bn} v)-\f{R\ep}{C_v\gamma}\p_{\bn}\tsigma\Pi u\big|_{L_t^2\tilde{H}^{m-1}} |v|_{L_t^2\tilde{H}^{m-1}}\\
&\lesssim \big(\mu\ep |(\p_{\bn}\tsigma, \Pi u)|_{L_t^2\tilde{H}^{m-\f{3}{2}}}|v|_{L_t^2\tilde{H}^{m-\f{1}{2}}}+\mu |(u, v, r_0^{-1})|_{L_t^2\tilde{H}^{m-1}} |v|_{L_t^2\tilde{H}^{m-1}}\big)\lat.
\end{align*}
Applying the trace inequality \eqref{traceL2} and the property \eqref{sec7:eq5} as well as the estimate \eqref{esr0} for $r_0, r_0^{-1},$
%\eqref{sec7:eq5.5}
one finds that: 
\begin{align}\label{ch32}
    &\cH_{42}\lesssim\ep %|\mu^{\f{1}{2}}\p_{\bn}\tsigma|_{L_t^2\tilde{H}^{m-\f{3}{2}}}
    \|\mu^{\f{1}{2}}(u,\na u)\|_{\hco^{m-1}}\|\mu^{\f{1}{2}}(u,\na u,\na\tsigma, \na^2\tsigma)\|_{\hco^{m-2}}\lat\notag\\
    &\quad+\big(\|\mu^{\f{1}{2}}\na u\|_{\hco^{m-1}}\|u\|_{\hco^{m-1}}+ \|u\|_{\hco^{m-1}}^2+\|\mu^{\f{1}{2}}(\Id,\na)(\theta, \ep\tsigma)\|_{\hco^{m-1}}^2\big)\lat \\
    &\quad\lesssim (T^{\f{1}{2}}+\ep)\lae.\notag
\end{align}
Finally, the term $\cH_{43}$ can be controlled roughly by 
\beq\label{sec7:eq2}
\cH_{43}\lesssim  \mu |r_0^{-1}\na v|_{L_t^2\tilde{H}^{m-2}}|v|_{L_t^2\tilde{H}^{m-1}}\lesssim \mu |(\theta,\ep\tsigma, \na v)|_{L_t^2\tilde{H}^{m-2}}|v|_{L_t^2\tilde{H}^{m-1}}\lat.
\eeq 
%{\color{red}give a lemma}
By the identity \eqref{normalofnormal-v}, the boundary condition \eqref{bc-v-2}, we find that:
\begin{align*}
   \mu^{\f{1}{2}} |\p_{\bn} v|_{L_t^2\tilde{H}^{m-2}}&\lesssim  \mu^{\f{1}{2}}|v|_{L_t^2\tilde{H}^{m-1}}+ \mu^{\f{1}{2}} |\Pi (\p_{\bn}v)|_{L_t^2\tilde{H}^{m-2}}\\
    &\lesssim   \mu^{\f{1}{2}}|v|_{L_t^2\tilde{H}^{m-1}}+ \mu^{\f{1}{2}}|(u,v, \theta, \ep\tsigma, \ep\p_{\bn}\tsigma)|_{L_t^2\tilde{H}^{m-2}}\lat
\end{align*}
The trace inequality \eqref{traceL2} and the estimates \eqref{Es-tsigma3}
thus yield that
\begin{align}\label{sec7:eq8}
     \mu^{\f{1}{2}} |\na v|_{L_t^2\tilde{H}^{m-2}} &\lesssim  \mu^{\f{1}{2}}|v|_{L_t^2\tilde{H}^{m-1}} %\big(\mu^{\f{1}{2}}\|(u,\na u)\|_{\hco^{m-1}}
    + \mu^{\f{1}{2}} 
    \|(\theta,\na\theta,\na(\ep\tsigma,\ep\na\tsigma))\|_{\hco^{m-2}}\big)\lat\notag\\
    &\lesssim \mu^{\f{1}{2}}|v|_{L_t^2\tilde{H}^{m-1}}+ (T^{\f{1}{2}}+\ep) \lae.
\end{align}
Moreover, by the fact \eqref{sec7:eq5} and the trace inequality \eqref{traceL2},
we can get more precisely that:
\begin{align}\label{sec7:eq9}
    \mu^{\f{1}{2}}|v|_{L_t^2\tilde{H}^{m-1}}\lesssim \mu^{\f{1}{2}}(\|v\|_{\hco^{m-1}}+\|\na v\|_{\hco^{m-1}}^{\f{1}{2}}\|v\|_{\hco^{m-1}}^{\f{1}{2}})\lesssim T^{\f{1}{4}}\lat\cE_{m,t}.
\end{align}
%Using the fact that $\|v\|_{\infco^{m-1}}\lesssim \|u\|_{\infco^{m-1}}\lesssim \cE_{m,t},$ we find that\begin{align*} \mu^{\f{1}{2}}|v|_{L_t^2\tilde{H}^{m-1}}\lesssim T^{\f{1}{4}}\cE_{m,t}.\end{align*}
Plugging the above two inequalities into \eqref{sec7:eq2}, we find:
\begin{align}
\cH_{43}\lesssim ( T^{\f{1}{2}}+\ep)\lat\cE_{m,t}^2,
\end{align}
which, together with \eqref{ch31} and \eqref{ch32}, yields:
\begin{align}\label{ch3}
    \cH_{4}\lesssim  (T^{\f{1}{2}}+\ep)\lae.
\end{align}
We now proceed to control the term $\cH_5=\mu\lambda_1\overline{\Gamma} \izto [Z^I, \div ](r_0^{-1}\na v)\cdot Z^I v\,\d x\d s.$ By using the identity \eqref{comu}, we integrate by parts to get that:
\begin{align*}
    \cH_5 \lesssim \|\mu^{\f{1}{2}}\na v\|_{\hco^{m-1}}\|\na v\|_{\hco^{m-2}}+\mu|r_0^{-1}\na v|_{L_t^2\tilde{H}^{m-2}}|v|_{L_t^2\tilde{H}^{m-1}},
\end{align*}
which, combined with \eqref{sec7:eq8}, \eqref{sec7:eq9} and the fact
\begin{align*}
 \|\mu^{\f{1}{2}}\na v\|_{\hco^{m-1}}+\|\na v\|_{\infco^{m-2}}&\lesssim \mu^{\f{1}{2}}\|\na (r_0u)\|_{\hco^{m-1}} + \|\na(r_0 u)\|_{\infco^{m-2}}\\
 &\lesssim \lat\cE_{m,t},
\end{align*}
leads to the estimate:
\begin{align}\label{ch4}
  \cH_5 \lesssim (T^{\f{1}{2}}+\ep) \lat\cE_{m,t}^2.  
\end{align}
It remains the control of $\cH_6=\izto Z^I v\cdot H^I\, \d x \d s$ 
 which can be split into two terms:
\beqs
\cH_6 =\izto Z^I  v\cdot Z^I G_1\, \d x \d s+ \izto Z^I  v\cdot Z^I  (H^I-Z^I G_1)\, \d x \d s
\eeqs
where $G_1$ is defined in \eqref{defG1}. Similar to the estimate of $\cH_2,$ we integrate by parts in space for the first term %and use estimate \eqref{es-G-1} to obtain:
use the Cauchy-Schwarz inequality for both of these two terms:
\begin{align*}
\cH_6 &=\izto Z^I  v\cdot Z^I G_1\, \d x \d s+ \izto Z^I  v\cdot Z^I  (H^I-Z^I G_1)\, \d x \d s\\
&\lesssim  \|\mu^{-\f{1}{2}} G_1\|_{\hco^{m-2}}\|\mu^{\f{1}{2}} v\|_{\uhco^m}+ T^{\f{1}{2}}\|v\|_{\infco^{m-1}}\|H^I-Z^I G\|_{\hco^{m-1}}.
\end{align*}
%which relies on the estimate of $\|\cH^I\|_{\hco^{m-1}}.$
Using the product estimate \eqref{roughproduct1} and the commutator estimate \eqref{roughcom}, one finds that:
\begin{align*}
   & \|\mu \na r_0^{-1}\cdot \na v\|_{\hco^{m-1}}\lesssim \|\mu^{\f{1}{2}}( \theta,\ep\tsigma, \na u, \na\theta,\ep\na\tsigma)\|_{\hco^{m-1}}\lat\lesssim \lae,\\
  & \| [Z^I, u\cdot\na ]v\|_{L_t^2L^2}\lesssim (\|\na v\|_{\hco^{m-2}}+\|u\|_{\hco^{m-1}})\lat\lesssim T^{\f{1}{2}}\lae,
\end{align*}
which, together with \eqref{es-G-2} and \eqref{es-bpcom}, leads to the estimate 
\beqs
\|\cH^I-Z^I G_1\|_{\hco^{m-1}}\lesssim \lae.
\eeqs
Therefore, we can estimate $\cH_6$ by using \eqref{es-G-1} for $G_1$ and the above inequality:
\beq\label{ch5}
\cH_6\lesssim (T^{\f{1}{2}}+\ep)\lae.
\eeq

Inserting \eqref{ch01}-\eqref{ch2}, \eqref{ch3}-\eqref{ch5} into \eqref{sec7:eq0}, 
we find by summing up for $|I|\leq m-1, I_0\leq m-2$ that
\begin{align}\label{sec7:eq10}
\|v\|_{\uinfco^{m-1}}^2+\mu\|\na v\|_{\uhco^{m-1}}^2\lesssim Y_m^2(0)+(T^{\f{1}{2}}+\ep)\lae.
\end{align}
In view of the  \eqref{timederv} and  \eqref{sec7:eq10},
we  finish the proof of \eqref{EE-v}.
\end{proof}
\begin{rmk}
We can now conclude that for any $\ep\in (0,1], (\mu,\kpa)\in A,$ any $0<t\leq T,$
\begin{align}\label{nauL2-uniform}
    \mu\|\na u\|_{\hco^{m-1}}^2\lesssim 
    Y_m^2(0)+(T+\ep)^{\f{1}{6}}\lae.
\end{align}
Indeed, in view of \eqref{com-uniL2}, \eqref{EE-v}, it stems from the estimate:
\beq\label{sec6:eq40}
\mu^{\f{1}{2}}\|\na\bp u\|_{\hco^{m-1}}\lesssim \mu^{\f{1}{2}} \|\na v\|_{\hco^{m-1}}+ (T^{\f{1}{2}}+\ep)\lae,
\eeq
which can be shown in the following way.  On the one hand, for any conormal vector fields
$Z^I, |I|\leq m-1,$ we write 
$$r_0 Z^I \na\bp u=Z^I(r_0 \na\bp u)-Z^I r_0\na\bp u-[Z^I, r_0, \na\bp u].$$
This gives rise to the estimate:
\begin{align}\label{sec6:eq41}
&\mu^{\f{1}{2}}\|\na\bp u\|_{\hco^{m-1}}\lesssim   \mu^{\f{1}{2}}\|r_0 \na\bp u\|_{\hco^{m-1}}\notag\\
  &\qquad+ T^{\f{1}{2}} \|\mu^{\f{1}{2}}Z r_0, \na\bp u\|_{\infco^{m-2}}
  \lab \mu^{\f{1}{2}}\il Z r_0\il_{m-4,\infty,t}+\il \na\bp u\il_{1,\infty,t}\big)\\
  &\lesssim   \mu^{\f{1}{2}}\|r_0 \na\bp u\|_{\hco^{m-1}}+T^{\f{1}{2}}\lae.\notag
\end{align}
On the other hand, 
by the identity $$\na v-r_0\na\bp u= \na\bp(r_0 u)-r_0\na\bp u=\na [\bp, r_0]u+\na r_0\cdot\bp u,$$
 the estimate 
\begin{align*}
\|\na [\bp, r_0]u\|_{\hco^{m-1}}&\lesssim \|\div ([\bp, r_0]u), \curl ([\bq, r_0]u)\|_{\hco^{m-1}}\\
&\lesssim \|(\na r_0\cdot \bp u, \na r_0\times \bq u)\|_{\hco^{m-1}},
\end{align*}
as well as \eqref{esr0-3}, one can find that:
\begin{align*}%\label{sec6:eq42}
  & \mu^{\f{1}{2}}\| \na v-r_0\na\bp u\|_{\hco^{m-1}}\notag\\
   &\lesssim %T^{\f{1}{2}} 
   \|(\mu^{\f{1}{2}}\na r_0, u)\|_{\hco^{m-1}}\lab \il\mu^{\f{1}{2}}\na r_0\il_{[\f{m}{2}]-1,\infty,t}+ \il(u, \bq u)\il_{[\f{m}{2}],\infty,t}\big)\\
   &\lesssim (T^{\f{1}{2}}+\ep)\lae.\notag
\end{align*}
This, together with \eqref{sec6:eq41}, yields \eqref{sec6:eq40}.
\end{rmk}
\begin{lem}
Denote $u^j=(\ep\pt)^j u,$ and $v^j=\bp (r_0 u^j),$ we have the following estimate for 
$v^j,$ 
\begin{align}\label{EE-vj}
     \|v^j\|_{\infcoch^{m-1-j}}^2\lesssim Y_m^2(0)+(T+\ep)^{\f{1}{2}}\lae, \,\, \forall j=0,1\cdots m-1,
\end{align}
where we use $\cH^{k}=\cH^{0,k}$ to denote the norm that involves only spatial conormal derivatives.
\begin{proof}
In fact, \eqref{EE-vj} can be derived from the estimate \eqref{EE-v} and the following property:
\begin{align*}
    \|v^j\|_{\infcoch^{m-1-j}}\lesssim \|v\|_{\infco^{m-1}}+  \ep\lat\cE_{m,t},
\end{align*}
which, in light of the definition $$v^j=(\ep\pt)^j v+ \bp([r_0,(\ep\pt)^j]u),$$  
is the consequence of the following estimate:
\begin{align*}
   & \| \bp([r_0,(\ep\pt)^j]u)\|_{\infcoch^{m-1-j}}\lesssim \|[r_0,(\ep\pt)^j]u\|_{_{\infcoch^{m-1-j}}}\\
    &\qquad\lesssim \ep \|(\pt r_0, u)\|_{\infco^{m-2}}\lab \il (\pt r_0, u)\il_{[\f{m}{2}]-1,\infty,t}
    \big)\lesssim \ep\lae.
\end{align*}
\end{proof}
\end{lem}
\section{Uniform estimates for the gradient of the velocity}
In the previous section, we have controlled uniformly the incompressible part $v^j=\bp(r_0 u^j)$ of the unknowns $r_0 u^j=r_0 (\ep\pt)^j u \,(j=0,\cdots m-1) $
 in the space $L_t^{\infty}\cH^{m-1-j}.$   
In this section, we focus on the uniform estimate of $\na v^j$ in $L_t^{\infty}\cH^{m-2-j}$ which is crucial in order to control the $\infco^{m-2}$ norm of the gradient of the velocity.
%stated in the next lemma.
\begin{prop}\label{prop-nauLinftyL2}
Suppose that \eqref{preasption}
 holds and $m\geq 7$, then there exists a constant $\vartheta_3>0$ such that for any $\ep \in (0,1], (\mu,\kpa)\in A,$ any $0<t\leq T,$  the following estimate holds:
 \begin{equation}\label{nablavLinfty}
    \begin{aligned}
  \|\nabla v^j\|_{L_t^{\infty}\cH^{m-2-j}}^2&\lesssim \lab Y_m(0)\big)
  +(T+\ep)^{\vartheta_3}\Lambda\big(\f{1}{c_0},\cN_{m,t}\big), \quad \forall\, j=0,\cdots m-2.
\end{aligned}
\end{equation}
\end{prop}
\begin{proof}
Since in the interior domain, the conormal spatial derivatives are equivalent to the standard spatial derivatives, we only have to estimate $\nabla v$ near the boundary, say
$\|\chi_{i} \nabla v^j\|_{L_t^{\infty}\cH^{m-2-j}}$  where
 $\chi_i, (i=1\cdots N)$ are smooth functions associated to the covering \eqref{covering1} and are compactly supported in $\Omega_i.$ Close to the boundary, it follows from the  identities \eqref{normalofnormal-v} and 
 %\beqs\label{usefulidentity}\partial_{\bn} v^j\cdot \bn =-(\Pi \partial_{y^1}v^j)^1-(\Pi \partial_{y^2}v^j)^2,\eeqs
\beq\label{tan-nor}
\Pi(\partial_{\bn} v^j)=\Pi ((\nabla  v^j-D v^j)\bn)+\Pi((D v^j)\bn)=\Pi(\curl v^j \times \bn)+\Pi \na (v^j\cdot\bn) +\Pi(-(D \bn) v^j)
\eeq
that (since $\Pi \na$ only involves tangential derivative):
\beqs
\begin{aligned}
\|\chi_{i}\nabla v^j\|_{L_t^{\infty}\cH^{m-2-j}}&\lesssim
\|\chi_{i}\Pi(\partial_{\bn} v^j)\|_{L_t^{\infty}\cH^{m-2-j}}+\|v^j\|_{L_t^{\infty}\cH^{m-1-j}}\\
&\lesssim \|\chi_i(\curl v^j\times\bn)\|_{L_t^{\infty}\cH^{m-2-j}}+\|v^j\|_{L_t^{\infty}\cH^{m-1-j}}.
\end{aligned}
\eeqs
Consequently, this proposition can be proved by combining the estimate \eqref{EE-vj} for
$v^j$ and the next lemma for $\curl v^j \times\bn.$
\end{proof}

%Since $v$ is divergence free and satisfies the boundary condition $v\cdot\bn|_{\p\Omega}=0,$ it suffices to estimate  $\curl v=\curl (r_0u)=\colon \omega_{r_0}.$  
\begin{prop}\label{omeganLinfty}
%Under the assumption \eqref{preasption}, 
Under the same assumption as in Proposition \ref{prop-nauLinftyL2}, there exists $\vartheta_3>0,$ such that
 for any $\ep \in (0,1], (\mu,\kpa)\in A,$  any $0<t\leq T, $ any $j=0,\cdots m-2$, it holds that:
\begin{equation}\label{ineqofomegatimesn}
    \begin{aligned}
  \sup_{i\in \{1,\cdots N\}}\|\chi_i(\omega_{r_0}^j\times \bn)\|_{L_t^{\infty}\cH^{m-2-j}}^2&\lesssim \lab Y_m(0)\big)
  %\|\chi_i(\omega_{r_0}^j\times \bn)(0)\|_{\cH^{m-2-j}}^2
  +(T+\ep)^{\vartheta_3}\Lambda\big(\f{1}{c_0},\cN_{m,t}\big), %\quad \forall j=0,\cdots m-2,
    \end{aligned}
\end{equation}
where $\omega_{r_0}^j=\curl v^j= \curl(r_0 (\ep\pt)^j u),$ and $\chi_{i}$ is a smooth function compactly supported in $\Omega_i.$
\end{prop}
 \begin{proof}
Let us first  derive the equation satisfied by  $\omer^j=\colon \curl(r_0 (\ep\pt)^j u).$ 
Applying time derivatives $(\ep\pt)^j $ on the equations $\eqref{NCNS-S2}_2,$ we find that $u^j=\colon (\ep\pt)^j u$ solves the equation:
\beq 
\f{1}{R\beta}(\pt+ u\cdot\na ) u^j+ \f{\na\tsigma^j}{\ep}+\mu\lambda_1\Gamma \curl\curl u^j=\cC_{u}^j,
\eeq
where 
\beqs 
\tilde{\sigma}^j= (\ep\pt)^j \tilde{\sigma}= (\ep\pt)^j (\sigma-\ep\mu(2\lambda_1+\lambda_2)\Gamma\div u),
\eeqs
\beq\label{defCuj}
\begin{aligned}
\cC_{u}^j=\colon&-[(\ep\pt)^j, \f{1}{R\beta}]\pt u- [(\ep\pt)^j, \f{1}{R\beta}u\cdot\na ]u\\
&-\mu\lambda_1 [(\ep\pt)^j,\Gamma]\curl\curl u -\mu(2\lambda_1+\lambda_2)(\ep\pt)^j(\div u\na\Gamma).
\end{aligned}
\eeq
Similar to the derivation of \eqref{eqr0u}, we rewrite  this equation as:
\beq\label{eqr0uj}
(\pt+u\cdot\na)(r_0 u^j)+\ep^{-1}(1+\ep r_1)\na{\tsigma^j}+\mu\lambda_1 \overline{\Gamma}r_0^{-1}\curl\curl(r_0 u^j)=G^j
\eeq
where 
\begin{align}\label{def-Gj}
G^j&=u^j(\pt+u\cdot\na)r_0+\mu\lambda_1\overline{\Gamma}r_0^{-1}[\curl\curl, r_0]u^j\notag\\
&\qquad- \ep\mu \lambda_1r_1\Gamma\curl\curl u^j-\mu\lambda_1(\Gamma-\overline{\Gamma})\curl\curl u^j+(1+\ep r_1)\cC_u^j. %\mu(2\lambda_1+\lambda_2)
%+ \ep\mu r_1\div\cL u^j+\mu(\Gamma-\overline{\Gamma}) \div\cL u^j\notag\\
%&\qquad\qquad\qquad-\mu\overline{\Gamma}r_0^{-1}[\div\cL, r_0]u^j-\mu(2\lambda_1+\lambda_2)\overline{\Gamma}\na(r_0^{-1})\div(r_0 u^j)
\end{align}
Applying the curl of the equations \eqref{eqr0uj}, we find that $\omega_{r_0}^j$ solves the equation:
 \beq\label{eqomegar0}
 (\pt+ u\cdot\na)\omer^j -\mu \lambda_1\overline{\Gamma} r_0^{-1} \Delta\omer^j =\curl G^j-\sum_{l=1}^3\na u_{l}\times \p_{l}(r_0 u^j) +\mu\lambda_1 \overline{\Gamma}  \na (r_0^{-1} )\times \curl \omer^j. 
 \eeq
Taking the cross product of the above equations with $\chi_i\bn,$ we can  derive further the equation for $\chi_i \omer^j \times \bn:$ 
 \beq\label{eqomern}
 (\pt+ u\cdot\na)(\chi_i \omer^j \times \bn) -\mu \lambda_1\overline{\Gamma} r_0^{-1} \Delta(\chi_i \omer^j \times \bn)=H_i^j,
 \eeq
where 
$H_i^j=H_{i1}^j+H_{i2}^j$ %+H_{i3}^j$
with 
\beq\label{defH}
\begin{aligned}
H_{i1}^j&=\colon \big(\curl G^j+\mu\lambda_1 \overline{\Gamma}  \na (r_0^{-1} )\times \curl \omer^j -\sum_{l=1}^3\na u_{l}\times \p_{l}(r_0 u^j)\big) \times \chi_i \bn,\\
H_{i2}^j&=\omer^j\times (u\cdot \na (\chi_i\bn))-\mu\lambda_1\overline{\Gamma} r_0^{-1}
\big( 2\div(\omer^j\times\na (\chi_i\bn))-\omer^j \times\Delta(\chi_i\bn)\big).
%(2\na \omer^j\times \na(\chi_i\bn)-\omer^j \times\Delta(\chi_i\bn)).
\end{aligned}
\eeq 
Moreover, it follows from the identity \eqref{tan-nor} and the boundary condition %\eqref{bc-v-1}, \eqref{bc-v-2}
 \eqref{bdryconditionofu} that: 
\ben\label{bcomernj}
%\begin{aligned}
 \omer^j \times \bn&=& \Pi(\p_{\bn}v^j)+\Pi((D \bn) v^j)\notag\\
 &=&\p_{\bn} r_0 \Pi u^j+ r_0\Pi(\p_{\bn} u^j)-\Pi(\p_{\bn}\bq (r_0 u^j))+\Pi((D \bn) v^j)\\ 
    &=&\f{R\ep}{C_v\gamma}r_0\p_{\bn}\tsigma \Pi u^j+2r_0\Pi\big(-a u^j+(D\bn)u^j\big) 
\qquad\quad \text{ on } \p\Omega. \notag
\een
The important feature here is that $\chi_i \omer^j \times \bn$ is governed by a transport-diffusion equation \eqref{eqomegar0} without any singular terms. Nevertheless, due to the non-homogeneous Dirichlet boundary condition \eqref{bcomernj}, 
the uniform estimates can not be obtained by simply performing energy estimates. Our general strategy is to split the system for  $\chi_i(\omer^j\times \bn)$ into two parts, one of which solves a heat equation (with variable coefficients) with the  nontrivial Dirichlet boundary condition and non-homogeneous source term.
The other satisfies a convection-diffusion equation with homogeneous Dirichlet boundary condition which is amenable to energy estimates. 
 To deal with the first  system,
the explicit formula for heat equation with variable coefficient will play an important role. It is thus helpful  to transform the problem to the half-space. %Such a strategy has essentially  been employed in the previous work \cite{masmoudi2021uniform} where uniform  in (and only in) Mach number regularity estimates are established for \textit{isentropic} compressible Navier-Stokes system. However, in the situation of the non-isentropic case, besides the computational complexity due to the  appearance of the non-constant temperature, there are several extra difficulties  in order to obtain the estimates which are uniform in Mach number, Reynolds number and P\'elect number. 
 %1, For the homogeneous part, we need to deal with the heat equation with variable coefficients, which makes the analysis more involved. (explain more in detail)2. The typical term  $\mu \izto Z^{I}r_0 \na \omer Z^I \na\omer\d x \d s$ requires further consideration.3. Appearance of $\big(\na (\pt+u\cdot\na)r_0-\mu\lambda_1 \overline{\Gamma}r_0^{-1}\na\Delta r_0\big) \times u $ lead us to make the assumption on $\mu$ and $\kpa$. 
 %The proof of Lemma \ref{omeganLinfty} is quite involved and is thus split into the following lemmas. %{\color{red} state the difference between isentropic and non-isentropic}
%Due to some technical issues, instead of working in the 
Instead of coordinates \eqref{local coordinates}, it is convenient here for us to use another parameterization, which is the so-called
 normal geodesic coordinates: (by shrinking $\Omega_i$ and taking smaller $\delta_i,\epsilon_i$ if necessary)
\begin{equation}\label{normal geodesic coordinates}
 \begin{aligned}
 {\Psi}_{i}:\quad & (-\delta_i,\delta_i)^2\times (0,\epsilon_i)
 \rightarrow \Omega_{i}\cap{\Omega}\\
 &\qquad\quad(y,z)^{t}\rightarrow {\Psi}_{i}(y,z)=(y,\varphi_i(y))^{t}-z \bn^i(y)
 \end{aligned}
 \end{equation}
where $\bn^i(y)=\textbf{N}^i/|\textbf{N}^i|$,
$\textbf{N}^i(y)=(\partial_{y^1}\varphi_i(y),\partial_{y^2}\varphi_i(y),-1)^{t}.$ Note that since we are considering at most $m-2$ order of regularities, %it does not require more regularity of the boundary, that is, 
the following estimates hold for a domain with $C^{m+2}$ smooth boundaries. As we work in the local coordinates \eqref{normal geodesic coordinates}, it is useful to know the expressions of gradient and Laplacian in the new coordinates. 
By straightforward calculations, one can find that: 
%$$(\na f)\circ \Psi_i=(D\Psi_i )^{-1}\na (f\circ \Psi_i),$$ 
\beq
 \begin{aligned}\label{change-variable}
(\nabla f)\circ \Psi_i=& P_i{\nabla} (f\circ{\Psi_i}), \quad (\div F)\circ \Psi_i=\f{1}{d_i}{\div}\big(d_i P_i^{*}(F\circ\Psi_i)\big),\quad\\
&(\Delta f)\circ \Psi_i=\f{1}{d_i}{\div}\big(d_i P_i^{*}P_i\na \big(f\circ\Psi_i)\big),
 \end{aligned}
 \eeq
where $ P_i= (D\Psi_i)^{-1}$, $d_i=\det (D\Psi_i)$ and the matrix %$D\Psi_i$ has the entries $(D\Psi_i)_{jk}=\f{\p\Psi_{i,k}}{\p {y^j}}.$
\beq\label{defdpsi}
D\Psi_i=
\left(\begin{array}{ccc}
    1 &0&\p_{y^1}\varphi_i  \\
     0& 1& \p_{y^2}\varphi_i\\
    % -\p_{y^1}\varphi_i&-\p_{y^2}\varphi_i &1
    n_1^i & n_2^i & n_3^i
\end{array}
\right)-z
\left(
\begin{array}{ccc}
   % \p_{y^1}^2\varphi_i &  \p_{y^1} \p_{y^2}\varphi_i&\p_{y^1} \p_{y^3}\varphi_i  \\
      \p_{y^1} \bn_1^i&  \p_{y^1} \bn_2^i& \p_{y^1} \bn_3^i\\
     \p_{y^2}\bn_1^i&\p_{y^2}\bn_2^i &  \p_{y^2} \bn_3^i\\
     0&0&0
\end{array}
\right).
\eeq
Note that since $|\det ((D\Psi_i)|_{z=0})|=|\bN^i|\geq 1,$  the matrix is invertible
as long as $z$ is sufficiently small.
%By further computation, we find that the Riemann metric in the new coordinates induced in the standard Euclidean metric in $\Omega_i\cap\Omega$ through pull back of $\Psi_{i}$ has the following form:
%from which, 
By further computation, one could 
deduce that the Laplacian in this new coordinates  has the form 
\begin{equation}\label{laplaceg}
    \f{1}{d_i}{\div}\big(d_i P_i^{*}P_i\na \cdot\big)=\partial_z^2 +\frac{1}{2}\partial_z(\ln |g_i|)\partial_z +\Delta_{{g_i}}%=\colon \Delta_{g^i}
\end{equation}
where $\Delta_{{g_i}}$ involves only the tangential derivatives:
\beq\label{defdeltag}
\Delta_{{g_i}}=|{g_i}|^{-\frac{1}{2}}\sum_{l,k=1}^2 \partial_{y^l}(g_{i}^{lk}|{g_i}|^{\frac{1}{2}}\partial_{y^k}),
\eeq
and $|g_i|$ denotes the determinate of the symmetric invertible matrix
\begin{equation}\label{defmatrixg}  g_i(y,z)= \left(   \begin{array}{cc}    g_{i}^{11} & g_{i}^{12} \\ g_{i}^{21} &  g_{i}^{22} \end{array} \right)(y,z)%\in \cM_2(\mR^3)
\end{equation}
whose entries  depends smoothly on $D\Psi_i.$ Note that $|g_i|$ has uniform lower bound and the matrix $g_i$ is positive definite
as long as $z$ is small enough.

For a function $f: \Omega_{i}\cap{\Omega}\rightarrow \mR,$ we 
use the notation $f^{\Psi_i}=f\circ \Psi_i.$ We will also denote 
$\tilde{\zeta}_i^j=(\chi_i\omer^j\times\bn)^{\Psi_i}.$ Since $\Supp \chi_i|_{\bar{\Omega}}\Subset \Omega_i\cap\bar\Omega,$
we can extend the definition of  $\tilde{\zeta}_i^j$ and $(H_{i}^{j})^{\Psi_i}$ from ${\Psi}_{i}^{-1}(\Omega_{i}\cap\bar{\Omega})$ to $\overline{\mathbb{R}_{+}^3}$ by zero extension,
which are still denoted as $\tilde{\zeta}_i^j$ and $(H_{i}^{j})^{\Psi_i}.$
By \eqref{eqomern}, $\tilde{\zeta}_i^j$ solves the equation: 
\begin{equation}\label{eqtildezeta}
\left\{
    \begin{array}{l}
   (\pt+ \tilde{u}_i \cdot \na)\tilde{\zeta}_i^j-\mu \lambda_1\overline{\Gamma} ( \f{1}{r_0})^{\Psi_i}(\partial_z^2 +\frac{1}{2}\partial_z(\ln |g_i|)\partial_z +\Delta_{{g_i}})\tilde{\zeta}_i^j=(H_{i}^{j})^{\Psi_i}
   ,  \,\, (t,x) \in \mR_{+}\times \mR^3,  \\[2.5pt]
\tilde{\zeta}_i^j|_{z=0}= \big[\f{R\ep}{C_v\gamma}r_0\chi_i\p_{\bn}\tsigma \Pi u^j+2r_0\chi_i\Pi\big((D\bn)u^j-a u^j\big)\big]^{\Psi_i}|_{z=0}, \\[2.5pt]
\tilde{\zeta}_i^j|_{t=0}= (\chi_i\omer^j|_{t=0}\times\bn)^{\Psi_i},
    \end{array}
\right.
\end{equation}
where $\na=(\p_{y^1}, \p_{y^2}, \p_{y^3})$ and $\div=(\na)^{*}$ 
stand for the gradient and the divergence in the new coordinates and $\tilde{u}_i$ is defined as:
\beq\label{deftildeu}
\tilde{u}_i=(\tilde{u}_{i,1},\tilde{u}_{i,2},\tilde{u}_{i,3})=P_i^{*}(\chi_i u)^{\Psi_i}.
\eeq
Note that on the boundary $z=0,$ 
\beq\label{tlu3-bd}
\tilde{u}_{i,3}|_{z=0}=\big[( D\Psi_i|_{z=0})^{-*}(\chi_i u)^{\Psi_i}\big]_3\big|_{z=0}=- ((\chi_i u\cdot \bN^i)|_{\p\Omega})^{\Psi_i}=0.
\eeq
 %$ P_i= (D\Psi_i)^{-1}$
Let us set %$\mathcal{Z}_0=\varepsilon\partial_t,$
 \beq\label{newvector}
 \mathcal{Z}_j=\partial_{y^j},\,\, j=1,2,\quad \mathcal{Z}_3=\phi(z)\partial_z
 \eeq
and define energy norm
\beq\label{newnotation}
\|f^{\Psi_i}\|_{m,t}=\sum_{|\alpha|\leq m}\|\mathcal{Z}^{\alpha} f^{\Psi_i}\|_{L^2([0,t]\times\mathbb{R}_{+}^3)},\, \|f^{\Psi_i}(t)\|_{m}=\sum_{|\alpha|\leq m}\|(\mathcal{Z}^{\alpha} f^{\Psi_i})(t)\|_{L^2(\mathbb{R}_{+}^3)}
\eeq
and $L_{t,x}^{\infty}$ norm: 
\beq\label{newinfty}
\il f^{\Psi_i}\il_{*,m,\infty,t}=\sum_{|\alpha|\leq m}\|\cZ^{\alpha} f^{\Psi_i}\|_{L^{\infty}([0,T]\times\mR_{+}^3)}, %\quad 
 %\il f^{\Psi_i}\il_{m,\infty,t}=\sum_{k\leq m}\|(\ep\pt)^k f^{\Psi_i}\|_{*,m-k,\infty,t}
\eeq
%\beq\label{newnotation}\|f^{\Psi_i}\|_{m,t}=\sum_{k+|\alpha|\leq m}\|\mathcal{Z}^{\alpha}(\ep\pt)^k f^{\Psi_i}\|_{L^2([0,t]\times\mathbb{R}_{+}^3)},\, \|f^{\Psi_i}(t)\|_{m}=\sum_{k+|\alpha|\leq m}\|(\mathcal{Z}^{\alpha}(\ep\pt)^k f^{\Psi_i})(t)\|_{L^2(\mathbb{R}_{+}^3)}\eeq
%\beq  \il f^{\Psi_i}\il_{m,\infty,t}=\sum_{k+|\alpha|\leq m}\|\cZ^{\alpha}(\ep\pt)^k f^{\Psi_i}\|_{L^{\infty}([0,T]\times\mR_{+}^3)}\eeq
where $\mathcal{Z}^{\alpha}=%\mathcal{Z}_0^{\alpha_0}
\mathcal{Z}_1^{\alpha_1}\mathcal{Z}_2^{\alpha_2}\mathcal{Z}_3^{\alpha_3}, \alpha=(\alpha_1,\alpha_2,\alpha_3).$ We emphasize that the above norms involves only spatial conormal derivatives and will be frequently used throughout this section. For later use, let us also define here:
\beq\label{newinfty-time}
\il f^{\Psi_i}\il_{m,\infty,t}=\sum_{k\leq m}\|(\ep\pt)^k f^{\Psi_i}\|_{*,m-k,\infty,t}.
\eeq
Since 
${\Psi_i}^{-1}\circ \Phi_i$ is a diffeomorphism between  $\Phi_i^{-1}(\Omega_i\cap \Omega)\rightarrow \Psi_i^{-1}(\Omega_i\cap \Omega)$ and satisfies $({\Psi_i}^{-1}\circ \Phi_i)|_{z=0}=\Id,$ we can deduce from the Leibniz's rule and the definition of the conormal spaces \eqref{not3} that:
 \beq\label{norm-equiv-bd}
 \|f^{\Psi_i}\|_{k,t}\approx \|f\|_{L_t^2\cH^k(\Omega)}, \quad  \|f^{\Psi_i}(t)\|_{k}\approx \|f (t)\|_{\cH^k(\Omega)}.
 \eeq
 %Let us recall that $\cH^{k}=\cH^{0, k}$ involves only spatial conormal derivatives.
 By \eqref{norm-equiv-bd}, the proof of \eqref{ineqofomegatimesn} can be reduced to the following estimate for $\tilde{\zeta}_i^j$.
  \end{proof}
  \begin{prop}
% Using the notation \eqref{newnotation},
Under the same assumption as in Proposition \ref{prop-nauLinftyL2}, 
 there exists a constant $\vartheta_3>0$ such that for any $\ep \in (0,1], (\mu,\kpa)\in A,$ any $0<t\leq T,$ any $ j=1\cdots m-2,$ any $i=1,\cdots N,$  the following estimate holds:
  \begin{align}\label{es-tildezeta}
\sup_{0\leq t\leq T}  
 \|\tilde{\zeta}_i^j(t)\|_{m-2-j}^2\lesssim \lab Y_m(0)\big)+(T+\ep)^{\vartheta_3}\Lambda\big(\f{1}{c_0}, \cN_{m,T}\big). 
  \end{align}
  \end{prop}
  \begin{rmk}
  Note that the notations \eqref{newnotation}-\eqref{newinfty} will be used frequently in the proof of this proposition.
  \end{rmk}
 \begin{proof}
 In the following, we will drop the subscript $i$ for the sake of the notational clarity
 and in particular, we will write $\tilde{\zeta}_i, (H_i^{j})^{\Psi_i}, \tilde{u}_i$ by 
 $\tilde{\zeta}^j, (H^{j})^{\Psi}, \tilde{u}.$ We introduce further 
 $$\zeta^j= |g|^{\f{1}{4}}\tilde{\zeta}^j$$ and derive from \eqref{eqtildezeta} 
 the equation satisfied by $\zeta^j:$
 \begin{equation}\label{eqzeta}
\left\{
    \begin{array}{l}
   (\pt+ \tilde{u} \cdot \na){\zeta}^j-\mu \lambda_1\overline{\Gamma} %r_0^{-1}
  ( \f{%\tilde{\chi}
  1}{r_0})^{\Psi}(\partial_z^2  +\Delta_{{g}})\zeta^j =|g|^{\f{1}{4}} (H^j)^{\Psi}+ H^{j,\dag},  \quad (t,x) \in \mR_{+}\times \mR^3,  \\[2.5pt]
\zeta^j|_{z=0}= \big[\f{R\ep}{C_v\gamma}r_0\p_{\bn}\tsigma |g|^{\f{1}{4}}\chi\Pi u^j+2r_0|g|^{\f{1}{4}}\chi\Pi\big((D\bn)u^j-a u^j\big)\big]^{\Psi}|_{z=0},  \\[2.5pt]
\zeta^j|_{t=0}=|g|^{\f{1}{4}}\tilde{\zeta}^j|_{t=0},
    \end{array}
\right.
\end{equation}
%where we denote $\tilde{\chi}$  a compactly supported smooth function such that $0\leq \tilde{\chi}\leq 1,$  $\tilde{\chi}\chi=\chi$ 
where
\beq\label{defHdag}
H^{j,\dag}=\big(\tilde{u}\cdot\na (|g|^{\f{1}{4}})\big) (|g|^{-\f{1}{4}}\zeta^j)%\tilde{\zeta}
- \mu \lambda_1\overline{\Gamma} r_0^{-1}\big( |g|^{-\f{1}{4}}\p_z^2(|g|^{\f{1}{4}})\zeta^j -[\Delta_g, |g|^{\f{1}{4}}](|g|^{-\f{1}{4}}\zeta^j)\big).
\eeq
Since $g$ involves two derivations of the boundary parameterization function $\varphi,$ one can deduce from the assumption on the smoothness of the boundary that $g, |g|, |g|^{-\f{1}{4}}\in C^m (\mR_{+}^3).$ Therefore, in order to show 
\eqref{es-tildezeta}, it suffices to establish the following estimate for $\zeta^j\, (j=0,1\cdots m-2),$ \begin{align}\label{es-zeta}
   \|\zeta^j(t)\|_{m-2-j}^2\lesssim  \lab Y_m(0)\big)+(T+\ep)^{\vartheta_3}\Lambda\big(\f{1}{c_0}, \cN_{m,T}\big) 
   \quad \forall\, 0< t\leq T. 
  \end{align}
Due to some technical reason, 
we have to distinguish the case $j=0$ or $1\leq j\leq m-2.$
%norm which involves $m-2$ spatial conormal derivatives and the others and control them separately. 
%and is done in the following two lemmas.
The estimate \eqref{es-zeta} is thus the consequence of the following two estimates: 
%\begin{lem}\label{zetalow}
%Using the notation \eqref{newnotation}, it holds that for any $t\leq T,$
\begin{align}
  \|\zeta^j(t)\|_{m-2-j}^2&\lesssim  \lab Y_m(0)\big)+(T+\ep)^{\vartheta_4}\Lambda\big(\f{1}{c_0}, \cN_{m,T}\big), \quad j=1,2\cdots m-2,\label{zetalow} \\
  \|\zeta^0(t)\|_{m-2}^2&\lesssim  \lab Y_m(0)\big)+(T+\ep)^{\vartheta_5}\Lambda\big(\f{1}{c_0}, \cN_{m,T}\big), \label{zetahigh}
\end{align}
where the constants $\vartheta_4, \vartheta_5\in (0,\f{1}{2}]$  are independent of $\ep,\mu,\kpa.$
The above two estimates can be achieved by splitting the system in suitable (different) ways, we will detail them in the following two subsections respectively.
\end{proof}
\subsection{Uniform high regularity with at least one time derivatives-proof of \eqref{zetalow}}
This section is dedicated to the high regularity estimate for $\zeta^j (1\leq j\leq m-2)$ which is summarized in 
 \eqref{zetalow}.  As explained before, such an estimate is unlikely to be obtained by sorely energy estimates due to the non-homogeneous boundary conditions. The strategy to prove \eqref{zetalow} is to
 split $\zeta^j$ into two unknowns:
$\zeta^j=\zeta_1^j+\zeta_2^j,$  where $\zeta_1^j, \zeta_2^j$ solve the following systems respectively:
\begin{equation}\label{eqzeta1}
\left\{
    \begin{array}{l}
   \pt{\zeta}_{1}^j-\mu \lambda_1\overline{\Gamma} ( \f{%\tilde{\chi}
   1}{r_0})^{\Psi}|_{z=0}\partial_z^2 \zeta_{1}^j =0 %\p_z F^{\dag\dag},  
   \quad (t,x) \in \mR_{+}\times \mR^3,  \\[2.5pt]
\zeta_{1}^j|_{z=0}= \zeta^j|_{z=0}=\big[\f{R\ep}{C_v\gamma}r_0\p_{\bn}\tsigma |g|^{\f{1}{4}}\chi\Pi u^j+2r_0|g|^{\f{1}{4}}\chi\Pi\big((D\bn)u^j-a u^j\big)\big]^{\Psi}|_{z=0},\\[2.5pt]
{\zeta}_{1}^j|_{t=0}=0,
    \end{array}
\right.
\end{equation}
\begin{equation}\label{eqzeta2}
\left\{
    \begin{array}{l}
   (\pt+ \tilde{u} \cdot \na){\zeta_2^j}-\mu \lambda_1\overline{\Gamma} ( \f{{1}}{r_0})^{\Psi}(\partial_z^2  +\Delta_{{g}})\zeta_2^j =\cS_2^j,  \quad (t,x) \in \mR_{+}\times \mR^3,  \\[2.5pt]
\zeta_2^j|_{z=0}= 0, \quad
\zeta_2^j|_{t=0}=\zeta^j|_{t=0},
    \end{array}
\right.
\end{equation}
with the source term 
\beq\label{source2}
\cS_2^j=- (\tilde{u} \cdot \na){\zeta_{1}^j}+\mu \lambda_1\overline{\Gamma}\bigg(\big( \f{%\tilde{\chi}
1}{r_0}\big)^{\Psi}\Delta_{{g}}\zeta_{1}^j+ \big(\big( \f{1}{r_0}\big)^{\Psi}-\big( \f{1}{r_0}\big)^{\Psi}\big|_{z=0
}\big)\p_z^2\zeta_{1}^j\bigg)+ H^{j,\dag}+|g|^{\f{1}{4}} (H^j)^{\Psi}.
\eeq
The estimate \eqref{zetalow} thus follows from 
the studies on $\zeta_1^j$ and $\zeta_2^j$ carried out in Proposition \ref{lemze1} and Proposition \ref{lemze2} respectively. %stated and proved in below.

We begin with the estimate for $\zeta_1^j (1\leq j\leq m-2).$
\begin{prop}\label{lemze1}
Let $\zeta_1^j$ be the solution of the equations \eqref{eqzeta1}. Suppose that $m\geq 7$ and the assumption \eqref{preasption} holds.
%There exists a polynomial $\tilde{\Lambda},$ such that if $T^{\f{1}{2}}\tilde{\Lambda}\big(\f{1}{c_0},\cN_{m,T}\big)\leq \f{1}{2},$ then,
By using the notation in \eqref{newnotation}, 
the following estimates hold for $\zeta_{1}^j:$ 
 for  any $\ep \in (0,1], (\mu,\kpa)\in A,$
any $ 0<t\leq T,$ 
%$t\leq T,$ %\begin{align}\label{zeta1infty}  \il \ep\pt\zeta_1\il_{1,\infty,t}\lesssim \lat,\end{align}
\begin{align}
 \il\zeta_1^j\il_{*,[\f{m}{2}]-1-j,t}&\lesssim \lat, \quad \forall j=1,2\cdots [\f{m}{2}]-1, \label{zeta1infty}\\
  \|\zeta_1^j(t)\|_{m-2-j}+ \|\zeta_1^j\|_{m-1-j,t}&\lesssim T^{\f{1}{4}}\lat\cE_{m,t}, \quad \forall j=1,2\cdots m-2.\label{zeta1energy}
\end{align}
\end{prop}

\begin{proof}
For notational convenience, we denote that   
\beq \label{defb}
b_0=b_0(t,y)=\big(\f{%\tilde{\chi}
1}{r_0}\big)^{\Psi}|_{z=0}.
\eeq
%where $\tilde{\chi}$ is a smooth function such that$0\leq \tilde{\chi}\leq 1$ and$\tilde{\chi}\chi=\chi.$
Note that by the assumption \eqref{preasption} and the property \eqref{preasption1}, %we find can find a constant $C(\f{1}{c_0})$ (which is a polynomial in $\f{1}{c_0}$) and a constant $C(R,\overline{\Gamma})$ such that: 
\beq\label{prop-b}
c_0\leq b_0(t,y)\leq \f{1}{c_0}, \quad \forall \, (t,y)\in [0,T]\times\mR^2.
\eeq
Applying  \eqref{green-heat} with $a_1=a_2=a_3=0,$ we get the following formula:
\begin{align}\label{sec7:eq12}
\zeta_1^j(t,y,z)&=\mu \lambda_1\overline{\Gamma}\izt  \big(\f{1}{(\pi Y)^{\f{1}{2}}}\f{z}{2Y} e^{-\f{z^2}{4Y}} \big)(t,t', y) (b_0 \zeta^j|_{z=0}) (t',y)\,\d t',
\end{align}
where 
\beq\label{defY1}
Y=Y(t,t',y)=\mu \lambda_1\overline{\Gamma} \int_{t'}^t b_0(\tau,y)\, \d\tau. %\quad K_{\pm}(t,t', y, z', z)=\f{1}{(4\pi Y)^{\f{1}{2}}} e^{-\f{|z\pm z'|^2}{4Y}}.
\eeq
Note that by \eqref{prop-b}, we have that:
\beq\label{prop-Y}
\mu\lambda_1\overline{\Gamma}c_0(t-t')\leq Y(t,t',y)\leq \mu\lambda_1\overline{\Gamma}(t-t')/c_0, \quad \forall\, 0\leq t'\leq t\leq T,  y\in \times\mR^2.
 %\mu C(R,\overline{\Gamma},\lambda_1)\f{t-t'}{C(\f{1}{c_0})}\leq  Y(t,t',y)\leq \mu C(R,\overline{\Gamma},\lambda_1)C\big(\f{1}{c_0}\big) \quad \forall (t,y)\in [0,T]\times\mR^2.
\eeq 

We first show the $L^{\infty}_{t,x}$ estimate \eqref{zeta1infty} by assuming $1\leq j\leq [\f{m}{2}]-1.$
Let $\alpha=(\alpha_1,\alpha_2,\alpha_3)$ be a multi-index with length $|\alpha|\leq [\f{m}{2}]-1-j.$ It follows from the property
\eqref{appen-4} found and proved in the appendix that: there is a polynomial $P_{2|\alpha|+1}$ with degree $2|\alpha|+1,$ such that: 
\begin{align}
   \| \cZ^{\alpha}\zeta_1^j(t,\cdot, z) \|_{L_y^{\infty}(\mR^2)}
    &\lesssim\mu  \izt \f{1}{Y} P_{2|\alpha|+1}\big( \f{z}{Y^{\f{1}{2}}}\big) e^{-\f{z^2}{4Y}} \d t' \notag\\
&\qquad\cdot  \lab |b|_{[\f{m}{2}]-1-j,\infty,t}\big) \sum_{\alpha'\leq \alpha}\big|\cZ^{\alpha'}\zeta^j|_{z=0}\big|_{L^{\infty}([0,t]\times\mR^2)}. \label{zeta1inf-proof}
\end{align}
Thanks to \eqref{prop-Y}, \eqref{prop-b}
we find that for any $z\in \mR_{+}, 0\leq t\leq T,$ 
\begin{align*}
\mu \izt  \f{1}{Y} P_{2|\alpha|+1}\big( \f{z}{Y^{\f{1}{2}}}\big) e^{-\f{z^2}{4Y}} \d t'\lesssim \mu z^{-2} \izt  \tilde{P}_{2|\alpha|+3}\bigg( \big(\f{z^2}{\mu(t-t')}\big)^{\f{1}{2}}
\bigg) e^{-\f{z^2}{\mu\lambda_1\overline{\Gamma}(t-t')/c_0}} \d t'\leq C(\f{1}{c_0}),
\end{align*}
where we denote $\tilde{P}_{2|\alpha|+3}(\cdot)=|\cdot|^2  P_{2|\alpha|+1}(\cdot).$ 
Moreover, in view of the expression of $\zeta|_{z=0}$ in $\eqref{eqzeta1}_2$ and the definition \eqref{defb} for $b,$  we get that for any $j=1,2\cdots [\f{m}{2}]-1,$
\begin{align}
&\il \zeta_1^j\il_{*,[\f{m}{2}]-1-j,\infty,t}\lesssim \lab \il (\p_{\bn}\tsigma , r_0)\il_{0,\infty,t}+\il (\ep\pt)^j u\il_{*,[\f{m}{2}]-1-j,\infty,t}\big)\lesssim \lat,\notag\\
&\il b_0\il_{*,[\f{m}{2}]-1-j,\infty,t}\lesssim \il b_0\il_{*, m-3,\infty,t}\lesssim \lab \il(\theta,\ep\tsigma)\il_{*, m-3,\infty,t}\big)\lesssim \lat. \label{binfty}
\end{align}
Inserting the previous three estimates into \eqref{zeta1inf-proof}, we 
find \eqref{zeta1infty}.
%that:\begin{align*}  \il\zeta_1^j\il_{[\f{m}{2}]-1-j,\infty,t}\lesssim \lat.\end{align*}

Let us now take a  vector field $\cZ^{\gamma}$ 
and start to estimate 
$\|\zeta_1^j(t)\|_{m-2-j}$ by assuming $0\leq |\gamma|\leq m-2-j\leq m-3.$
Thanks to the properties \eqref{prop-Y} and \eqref{appen-4}, we have by using Minkowski inequality that:
\begin{align}\label{sec7:eq13.5}
  & \|\cZ^{\gamma}\zeta_1^j(t)\|_{L^2(\mR_{+}^3)}\lesssim \mu \izt \f{1}{(\mu(t-t'))^{\f{3}{4}}}%\sum_{k\leq m-2-j}|\p_y^k  \zeta^j|_{z=0}|(t',y)
 |\zeta^j|_{z=0}(t',\cdot)|_{H^{m-2-j}(\mR^2)}   \,\d t'  \lab|b_0|_{ m-3,\infty,t}\big),
\end{align}
%Here we use the notation $\p_y^k f=\sum_{ \beta_1+\beta_2= k} \p_{y^1}^{\beta_1} \p_{y^2}^{\beta_2} f.$ 
%Noticing the definition \eqref{defb} for $b$  we get that:
which, combined with \eqref{binfty}, yields,
\begin{align}\label{sec7:eq19}
\|\cZ^{\gamma}\zeta_1^j(t)\|_{L^2(\mR_{+}^3)}&\lesssim (\mu T)^{\f{1}{4}} | \zeta^j|_{z=0}|_{L_t^{\infty}{H}^{m-2-j}(\mR_{+}^2)}\lat.%\lab\il (\theta, \ep\sigma)\il_{m-3,\infty,t}%+|\zeta|_{z=0}|_{L_{t,y}^{\infty}}
%&\lesssim (\mu T)^{\f{1}{4}} \|(\text{Id}, \na)(u,\theta,\ep\sigma)\|_{\infco^{m-2}}\lab\il (\theta,\ep\sigma)\il_{m-3,\infty,t}+\il u\il_{0,\infty,t}\big)\\&\lesssim (\mu T)^{\f{1}{4}}\lat\cE_{m,t}.\notag
\end{align}
We now detail the estimate of $\|\zeta_1^j(t)\|_{m-1-j,t}.$ Taking a multi-index $\alpha=(\alpha_1,\alpha_2,\alpha_3),$ with $0\leq |\alpha|\leq m-1-j,$
we use the property \eqref{appen-5} if $j=1$ and  \eqref{appen-4} if $2\leq j\leq m-2$  to find that
\begin{align*}%\label{sec7:eq25}
&\|\cZ^{\alpha}\zeta_1^j(s)\|_{L^2(\mR_{+}^3)}\notag\\
&\lesssim 
\int_0^s \f{\mu^{\f{1}{4}}}{(s-s')^{\f{3}{4}}} \big(\big|\zeta^j|_{z=0}(s')\big|_{{H}^{m-1-j}(\mR^2)}+|b_0|_{L_t^{\infty}{H}^{m-2}(\mR^2)}\big)\d s'\lab | b_0|_{m-3,\infty,t}+|\zeta^1|_{z=0}|_{L_{t,y}^{\infty}}\big).
\end{align*}
By using Young's inequality in time (after extending $\zeta|_{z=0}$ by zero to $t\leq 0$ and $t\geq T$), and the trace inequality \eqref{traceL2},
we obtain that for any $0<t\leq T,$
\begin{align}\label{sec7:eq26}
    \|\cZ^{\alpha}\zeta_1^j\|_{L^2([0, t]\times \mR_{+}^3)}
    &\lesssim (\mu T)^{\f{1}{4}} \big(
    \big|\zeta^j|_{z=0}\big|_{L_t^2H^{m-1-j}(\mR^2)}+|b_0|_{L_t^{\infty}\tilde{H}^{m-2}}\big)\lat.
\end{align}
In view of the estimates \eqref{sec7:eq19} and \eqref{sec7:eq26},  we can finish the proof of \eqref{zeta1energy} by showing the following estimates for 
$\zeta^j|_{z=0}$ and $b:$
\begin{align}\label{sec7:eq27}
    \mu^{\f{1}{4}} |\zeta^j|_{z=0}|_{L_t^{\infty}{H}^{m-2-j}\cap L_t^{2}{H}^{m-1-j}(\mR^2)}+|b_0|_{L_t^{\infty}{H}^{m-2}(\mR^2)}\lesssim \lae, \, 1\leq j\leq m-2.
\end{align}
%Their estimates  are based on the definitions $\eqref{eqzeta1}_2$ of $\zeta^j|_{z=0}$ and \eqref{defb} for $b$ and the trace inequality \eqref{traceLinfty}, \eqref{traceL2}.  
The estimate for $b$ follows from the definition \eqref{defb}, the  trace inequality \eqref{traceLinfty} and the estimate \eqref{esr0} for $r_0^{-1},$ we omit the detail.
Let us focus on the estimate of $\zeta^j|_{z=0}$
and detail the estimate of $ \mu^{\f{1}{4}} |\zeta^j|_{z=0}|_{L_t^{\infty}{H}^{m-2-j}(\mR^2)}.$ %and skip the other two since it is similar.
In view of its definition in $\eqref{eqzeta1}_2$
we can control this term in the following way.
If $m-2-j\leq [\f{m}{2}]-1,$ we use the trace inequality \eqref{traceLinfty}, the estimate \eqref{Es-tsigma4} for $\tsigma$ to find that:
\begin{align*}
 |\zeta^j|_{z=0}|_{L_t^{\infty}{H}^{m-2-j}(\mR^2)}&\lesssim
     |u^j|_{L_t^{\infty}H^{m-2-j}(\p\Omega)}\lab  \il \theta, \ep\tsigma, \ep\na\tsigma\il_{[\f{m}{2}]-1,\infty,t} \big)\\
  &\lesssim  \|(u,\na u)\|_{\infco^{m-2}}\lat\lesssim \lat\cE_{m,t}.
\end{align*}
If $m-2-j\geq [\f{m}{2}],$ or $j\leq [\f{m-3}{2}],$  we use successively the classical Gagliardo-Nirenberg inequality (in $\mR^2$),  the 
trace inequality \eqref{traceLinfty}, the estimate \eqref{esr0} for $r_0,$ \eqref{Es-tsigma3} for $\tsigma$
to obtain that
\begin{align*}
 \mu^{\f{1}{4}}|\zeta^j|_{z=0}|_{L_t^{\infty}{H}^{m-2-j}(\mR^2)}&\lesssim
     \mu^{\f{1}{4}} \big|(u^j, \ep\p_{\bn}\tsigma, r_0)
     \big|_{L_t^{\infty}H^{m-2-j}(\p\Omega)}\lab  \il (u^j, \p_{\bn}\sigma, r_0)\il_{0,\infty,t} \big)\\
  &\lesssim \big( \|(u,\na u)\|_{\infco^{m-2}}+\|
  \na r_0\|_{\infco^{m-3}}+ |\ep\p_{\bn}\tsigma |_{L_t^{\infty}H^{m-3}(\p\Omega)} \big)\lat\\
  &\lesssim \lat\cE_{m,t}.
\end{align*}
The  $L_t^{2}{H}^{m-1-j}(\mR_{+}^2)$ norm of $\zeta^j|_{z=0}$ can be controlled in a similar way. More precisely, one can  use the trace inequality \eqref{traceL2} and the estimates \eqref{Es-tsigma2} \eqref{Es-tsigma3} for $\tsigma$ to get:
\begin{align*}
    & \mu^{\f{1}{4}}|\zeta^j|_{z=0}|_{L_t^{2}{H}^{m-1-j}(\mR^2)}\lesssim \mu^{\f{1}{4}} \big|(u^j, \ep\p_{\bn}\tsigma, r_0)
    \big|_{L_t^{2}H^{m-1-j}(\p\Omega)}\lat\\
     &\qquad\lesssim \big(\|(u,\mu^{\f{1}{2}}\na u)\|_{\hco^{m-1}}+
    \|(\Id, \na)(\theta, \ep\tsigma)\|_{\hco^{m-2}}+\ep|\p_{\bn}\tsigma |_{L_t^{2}H^{m-2}(\p\Omega)}\big)\lat\\
    &\qquad\lesssim \lat\cE_{m,t}.
\end{align*}
The proof of \eqref{sec7:eq27} is now complete. %and so is \eqref{zeta1energy}. 
\end{proof}
%The following lemma is   We to the estimate of $\zeta_2^j.$The following proposition states some estimates for $\zeta_2^j.$ 
To show \eqref{zetalow}, it remains to control $\zeta_2^j$ (defined in \eqref{eqzeta2}),
which is the aim of the following proposition.
\begin{prop}\label{lemze2}
Let $\zeta_2^j\, (j=1,2\cdots m-2)$ be the solution to the system \eqref{eqzeta2}. Suppose that the assumption \eqref{preasption} holds and $m\geq 7.$  There exists $\vartheta_4\in (0,\f{1}{2}],$ 
such that for  any $\ep \in (0,1], (\mu,\kpa)\in A,$ any $0< t\leq T,$
\begin{equation}\label{EE-zeta2}
  \|\zeta_2^j\|_{m-2-j,t}^2+\mu \|\na \zeta_2^j\|_{m-2-j,t}^2 \lesssim \lab Y_m(0)\big) + (T+\ep)^{\vartheta_4}\lae.
\end{equation}
\end{prop}
\begin{proof}
Taking a multi-index $\gamma=(\gamma_1,\gamma_2,\gamma_3)$ with $|\gamma|=k\leq m-2-j$ and applying $\cZ^{\alpha}$ on the equations \eqref{eqzeta2}, we find that $\zeta_{2}^{j,\gamma}=\colon\cZ^{\gamma}\zeta_2^j$ solves the following equation:
\begin{align}\label{eqzeta2jga}
\left\{\begin{array}{l}
      (\pt+\tilde{u}\cdot\na)\zeta_{2}^{j,\gamma}-\mu\lambda_1 \overline{\Gamma}\big(\p_z (\cZ^{\gamma}\p_z\zeta_2^{j}/{r}^{\Psi}_0)+\p_{y^k}(({g^{kl}}/{r}^{\Psi}_0)\p_{y^l}\zeta_2^{j,\gamma})\big)=\cZ^{\gamma}\cS_2^j+\sum_{i=1}^{5}\cR_{i}^{j,\gamma}, \\
   \zeta_{2}^{j,\gamma}|_{z=0}=0,\quad \zeta_{2}^{j,\gamma}|_{t=0}=\cZ^{\gamma}(\zeta^j|_{t=0}),
\end{array}
\right.
\end{align}
where $\cS_2^j$ is defined in \eqref{source2} and 
\beq\label{defcR}
\begin{aligned}
\cR_{1}^{j,\gamma}&=-[\cZ^{\gamma},\tlu \cdot\na]\zeta_2^j,\quad
\cR_{2}^{j,\gamma}=\mu\lambda_1 \overline{\Gamma}\cZ^{\gamma}\big(\p_z^2(1/r^{\Psi}_0)\zeta_2^j\big),\\
\cR_{3}^{j,\gamma}&=-\mu\lambda_1 \overline{\Gamma}([\cZ^{\gamma},\p_z]+\p_z\cZ^{\gamma})\big(\p_z(1/r^{\Psi}_0)\zeta_2^j\big),\\%=\colon \cR_{4,1}^{\gamma}+\cR_{4,2}^{\gamma}.\\
\cR_{4}^{j,\gamma}&=\mu\lambda_1 \overline{\Gamma}\big([\cZ^{\gamma},\p_z](\p_z\zeta_2^{j}/{r}^{\Psi}_0)+\p_z([\cZ^{\gamma}, 1/r^{\Psi}_0]\p_z\zeta_2^j)\big),\\
\cR_{5}^{j,\gamma}&=\mu\lambda_1 \overline{\Gamma}\bigg(\p_{y^k}\big([\cZ^{\gamma}, {g^{kl}}/{r}^{\Psi}_0]\p_{y^l}\zeta_2^j\big)-\cZ^{\gamma}\big(\p_{y^k}({|g|^{-\f{1}{2}}}/{r}^{\Psi}_0)g^{kl}|g|^{\f{1}{2}}\p_{y^l}\zeta_2^j\big)\bigg).
\end{aligned}
\eeq
Note that Einstein's summation convention has been employed. %Moreover, for notational simplicity  and without much ambiguity, we denote hereafter $(\f{\tilde{\chi}}{r_0})^{\Psi}$  by $(\f{1}{r_0})^{\Psi}.$

Performing the standard energy estimates on  \eqref{eqzeta2jga} and using the boundary condition \eqref{tlu3-bd}, we find that:
\begin{align}\label{EI-zeta2}
   & \f{1}{2}\int_{\mR_{+}^3} |\zeta_2^{j,\gamma}|^2(t)\,\d x 
    +\mu\lambda_1 \overline{\Gamma}\izt\int_{\mR_{+}^3}\big( | \p_z \zeta_2^{j,\gamma}|^2+ g^{kl}\p_{y^k} \zeta_2^{j,\gamma} \p_{y^l}\zeta_2^{j,\gamma}\big)/ r^{\Psi}_0\,\d x\d s\notag\\
    & =\f{1}{2}\int_{\mR_{+}^3} |\zeta_2^{j,\gamma}|^2(0)\,\d x -\mu\lambda_1 \overline{\Gamma}\izt\int_{\mR_{+}^3} [\cZ^{\gamma},\p_z] \zeta_2^{j}\cdot \p_z \zeta_2^{j,\gamma} /r^{\Psi}_0 \,\d x\d s\\
    & \quad+
  \f{1}{2}  \izt\int_{\mR_{+}^3}|\zeta_2^{j,\gamma}|^2 \div \tlu  \,\d x\d s  + \izt\int_{\mR_{+}^3}\zeta_2^{j,\gamma}  \cdot \big(\cZ^{\gamma}\cS_2^j+\sum_{i=1}^{4}\cR_{i}^{j,\gamma}\big)
    \,\d x\d s. \notag
\end{align}
Since the matrix $g$ defined in \eqref{defmatrixg} is positive definite, and $r_0^{\Psi}\leq \f{1}{c_0}$ by the assumption \eqref{preasption} and the property \eqref{preasption1},
we can find a constant $\iota>0$ which are independent of $\ep,\mu,\kpa,$ such that
\beq\label{sec7:eq40}
\mu\lambda_1 \overline{\Gamma}\izt\int_{\mR_{+}^3}\big( | \p_z \zeta_2^{j,\gamma}|^2+ g^{kl}\p_{y^k} \zeta_2^{j,\gamma} \p_{y^l}\zeta_2^{j,\gamma}\big)/ r^{\Psi}_0\,\d x\d s\geq \iota c_0\mu\|\na \zeta_2^{j,\gamma} \|_{0,t}.%_{L_t^2L^2(\mR_{+}^3)}^2.
\eeq
Moreover, by induction, one gets that for $k=|\gamma|\geq 1,$
\begin{equation}\label{com-identity}
\begin{aligned}
 [\mathcal{Z}^{\gamma},\partial_z]=
\sum_{\beta<\gamma}
 C_{\phi,\beta,\gamma}\partial_z\mathcal{Z}^{\beta}
\end{aligned}
\end{equation}
where $C_{\phi,\beta,\gamma}$ are smooth functions that depend on $\phi$ and its derivatives. We can thus control the second term in the right hand side of \eqref{EI-zeta2} by using the 
Young's inequality: 
\begin{align}\label{sec7:eq40.5}
   & \mu\lambda_1 \overline{\Gamma}\izt\int_{\mR_{+}^3} [\cZ^{\gamma},\p_z] \zeta_2^{j}\cdot \p_z \zeta_2^{j,\gamma} /r^{\Psi}_0 \,\d x\d s\notag\\
    &\leq \delta c_0\mu
    \|\na \zeta_2^{j,\gamma} \|_{0,t}^2+ C(\delta,\f{1}{c_0})\mu \sum_{0\leq |\beta|\leq |\gamma|-1}
    \|\na\cZ^{\beta}\zeta_2^j\|_{0,t}.%{L_t^2L^2(\mR_{+}^3)}^2
\end{align}
Next, by using the fact \eqref{norm-equiv-bd} and the estimate \eqref{zeta1energy}, we have that: 
\beq\label{zeta2j-pre}
\begin{aligned}
\|\zeta_2^{j}\|_{m-2-j,t}&\lesssim \|(\zeta^j, \zeta_2^{j})\|_{m-2-j,t}\\
&\lesssim T^{\f{1}{2}}(\|\curl(r_0 (\ep\pt)^j u)\|_{\infcoch^{m-2-j}}+\sup_{0\leq s\leq t} \|\zeta_2^{j}(s)\|_{m-2-j}\big)\\
&\lesssim T^{\f{1}{2}}\lat\cE_{m,t},
\end{aligned}
\eeq
which, allows us to control the third term  as:
\begin{align}\label{sec7:eq41.5}
 \izt\int_{\mR_{+}^3}|\zeta_2^{j,\gamma}|^2 \div \tlu  \,\d x\d s &\leq   \|\zeta_2^{j}\|_{m-2-j,t}^2\il \div \tilde{u}\il_{0,\infty,t}\lesssim T\lat\cE_{m,t}^2.
\end{align}
Our following task is to control the commutator terms 
$ \izt\int_{\mR_{+}^3}\zeta_2^{j,\gamma}  \cdot \cR_{i}^{j,\gamma} \,\d x\d s$
($\cR_1^{j,\gamma}-\cR_5^{j,\gamma}$ are defined in \eqref{defcR}), 
which rely on the estimates for $\cR_i^{j,\gamma}$ presented in 
Lemma \ref{lemcR}.
We first use  \eqref{zeta2j-pre} and the estimate \eqref{es-cR1-2} to get that,
\begin{align}\label{sec7:eq41}
    \izt\int_{\mR_{+}^3}\zeta_2^{j,\gamma}  \cdot (\cR_{1}^{j,\gamma}+\cR_{2}^{j,\gamma}) \,\d x\d s\leq \|\zeta_2^{j}\|_{m-2-j,t}
    \| (\cR_1^{j,\gamma},\cR_2^{j,\gamma})\|_{0,t}\lesssim T^{\f{1}{2}}\lat\cE_{m,t}^2.
\end{align} 

Next, in view of the identity \eqref{com-identity},
the  vanishing boundary condition \eqref{eqzeta2jga}, we can integrate by parts in space and use the Young's inequality as well as the estimates \eqref{es-cR3} \eqref{es-cR4} to obtain that:
\begin{align}
  \izt\int_{\mR_{+}^3}\zeta_2^{j,\gamma}  \cdot \cR_{3}^{j,\gamma} \,\d x\d s&\lesssim \mu \|\p_z \zeta_2^{j,\gamma}\|_{0,t} \|\p_z(1/r^{\Psi}_0)\zeta_2^j\|_{m-2-j,t}\notag\\
&\leq \delta c_0\mu  \|\p_z \zeta_2^{j,\gamma}\|_{0,t}^2+T\lat\cE_{m,t}^2, \\
\izt\int_{\mR_{+}^3}\zeta_2^{j,\gamma}  \cdot \cR_{4}^{j,\gamma} \,\d x\d s&\lesssim \mu \|\p_z \zeta_2^{j,\gamma}\|_{0,t} \|\p_z\zeta_2^{j}\|_{m-3-j,t}\il 1/r^{\Psi}_0\il_{*,m-3,\infty,t} \notag\\%\lab \big)\\
&\leq \delta c_0\mu  \|\p_z \zeta_2^{j,\gamma}\|_{0,t}^2+C_{\delta}
\mu\|\p_z\zeta_2^{j}\|_{k-1,t}^2\lab\il(\theta,\ep\tsigma)\il_{m-3,\infty,t}^2\big).
\end{align}
%By Cauchy-Schwarz inequality, the estimate \eqref{zeta2j-pre} and \eqref{es-cR4}, we find that: 
%\beq\begin{aligned}\izt\int_{\mR_{+}^3}\zeta_2^{j,\gamma}  \cdot \cR_{4,1}^{\gamma} \,\d x\d s\lesssim \|\zeta_2^{j}\|_{m-2-j,t}\|\cR_{4,1}^{\gamma}\|_{0,t}\lesssim T\lat\cE_{m,t}^2.\end{aligned}\eeq
%By integration by parts and Young's inequality, we get 
%\beq\label{sec7:eq42}\begin{aligned}\izt\int_{\mR_{+}^3}\zeta_2^{j,\gamma}  \cdot \cR_{4,2}^{\gamma} \,\d x\d s&\lesssim\mu\|\p_z\zeta_2^{j,\gamma}\|_{0,t}\|\p_z(1/r^{\Psi}_0)\zeta_2^j\|_{m-2-j,t}\\&\leq \delta c_0\mu \|\p_z\zeta_2^{j,\gamma}\|_{0,t}^2+ T\lat\cE_{m,t}^2.\end{aligned}\eeq
Finally, it follows from the estimates \eqref{zeta2j-pre}, \eqref{es-cR5} and Young's inequality that:
\beq\label{sec7:eq44}
\begin{aligned}
\izt\int_{\mR_{+}^3}\zeta_2^{j,\gamma}  \cdot \cR_{5}^{j,\gamma} \,\d x\d s&\leq \|\zeta_2^{j}\|_{m-2-j,t}
    \| \cR_5^{\gamma}\|_{0,t}\\
    &\leq \delta c_0\mu
    \|\na \zeta_2^{j} \|_{m-2-j,t}^2+ C(\delta) T\lat\cE_{m,t}^2.
 \end{aligned}
\eeq
Collecting \eqref{sec7:eq41}-\eqref{sec7:eq44}, 
we find:
\beq\label{sec7:eq45}
\begin{aligned}
   \sum_{i=1}^5 \izt\int_{\mR_{+}^3}\zeta_2^{j,\gamma}  \cdot \cR_{i}^{j,\gamma} \,\d x\d s&\leq 3\delta c_0\mu
    \|\na \zeta_2^{j} \|_{m-2-j,t}^2+ T^{\f{1}{2}}\lat\cE_{m,t}^2\notag\\
   &\quad+C_{\delta}\mu\|\p_z\zeta_2^{j}\|_{m-3-j,t}^2\lab\il(\theta,\ep\tsigma)\il_{m-3,\infty,t}^2\big).
\end{aligned}
\eeq
We are now left to control the term 
 $\izt\int_{\mR_{+}^3}\zeta_2^{j,\gamma}  \cdot \cZ^{\gamma}\cS_2^j\,\d x\d s$ appearing in 
 \eqref{EI-zeta2}. According to the definition \eqref{source2} for $\cS_2^j,$ it can be written as the sum of the following four terms:
 \begin{align*}
  \mathscr{S}_1&=-\izt\int_{\mR_{+}^3}\zeta_2^{j,\gamma}  \cdot  \cZ^{\gamma}((\tilde{u} \cdot \na){\zeta_{1}^j} ) \,\d x\d s,\\
  \mathscr{S}_2&=\mu \lambda_1\overline{\Gamma}\izt\int_{\mR_{+}^3}\zeta_2^{j,\gamma}  \cdot  \cZ^{\gamma} \bigg(\big( \f{1}{r_0}\big)^{\Psi}\Delta_{{g}}\zeta_{1}^j+ \big(\big( \f{1}{r_0}\big)^{\Psi}-\big( \f{1}{r_0}\big)^{\Psi}\big|_{z=0
}\big)\p_z^2\zeta_{1}^j\bigg) \,\d x\d s,\\
   \mathscr{S}_3&=\izt\int_{\mR_{+}^3}\zeta_2^{j,\gamma}  \cdot  \cZ^{\gamma} H^{j,\dag} \,\d x\d s , \qquad \mathscr{S}_4=\izt\int_{\mR_{+}^3}\zeta_2^{j,\gamma}  \cdot \cZ^{\gamma}( |g|^{\f{1}{4}} (H^j)^{\Psi} )\,\d x\d s .
 \end{align*}
 The control of these terms are based on the properties for the terms contained in $\cS_2^j,$ 
 which are listed in Lemma \ref{lemsource2}.
Thanks to the estimate \eqref{zeta2j-pre}, the estimate \eqref{convection-zeta1} for the convection term $(\tilde{u} \cdot \na){\zeta_{1}^j},$ we control the first term directly as:
\begin{align}\label{sec7:eq46}
    \mathscr{S}_1\lesssim \|\zeta_2^{j}\|_{m-2-j,t}  \| (\tilde{u} \cdot \na){\zeta_{1}^j} \|_{m-2-j,t}\lesssim T^{\f{1}{2}}\lat\cE_{m,t}^2.
\end{align}
Moreover, by the virtue of the identities \eqref{diffusion-zeta1}, \eqref{com-identity}, we integrate  by parts in space to control the term $\mathscr{S}_2$ as 
\begin{align*}
    \mathscr{S}_2 \lesssim \mu \|\na \zeta_2^{j,\gamma}\|_{0,t} \|(\cS_{21}^j, \cS_{22}^j, \cS_{23}^j)\|_{m-2-j,t}+ \|\zeta_2^j\|_{m-2-j,t} \|\cS_{24}^j\|_{0,t}.
\end{align*}
 It then follows from Young's inequality, the estimates \eqref{zeta2j-pre}, \eqref{es-viscous-zeta1} that 
\beq 
\mathscr{S}_2\leq \delta c_0\mu  \|\na \zeta_2^{j,\gamma}\|_{0,t}^2+T^{\f{1}{2}}\lat\cE_{m,t}^2.
\eeq
Next, in view of the definition \eqref{defHdag} for $H^{j,\dag}$ and the fact \eqref{norm-equiv-bd},
%and the estimate for $\zeta_1^j$ in \eqref{zeta1energy}, 
we can have that:
%:
\begin{align*}
    \|H^{j,\dag}\|_{m-2-j,t}&\lesssim  \|(\tlu, \zeta^j, \mu^{\f{1}{2}}\p_y \zeta^j)\|_{m-2-j,t}\lat\\
 &\lesssim \|(u, (\Id, \mu^{\f{1}{2}}\p_y) \curl(r_0 u^j)) \|_{L_t^2\cH^{m-2}}\lat\lesssim \lat\cE_{m,t}.
 %\mu^{\f{1}{2}} \|\na \zeta_2^j\|_{m-2-j,t}+ T^{\f{1}{4}} \lat \cE_{m,t}.
\end{align*}
Consequently, one finds by the estimate \eqref{zeta2j-pre} that:
\begin{align}
   \mathscr{S}_3\lesssim  \|\zeta_2^{j}\|_{m-2-j,t} \|H^{j,\dag}\|_{m-2-j,t}\leq
   T^{\f{1}{2}}\lat\cE_{m,t}^2.%\delta c_0 \mu \|\na \zeta_2^j\|_{m-2-j,t}+ C_{\delta} T^{\f{3}{4}}\lat\cE_{m,t}^2.
\end{align}
It remains to control $\mathscr{S}_4,$ which requires a closer look on $(H^j)^{\Psi}$ defined in \eqref{defH}.  We will show in Lemma \ref{lemsource2} that it takes the form %can be written in the form
\beqs 
(H^j)^{\Psi}= \mu^{\f{1}{2}} (\p_{y^1}\cH^{j}_1+\p_{y^2}\cH^{j}_2+\p_z \cH^{j}_3) +\cH_{4}^j, 
\eeqs
where $\cH_1^j-\cH_4^j$ are some vector valued functions with amenable property \eqref{prop-ch}. %Consequently, successive uses of integration by parts, Young's inequality, the estimate \eqref{zeta2j-pre} give rise to:
As a result, it follows successively the integration by parts, Young's inequality and the estimate \eqref{zeta2j-pre} that:
\beq\label{sec7:eq47}
\begin{aligned}
   \mathscr{S}_4&\lesssim \mu \|\na \zeta_2^j\|_{m-2-j,t}
   \|(\cH_1^j,\cH_2^j,\cH_3^j)\|_{m-2-j,t}+ \| \zeta_2^j\|_{m-2-j,t} \|\cH_4^j\|_{m-2-j,t}\\
  & \leq \delta c_0 \mu \|\na \zeta_2^j\|_{m-2-j,t}+ C_{\delta} (T+\ep)^{\f{1}{2}}\lat\cE_{m,t}^2.
\end{aligned}
\eeq
%Collecting 
In view of \eqref{sec7:eq46}-\eqref{sec7:eq47}, we have obtained that:
\begin{align}\label{sec7:eq48}
    \izt\int_{\mR_{+}^3}\zeta_2^{j,\gamma}  \cdot \cZ^{\gamma}\cS_2^j\,\d x\d s\leq 3 \delta c_0 \mu \|\na \zeta_2^j\|_{m-2-j,t}+ C_{\delta} T^{\f{1}{2}}\lat\cE_{m,t}^2.
\end{align}
 Plugging the estimates \eqref{sec7:eq40} \eqref{sec7:eq40.5},  \eqref{sec7:eq41.5}, \eqref{sec7:eq45} \eqref{sec7:eq48} into \eqref{EI-zeta2}, one finds by choosing $\delta$  sufficiently small compared to $\iota$ that for any $0\leq k\leq m-2-j,$
 \beqs
 \begin{aligned}
    &\qquad \|\zeta_2^j\|_{k,t}^2+\mu \|\na \zeta_2^j\|_{k,t}^2\\
     &\lesssim_{\iota} Y_m^2(0)+ T^{\f{1}{2}}\lat\cE_{m,t}^2+\mu\|\p_z\zeta_2^{j}\|_{k-1,t}^2\lab\il(\theta,\ep\tsigma)\il_{m-3,\infty,t}^2\big).
 \end{aligned}
 \eeqs
 Note that the convention $\|\cdot\|_{l,t}=0$ if $l=-1$ has been employed.  This estimate in hand, we can derive by induction on $k$ that:
  \begin{align}\label{zeta2j-final-1}
     &\qquad \|\zeta_2^j\|_{m-2-j,t}^2+\mu \|\na \zeta_2^j\|_{m-2-j,t}^2 \lesssim \lab\il(\theta,\ep\tsigma)\il_{m-3,\infty,t}^2+Y_m^2(0)\big) + (T+\ep)^{\f{1}{2}}\lae. 
 \end{align}
 However, by the Sobolev embedding \eqref{sobebd}, the estimate \eqref{Es-tsigma2} for $\tsigma,$
 one has:
 \beq\label{tsigmainfty}
 \begin{aligned}
      \ep \il\tsigma\il_{m-3,\infty,t}\lesssim \ep \|\tsigma\|_{\infco^{m-1}}+\ep\|\na\tsigma\|_{\infco^{m-2}}\lesssim \ep\lae.
 \end{aligned}
 \eeq
 Inserting the above estimate and \eqref{thetainfty} into \eqref{zeta2j-final-1} and recalling that 
 $\Lambda$ are polynomials, we thus find a constant $\vartheta_4\in (0,\f{1}{2}],$ which is independent of $\ep\in(0,1], (\mu,\kpa)\in A,$ such that, %, by changing $\Lambda$ correspondingly,
\begin{align*}
    \|\zeta_2^j\|_{m-2-j,t}^2+\mu \|\na \zeta_2^j\|_{m-2-j,t}^2 \lesssim \lab Y_m(0)\big) + (T+\ep)^{\vartheta_4}\lae.
\end{align*}
 %\begin{align*}   &\qquad \|\zeta_2^j\|_{k,t}^2+\mu \|\na \zeta_2^j\|_{k,t}^2 \lesssim \lab Y_m^2(0)\big) + T^{\f{1}{2}}\lae \quad \forall\, 0\leq k\leq m-3,\end{align*}
 %which is \eqref{EE-zeta2}. Inserting 
\end{proof}
In the following two lemmas, we state and prove some estimates for $\cR_1^{j,\gamma}-\cR_5^{j,\gamma}$ and 
$\cS_2^j$ respectively, which have been used in the proof of the previous proposition.
\begin{lem}\label{lemcR}
%Using the notation \eqref{newnotation} and recalling the definition  for
Let $\cR_1^{j,\gamma}-\cR_5^{j,\gamma}$ be defined in \eqref{defcR}. Under the assumption \eqref{preasption}
we have 
for any $1\leq j\leq m-2, |\gamma|\leq m-2-j,$ any 
$\ep\in (0,1], (\mu,\kpa)\in A,$ any $0<t\leq T,$ 
\begin{align}
&\|(\cR_{1}^{j,\gamma}, \cR_{2}^{j,
\gamma})\|_{0,t}\lesssim %T^{\f{1}{2}}
\lat\cE_{m,t}, \label{es-cR1-2}\\
%&\|\cR_{2}^{\gamma}\|_{0,t}\lesssim T^{\f{1}{2}}\lat\cE_{m,t}, \label{es-cR2}     \\
&\|\p_z(1/r^{\Psi}_0)\zeta_2^j\|_{m-2-j,t}\lesssim T^{\f{1}{2}} \lat\cE_{m,t},\label{es-cR3}\\
&\|\p_z\zeta_2^{j}/{r}^{\Psi}_0\|_{m-3-j,t}+ \|[\cZ^{\gamma}, 1/r^{\Psi}_0]\p_z\zeta_2^j\|_{0,t} \lesssim   \|\p_z\zeta_2^{j}\|_{m-3-j,t}\lab\il(\theta,\ep\tsigma)\il_{m-3,\infty,t}\big),  \label{es-cR4} \\
&\|\cR_{5}^{j,\gamma}\|_{0,t}\lesssim \mu^{\f{1}{2}}\|\na\zeta_2^j\|_{m-2-j,t}\lat+ T^{\f{1}{2}}\lat\cE_{m,t}.\label{es-cR5}
\end{align}
%We refer to \eqref{newnotation} for the definition of the norms. 
\end{lem}
\begin{proof}
 We first control $\cR_1^{j,\gamma}$ in $L_t^2L^2$ norm. Let us write:
\beq\label{com-convetion-id}
[\cZ^{\gamma},\tlu \cdot\na]=[\cZ^{\gamma},\tlu_1 ]\p_{y^1}+[\cZ^{\gamma},\tlu_2 ]\p_{y^2}+[\cZ^{\gamma},\tlu_3/\phi ]\cZ_3.%\zeta_2^j
\eeq
Therefore, by arguing separately for $1\leq j\leq [\f{m}{2}]-2$ and $j\geq [\f{m}{2}]-1,$ we can have that:
\begin{align}\label{escR1-proof}
    \|[\cZ^{\gamma},\tlu \cdot\na]\zeta_2^j\|_{0,t}&\lesssim %\left\{ \begin{array}{l}
       \|(\tlu_1, \tlu_2, \tlu_3/\phi)\|_{m-3,t}\il\zeta_2^j\il_{[\f{m}{2}]-1-j,\infty,t} \mathbb{I}_{\{1\leq j\leq [\f{m}{2}]-2\}}\notag\\
       &\quad +\sup_{0\leq s\leq t}\|\zeta_2^j(s)\|_{m-2-j} \big(\izt \|(\tlu_1, \tlu_2, \tlu_3/\phi)(s)\|_{[\f{m-1}{2}],\infty}^2\big)^{\f{1}{2}}.
\end{align}
For instance, if $1\leq j\leq [\f{m}{2}]-2,$ we control 
$[\cZ^{\gamma},\tlu_3/\phi ]\cZ_3\zeta_2^j$ as:
\begin{align*}
   \|[\cZ^{\gamma}, \tlu_3/\phi ]\cZ_3\zeta_2^j\|_{0,t}\lesssim \|\tlu_3/\phi\|_{m-3,t}\il\zeta_2^j\il_{[\f{m}{2}]-1-j,\infty,t}+\sup_{0\leq s\leq t}\|\zeta_2^j(s)\|_{m-2-j} \big(\izt \|( \tlu_3/\phi)(s)\|_{[\f{m-1}{2}],\infty}^2\big)^{\f{1}{2}} %\|\zeta_2^j\|_{m-2-j,t}  \il\tlu_1\il_{[\f{m-1}{2}],\infty,t}
\end{align*}
while if $j\geq [\f{m}{2}]-1, |\gamma|\leq m-2-j\leq [\f{m-1}{2}],$ we have: 
\begin{align*}
   \|[\cZ^{\gamma}, \tlu_3/\phi ]\cZ_3\zeta_2^j\|_{0,t}\lesssim
   \sup_{0\leq s\leq t}\|\zeta_2^j(s)\|_{m-2-j} \big(\izt \|( \tlu_3/\phi)(s)\|_{[\f{m-1}{2}],\infty}^2\big)^{\f{1}{2}}.%\|\zeta_2^j\|_{m-2-j,t}  \il\tlu_1\il_{[\f{m-1}{2}],\infty,t}.
\end{align*}
Since $\tlu_3|_{z=0}=0,$ we can use the Hardy inequality and the fundamental theorem of calculus respectively to obtain that:
\begin{align}\label{sec7:eq80}
   \|\tlu_3/\phi\|_{k,t}\lesssim \|\na \tlu_3\|_{k,t},\quad
   \|\tlu_3(s)/\phi \|_{k,\infty}\lesssim \il \p_z \tlu_3(s)\|_{k,\infty}.
\end{align}
Note that by the definition of $\tlu$ in \eqref{deftildeu}, one could find a smooth vector valued function $\psi$ which depends only on the boundary parametrization $\varphi,$ such that 
\beqs 
\p_z \tlu_3=(\p_{\bn}(\chi u)\cdot \bn)^{\Psi}+\psi \cdot z\p_z(\chi u)^{\Psi}+ \p_z (D\Psi)^{-*} (\chi u)^{\Psi}.
\eeqs
This, together with \eqref{sec7:eq80}, leads to that:
\begin{align}\label{sec7:eq80.5}
 \izt \| \tlu_3/\phi(s)\|_{[\f{m-1}{2}],\infty}^2\,\d s   \lesssim  \izt \|\div u(s)\|_{[\f{m-1}{2}],\infty}^2+\|u(s)\|_{[\f{m+1}{2}],\infty}^2\,\d s\lesssim \lat.
\end{align}
Moreover, in view of \eqref{zeta1infty}, for any $1\leq j\leq [\f{m}{2}]-2$,
\beq\label{sec7:eq81}
\il\zeta_2^j\il_{[\f{m}{2}]-1-j,\infty,t} \lesssim \il(\zeta^j, \zeta_1^j)\il_{[\f{m}{2}]-1-j,\infty,t}\lesssim \il(\curl (r_0 u^j))^{\Psi}, \zeta_1^j\il_{[\f{m}{2}]-1-j,\infty,t}\lesssim \lat.
\eeq
The estimate for $\cR_1^{j,\gamma}$ in \eqref{es-cR1-2} then follows from 
the facts \eqref{norm-equiv-bd}, \eqref{zeta2j-pre} and the estimates \eqref{escR1-proof}-\eqref{sec7:eq81}.

Let us switch to the control of $\|\cR_2^{\gamma}\|_{0,t}.$ We distinguish the case $1\leq j\leq [\f{m}{2}]-1,$ and
$j\geq [\f{m}{2}]$ to find  that:
%If $1\leq j\leq [\f{m}{2}]-1,$ we can apply the generalized Gagliardo-Nirenberg inequality \eqref{} to obtain that:we thus can have naively that:
%\begin{align*} \|\cR_2^{\gamma}\|_{0,t}\lesssim \il\mu   \p_z^2(1/r^{\Psi}_0)\il_{[\f{m}{2}]-1,\infty,t}\|\zeta_2^j\|_{m-2-j,t}\end{align*}
\begin{align*}   \|\cR_2^{j,\gamma}\|_{0,t}&\lesssim  \mu \| \p_z^2(1/r^{\Psi}_0)\zeta_2^j \|_{m-3,t}\\  &\lesssim \mu \|\p_z^2(1/r^{\Psi}_0)\|_{m-2-j,t}\il\zeta_2^j\il_{[\f{m}{2}]-1,\infty,t}+\|\zeta_2^j\|_{m-2-j,t} \il\mu\p_z^2(1/r^{\Psi}_0)\il_{[\f{m}{2}]-1,\infty,t}.
\end{align*}
It results from \eqref{esr0-sec} and the definition \eqref{defr0} for $r_0$ that:
\begin{align*}
      \mu \| \p_z^2(1/r^{\Psi}_0)\|_{m-3,t}&\lesssim T^{\f{1}{2}} \mu \|(\na, \na^2) r_0^{-1}\|_{\infco^{m-3}}\lesssim T^{\f{1}{2}}\lat\cE_{m,t},\\
        \il\mu\p_z^2(1/r^{\Psi}_0)\il_{[\f{m}{2}]-1,\infty,t}&\lesssim \il (\Id, \na, \mu \na^2)(\theta,\ep\tsigma) \il_{[\f{m}{2}]-1,\infty,t}\lesssim \lat.
\end{align*}
 The proof of \eqref{es-cR1-2} is thus finished. 

Next, the property \eqref{es-cR3} follows from similar arguments as above,
    \begin{align*}
       \|\p_z(1/r^{\Psi}_0)\zeta_2^j\|_{m-2-j,t}
      & \lesssim (\|\p_z(1/r^{\Psi}_0)\|_{m-3,t}+\|\zeta_2^j\|_{m-2-j,t})(\|\p_z(1/r^{\Psi}_0)\|_{[\f{m}{2}]-1,\infty,t}+\il\zeta_2^j\il_{[\f{m}{2}]-1-j,\infty,t})\\
      & \lesssim T^{\f{1}{2}} \lat\cE_{m,t}.
    \end{align*}

As for the proof of \eqref{es-cR4},  by noticing the fact $|\gamma|\leq m-2-j\leq m-3,$ we attribute the $L_{t,x}^{\infty}$ norm on $1/r^{\Psi}_0$ to obtain:
\begin{align*}
    \|\p_z\zeta_2^{j}/{r}^{\Psi}_0\|_{m-3-j,t}+ \|[\cZ^{\gamma}, 1/r^{\Psi}_0]\p_z\zeta_2^j\|_{0,t} \lesssim   \|\p_z\zeta_2^{j}\|_{m-3-j,t}\il 1/r^{\Psi}_0\il_{m-3,\infty,t}.
\end{align*}
Observe that by the definition \eqref{defr0} for $r_0,$
\begin{align}\label{r0infty}
\il 1/r^{\Psi}_0\il_{m-3,\infty,t}&\lesssim 
 \lab\il(\theta,\ep\tsigma)\il_{m-3,\infty,t}\big).%\lesssim \lat. 
 \end{align}
 
Finally, to show \eqref{es-cR5}, we count the number of derivatives hitting on
$1/r^{\Psi}_0.$ The highest one $m-2$ happens only when $j=1$ and the vector field $\cZ^{\gamma}$ hits on $\p_{y^k}(1/r^{\Psi}_0).$ Therefore, we can conclude that
\beqs 
\|\cR_5^{j,\gamma}\|_{0,t}\lesssim \mu^{\f{1}{2}}\|\na \zeta_2^j\|_{m-2-j,t} \il 1/r^{\Psi}_0\il_{m-3,\infty,t}+
\|\p_y (1/r^{\Psi}_0)\|_{m-3,t}\il\zeta_{2}^1\il_{1,\infty,t}.
\eeqs
Note that by \eqref{esr0}, it holds %the definition \eqref{defr0} for $r_0,$ we 
that:
\begin{align}
\|\p_y (1/r^{\Psi}_0)\|_{m-3,t}&\lesssim \|(\theta,\ep\tsigma)\|_{\hco^{m-2}}\lat\lesssim T^{\f{1}{2}}\lat.\notag
\end{align}
\end{proof}
In the next lemma, we give some estimates concerning $\cS_2^j.$
\begin{lem}\label{lemsource2}
Under the same assumption as in Lemma \ref{lemcR}, 
the following three aspects hold for any $1\leq j\leq m-2$ any 
$\ep\in (0,1], (\mu,\kpa)\in A,$ any $0<t\leq T:$\\[2pt]
(1). The convection term enjoys the estimate:
\begin{align}
    \|(\tilde{u} \cdot \na){\zeta_{1}^j}\|_{m-2-j,t}\lesssim \lat\cE_{m,t}. \label{convection-zeta1}
\end{align}
(2). There exist four functions $\cS_{21}^j-\cS_{24}^j$ such that: 
\begin{align}\label{diffusion-zeta1}
   (1/r^{\Psi}_0)\Delta_{{g}}\zeta_{1}^j+ \big(%\big( \f{1}{r_0}\big)^{\Psi}-\big( \f{1}{r_0}\big)^{\Psi}\big|_{z=0}
    (1/r^{\Psi}_0)- (1/r^{\Psi}_0)|_{z=0}\big)\p_z^2\zeta_1^j=
      %  \sum_{k=1}^2 \p_{y^k}\cS_{2k}^j
     \p_{y^1}\cS_{21}^j +\p_{y^2}\cS_{22}^j +\p_z \cS_{23}^j +\cS_{24}^j,
\end{align}
and 
\begin{align}\label{es-viscous-zeta1}
\mu^{\f{1}{2}}\|(\cS_{21}^j,\cS_{22}^j,\cS_{23}^j)\|_{m-2-j,t}\lesssim T^{\f{1}{4}} \lat\cE_{m,t}, \quad \mu\|\cS_{24}^j\|_{m-2-j,t}\lesssim \lat\cE_{m,t}.
\end{align}
(3). Let  $H^j$ be defined in \eqref{defH}, 
$(H^j)^{\Psi}= H^j\circ \Psi$ has the following form:
\beqs 
(H^j)^{\Psi}= \mu^{\f{1}{2}} (\p_{y^1}\cH^{j}_1+\p_{y^2}\cH^{j}_2+\p_z \cH^{j}_3) +\cH_{4}^j, 
\eeqs
where the vector valued functions $\cH^{j}_1-\cH^{j}_4$ enjoy the property:
\beq\label{prop-ch}
\sum_{l=1}^{3}\|\cH^{j}_l%,\cH^{j}_2,\cH^{j}_3, \cH^{j}_4)
\|_{m-2-j,t}\lesssim (T+\ep)^{\f{1}{3}}\lat\cE_{m,t}, \quad 
\|\cH_4^j\|_{m-2-j}\lesssim \lat\cE_{m,t}.
\eeq
\end{lem}
\begin{proof}
$\underline{Proof \,of \,(1)}.$
Let us write
\beq\label{identity-convection}
\tilde{u} \cdot \na%\zeta_{1}^j
=\tlu_1\p_{y^1}+\tlu_2\p_{y^2}+(\tlu_3/\phi) \cZ_3. %\zeta_{1}^j.
\eeq
Following the similar arguments as the estimate of 
$\cR_1^{\gamma}$ in the last lemma,
we can control the convection term as:
\begin{align*}
 \|{(\tilde{u} \cdot \na)\zeta_{1}^j}\|_{m-2-j,t} &\lesssim    \|(\tlu_1, \tlu_2, \tlu_3/\phi)\|_{m-2,t}\il\zeta_1^j\il_{[\f{m}{2}]-1-j,\infty,t} \mathbb{I}_{\{1\leq j\leq [\f{m}{2}]-2\}}\notag\\
       &\quad +\sup_{0\leq s\leq t}\|\zeta_1^j(s)\|_{m-2-j} \big(\izt \|(\tlu_1, \tlu_2, \tlu_3/\phi)(s)\|_{[\f{m-1}{2}],\infty}^2\big)^{\f{1}{2}}\\
       &\quad + \|\zeta_1^j(s)\|_{m-1-j,t}\il (\tlu_1, \tlu_2, \tlu_3/\phi)\il_{0,\infty,t},
\end{align*}
which, combined with \eqref{zeta1infty}, \eqref{zeta1energy}, \eqref{sec7:eq80}, \eqref{sec7:eq80.5}, yields \eqref{convection-zeta1}. \\[3pt]
$\underline{Proof of\, (2)}.$ By direct computation, we can set:
\begin{align*}
  &  \cS_{2k}^j= \sum_{l=1}^2 (1/r^{\Psi}_0) g^{kl}\p_{y^l}\zeta_1^j, \quad k=1,2, \quad   \cS_{23}^j=\big((1/r^{\Psi}_0)- (1/r^{\Psi}_0)|_{z=0}\big)\p_z\zeta_1^j-\p_z(1/r^{\Psi}_0)\zeta_1^j,\\
   & \cS_{24}^j=- \p_{y^k}(|g|^{-\f{1}{2}}/r^{\Psi}_0)
    g^{kl}|g|^{\f{1}{2}}\p_{y^l}\zeta_1^j+\p_z^2  (1/r^{\Psi}_0) \zeta_1^j.
\end{align*}
Since $j\geq 1,$ we can use \eqref{zeta1energy} and \eqref{r0infty} to control $\cS_{2k}^j$: 
\beqs 
\|\cS_{2k}^j\|_{m-2-j,t}\lesssim \|\zeta_1^j\|_{m-1-j,t} \il1/r^{\Psi}_0 \il_{m-3,\infty,t}\lesssim T^{\f{1}{4}}\lat\cE_{m,t}.
\eeqs
We now sketch the proof of \eqref{es-viscous-zeta1}.
For the first term of $\cS_{23}^j,$ %by writing for notational convenience $\mathscr{H}=(1/r^{\Psi}_0)- (1/r^{\Psi}_0)|_{z=0},$ 
we take benefit of the vanishing boundary condition for
 $\mathscr{H}=(1/r^{\Psi}_0)- (1/r^{\Psi}_0)|_{z=0}$ to obtain that:
 \begin{align}\label{sec7:eq85}
\|\mathscr{H}\cdot \p_z\zeta_1^j\|_{m-2-j,t}\lesssim 
     \|\mathscr{H}/\phi\|_{m-3,t} \il \cZ_3 \zeta_1^j\il_{[\f{m}{2}]-j-2} \mathbb{I}_{\{1\leq j\leq [\f{m}{2}]-2\}}+  \|\cZ_3 \zeta_1^j\|_{m-2-j, t} \il \mathscr{H}/\phi\il_{[\f{m-1}{2}],\infty,t}.
 \end{align}
 It follows from the fundamental theorem of calculus, the Hardy inequality as well as the estimate for $r_0^{-1}$ in \eqref{esr0} that:
 \begin{align*}
&\mu^{\f{1}{2}}\il \mathscr{H}/\phi\il_{[\f{m-1}{2}],\infty,t}\lesssim \mu^{\f{1}{2}}\il (\Id,\na)(\theta,\ep\tsigma)\il_{[\f{m-1}{2}],\infty,t}\lesssim\lat,\\ 
   &\|\mathscr{H}/\phi\|_{m-3,t}\lesssim \|\p_z\mathscr{H}\|_{m-3,t}\lesssim \|(\Id, \na)(\theta,\ep\tsigma)\|_{\hco^{m-3}}\lat\lesssim T^{\f{1}{2}}\lat\cE_{m,t},
 \end{align*}
 which, together with \eqref{sec7:eq85}, \eqref{zeta1energy} yield,
 \beqs 
 \mu^{\f{1}{2}} \|\mathscr{H}\cdot \p_z\zeta_1^j\|_{m-2-j,t}\lesssim T^{\f{1}{4}}\lat\cE_{m,t}.
 \eeqs
The second term in $\cS_{23}^j$ can be controlled in the same way as in \eqref{es-cR3}, the estimate for $ \cS_{24}^j$ bears some resemblances with the control of $\cR_2^{\gamma}$ and $\cR_5^{\gamma}$ in the previous lemma, we omit the detail and thus finish the proof of \eqref{es-viscous-zeta1}.\\[3pt] 
\underline{\textit{Proof of \,(3)}}.
For convenience, we write here the definitions of $H^j=H_1^j+H_2^j:$
\beqs
\begin{aligned}
H_{1}^j&=\big(\curl G^j -\mu\lambda_1 \overline{\Gamma}  \na (r_0^{-1} )\times \curl \omer^j -\sum_{l=1}^3\na u_{l}\times \p_{l}(r_0 u^j)\big) \times \chi \bn,\\
H_{2}^j&=\omer^j\times (u\cdot \na (\chi_i\bn))-\mu\lambda_1\overline{\Gamma} r_0^{-1}(2  \div(\omer^j\times\na (\chi\bn)-\omer^j \times\Delta(\chi\bn)),
\end{aligned}
\eeqs
and 
\beq\label{def-Gj-cuj}
\begin{aligned}
G^j-(1+\ep r_1)\cC_u^j&=u^j(\pt+u\cdot\na)r_0+\mu\lambda_1\overline{\Gamma}r_0^{-1}[\curl\curl, r_0]u^j\notag\\
&\quad- \ep\mu r_1\lambda_1\Gamma\curl\curl u^j-\mu\lambda_1(\Gamma-\overline{\Gamma})\curl\curl u^j,
\end{aligned}
\eeq
where $\cC_u^j$ is defined in \eqref{defCuj}.
By looking at every term appearing in the definition of $H_1^j$ and $H_2^j,$ we find that there are three troublesome terms that may involve one more normal derivative than what we can control, namely, 
\begin{align*}
  &H_{11}^j=\colon \curl\big(G^j-(1+\ep r_1)\cC_u^j \big)\times\chi\bn,\\ %(u^j(\pt+u\cdot\na)r_0)\times\chi\bn, \quad  H_{12}^j=\colon \mu\lambda_1\overline{\Gamma}\curl\big(r_0^{-1}[\curl\curl, r_0]u^j\big)\times\chi\bn,\\
& H_{12}^j=\colon \mu\lambda_1\overline{\Gamma}\big(\na (r_0^{-1})\times \curl \omega_{r_0}^j\big)\times\chi\bn,  \quad  H_{21}^j=\colon- 2\mu\lambda_1\overline{\Gamma} r_0^{-1}  \div(\omer^j\times\na (\chi\bn)).%\na \omer^j\times \na(\chi\bn).
\end{align*}
To continue, it is convenient to write these four terms %$ \circled{1}, \circled{3}, \circled{4}$ 
as a perfect derivative plus reminders. 
In view of the identity 
\begin{align}\label{id-curlftimesn}
    \curl f \times \chi\bn =\p_{\bn}(\chi f)-\na (\chi\bn\cdot f)+ D(\chi\bn) f-f\p_{\bn}\chi,
\end{align}
we split further 
\begin{align*}
H_{11}^j=  H_{111}^j +(H_{11}^j-  H_{111}^j) , %\quad H_{12}^j=H_{121}^j+ H_{122}^j.
\end{align*}
where 
\beq\label{defh111}
\begin{aligned}
  H_{111}^j = \big(\p_{\bn}(\chi \cdot)-\na (\chi\bn\cdot)\big)\big(G^j-(1+\ep r_1)\cC_u^j \big). %\big(u^j(\pt+u\cdot\na)r_0\big), \quad  %H_{112}^j= H_{11}^j- H_{111}^j, \\
%   H_{121}^j= \mu\lambda_1\overline{\Gamma}\big(\chi\p_{\bn}-\na (\chi\bn\cdot)\big)\big(r_0^{-1}[\curl\curl, r_0]u^j\big), \quad H_{122}^j= H_{12}^j- H_{121}^j,
\end{aligned}
\eeq
Moreover, let us write 
\beq\label{defh131}
\begin{aligned}
   H_{12}^j&=-\mu\lambda_1\overline{\Gamma} \big(\chi\p_{\bn}(r_0^{-1})\curl \omega_{r_0}^j-\chi\na(r_0^{-1}) \curl \omega_{r_0}^j\cdot\bn\big)\\
   &=-\mu\lambda_1\overline{\Gamma} \curl \big(\chi\p_{\bn}(r_0^{-1})\omer^j\big)+\mu\lambda_1\overline{\Gamma} \big(\na (\chi\p_{\bn}(r_0^{-1})\times \omer^j+\chi\na(r_0^{-1}) \curl \omega_{r_0}^j\cdot\bn \big)\\
   &=\colon   H_{121}^j+ H_{122}^j
\end{aligned}
\eeq
and 
\beq\label{defh211}
\begin{aligned}
  H_{21}^j&=-2\mu \lambda_1\overline{\Gamma} \div \big(r_0^{-1}\omer^j\times \na (\chi\bn)\big)+2\mu \lambda_1\overline{\Gamma} \na r_0^{-1}\cdot \big(\omer^j\times \na (\chi\bn)\big)=\colon    H_{211}^j+ H_{212}^j.
\end{aligned}
\eeq
By the change of variable, one can find three functions $\cH_1^j,\cH_2^j,\cH_3^j$ such that: 
\begin{align}\label{cv-newidentity}
 (H_{111}^j+ H_{121}^j+ H_{211}^j\big)^{\Psi}= \mu^{\f{1}{2}}\big(\p_{y^1}\cH^{j}_1+\p_{y^2}\cH^{j}_2+\p_z \cH^{j}_3\big)
\end{align}
and $\cH_1^j-\cH_3^j$ are the linear combinations of the following quantities (with the coefficients depending on $\chi, (D\Psi)^{-1}$ and their first two derivatives), 
$$\circled{1}=\mu^{-\f{1}{2}}\big(
G^j-(1+\ep r_1)\cC_u^j \big)^{\Psi}, \quad \circled{2}=\mu^{\f{1}{2}}\lambda_1\overline{\Gamma}  \big(\chi\p_{\bn}(r_0^{-1})\omer^j\big)^{\Psi}, \qquad \circled{3}=2 \mu^{\f{1}{2}}\lambda_1\overline{\Gamma}  \big(r_0^{-1}\omer^j\times \na (\chi\bn)\big)^{\Psi}.$$
We can then define 
$\cH_4^j$ as: 
\beq\label{defcH4j} 
\cH_4^j=\big(H_1^j-\sum_{l=1}^2 H_{1l1}^j+H_{2}^j-H_{211}^j\big)^{\Psi},
\eeq
where $H_{111}^j, H_{121}^j, H_{211}^j$ are defined in \eqref{defh111}-\eqref{defh211}.

Now it remains to justify the property \eqref{prop-ch}. %We will begin with the estimate \eqref{}. 
First, by the definition, %the estimate of $\ch_{111}, \cH_{121},\cH_{131}$ 
we have that:
\begin{align}\label{sec7:eq85.2}
  \|(\cH_{1}^j, \cH_{2}^j,\cH_{3}^j)\|_{m-2-j,t}\lesssim \big\|(\circled{1},\circled{2},\circled{3})\big\|_{m-2-j,t}.
\end{align}
Similar (and indeed easier) to the proof of \eqref{es-cR3}, it can be verified  that:
\beq\label{sec7:eq85.5}
\begin{aligned}
    \big\|\circled{2}, \circled{3}\big\|_{m-2-j,t}&\lesssim 
    \mu^{\f{1}{2}}\|(\chi\p_{\bn}(r_0^{-1})\omer^j, r_0^{-1}\omer^j)\|_{L_t^2\cH^{m-2-j}}\\
    &\lesssim T^{\f{1}{2}}\big(\|\omer^j\|_{L_t^{\infty}\cH^{m-2-j}}+ \|\chi\p_{\bn}(r_0^{-1}), (Z r_0^{-1})\|_{L_t^{\infty}\cH^{m-3}}\big) \lat\\
    &\lesssim T^{\f{1}{2}}\lat\cE_{m,t}.
   % \lab \il \chi\p_{\bn}(r_0^{-1}) \il \big)
\end{aligned}
\eeq
The estimate of $\circled{1}$ requires some efforts since it contains four terms (see \eqref{def-Gj-cuj}).
In view of the identity \eqref{timeder-r0}, we have that
modulus some constants $C=C(R, C_v,\gamma),$
\beqs%\label{timeder-r0}
\begin{aligned}
\mu^{-\f{1}{2}}(\pt+u\cdot\na)r_0%&=-\f{R}{C_v}r_0 (\pt+u\cdot\na)(\tilde{\theta}-\mu(2\lambda_1+\lambda_2)\ep^2 \Gamma\div u/\gamma)\\
\approx -\kpa\mu^{-\f{1}{2}}r_0 \Gamma \div(\beta\na\theta)
-\mu^{\f{1}{2}}\ep r_0 (\ep\pt+\ep u\cdot\na)( \Gamma\div u).
\end{aligned}
\eeqs
%Since $\kpa\lesssim  \mu \lesssim \kpa,$ 
Since $(\mu,\kpa)\in A,$
it holds by  Remark \ref{rmkmuapproxkpa} that
$\kpa \mu^{-\f{1}{2}}\lesssim \mu^{\f{1}{2}}\lesssim \kpa^{\f{1}{2}}.$  Denoting $\mathscr{P}$ as a placeholder of $$(\kpa^{\f{1}{2}} \div(\beta\na\theta), u, Z r_0, Z\div u, Z\Gamma),$$ one finds that:
  \begin{align*}
    \il\mu^{-\f{1}{2}}(\pt+u\cdot\na)r_0\il_{0,\infty,t}\lesssim \lab \il \mathscr{P}\il_{0,\infty,t}\big)\lesssim \lat.
\end{align*}
It then follows from the generalized Gagliardo-Nirenberg inequality (in space) that:
\begin{align*}
  & \|\mu^{-\f{1}{2}}(\pt+u\cdot\na)r_0\|_{L_t^{\infty}\cH^{m-3}}\lesssim \|\mathscr{P}\|_{L_t^{\infty}\cH^{m-3}}\lab \il \mathscr{P}\il_{0,\infty,t}\big)\lesssim \lat\cE_{m,t}.
\end{align*}
We thus can control the term $\mu^{-\f{1}{2}} u^j ((\pt+u\cdot\na)r_0)$ as (note that $j\geq 1$):
\beq\label{sec7:eq86}
\begin{aligned}
   &\|\mu^{-\f{1}{2}} u^j (\pt+u\cdot\na)r_0)\|_{L_t^2\cH^{m-2-j}}\\
   &\lesssim
   \|u\|_{L_t^2\cH^{m-2}}\il\mu^{-\f{1}{2}}(\pt+u\cdot\na)r_0\il_{0,\infty,t} + \|\mu^{-\f{1}{2}}(\pt+u\cdot\na)r_0\|_{L_t^{2}\cH^{m-3}}  \il u\il_{m-3,\infty,t}\\
   &\lesssim T^{\f{1}{2}}\lat\cE_{m,t}.
\end{aligned}
\eeq
 Since $j\geq 1$ and $\il(\sigma, \tsigma)\il_{m-3,\infty,t}$ is uniformly bounded,  we control the  next two terms 
 by attributing all the $L_{t,x}^{\infty}$ norm on $r_1\Gamma, \, \ep^{-1}(\Gamma-\overline{\Gamma})$ 
 \beq\label{sec7:eq86.5}
 \begin{aligned}
    &\mu^{\f{1}{2}} \|(\ep r_1\Gamma\curl\curl u^j, (\Gamma-\overline{\Gamma})\curl\curl u^j)\|_{L_t^2\cH^{m-2-j}}\\
    % &\lesssim \|\ep\mu^{\f{1}{2}} \na^2 u\|_{L_t^{2}H_{co}^{m-2}} \lab \il(r_1 \Gamma, \ep^{-1}(\Gamma-\overline{\Gamma}))\il_{0,\infty,t} \big)\\
    % &\quad + \ep^{\f{1}{2}} \|(\ep\mu)^{\f{1}{2}} \na^2 u\|_{L_t^{2}H_{co}^{m-3}}\lab \il(r_1 \Gamma, \ep^{-1}(\Gamma-\overline{\Gamma}))\il_{m-3,\infty,t} \big)\\
      &\lesssim \ep^{\f{1}{3}}\|\ep^{\f{2}{3}}\mu^{\f{1}{2}} \na^2 u^j\|_{L_t^{2}\cH^{m-2-j}} \lab \il(r_1 \Gamma, \ep^{-1}(\Gamma-\overline{\Gamma}))\il_{m-3,\infty,t}\big)\\%+\ep^{\f{1}{2}}\lae\\
      &\lesssim \ep^{\f{1}{3}}\lae\cE_{m,t}.
 \end{aligned}
\eeq 
Finally, for the commutator term $ \mu^{\f{1}{2}}r_0^{-1}[\curl\curl, r_0]u^j,$
 one has by noticing the identity \eqref{sec7:eq-1} that:
\begin{align*}
   \mu^{\f{1}{2}}r_0^{-1}[\curl\curl, r_0]u^j\approx 
    \mu^{\f{1}{2}}(r_0^{-1}\na^2 r_0) u^j+  \mu^{\f{1}{2}}(r_0^{-1}\na r_0) \na u^j,
\end{align*}
which enables us to obtain:
\beq\label{sec7:eq87}
\begin{aligned}
  &\|\mu^{\f{1}{2}}r_0^{-1}[\curl\curl, r_0]u^j\|_{L_t^2\cH^{m-2-j}}\\
  &\lesssim 
T^{\f{1}{2}}  \big(\|\mu^{\f{1}{2}}r_0^{-1}(\na^2 r_0, \na r_0)\|_{L_t^{\infty}\cH^{m-3}}+\|(u,\na u)\|_{L_t^{\infty}\cH^{m-2}}\big)\cdot\\
&\quad  \big(\il u\il_{m-3,\infty,t}+ \il \mu^{\f{1}{2}}r_0^{-1}\na^2 r_0\il_{0,\infty,t}+\il \na u, r_0^{-1}\na r_0 \il_{[\f{m}{2}]-1,\infty,t}\big)\\
&\lesssim T^{\f{1}{2}}\lat\cE_{m,t}.
\end{aligned}
\eeq
Note that by \eqref{esr0}, \eqref{esr0-sec},
\begin{align*}
    \|\mu^{\f{1}{2}}r_0^{-1}(\na^2 r_0, \na r_0)\|_{L_t^{\infty}\cH^{m-3}}\lesssim \|(\Id, \na, \mu^{\f{1}{2}}\na^2)(\theta,\ep\tsigma)\|_{L_t^{\infty}\cH^{m-3}}\lat\lesssim \lat\cE_{m,t}.
\end{align*}
 We  can now conclude from \eqref{sec7:eq86}-\eqref{sec7:eq87}  that 
 \beqs 
 \|\circled{1}\|_{m-2-j,t}\lesssim T^{\f{1}{2}}\lat\cE_{m,t},
 \eeqs
which, together with \eqref{sec7:eq85.5} and \eqref{sec7:eq85.2}, yields the first inequality in \eqref{prop-ch}. 
%\beq \|(\cH_{1}^j, \cH_{2}^j,\cH_{3}^j)\|_{m-2-j,t}\lesssim T^{\f{1}{2}}\lat\cE_{m,t}.\eeq 
   %The proof of the second inequality follows the similar arguments, 
   
The estimate for $\cH_4^j$ (defined in \eqref{defcH4j}) is a bit lengthy since it contains many terms. However, most of them can be dealt with in the similar manner as we performed before. We thus only detail one term appearing in $\curl\big((1+\ep r_1)\cC_{u}^j\big)$
which may not be obvious, namely  $$%[(\ep\pt)^j,\f{1}{R\beta}]\pt u, \quad 
\mu\lambda_1 \curl\big( (1+\ep r_1) [(\ep\pt)^j,\Gamma] \curl\curl u\big)=
\mu\lambda_1[(\ep\pt)^j,\Gamma]\curl \curl\curl u+ \cdots ,$$
where we denote $\cdots$ as the terms which are similar or easier to control. We thus focus on the first term in the right hand side.
The equation $\eqref{NCNS-S2}_2$ can be rewritten 
\beq 
-\ep\mu\lambda_1\Gamma\curl\curl u=\f{1}{R\beta}(\ep\pt+\ep u\cdot\na )u+ \na\tsigma+\ep\mu (2\lambda_1+\lambda_2)(\div u) \na\Gamma,
\eeq
from which we can find that:
%which lead to the estimate:
\begin{align}
   \ep \mu\|\curl\curl\curl u\|_{L_t^{2}H_{co}^{m-3}}&\lesssim \big(\|(\Id,\na) (u, \theta, \sigma, \ep\mu\div u)\|_{\hco^{m-2}}+\|\na^2\tsigma\|_{\hco^{m-3}}\big) \lat \notag\\
    &\lesssim \lat\cE_{m,t}, \notag \\
  \ep\mu\il\curl\curl\curl u\il_{\f{[m]}{2}-2,\infty,t}&\lesssim \lab \il(\Id,\na)(\sigma, u,\theta)\il_{[\f{m}{2}]-1,\infty,t}+ \il\na^2\tsigma\il_{[\f{m}{2}]-2,\infty,t}\big)
  \lesssim \lat. \label{tcurl-infty}
\end{align}
Note that the properties \eqref{Es-tsigma2}, \eqref{Es-tsigma4} for $\tsigma$ have been employed.
One thus obtains from the above two estimates that:
\begin{align*}
    &\mu\lambda_1 \|\curl\big( (1+\ep r_1) [(\ep\pt)^j,\Gamma] \curl\curl u\big)\|_{L_t^2\cH^{m-2-j}}\\
   & \lesssim 
    \|(\pt \Gamma,  \mu\lambda_1\ep \curl\curl\curl u)\|_{\hco^{m-3}}\lat\lesssim \lat\cE_{m,t}.
\end{align*}
We now finished the proof of the second estimate in \eqref{prop-ch} and thus the proof of the aspect $(3).$
\end{proof}
\subsection{Uniform estimates for high spatial regularity-proof of  \eqref{zetahigh}}
In this subsection, we devote ourselves to proving the estimate \eqref{zetahigh} for $\zeta=\zeta^0=(\chi \curl(r_0 u) \times\bn)^{\Psi},$ which relies again on a suitable splitting of the system \eqref{eqzeta}.
%Before doing so, it will be helpful to look more carefully one source term in the right hand side of \eqref{eqzeta}, namely $|g|^{\f{1}{4}}(H_2^0)^{\Psi},$ where $H_2^0$ is defined in \eqref{defH}. It may involve one normal derivative of $\omega_{r_0}^{\Psi}.$ More precisely, by the change of variable, we find that it can be written as the form:
%\beq\label{defFdag}\begin{aligned}|g|^{\f{1}{4}}(H_2^0)^{\Psi}&= \mu\p_z \bigg(F^{\dag\dag} \big(\omega_{r_0}^{\Psi}, \na(\chi\bn)^{\Psi}, (r_0^{-1})^{\Psi}, |g|^{\f{1}{4}}\big) \bigg)\\&\quad + F^{\dag}\big(u^{\Psi},\omega_{r_0}^{\Psi}, \mu\p_{y}\omega_{r_0}^{\Psi}, (r_0^{-1})^{\Psi}, \mu\p_z (r_0^{-1})^{\Psi}, \na (\chi\bn)^{\Psi}, \na^2 (\chi\bn)^{\Psi}, |g|^{\f{1}{4}}, \p_z (|g|^{\f{1}{4}})\big).\end{aligned}\eeq  where the functions $F^{\dag}$ and $F^{\dag\dag}$ are smooth functions (indeed cubic polynomials) with their arguments. Hereafter, we denote $\omer=\omer^0=\curl(r_0 u)$ and $\zeta=\zeta^0=(\chi\omer \times\bn)^{\Psi}.$
Before doing so, we define
  \beq\label{defr2}
 r_2=\f{1}{R\beta(0)}\exp\big(-\theta+\f{R}{C_v\gamma}\ep\sigma\big), \quad b_2(t,y)=\big(\f{
   1}{r_2}\big)^{\Psi}|_{z=0},\,\, b_0(t,y)=\big(\f{
   1}{r_0}\big)^{\Psi}|_{z=0}.
 \eeq
  By  \eqref{defr0}, the following relation holds:
  \begin{align}\label{relationr0r2}
     r_0 =r_2 \exp(-\ep^2\mu \bar{C}\Gamma\div u) \quad \text{ with } \bar{C}=(2\lambda_1+\lambda_2) R/{C_v\gamma}.
  \end{align}
  
 Let us write $\zeta=\zeta_{3}+\zeta_{4}+\zeta_5,$ where $\zeta_{3}, \zeta_{4}, \zeta_5$ 
 solve the following three problems:
 \begin{equation}\label{eqzeta3}
\left\{
    \begin{array}{l}
   \pt{\zeta}_{3}-\mu \lambda_1\overline{\Gamma}  %\f{%\tilde{\chi}1}{r_2})^{\Psi}|_{z=0}
   b_2\partial_z^2 \zeta_{3} =0,%\mu\p_z F^{\dag\dag}, 
   \quad (t,x) \in \mR_{+}\times \mR^3,  
   \\[2.5pt]
\zeta_{3}|_{z=0}=\big[\f{\ep^2  }{C_v\gamma}((\f{r_0}{\beta} u\cdot\na)\bn \cdot u )|g|^{\f{1}{4}}\chi\Pi u+2r_0|g|^{\f{1}{4}}\chi\Pi\big((D\bn)u-a u\big)\big]^{\Psi}|_{z=0},
\\[2.5pt]
{\zeta}_{3}|_{t=0}=0,
    \end{array}
\right.
\end{equation}
\begin{equation}\label{eqzeta4}
\left\{
    \begin{array}{l}
   \pt{\zeta}_{4}+\mu(1-\Delta)\zeta_{4}=0,  \quad (t,x) \in \mR_{+}\times \mR^3,  \\[2.5pt]
\zeta_{4}|_{z=0}=\big[\f{R\ep}{C_v\gamma}r_0 
(\p_{\bn}\tsigma-\f{1}{R\beta}(\ep u\cdot\na)\bn \cdot u)
|g|^{\f{1}{4}}\chi\Pi u\big]^{\Psi}|_{z=0},  
\\[2.5pt] {\zeta}_{4}|_{t=0}= 0,
    \end{array}
\right.
\end{equation}
\begin{equation}\label{eqzeta5}
\left\{
    \begin{array}{l}
   (\pt+ \tilde{u} \cdot \na){\zeta_5}-\mu \lambda_1\overline{\Gamma} ( \f{1}{r_0})^{\Psi}(\partial_z^2  +\Delta_{{g}})\zeta_5 =\cS_5,  \quad (t,x) \in \mR_{+}\times \mR^3,  \\[2.5pt]
\zeta_5|_{z=0}= 0,\quad
\zeta_5|_{t=0}=\zeta|_{t=0}, 
    \end{array}
\right.
\end{equation}
with the source term 
\beq\label{source5}
\begin{aligned}
\cS_5&=-(\tilde{u} \cdot \na)(\zeta_{3}+\zeta_4)+\cS_3+\cS_4+|g|^{\f{1}{4}} (H^0)^{\Psi}+ H^{0,\dag},
%+F^{\dag},
\end{aligned}
\eeq
where $(H^0)^{\Psi}, H^{0,\dag}$ are defined in \eqref{defH}, \eqref{defHdag} and 
\begin{align}
   & \cS_3=\mu \lambda_1\overline{\Gamma}\bigg(
\big( \f{%\tilde{\chi}
1}{r_0}\big)^{\Psi}\Delta_{{g}}\zeta_{3}+ \big(\big( \f{1}{r_0})^{\Psi}-b_2\big)\p_z^2\zeta_{3}
\bigg), \label{defcS3}\\
&\cS_4=\mu \lambda_1\overline{\Gamma}
\big( \f{1}{r_0}\big)^{\Psi}(\Delta_{{g}}+\p_z^2)\zeta_{4}-\p_t\zeta_4. \label{defcS4}
\end{align}
Note that the system \eqref{eqzeta4} and \eqref{eqzeta5} satisfy the compatibility condition thanks to the assumption $u_0|_{\p\Omega}=0.$ 
Moreover,
by the definition of $\zeta_3-\zeta_5,$ 
it holds that:
$$\zeta_3|_{z=0}+\zeta_4|_{z=0}+\zeta_5|_{z=0}=\zeta|_{z=0}.$$
The desired estimate \eqref{es-zeta} is thus the consequence of the Proposition \ref{lemze3}, Proposition \ref{lemze4} 
and Proposition \ref{lemze5} for $\zeta_3, \zeta_4$ and $\zeta_5$ respectively.
\begin{prop}\label{lemze3} 
 Suppose that $m\geq 7$ and the assumption \eqref{preasption} holds. We have the following estimates for $ \zeta_3$ which solves the system \eqref{eqzeta3}: for  any $\ep \in (0,1], (\mu,\kpa)\in A,$
any $ 0<t\leq T,$ 
 \begin{align}
    &\qquad \qquad\qquad\qquad \il\zeta_3\il_{1,\infty,t}\lesssim \lat, \label{zeta3infty}\quad   \\   &\|\zeta_3(t)\|_{m-2}\lesssim  (\mu T)^{\f{1}{4}} \lat\cE_{m,t}, \quad \|\zeta_3\|_{m-1,t}\lesssim  T^{\f{1}{4}} \lat\cE_{m,t},\label{zeta3energy1}\\
 % &\ep\mu^{\f{1}{2}}\il \p_z \zeta_3\il_{0,\infty,t} \lesssim Y_m(0)+\lat,\\ %\label{zeta3energy3}\\ 
 &\ep\mu^{\f{1}{2}}\|\p_z \zeta_3\|_{m-2,t}
   \lesssim\lab Y_m(0)\big)+ (T^{\f{1}{2}}+\ep)\lat\cE_{m,t},  \label{zeta3energy3-1}\\
  & \quad  \|\ep\pt\zeta_3\|_{0,t}\lesssim \lab Y_m(0)\big)+\lae. \label{zeta3energy3-2}
 \end{align}
\end{prop}
\begin{proof}

We apply \eqref{green-heat} with 
$a_1=a_2=a_3=0$ to obtain the following formula:
\beq\label{green-zeta3}
 \begin{aligned}
\zeta_3(t,y,z)&=
\mu \lambda_1\overline{\Gamma}\izt  \f{1}{(\pi Y)^{\f{1}{2}}}\f{z}{2Y} e^{-\f{z^2}{4Y}}b_2(t',y)  \zeta_3|_{z=0} (t',y)\,\d t'
\end{aligned}  
\eeq
where
\begin{align*}
    b_2= (\f{1}{r_2})^{\Psi}|_{z=0}, \quad
Y=Y(t,t',y)=\mu \lambda_1\overline{\Gamma} \int_{t'}^t b_2(\tau,y)\, \d\tau.
\end{align*}
By the assumption \eqref{preasption} and the property \eqref{preasption1}, it holds that for any $0\leq t'\leq t\leq T,  y\in \mR^2,$
\beq\label{prop-b2-Y2}
c_0\leq b_2(t,y)\leq {1}/{c_0}, \,\quad \mu\lambda_1\overline{\Gamma}c_0(t-t')\leq Y(t,t',y)\leq \mu\lambda_1\overline{\Gamma}(t-t')/c_0.
\eeq
The first estimate \eqref{zeta3infty} can be shown by  following the similar arguments as in the proof of 
\eqref{zeta1infty} done in Proposition \ref{lemze1},
we thus omit the details.

Let us now focus on the proof of \eqref{zeta3energy1}.
Thanks to \eqref{prop-Y} and the properties \eqref{appen-4}, \eqref{appen-5} in the appendix, 
we find for any vector fields $\cZ^{\gamma}$ with $|\gamma|\leq m-2,$
\begin{align}\label{sec7:eq13}
  & \|\cZ^{\gamma}\zeta_{3}(t)\|_{L^2(\mR_{+}^3)}\lesssim \mu \lab| b_2|_{m-3,\infty,t}+|\zeta_3|_{z=0}|_{L_{t,y}^{\infty}}\big)\cdot\notag\\
    &\bigg\|\izt \f{1}{(\mu(t-t'))^{\f{3}{4}}}\bigg(
\big(|\p_y^{m-2}b_2(t',\cdot)|+\big|\f{1}{t-t'}\int_{t'}^t \p_y^{m-2}b_2(\tau, \cdot)\d\tau\big|\big)+ \sum_{k\leq m-2}|\p_y^k \zeta_3|_{z=0}|\bigg)(t',\cdot)\,\d t' \bigg\|_{L_y^2}.
\end{align}
In view of the explicit expression of $\zeta_3|_{z=0}$ in $\eqref{eqzeta3}_2,$ we
apply successively the Minkowski inequality, 
the classical Gagliardo-Nirenberg inequality in $\mR^2$ %the fact \eqref{equ}
to get that: 
\beq\label{zeta31-pre}
\begin{aligned}
\|\cZ^{\gamma}\zeta_{3}(t)\|_{L^2(\mR_{+}^3)}&\lesssim (\mu T)^{\f{1}{4}} \big(| \zeta_3|_{z=0}|_{L_t^{\infty}{H}^{m-2}(\mR^2)}+|\p_y b_2|_{L_t^{\infty}{H}^{m-3}(\mR^2)}\big)\cdot\\
&\qquad\qquad\qquad\qquad \lab| b_2|_{m-3,\infty,t}+|\zeta_3|_{z=0}|_{L_{t,y}^{\infty}}\big) \\
&\lesssim (\mu T)^{\f{1}{4}} |( u, r_0, r_0^{-1})|_{L_t^{\infty}{H}^{m-2}(\p\Omega)}\lat.
\end{aligned}
\eeq
Thanks to %the estimate \eqref{Es-tsigma3}, 
the trace inequality \eqref{traceLinfty} and the estimate \eqref{esr0}, we can conclude that: 
\begin{align}\label{eszeta31}
 \|\zeta_{31}(t)\|_{m-2}\lesssim (\mu T)^{\f{1}{4}} \lat\cE_{m,t}.
\end{align}
Let us proceed to control the $\|\cdot\|_{m-1,t}$ norm.
%For $\zeta_{31},$  
We use \eqref{prop-b2-Y2} and properties \eqref{appen-4}, \eqref{appen-5} in appendix to find that, for any vector field $\cZ^{\gamma}$ with $|\gamma|\leq m-1,$
\begin{align*}%\label{sec7:eq25}
&\qquad\|\cZ^{\gamma}\zeta_{31}(t)\|_{L^2(\mR_{+}^3)}\notag\\
&\lesssim 
\int_0^t \f{\mu^{\f{1}{4}}}{(t-t')^{\f{3}{4}}} \big(\big|\zeta_{3}|_{z=0}(t')\big|_{{H}^{m-1}}+|\p_y b_2|_{L_t^{\infty}{H}^{m-2}}\big)\d s'\lab| b_2|_{m-3,\infty,t}+|(\Id , \cZ)\zeta_{3}|_{z=0}|_{L_{t,y}^{\infty}}\big). 
\end{align*}
Hereafter, we denote for simplicity $H^k=H^k(\mR^2).$
By the convolution inequality in time (after extending $\zeta|_{z=0}$ by zero to $t\leq 0$ and $t\geq T$), we get further
\begin{align}\label{sec7:eq95}
 \|\zeta_{31}\|_{m-1,t}&\lesssim 
   (\mu T)^{\f{1}{4}} \big(|\zeta_3|_{z=0}|_{L_t^{2}{H}^{m-1}}+|\p_y b_2|_{L_t^{\infty}{H}^{m-2}}\big)\lat.
\end{align}
Moreover, one can apply the Gagliardo-Nirenberg inequality in 
$\mR^2$ and the trace inequality \eqref{traceL2} to derive that:
\begin{align*}
  \mu^{\f{1}{4}} |\zeta_3|_{z=0}|_{L_t^{2}{H}^{m-1}}&\lesssim
   \mu^{\f{1}{4}}  |(u, \beta^{-1}, r_0)|_{L_t^{2}{H}^{m-1}(\p\Omega)}\lat\notag\\
 &\lesssim   \|(\Id, \mu^{\f{1}{2}}\na) (u,\theta,\ep\sigma)\|_{\hco^{m-1}}\lat\lesssim \lat\cE_{m,t},\notag\\
  \mu^{\f{1}{4}} |\p_y b_2|_{L_t^{\infty}{H}^{m-2}}&\lesssim
  \|\mu^{\f{1}{2}}\na b_2\|_{\infco^{m-1}}+  \|\p_y b_2\|_{\infco^{m-2}}\\
 & \lesssim \|(\Id, \mu^{\f{1}{2}}\na)(\theta,\ep\sigma)\|_{\infco^{m-1}}\lat\lesssim \lat\cE_{m,t}. %\label{esb2}
\end{align*}
Inserting these two estimates into \eqref{sec7:eq95},
we achieve  the second estimate in \eqref{zeta3energy1}.
%Let us remark here that the estimate \eqref{esb2} for $b_2$ explains why we change $b_0$ by $b_2$ in the equation \eqref{eqzeta3} of $\zeta_3,$ since otherwise, there would involve the term  $\mu^{\f{1}{2}}\ep \|\na \tsigma\|_{L_t^{\infty}H_{co}^{m-1}}$ which is out of control. 

\textit{Proof of \eqref{zeta3energy3-1}-\eqref{zeta3energy3-2}}.
Let us define 
$$\tilde{\zeta}_3=\ep(\zeta_3- h), \quad h=\colon\big[ \f{\ep^2  }{C_v\gamma}\big((\f{r_0}{\beta} u\cdot\na\big)\bn \cdot u )|g|^{\f{1}{4}}\chi\Pi u+2r_0|g|^{\f{1}{4}}\chi\Pi\big((D\bn)u-a u\big)\big]^{\Psi}.$$
By equations \eqref{eqzeta3}, we find that $\tilde{\zeta}_3$ solves the system:
\begin{align}\label{tzeta3}
\left\{
\begin{array}{l}
\pt{\tilde{\zeta}}_{3}-\mu \lambda_1\overline{\Gamma}  b_2(t,y)\partial_z^2 \tilde{\zeta}_3=-\ep\pt h+\ep\mu \lambda_1\overline{\Gamma}  b_2(t,y)\partial_z^2 h=\colon h_1,\\[2pt]
  {\tilde{\zeta}}_{3}|_{z=0}=0, \quad {\tilde{\zeta}}_{3}|_{t=0}=-\ep h|_{t=0}.
     \end{array}
     \right.
\end{align}
As one has:
\beqs
%\ep \mu^{\f{1}{2}}\il \p_z h\il_{0,\infty,t}\lesssim \ep \mu^{\f{1}{2}} \lat, \quad
\ep \mu^{\f{1}{2}}\|\p_z h\|_{m-2,t}\lesssim \ep \mu^{\f{1}{2}} \lat\cE_{m,t},
\eeqs
to prove \eqref{zeta3energy3-1},
it suffices to show:
\begin{align}\label{estzeta3}
   \mu^{\f{1}{2}} \| \p_z\tilde{\zeta}_{3} \|_{m-2,t}%+\|\pt\tilde{\zeta_3}\|_{0,t}
   \lesssim Y_m(0)+T^{\f{1}{2}}\lat\cE_{m,t}.
\end{align}
The solution of the system \eqref{tzeta3} has the form
\begin{align*}
  \tilde{\zeta}_{3}
 & =\izt \int_{\mR_{+}}(K_{-}-K_{+})(t,t', y, z', z)
 h_1(t', y, z') \,\d z' \d t'- \int_{\mR_{+}}(K_{-}-K_{+})(t,0, y, z', z) \ep h|_{t=0}(y,z')\, \d z'\\
 &=\colon  \tilde{\zeta}_{31}+ \tilde{\zeta}_{32}
\end{align*}
where 
\beqs %\label{defKpm}
 K_{\pm}(t,t', y, z', z)=\f{1}{(4\pi Y)^{\f{1}{2}}} e^{-\f{|z\pm z'|^2}{4Y}}, \quad Y=Y(t,t',y)=\mu \lambda_1\overline{\Gamma} \int_{t'}^t b_2(\tau,y)\, \d\tau.
\eeqs
It can be  verified that 
%{\color{red} for Linfty in time estimate, you need to satisfy the compatiability condition}
\begin{align}
  % & \il h_1\il_{0,\infty,t}\lesssim  \lab \ep \mu\il\na^2(u, \theta, \ep\tsigma)\il_{0,\infty,t}+\il (u, \theta, \ep\tsigma) \il_{1,\infty,t}\big)\lesssim \lat,\notag\\
   \| h_1\|_{m-2,t}\lesssim \|(\Id, \na, \mu\ep \na^2)(u,\theta, \ep\tsigma )\|_{L_t^2\cH^{m-2}}\lat\lesssim \lat\cE_{m,t}.  \label{sec7:eq122}
\end{align}
%By applying \eqref{appen-7.5} with $\gamma_1+\gamma_2=0,$ we get thanks to Young's inequality that:
%\begin{align*}   \mu^{\f{1}{2}}\| \p_z\tilde{\zeta_3}(t) \|_{L^{\infty}(\mR_{+}^3)}&\lesssim \il h_1\il_{0,\infty,t} \int_{0}^t \|\mu^{\f{1}{2}} \p_z (K_{-}-K_{+})\|_{L_y^{\infty}L_{z}^1}\d t' +\|\ep\p_z h|_{t=0}\|_{L^{\infty}} \\ &\leq \Lambda\big(\f{1}{c_0}\big) \il h_1\il_{0,\infty,t}\izt \f{1}{(t-t')^{\f{1}{2}}}\d t' +Y_m(0)\lesssim T^{\f{1}{2}}\lat+Y_m(0)\end{align*}which proves the first estimate in \eqref{estzeta3}. 
%The proof of the first quantity in the second inequality of \eqref{estzeta3} can be shown by the similar arguments. More precisely,
By taking benefits of the properties \eqref{appen-6}, \eqref{appen-7.5}, \eqref{appen-7} proved in the appendix, we have that, for any vector fields $\cZ^{\gamma}$ with $|\gamma|\leq m-2,$
\beq\label{sec7:eq14}
\begin{aligned}
 \cZ^{\gamma} \tilde{\zeta}_{31} = C_{\gamma_1}^{\tilde{\gamma}_1}C_{\gamma_2}^{\tilde{\gamma}_2} \sum_{0\leq \tilde{\gamma}_j\leq \gamma_j, j=1,2,3  } \izt
 \int_{\mR_{+}} (z'\p_{z'})^{\tilde{\gamma}_3}\big(&F_{\tilde{\gamma}_1, \tilde{\gamma}_2}^{{\gamma}_3,\tilde{\gamma}_3}(t,t', y, z-z')- F_{\tilde{\gamma}_1, \tilde{\gamma}_2}^{{\gamma}_3,\tilde{\gamma}_3}(t, t', y, z+ z')\big) \\
    &\cdot \cZ_1^{{\gamma}_1-\tilde{\gamma}_1} \cZ_2^{{\gamma}_1-\tilde{\gamma}_2} h_1(t', y, z') \,\d z'\d t'
\end{aligned}
\eeq
where $F_{\tilde{\gamma}_1, \tilde{\gamma}_2}^{{\gamma}_3,\tilde{\gamma}_3}(t,t', \cdot, \cdot)$ are smooth functions that satisfy:
\text{ if} $ \gamma_1+\gamma_2=\tilde{\gamma}_1+\tilde{\gamma}_2=m-2,$
\beq\label{sec7:eq20}
\begin{aligned}
\|\mu^{\f{1}{2}}\p_z F_{\tilde{\gamma}_1, \tilde{\gamma}_2}^{{\gamma}_3,\tilde{\gamma}_3}(t,t', \cdot, \cdot)\|_{L_y^2L_z^1(\mR_{+})}&\lesssim
\f{1}{(t-t')^{\f{1}{2}}}
\lab  | b_2|_{m-3,\infty,t}\big)|b_2|_{L_t^{\infty}H^{m-2}(\mR^2)} \\
&\lesssim \f{1}{(t-t')^{\f{1}{2}}}\lat\cE_{m,t}
\end{aligned}
\eeq
and if $\gamma_1+\gamma_2\leq m-3,$
\beq\label{sec7:eq21}
\begin{aligned}
\sup_{y\in \mR^2}\|\mu^{\f{1}{2}}\p_z F_{\tilde{\gamma}_1, \tilde{\gamma}_2}^{{\gamma}_3,\tilde{\gamma}_3}(t,t', \cdot, \cdot)\|_{L_z^1(\mR_{+})}&\lesssim\f{1}{(t-t')^{\f{1}{2}}}  \lab  | b_2|_{\gamma_1+\gamma_2,\infty,t}\big)\\
&\lesssim \f{1}{(t-t')^{\f{1}{2}}}\lat.
\end{aligned}
\eeq
Note that the following estimates  for $b_2$ have been employed to derive the second inequality in the above two estimates:
\begin{align*}
   & |b_2|_{m-3,\infty,t}\lesssim \lab \il(\theta,\ep\sigma)\il_{m-3,\infty,t}\big)\lesssim \lat, \\
 &|b_2|_{L_t^{\infty}H^{m-2}(\mR^2)}\lesssim \|(\Id,\na)(\theta,\ep\sigma)\|_{\infco^{m-2}}\lat\lesssim \lat\cE_{m,t}.
\end{align*}
In view of the estimates \eqref{sec7:eq14}-\eqref{sec7:eq21}, we find that by the convolution inequality in time variable (after zero extensions of $h_1$ on $t\leq 0,$ and $t\geq T$),
\beqs
\begin{aligned}
\mu^{\f{1}{2}}\|\p_z\cZ^{\gamma}\tilde{\zeta}_{31}(t)\|_{L^2(\mR_{+}^3)}&\lesssim   \izt  \|\mu^{\f{1}{2}}\p_z F_{\tilde{\gamma}_1, \tilde{\gamma}_2}^{{\gamma}_3,\tilde{\gamma}_3}(t,t', \cdot, \cdot)\|_{L_y^2L_z^1(\mR_{+})} \il h_1\il_{0,\infty,t} \mathbb{I}_{\{\gamma_1+\gamma_2=m-2\}}\\
&\qquad\quad+ \|\mu^{\f{1}{2}}\p_z F_{\tilde{\gamma}_1, \tilde{\gamma}_2}^{{\gamma}_3,\tilde{\gamma}_3}(t,t', \cdot, \cdot)\|_{L_y^{\infty}L_z^1(\mR_{+})}  \|h_1(t')\|_{m-2} \mathbb{I}_{\{\gamma_1+\gamma_2\leq m-3\}} \,\d t' \\
&\quad \lesssim T^{\f{1}{2}} \big( \|h_1\|_{m-2,t}+\cE_{m,t}\big)\lat.
%\izt \f{1}{(t-t')^{\f{1}{2}}} \|h_1(t')\|_{m-2} \d t' \lat+ T\lat\cE_{m,t}
\end{aligned}
\eeqs
Summing up for $|\gamma|\leq m-2,$ and using \eqref{sec7:eq122}, we can conclude that 
\beqs 
\mu^{\f{1}{2}}\|\p_z\tilde{\zeta}_{31}\|_{m-2,t}\lesssim T^{\f{1}{2}}\lat\cE_{m,t}.
\eeqs
Moreover, by integration by parts in $z$ variable, we can verify that:
\begin{align*}
   \mu^{\f{1}{2}} \|\p_{z}\tilde{\zeta}_{32}\|_{m-2,t}\lesssim \mu^{\f{1}{2}}\|\p_z h|_{t=0}\|_{m-2}\lesssim \lab Y_m(0)\big).
\end{align*}
We thus obtain \eqref{zeta3energy3-1} from the above two estimates.

Finally, as the proof of \eqref{zeta3energy3-2} can be reduced to the control $\|\pt \tilde{\zeta_3}\|_{0,t},$ we  multiply the system \eqref{tzeta3} by $\pt\tilde{\zeta}_3,$ integrate in space and time and use Young's inequality
 to obtain:
 \begin{align*}
     \|\pt \tilde{\zeta}_3\|_{0,t}^2+\mu \|\p_z\tilde{\zeta}_3(t)\|_{0}^2\lesssim \mu \|\p_z\tilde{\zeta}_3(0)\|_{0}^2+ \|h_1\|_{0,t}^2+ \lab \il \pt b\il_{0,\infty,t}\big)\mu \|\p_z \tilde{\zeta}_3\|_{0,t}^2 .
 \end{align*}
 The above estimate, together with the fact: 
 \begin{align*}
     \|\ep\pt h\|_{0,t}\lesssim \lae,
 \end{align*}
 leads to \eqref{zeta3energy3-2}.
\end{proof}
\begin{cor}
Let $\cS_3$ be defined in \eqref{defcS3}, the following estimate holds:
\beq\label{s3L2}
\begin{aligned}
    \mu^{-\f{1}{2}}\|\cS_3\|_{0,t}\lesssim \lab Y_m(0)\big)+ (T^{\f{1}{2}}+\ep) \lae.
\end{aligned}
\eeq
\end{cor}
\begin{proof}
By the assumption \eqref{preasption}, the property \eqref{preasption1}
and estimate \eqref{zeta3energy1}, we have directly:
\beqs
\mu^{\f{1}{2}}\|\big(\f{1}{r_0}\big)^{\Psi}\Delta_{{g}}\zeta_{3}\|_{0,t}\leq (\mu T)^{\f{1}{2}}\Lambda\big(\f{1}{c_0}\big)\sup_{0\leq s\leq t}\|\zeta_3\|_2\lesssim (\mu T)^{\f{1}{2}}
\lat\cE_{m,t}.
\eeqs
Moreover, we write by using %the definition \eqref{defr2}
the relation \eqref{relationr0r2}, the equation $\eqref{eqzeta3}$ that:
\begin{align}
&\qquad \mu^{\f{1}{2}} \big( (r_0^{-1})^{\Psi}-b_2\big)\p_z^2\zeta_{3}\notag\\
&=  \mu^{\f{1}{2}}\big( (r_0^{-1})^{\Psi}-(r_0^{-1})^{\Psi}|_{z=0}\big)\p_z^2\zeta_{3}+ \f{1}{\lambda_1\overline{\Gamma}\mu^{\f{1}{2}}}\big(\exp(\ep^2\mu \bar{C}\Gamma\div u)-1\big)^{\Psi}|_{z=0 }\pt\zeta_3.
\label{zeta3-secnor}
\end{align}
Denote $K=K(\ep\tsigma)=\exp(-\f{R\ep}{C_v\gamma}\tsigma),$  we have by the definition \eqref{defr0} that
$r_0^{-1}=R\beta(\theta)K(\ep\tsigma),$  which gives rise to:
\beq\label{r0-r0b}
\begin{aligned}
(r_0^{-1})^{\Psi}-(r_0^{-1})^{\Psi}|_{z=0}=R (\beta^{\Psi}-\beta^{\Psi}|_{z=0})K^{\Psi}+ R\beta^{\Psi}|_{z=0}(K^{\Psi}-K^{\Psi}|_{z=0}).
\end{aligned}
\eeq
It follows from the Taylor expansion and the fact 
$(\p_z\theta^{\Psi})|_{z=0}=(\p_{\bn}\theta|_{\p\Omega})^{\Psi}=0$ that:
\beq\label{defbetaK-betaKb}
\begin{aligned}
&\beta^{\Psi}-\beta^{\Psi}|_{z=0}=z^2\int_0^1
\p_{z}^2(\beta^{\Psi}) (t,y, \tau z)(1-\tau)\d\tau, \\
& K^{\Psi}-K^{\Psi}|_{z=0}=\ep z\int_0^1 \big((K^{\Psi})'\p_z\tsigma^{\Psi}\big) (t,y,\tau z) \d \tau,
\end{aligned}
\eeq
which, together with the estimate \eqref{zeta3energy1}, yield: 
%We thus derive from the above two identities and the estimates \eqref{zeta3energy1} that:
\beq\label{sec7:eq-10}
\begin{aligned}
&\mu^{\f{1}{2}}\|\big( (r_0^{-1})^{\Psi}-(r_0^{-1})^{\Psi}|_{z=0}\big)\p_z^2\zeta_{3}\|_{0,t}\\
&\lesssim(\|\zeta_3\|_{2,t}+\|\ep\mu^{\f{1}{2}}\p_z \zeta_3\|_{1,t})\lab \il  \na (\theta,\tsigma), \mu^{\f{1}{2}}\na^2\theta\il_{0,\infty,t} \big) \\
&\lesssim \lab Y_m(0)\big)+(T^{\f{1}{2}}+\ep)\lat\cE_{m,t}.
\end{aligned}
\eeq
Moreover, by again the Taylor expansion, 
\begin{align*}
%\big(
\exp(\ep^2\mu \bar{C}\Gamma\div u)-1=\ep^2\mu \bar{C}\int_0^1 \exp(\tau \ep^2\mu \bar{C}\Gamma\div u)\Gamma\div u\, \d \tau,
\end{align*}
which, together with \eqref{zeta3energy3-2}, enable us to  control the second term in \eqref{zeta3-secnor} as:
\beqs 
\ep \mu^{\f{1}{2}}\|\ep\pt\zeta_3\|_{0,t}\lab \il(\div u,\sigma)\il_{0,\infty,t}\big)
\lesssim \lab Y_m(0)\big)+\ep\mu^{\f{1}{2}}\lae.
\eeqs
We thus finish the proof of \eqref{s3L2} by collecting the above estimate and \eqref{sec7:eq-10}.
\end{proof}

\begin{prop}\label{lemze4}
Let $\zeta_4$ be the solution to the system \eqref{eqzeta4}. Under the same assumption as in Proposition \ref{lemze3}, it holds that for  any $\ep \in (0,1], (\mu,\kpa)\in A,$
any $ 0<t\leq T,$ 
\begin{align}
&\il\zeta_4\il_{*,1,\infty,t}\lesssim \lat, \label{zeta4infty}\\ 
&\|\zeta_4(t)\|_{m-2}+
 \|\zeta_4\|_{m-1,t}+ \|\pt\zeta_4\|_{m-2,t}
\lesssim \ep\mu^{\f{1}{4}} \lae.\label{eszeta4}
\end{align}
\end{prop}
\begin{proof}
We will only prove \eqref{eszeta4}, since \eqref{zeta4infty} can be shown by following the similar (and easier) arguments as in the proof of \eqref{zeta1infty}. 
By %the boundary condition of $\p_{\bn}\tsigma$ in 
$\eqref{eq-tiltasigma}_2,$ it holds that: 
\beqs
\p_{\bn}\tsigma-\f{1}{R\beta}(\ep u\cdot\na)\bn \cdot u=-\ep\mu(2\lambda_1+\lambda_2)(\div u)\p_{\bn}\Gamma+\ep\mu\lambda_1\Gamma \curl\curl u\cdot\bn \text{ on } \p\Omega.
\eeqs
Consequently, in light of the boundary condition:
\begin{align*}
  \zeta_4|_{z=0}= \big[\f{R\ep}{C_v\gamma}|g|^{\f{1}{4}} r_0 \chi\Pi u
(\p_{\bn}\tsigma-\f{1}{R\beta}(\ep u\cdot\na)\bn \cdot u)
\big]^{\Psi}|_{z=0},
\end{align*}
we have by using  Gagliardo-Nirenberg inequality in $\mR^2,$ the estimates \eqref{curlcurl-infty0}, \eqref{sec4:eq12} and the trace inequality \eqref{traceL2} that: 
\begin{align}\label{zeta4bd1}
|\zeta_4|_{z=0}|_{L_t^2H^{m-\f{3}{2}}}&\lesssim 
\ep \big|\big( r_0\chi\Pi u(\p_{\bn}\tsigma-\f{1}{R\beta}(\ep u\cdot\na)\bn \cdot u)\big)^{\Psi}\big|_{L_t^2H^{m-\f{3}{2}}}\notag\\
&\lesssim 
\ep^2\mu |(\chi r_0, \chi\Gamma, u, \div u, \p_{\bn}\Gamma,  \curl\curl u\cdot \bn)|_{L_t^2H^{m-\f{3}{2}}(\p\Omega)}\lat\\
&\lesssim \ep^2\mu^{\f{1}{2}}\lat\cE_{m,t}. \notag
\end{align}
Accordingly, by noticing again the estimate \eqref{sec4:eq12}, one can show that
\beq\label{zeta4bd2}
|\zeta_4|_{z=0}|_{L_t^{\infty}H^{m-2}}+
|\pt\zeta_4|_{z=0}|_{L_t^2H^{m-\f{5}{2}}}\lesssim \ep\mu^{\f{1}{2}}\lat\cE_{m,t}. 
\eeq
The heat type equation \eqref{eqzeta4} admits the explicit formula:
\begin{align}
    \zeta_4=2\mu\izt\int_{\mR^2}\f{1}{(4\pi\mu(t-t'))^{\f{3}{2}}}e^{-(t-t')} \p_z\big(e^{-\f{z^2+|y-y'|^2}{4\mu(t-t')}}\big)
    \zeta_4|_{z=0}(t',y') \,\d y' \d t'.
\end{align}
We can follow the arguments as in the proof of 
\eqref{zeta31-pre} (by setting $b_2=1$) to obtain that:
\begin{align*}
   \| \zeta_4(t)\|_{m-2}\lesssim (T\mu)^{\f{1}{4}} \big|\zeta_4|_{z=0}\big|_{L_t^{\infty}H^{m-2}}\lesssim \ep(T\mu)^{\f{1}{4}}\lat\cE_{m,t}.
\end{align*}
To control the quantity  $\|\zeta_4\|_{m-1,t},$ we need to take benefit of the horizontal smoothing effects. %An easy way to use this is applying the Fourier transform. 
To do so, we first extend 
$\zeta_4$ from $[0,T]$ to $\mR$ by zero extension and then perform Fourier transform in both $t$ and $y$ variable. By denoting $\widehat{f}(\tau, \xi)=\cF_{t\rightarrow \tau, y\rightarrow \xi}(f),$ it holds that:
\begin{align*}
    \widehat{\zeta_4}(\tau,\xi,z)= e^{-z\sqrt{\mu(1+i\tau+|\xi|^2})} \widehat{\zeta_4|_{z=0}}(\tau,\xi).
\end{align*}
Therefore, for any $0\leq k\leq m-1,$ 
one has that:
\begin{align*}
   & \|(z\p_z)^{k}(1+|\xi|^2)^{\f{m-1-k}{2}} \widehat{\zeta_4}\|_{L^2(\mR\times\mR_{+}^3)}^2\notag\\
    &\lesssim 
  \int_{\mR^3}  (1+|\xi|^2)^{m-1-k}\widehat{\zeta_4|_{z=0}}(\tau,\xi) \d \xi\d\tau \int_{\mR_{+}} P_{2k}\big(z\sqrt{\mu(1+i\tau+|\xi|^2})\big)e^{-z\sqrt{\mu(1+i\tau+|\xi|^2})}\d z\\
 &\lesssim \mu^{-\f{1}{4}}\big|\zeta_4|_{z=0}\big|_{L_t^2H^{m-\f{3}{2}-k}}
\end{align*}
where $P_{2k}$ is a polynomial with degree $2k.$
We thus obtain by using \eqref{zeta4bd1} that
\beqs
\|\zeta_4\|_{m-1,t}\lesssim |\zeta_4|_{z=0}|_{L_t^2H^{m-\f{3}{2}}}\lesssim \ep^2\mu^{\f{1}{4}}\lat\cE_{m,t}.
\eeqs
Finally, to control the quantity $\|\pt\zeta_4\|_{m-2,t},$ we first take time derivative to the system and then follow the similar arguments as in the proof of $\|\zeta_4\|_{m-1,t}.$ Doing so and using the estimate \eqref{zeta4bd2}, we find that
\begin{align*}
    \|\pt\zeta_4\|_{m-2,t}\lesssim |\pt\zeta_4|_{z=0}|_{L_t^2H^{m-\f{5}{2}}}\lesssim
    \ep\mu^{\f{1}{2}}\lat\cE_{m,t}.
\end{align*}
\end{proof}
\begin{cor}
Let $\cS_4$ be defined in \eqref{defcS4}, we have:
\begin{align}\label{es-cS4}
  \mu^{-\f{1}{2}} \| \cS_4\|_{0,t}\lesssim \ep\lat\cE_{m,t}.
\end{align}
\end{cor}
\begin{proof}
It is just a matter to substitute $\p_z^2\zeta_4$ by 
$(\p_t-1-\p_{y^1}^2-\p_{y^2}^2)\zeta_4$ in the expression of $\cS_4$ and to use the property \eqref{eszeta4}.
\end{proof}

To finish the estimate for $\zeta,$ it remains to control $\zeta_5$ defined as the solution to the system \eqref{eqzeta5}. As a preparation, we begin with the following lemma:
\begin{lem}\label{lemzeta5-prop}
Let $\zeta_5$ be the solution to the system \eqref{eqzeta5}
and assume $\mu$ and $\kpa$ satisfy the relation
\eqref{assumption-mukpa}.
Then $\zeta_5$ can be written as the sum of two functions:
$\zeta_5=\zeta_{51}+\zeta_{52},$ where $\zeta_{51}$ and $\zeta_{52}$ admit the properties: 
\begin{align}
&\|\p_z \zeta_{51}\|_{L_t^2L^2}\lesssim  (T^{\f{1}{2}}+\ep)\big(\lab Y_m(0)\big)+\lae\big), \label{zeta5-prop1} \\
&\mu^{\f{1}{2}}\il \p_z \zeta_{52} \il_{0,\infty,t}\lesssim Y_m(0)+
\lat. \label{propzeta52}
\end{align}
\end{lem}
\begin{rmk}
The estimates \eqref{zeta5-prop1} \eqref{propzeta52} will be very useful later in order to close the energy estimates for $\zeta_5$.
\end{rmk}
\begin{rmk}\label{rmk-mukpa-1}
Since $\zeta_5$ solves a transport-diffusion equation \eqref{eqzeta5} with the viscosity parameter $\mu,$ one may expect $\mu^{\f{1}{2}}\na\zeta_5$ to be uniformly bounded in $L_{t,x}^{\infty}$ which is the case in the purely inviscid limit problem \cite{MR2885569, MR3485413, MR3419883}.
Nevertheless, such an estimate is unlikely to be uniform with respect to the Mach number $\ep.$ Indeed, incorporated with the fast oscillation effects, the width of  
 the viscous boundary layer is expected to be $(\ep\mu)^{\f{1}{2}},$ which means that the best one can hope is $\il (\ep\mu)^{\f{1}{2}}\na\zeta_5\il_{0,\infty,t}$ being uniformly bounded, which is not useful here. Moreover, since the source term $\cS_5$ involve the high order derivatives of unknown $r_0=r_0(\theta,\tsigma),$ there will be some interactions between the thermal boundary layer %(with width $\kpa^{\f{1}{2}}$) 
 and  the viscous boundary layer. Consequently,  we are forced to make some suitable assumption (see \eqref{assumption-mukpa}) on  $\kpa$ and $\mu.$ 
 Let us comment that all the  estimates established previously hold indeed for any $\ep\in(0,1], (\mu,\kpa)\in (0,1]^2$ satisfying $\mu\sim \kpa.$ The more restrictive assumption \eqref{assumption-mukpa} on  $\kpa$ and $\mu$  is made here to prove \eqref{zeta5-prop1}.
\end{rmk}
\begin{proof}
We define $\zeta_{51}$ and $\zeta_{52}$ to be the solutions of the following two systems: 
\begin{align}\label{eqzeta51}
 \left\{  
 \begin{array}{l}
   \big(\p_{t}-\mu\lambda_1\overline{\Gamma}%(\f{1}{r_0})^{\Psi}
   b_0\p_z^2\big)  \zeta_{51}=\cM_1,    \\
 \zeta_{51}|_{z=0}=0,\quad   \zeta_{51}|_{t=0}=0,
    \end{array}
    \right.
    \quad 
    \left\{  
 \begin{array}{l}
   (\p_{t}-\mu\lambda_1\overline{\Gamma}b_0\p_z^2)  \zeta_{52}=\cM_2,    \\
 \zeta_{52}|_{z=0}=0,\quad   \zeta_{52}|_{t=0}=\zeta|_{t=0},
    \end{array}
    \right.
\end{align}
%\begin{align}\label{eqzeta52}
%\end{align}
where $\cM_1,\cM_2$ are defined as:
\beq\label{defcM}
\begin{aligned}
    \cM_1&=\colon \cS_3
    +\cS_4+\tilde{\cS}_5+\big(\curl \cQ \times\chi\bn\big)^{\Psi} \\
    &\qquad -\mu\lambda_1\overline{\Gamma}|g|^{\f{1}{4}} \bigg(\big(\na(r_0^{-1})\times \curl\omer \big)\times\chi\bn+ 2 r_0^{-1} \div\big(\omer\times\na (\chi\bn)\big)\bigg)^{\Psi} %-\bigg(\big(\mu\lambda_1(\Gamma-\overline{\Gamma})+\f{\kpa}{C_v\gamma}(\Gamma e^{\f{R\ep}{C_v\gamma}\tsigma}-\overline{\Gamma})\big)\na\Delta\theta\times u \bigg)^{\Psi}
    %+  \mu \lambda_1\overline{\Gamma}\bigg(\big( \f{%\tilde{\chi}1}{r_0}\big)^{\Psi}\Delta_{{g}}\zeta_{5}+ \big(\big( \f{1}{r_0})^{\Psi}-b_0\big)\p_z^2\zeta_{5}\bigg)
    ,\\
    \cM_2&=\colon\cS_5-\cM_1-\tilde{u}\cdot\na\zeta_5,
\end{aligned}
\eeq
with $ \cS_3, \cS_4$ being defined in \eqref{source5}-\eqref{defcS4}
and 
\beqs 
\tilde{\cS}_5=\colon\mu\lambda_1\overline{\Gamma}\bigg(\big(\f{1}{r_0} \big)^{\Psi}\Delta_g\zeta_5+ \big(\big(\f{1}{r_0} \big)^{\Psi}-b_0\big)\p_{z}^2\zeta_{5} \bigg),
\eeqs
\beq\label{defcQ}
\cQ= u\bigg((\pt+u\cdot\na)r_0+\f{\kpa}{C_v\gamma}\big(\exp\big({\f{R\ep}{C_v\gamma}\tsigma}\big)\Gamma-\overline{\Gamma}\big)\Delta\theta \bigg)
+\mu\lambda_1\overline{\Gamma}r_0^{-1}[\curl\curl, r_0]u.
\eeq
Let us prove \eqref{zeta5-prop1}, \eqref{propzeta52} %some estimates for $\zeta_{51}.$  first show  \eqref{zeta5-prop1}. %
and start with the estimate for $\zeta_{51}.$ 
First,
corresponding to \eqref{cv-newidentity}, we can find that the last two terms of $\cM_1$ can be written as the 
gradient of some quantities 
\begin{align*}
    \cM_1=\colon \p_{y^1}\cH_1+\p_{y^2}\cH_2+\p_{z}\cH_3
\end{align*}
where $\cH_1-\cH_3$ are linear combinations of the following typical terms (with the coefficients depending on $\chi, (D\Psi)^{-1}$ and their first two derivatives):
\begin{align}\label{sec7:eq100}
\cQ^{\Psi},%=\big(u(\pt+u\cdot\na)r_0+\mu\lambda_1\overline{\Gamma}r_0^{-1}[\curl\curl, r_0]u\big)^{\Psi}, 
\quad 
\mu\lambda_1\overline{\Gamma}  \big(\chi\p_{\bn}(r_0^{-1})\omer\big)^{\Psi}, \qquad 2 \mu\lambda_1\overline{\Gamma}  \big(r_0^{-1}\omer\times \na (\chi\bn)\big)^{\Psi},
\end{align}
plus a reminder $\cH_4$ which satisfy: 
\beq\label{esch4}
\mu^{-\f{1}{2}}\|\cH_4\|_{0,t}\lesssim \lat\cE_{m,t}, \quad \|\cH_4\|_{m-2,t}\lesssim T^{\f{1}{2}}\lat\cE_{m,t}.
\eeq

Now, taking the inner product with the system 
\eqref{eqzeta51} and $\mu^{-1}\zeta_{51},$ we find,
\beq\label{zeta51-zeroenergy}
\begin{aligned}
 \mu^{-1} \|&\zeta_{51}(t)\|_0^2+ \iota  c_0 \|\p_z\zeta_{51} \|_{0,t}^2\leq \big|\mu^{-1}\izt\int_{\mR_{+}^3}
 \zeta_{51}\cdot \cM_1 \d x\d s\big|\\
 &\lesssim \|\mu^{-\f{1}{2}}\zeta_{51}\|_{0,t}\|\mu^{-\f{1}{2}}(\cS_3,\cS_4, \tilde{\cS}_5, \p_{y^1}\cH_1, \p_{y^2}\cH_2, \cH_4)\|_{0,t}+\|\p_z \zeta_{51}\|_{0,t}\|\mu^{-1}\cH_3\|_{0,t}.
\end{aligned}
\eeq
Analogues to \eqref{s3L2}, we can verify that,
by using \eqref{zeta3-secnor}-\eqref{defbetaK-betaKb}, \begin{align*}
    \mu^{-\f{1}{2}}\|\tilde{\cS}_5\|_{0,t}\lesssim \mu^{\f{1}{2}}\big(\|\zeta_5\|_{2,t}+\ep \|\p_z\zeta_5\|_{1,t}\big)\lat\lesssim \lab Y_m(0)\big)+\lae,
\end{align*}
which, combined with the estimates \eqref{s3L2}, \eqref{es-cS4}, \eqref{esch4}, \eqref{zeta51-zeroenergy} and Young's inequality, gives rise to:
\begin{align*}
    \mu^{-1} \|\zeta_{51}(t)\|_0^2+ \iota  c_0 \|\na\zeta_{51} \|_{0,t}^2\lesssim %\lab Y_m(0)\big)+ 
    (T^{\f{1}{2}}+\ep)\big(\lab Y_m(0)\big)+\lae\big)+\|\mu^{-\f{1}{2}}(\cH_1,\cH_2)\|_{1,t}^2+\|\mu^{-1}\cH_3\|_{0,t}^2.
\end{align*}
%In view of the expression \eqref{sec7:eq100}, the last 
The proof of  \eqref{zeta5-prop1} shall thus be finished once it is shown that:
\begin{align}\label{es-ch1-3}
    \|\mu^{-1}\cH_3\|_{0,t}+ \|\mu^{-\f{1}{2}}(\cH_1,\cH_2,\cH_3)\|_{m-2,t} \lesssim (T^{\f{1}{2}}+\ep)\lae,
\end{align}
which relies on the control
of three terms in \eqref{sec7:eq100}. Note that the estimate of $\|\mu^{-\f{1}{2}}\cH_3\|_{m-2,t}$ is not necessary here but rather useful later.
% In the following, we will prove an extra estimate which is useful later for $\cH_1,\cH_2,\cH_3$
Let us estimate the last two terms in \eqref{sec7:eq100} as:
\begin{align}\label{es-otherterm}
    \| \big(\chi\p_{\bn}(r_0^{-1})\omer\big)^{\Psi}, \big(r_0^{-1}\omer\times \na (\chi\bn)\big)^{\Psi}\|_{m-2,t}&\lesssim T^{\f{1}{2}}\|(\omer, \na r_{0}^{-1})\|_{L_t^{\infty}\cH^{m-2}}\lat
    %\lab \il (\Id,\na) r_{0}^{-1}\il_{0,\infty,t}\big)
    \notag\\
    &\lesssim T^{\f{1}{2}}\lat\cE_{m,t}.
\end{align}
It now remains  to control $\|\mu^{-1}\cQ^{\Psi}\|_{0,t}$ which is the very place where the relation \eqref{assumption-mukpa} is used. Let us look at the terms in \eqref{defcQ} more carefully. 
By the straightforward calculation,
\begin{align*}
  [\curl\curl, r_0]u= -(\Delta{r_0}) u+ \big(\div u\na r_0 -(\na r_0\cdot\na) u +(u \cdot \na)\na r_0+\na r_0\times\curl u\big).
\end{align*}
Moreover, by the definition \eqref{defr0}, 
\beqs 
-r_0^{-1}\Delta{r_0}=\Delta\theta-\f{R\ep}{C_v\gamma}\Delta\tsigma+\big|\f{R\ep}{C_v\gamma}\na\tsigma-\na\theta\big|^2,
\eeqs
which, together with the above identity, yields:
\begin{align}\label{identity1}
  \mu\lambda_1\overline{\Gamma}r_0^{-1}[\curl\curl, r_0]u
  = \mu\lambda_1\overline{\Gamma}(\Delta\theta)u+ l.o.t,
\end{align}
where hereafter we denote $l.o.t$ for the terms that admit the following property:
\begin{align}
    \|\mu^{-1}(l.o.t )^{\Psi}\|_{0,t}+\|\mu^{-\f{1}{2}}(l.o.t )^{\Psi}\|_{m-2,t}\lesssim (T^{\f{1}{2}}+\ep)\lat\cE_{m,t}. %\quad \|\mu^{-\f{1}{2}}(l.o.t )^{\Psi}\|_{m-2,t}\lesssim T^{\f{1}{2}}\lat\cE_{m,t}.
\end{align}
Next, by the equation \eqref{timeder-r0}, and the 
fact $R\beta r_0=\exp(\f{R\ep}{C_v\gamma}\tsigma),$ %(see the definition \eqref{defr0})yields:
\begin{align}\label{timeder-r0-1}
(\pt+u\cdot\na)r_0=-\f{\kpa}{C_v\gamma} \exp\big(\f{R\ep}{C_v\gamma} \tsigma\big)\Gamma\Delta\theta + l.o.t,
\end{align} 
we thus obtained 
\begin{align}\label{identity2}
    u\bigg((\pt+u\cdot\na)r_0+\f{\kpa}{C_v\gamma}\big(\exp\big({\f{R\ep}{C_v\gamma}\tsigma}\big)\Gamma-\overline{\Gamma}\big)\Delta\theta \bigg)=-\f{\kpa}{C_v\gamma}\overline{\Gamma}u\Delta\theta+ l.o.t,
\end{align}
which, combined with \eqref{identity1}, leads to that:
\begin{align*}
    \cQ= \lambda_1\overline{\Gamma} \big(\mu-\f{\kpa}{C_v\gamma \lambda_1} \big)u\Delta\theta + l.o.t.
\end{align*}
Consequently, by the virtue of the fact:
\begin{align*}
    \kpa^{\f{1}{2}}\|\Delta\theta\|_{L_t^{\infty}L^2(\Omega)}+ \kpa \|\Delta\theta\|_{L_t^{\infty}\cH^{m-2}}\lesssim \cE_{m,t},
\end{align*}
once the relation \eqref{assumption-mukpa} holds, we can conclude that: 
\begin{align}\label{escQ}
   \mu^{-1} \|\cQ^{\Psi}\|_{0,t}+\mu^{-\f{1}{2}}
\|\cQ^{\Psi}\|_{m-2,t}\lesssim (T^{\f{1}{2}}+\ep)\lat\cE_{m,t}. %\quad \lesssim T^{\f{1}{2} \lat\cE_{m,t}
\end{align}
Collecting \eqref{es-otherterm} and \eqref{escQ}, one finds the estimate \eqref{es-ch1-3} and thus finishes the proof of the first estimate in \eqref{zeta5-prop1}.
%\begin{align*}   \end{align*}

  Let us now switch to the estimate of $\zeta_{52}$ stated in \eqref{propzeta52}. 
  Defined as the solution of \eqref{eqzeta51}, $\zeta_{52}$ has the following form:
\beq\label{zeta52-explicit}
\begin{aligned}
\zeta_{52}&=\int_{\mR_{+}}(K_{-}-K_{+})(t,0, y, z', z)\zeta_{52}|_{t=0}(y,z') \, \d z'\\
&\quad+\izt \int_{\mR_{+}}(K_{-}-K_{+})(t,t', y, z', z) \cM_2(t', y, z') \,\d z' \d t' =\colon \zeta_{521}+\zeta_{522}
\end{aligned}
\eeq
 where 
 \begin{align*}
     K_{\pm}(t,t', y, z', z)=\f{1}{(4\pi Y)^{\f{1}{2}}} e^{-\f{|z\pm z'|^2}{4Y}},\quad Y=Y(t,t',y)=\mu \lambda_1\overline{\Gamma} \int_{t'}^t b_0(\tau,y)\, \d\tau.
 \end{align*}
%  By using the identity 
 % \beqs  \p_z (K_{-}-K_{+})=\p_{z'}(K_{-}+K_{+}), \eeqs
 First,  we integrate by parts to obtain that % the first quantity  in \eqref{zeta52-explicit} ca
  \begin{align}\label{zeta52intial}
     \mu^{\f{1}{2}}\il \p_z \zeta_{521}\il_{0,\infty,t}\lesssim  \mu^{\f{1}{2}}
     \| \p_z \zeta|_{t=0}\|_{L^{\infty}(\Omega)}\lesssim Y_m(0).
  \end{align}
  Moreover, by applying the convolution inequality in $z$ variable and the estimate \eqref{appen-7}, we find that
  \begin{align*}
       \mu^{\f{1}{2}}\il \p_z \zeta_{522}\il_{0,\infty,t}
      & \lesssim \izt\|\mu^{\f{1}{2}}\p_{z}(K_{-}-K_{+})\|_{L_y^{\infty}L_z^1}(t, t') \,\d t' \il \cM_2\il_{0,\infty,t}\lesssim  T^{\f{1}{2}}\il \cM_2\il_{0,\infty,t}.
  \end{align*}
  By the definitions \eqref{defcM} and \eqref{source5}
  \beq\label{cM2-precise}
  \begin{aligned}
      \cM_2&=\cS_5-\cM_1-u\cdot\na\zeta_5=-(\tilde{u} \cdot \na)\zeta+\bigg(\curl \big( \f{\kpa}{C_v\gamma}u\big(\exp\big({\f{R\ep}{C_v\gamma}\tsigma}\big)\Gamma-\overline{\Gamma}\big)\Delta\theta \big)\times\chi\bn\bigg)^{\Psi}\\
     % &\quad +\mu\lambda_1\overline{\Gamma}\bigg(\big(\f{1}{r_0} \big)^{\Psi}\Delta_g\zeta_5+\big(\big(\f{1}{r_0} \big)^{\Psi}-b_0\big)\p_{z}^2\zeta_{5} \bigg) \\
      &\quad+|g|^{\f{1}{4}} \bigg(H^0+\mu\lambda_1\overline{\Gamma}\big(\big(\na(r_0^{-1})\times \curl\omer \big)\times\chi\bn+ 2 r_0^{-1} \div(\omer\times\na (\chi\bn))\big)\bigg)^{\Psi}+ H^{0,\dag}.
  \end{aligned}
  \eeq
We have first by the definition of $\zeta,$
 %by using the estimate \eqref{zeta3infty}, \eqref{zeta4infty}, 
 \beqs 
 \il(\tilde{u} \cdot \na)\zeta\il_{0,\infty,t}\lesssim \lab \il\big(u, \na (r_0 u)\big)\il_{0,\infty,t} \big)\lesssim \lat.
 \eeqs
Next, by using the equation $\eqref{NCNS-S2}_3$ for $\theta,$
it can be verified that:
\begin{align*}
    \ep\kpa \|\na \Delta\theta\|_{0,\infty,t}\lesssim \il \ep\pt \na\theta\il_{0,\infty,t}+\lat \lesssim \lat.
\end{align*}
We thus find that the $\il\cdot\il_{0,\infty,t}$ norm of the first line in the right hand side of  \eqref{cM2-precise} can be controlled as $\lat.$
Moreover,  in view of the definition of $H^0$ in \eqref{defH}, 
 and $H^{0,\dag}$ in \eqref{defHdag},
 we find that the last line in \eqref{cM2-precise} contains the terms which are the smooth functions of quantities:
 \begin{align*}
     \zeta, (\omer)^{\Psi},\, ((\Id,\na) u)^{\Psi},\, (\na (r_0 u))^{\Psi},\, \big((\Id ,\na )r_1 \Gamma \big)^{\Psi},\, \ep^{-1}\big((\Id ,\na )r_1 (\Gamma-\bar{\Gamma}) \big)^{\Psi} , \,\ep\mu \big((\Id, \curl)\curl\curl u \big)^{\Psi},\,
 \end{align*}
 whose $L^{\infty}([0,t]\times\mR_{+}^3)$ norm, by using the estimate \eqref{tcurl-infty} for the last quantity, can be controlled by $\cA_{m,t}.$
The above arguments  enable us to conclude that
$\il\cM_2\il_{0,\infty,t}\lesssim \lat,$ and thus
\begin{align*}
    \il\zeta_{522}\il_{0,\infty,t}\lesssim T^{\f{1}{2}}\lat,
\end{align*}
which, together with \eqref{zeta52intial}, yields the second estimate in \eqref{zeta5-prop1}.
\end{proof}
\begin{prop}\label{lemze5}
Let $\zeta_5$ be the solution to the system \eqref{eqzeta5}. %and the relation \eqref{asptionmukpa} holds. 
There exists a constant $\vartheta_5\in(0,\f{1}{2}],$ such that
 for any $\ep\in (0,1],$ any $(\mu,\kpa)\in (0,1]^2$ satisfying the relation \eqref{assumption-mukpa}, 
 any $0<t \leq T,$ 
\begin{align}\label{zeta4-energy}
    \|\zeta_5(t)\|_{m-2}^2+\mu\|\na\zeta_5\|_{m-2,t}^2\lesssim \lab Y_m(0)\big)+(T+\ep)^{\vartheta_5}\lae.
    %\lat\cE_{m,t}^2.
\end{align}
\end{prop}
\begin{rmk}
This property can be shown by performing energy estimates for $\zeta_5,$ however, such estimates relies
crucially on the Lemma \ref{lemzeta5-prop} when dealing with some commutator terms. 
\end{rmk}
\begin{proof}
Taking a multi-index $\gamma=(\gamma_1,\gamma_2,\gamma_3)$  %$|\gamma|=k\leq m-2-j$
and applying $\cZ^{\gamma}$ on the equations \eqref{eqzeta5}, we find that $\zeta_{5}^{\gamma}=\colon\cZ^{\gamma}\zeta_5$ solves the following equation:
\begin{align}\label{eqzeta4ga}
\left\{\begin{array}{l}
      (\pt+\tilde{u}\cdot\na)\zeta_{5}^{\gamma}-\mu\lambda_1 \overline{\Gamma}\big(\p_z (\cZ^{\gamma}\p_z\zeta_{5}^{\gamma}/{r}^{\Psi}_0)+\p_{y^k}(({g^{kl}}/{r}^{\Psi}_0)\p_{y^l}\zeta_{5}^{\gamma})\big)=\cZ^{\gamma}\cS_5+\sum_{i=1}^{5}\cR_{i}^{\gamma}, \\
   \zeta_{5}^{\gamma}|_{z=0}=0, \quad \zeta_{5}^{\gamma}|_{t=0}=\cZ^{\gamma}(\zeta|_{t=0}),
\end{array}
\right.
\end{align}
where $\cS_5$ is defined in \eqref{source5} and 
\beq\label{defcR-4}
\begin{aligned}
\cR_{1}^{\gamma}&=-[\cZ^{\gamma},\tlu \cdot\na]\zeta_5,\quad
\cR_{2}^{\gamma}=\mu\lambda_1 \overline{\Gamma}\cZ^{\gamma}\big(\p_z^2(1/r^{\Psi}_0)\zeta_5\big),\\
\cR_{3}^{\gamma}&=-\mu\lambda_1 \overline{\Gamma}([\cZ^{\gamma},\p_z]+\p_z\cZ^{\gamma})\big(\p_z(1/r^{\Psi}_0)\zeta_5\big),\\
\cR_{4}^{\gamma}&=\mu\lambda_1 \overline{\Gamma}\big([\cZ^{\gamma},\p_z](\p_z\zeta_5/{r}^{\Psi}_0)+\p_z([\cZ^{\gamma}, 1/r^{\Psi}_0]\p_z\zeta_5)\big),\\
\cR_{5}^{\gamma}&=\mu\lambda_1 \overline{\Gamma}\bigg(\p_{y^k}\big([\cZ^{\gamma}, {g^{kl}}/{r}^{\Psi}_0]\p_{y^l}\zeta_5\big)-\cZ^{\gamma}\big(\p_{y^k}({|g|^{-\f{1}{2}}}/{r}^{\Psi}_0)g^{kl}|g|^{\f{1}{2}}\p_{y^l}\zeta_5\big)\bigg).
\end{aligned}
\eeq
Before going into details, we summarize some useful properties for $\cR_{1}^{\gamma}-\cR_{5}^{\gamma}$ and $\cZ^{\gamma}\cS_5.$ 
%Let us begin with the case $|\gamma|\leq m-3.$
%\begin{lem}\label{lem-zeta4-1}
%All the properties listed in Lemma \ref{lemcR} and Lemma \ref{lemsource2}  hold true via replacing $\zeta_{2}^j$ by $\zeta_4,$ $\zeta_{1}^j$ by $\zeta_3,$ $\omer^j$ by $\omer$ and the norm $\|\cdot\|_{m-2-j,t}$ by $\|\cdot\|_{m-3,t}.$\end{lem}
 We first gather some estimates for $\cR_i^{\gamma}:$
\begin{lem}\label{lemcR-zeta4}
Let  $\cR_i^{\gamma} (i=1,\cdots 5)$  be defined in  \eqref{defcR-4}. Suppose that $|\gamma|=m-2,$ we have under the assumption \eqref{preasption}
the following estimates: for any $\ep\in(0,1], (\mu,\kpa)\in(0,1]^2$ satisfying  $\mu\sim\kpa$, any $0<t\leq T,$ 
\begin{align}
    \|(\cR_1^{\gamma},\cR_{2}^{\gamma})\|_{0,t}+\| \mu^{\f{1}{2}}\p_z(1/r_{0}^{\Psi})\zeta_5)\|_{m-2,t}\lesssim T^{\f{1}{2}}\lat\cE_{m,t}, \label{es-cR-124}\\
    \|\cR_5^{\gamma}\|_{0,t} \lesssim \mu \|\nabla \zeta_5\|_{m-2,t}\lat+T^{\f{1}{2}}\lat\cE_{m,t}. \label{es-cR-5}
\end{align}
Moreover, if $\mu$ and $\kpa$ satisfy the relation
\eqref{assumption-mukpa}, then it holds that:
\beq\label{cR-4}
\begin{aligned}
&\mu^{\f{1}{2}}\|\p_z\zeta_5/r_{0}^{\Psi}\|_{m-3,t}+\mu^{\f{1}{2}} \|[\cZ^{\gamma}, 1/r_{0}^{\Psi}]\p_z\zeta_5\|_{0,t}\\
&\lesssim \mu^{\f{1}{2}}\|\p_z\zeta_5\|_{m-3,t}\lab \il(\theta,\ep\tsigma)\il_{m-3,\infty,t}\big)
    +(T^{\f{1}{2}}+\ep)\big(\lab Y_m(0)\big)+\lae\big).
\end{aligned}
\eeq
\end{lem}
\begin{rmk}
We remark that the estimate \eqref{cR-4} can be derived with the help of the properties \eqref{zeta5-prop1} and \eqref{propzeta52} in Lemma \ref{lemzeta5-prop}, which is the place where
the relation \eqref{assumption-mukpa} is used. 
\end{rmk}
The following lemma states some estimates for $\cZ^{\gamma}\cS_5, |\gamma|=m-2:$
\begin{lem}\label{lemsource5}
Suppose that $|\gamma|=m-2$ and \eqref{preasption} is satisfied. 
 For any $\ep\in(0,1], (\mu,\kpa)\in(0,1]^2$ satisfying the relation \eqref{assumption-mukpa}, any $0<t\leq T,$ the following aspects hold: \\ 
%Let $\zeta_3=\zeta_5+\zeta_6,$ where $\zeta_6,\zeta_6$ admit the properties \eqref{zeta56-infty}, \eqref{zeta3energy2}, then we have: 
(i) It holds that
\begin{align}\label{convection-zeta34}
\|\tlu\cdot\na (\zeta_3+\zeta_4)\|_{m-2,t}\lesssim T^{\f{1}{2}}\lat\cE_{m,t}.
\end{align}
(ii) Let $\cS_3, \cS_4$ be defined in \eqref{defcS3}, \eqref{defcS4} respectively, then we have:
\begin{align}\label{cs34}
    \cS_3+\cS_4= \mu\lambda_1\overline{\Gamma} \big(\p_{y^1} \cS_{341}+\p_{y^2}\cS_{342}+\p_z\cS_{343}\big)+\cS_{344}
\end{align}
where $\cS_{341}-\cS_{344}$ are some functions that admit 
the properties:
\begin{align}\label{cs344-prop}
   \sum_{j=1}^3 \mu^{\f{1}{2}} \|\cS_{34j}\|_{m-2,t}\lesssim (T^{\f{1}{4}}+\ep)\lat\cE_{m,t}, \quad \|\cS_{344}\|_{m-2,t}\lesssim T^{\f{1}{2}}\lat\cE_{m,t}.
\end{align}
(iii) There exist $\cH_{51}-\cH_{54},$ such that:
\begin{align}\label{factHHdag}
    |g|^{\f{1}{4}} (H_1^0)^{\Psi}+ H^{0,\dag}=\p_{y^1} \cH_{51}+\p_{y^2}\cH_{52}+\p_z\cH_{53}+\cH_{54}
\end{align}
and 
\begin{align}\label{factHHdag1}
   \mu^{-\f{1}{2}} \|(\cH_{51},\cH_{52},\cH_{53})\|_{m-2,t}\lesssim (T+\ep)^{\f{1}{3}}\lae, \quad \|\cH_{53}\|_{m-2,t}\lesssim \lae.
\end{align}
\end{lem}

We will postpone the proof of the above two lemmas until the end of this section and proceed to prove \eqref{zeta4-energy}. 
 
Analogue to \eqref{EI-zeta2}, 
we carry out direct energy estimates for \eqref{eqzeta4ga}
to obtain: 
\beq\label{EI-zeta4}
\begin{aligned}
   & \int_{\mR_{+}^3} |\zeta_5^{\gamma}|^2(t)\,\d x 
    + \mu \iota c_0 \izt\int_{\mR_{+}^3} |\na \zeta_5^{\gamma}|^2 \d x\d s\lesssim \int_{\mR_{+}^3} |\zeta_5^{\gamma}|^2(0)\,\d x\\
    &  +\underbrace{\bigg|\izt\int_{\mR_{+}^3} 2\mu [\cZ^{\gamma},\p_z] \zeta_5\cdot \p_z \zeta_5^{\gamma} /r^{\Psi}_0+ |\zeta_5^{\gamma}|^2 \div \tlu   \,\d x\d s\bigg|}_{\eqref{EI-zeta4}_1}
   +\underbrace{ \bigg|2\izt\int_{\mR_{+}^3}\zeta_5^{\gamma}  \cdot \big(\cZ^{\gamma}\cS_5+\sum_{i=1}^{5}\cR_{i}^{\gamma}\big)\,\d x\d s\bigg|}_{\eqref{EI-zeta4}_2}
\end{aligned}
\eeq
where $\iota>0$ is a universe constant.

First of all, corresponding to \eqref{zeta2j-pre}, we have by using \eqref{norm-equiv-bd},  \eqref{zeta3energy1} and \eqref{zeta4-energy} that, for any $0\leq k\leq m-2,$
\begin{align}\label{zeta4-pre}
    \|\zeta_5\|_{k,t}\lesssim \|(\zeta,\zeta_3,\zeta_4)\|_{k,t}
    \lesssim T^{\f{1}{2}} \sup_{0\leq s\leq t}\|(\zeta,\zeta_3,\zeta_4)(s)\|_{k}\lesssim T^{\f{1}{2}}\lat\cE_{m,t}.
\end{align}
This, together with  the identity \eqref{com-identity} and Young's inequality, enables us to control 
the term $\eqref{EI-zeta4}_1$  as 
\begin{align}\label{sec7:168}
 \eqref{EI-zeta4}_1\leq   \delta\mu c_0 \|\p_z \zeta_5^{\gamma}\|_{0,t}^2+ C_{\delta} \mu\|\p_z \zeta_5\|_{m-3,t}^2+T\lat\cE_{m,t}^2.
\end{align}
%To control the last term in \eqref{EI-zeta4}, %we need to distinguish the case $|\gamma|\leq m-3$ and $|\gamma|=m-2.$ The former is easier and can be treated by following the similar arguments as the control of $\sup_{0\leq t \leq T}\|\zeta_{2}^1 (t)\|_{m-3}$ in Proposition \ref{lemze2},  which relies heavily on Lemma \ref{lemcR} and Lemma \ref{lemsource2}. However, as stated in Lemma \ref{lem-zeta4-1}, all the  properties listed in the above two lemmas hold true via replacing $\zeta_{2}^1$ by $\zeta_5,$ $\zeta_{1}^1$ by $\zeta_3+\zeta_4,$ $\omer^1$ by $\omer.$ Consequently, we can find: for any $|\gamma|\leq m-3,$
%\begin{align*}\eqref{EI-zeta4}_2\leq  \delta\mu c_0 \|\p_z \zeta_5^{\gamma}\|_{0,t}^2+ C_{\delta} \|\p_z \zeta_5\|_{|\gamma|-1}^2 \lab \il(\theta,\ep\tsigma)\il_{m-3,\infty,t}\big)+ (T+\ep)^{\f{1}{2}}\lat\cE_{m,t}^2.\end{align*}
%Inserting the previous two estimates into \eqref{EI-zeta4} and noticing \eqref{thetasigmainfty}, we can get by induction on $|\gamma| (\leq m-3)$ the similar estimate as in \eqref{EE-zeta2}:
%\beq \|\zeta_5\|_{m-3,t}^2+\mu \|\na \zeta_5\|_{m-3,t}^2 \lesssim \lab Y_m^2(0)\big) + (T+\ep)^{\vartheta_2}\lae.\eeq

Let us now focus on the estimate of $\eqref{EI-zeta4}_2.$
By the virtue of favorable estimates established in Lemma \ref{lemcR-zeta4},
the integral involving the terms $\cR_1^{\gamma}-\cR_5^{\gamma}$ can be treated in the same way as we did in \eqref{sec7:eq41}-\eqref{sec7:eq44}. %once some favorable estimates are established for $\cR_1^{\gamma}-\cR_5^{\gamma},$ which is summarized in the following lemma.   
%the similar estimates in Lemma \ref{lemcR} are established for $\cR_1^{\gamma}-\cR_5^{\gamma},$ which is the 
More precisely, by using the estimate \eqref{es-cR-124} for $\cR_1^{\gamma}-\cR_3^{\gamma},$ \eqref{cR-4} for $\cR_4^{\gamma},$ \eqref{es-cR-5} for $\cR_5^{\gamma},$
we can obtain after suitable integration by parts 
the following estimate corresponding to 
\eqref{sec7:eq45}:
\begin{align}\label{sec7:eq96}
  \bigg| \sum_{i=1}^5 \izt\int_{\mR_{+}^3}\zeta_5^{\gamma}  \cdot \cR_{i}^{\gamma} \,\d x\d s\bigg|&\leq \delta c_0\mu
    \|\na \zeta_5 \|_{k,t}^2+ (T+\ep)\big(\lab Y_m(0)\big)+\lae\big)\notag\\
   &\quad+C_{\delta}\mu\|\nabla\zeta_5\|_{k-1,t}^2\lab\il(\theta,\ep\tsigma)\il_{m-3,\infty,t}^2\big).
\end{align}
Let us now control the term $\big|  \izt\int_{\mR_{+}^3}\zeta_5^{\gamma}  \cdot \cZ^{\gamma}\cS_5 \,\d x\d s\big|.$ 
Thanks to the properties for the source term $\cS_5$ listed in Lemma \ref{lemsource5}, we get after suitable integration by parts that:
\beq\label{sec7:eq170}
\begin{aligned}
    \big|  \izt\int_{\mR_{+}^3}\zeta_5^{\gamma}  \cdot \cZ^{\gamma}\cS_5 \,\d x\d s\big|&\lesssim \mu^{\f{1}{2}}\|\na\cZ^{\gamma}\zeta_5\|_{0,t}\sum_{l=1}^3\mu^{-\f{1}{2}}\|(\mu\cS_{34l},\cH_{5l})\|_{m-2,t}\\
    &\quad+ \|\zeta_5\|_{m-2,t}\|\big(\tlu\cdot\na (\zeta_3+\zeta_4), \cS_{344}, \cH_{54}\big)\|_{m-2,t}\\
    &\leq \delta  c_0\mu
    \|\na \zeta_5 \|_{k,t}^2+(T+\ep)^{\f{1}{2}}\lat\cE_{m,t}^2.
\end{aligned}
\eeq
Inserting \eqref{sec7:168}-\eqref{sec7:eq170} into 
\eqref{EI-zeta4}, we find, by summing up for $|\gamma|=k$ and by choosing $\delta$ small enough,
\beqs
\begin{aligned}
   \|\zeta_5\|_{k,t}^2+\mu\|\na\zeta_5\|_{k,t}^2
     &\lesssim \lab Y_m(0)\big)+ (T+\ep)^{\f{1}{2}}\lae\\
    &\quad+\|\p_z \zeta_5\|_{k-1,t}^2 \lab \il(\theta,\ep\tsigma)\il_{m-3,\infty,t}\big).
\end{aligned}
\eeqs
Consequently, by noticing \eqref{thetainfty}, \eqref{tsigmainfty},
we can do induction on $k$ to find that, there exists a constant $\vartheta_5\in (0,\f{1}{2}]$ such that:
%by choosing $\vartheta_3$ found in Proposition \ref{lemze2} smaller if necessary
\beqs
\|\zeta_5\|_{m-2,t}^2+\mu \|\na \zeta_5\|_{m-2,t}^2 \lesssim \lab Y_m(0)\big) + (T+\ep)^{\vartheta_5}\lae.
\eeqs
%\begin{align*}\eqref{EI-zeta4}_2\leq \lab Y_m^2(0)\big)+2 \delta\mu c_0 \|\na \zeta_5^{\gamma}\|_{k,t}^2+ C_{\delta} \|\p_z \zeta_5\|_{k-1,t}^2 \lab \il(\theta,\ep\tsigma)\il_{m-3,\infty,t}\big)+ (T+\ep)\lat\cE_{m,t}^2.\end{align*}
\end{proof}

%\begin{proof}[\textbf{Proof of Lemma \ref{lem-zeta4-1}}]It is just a matter to recheck such properties in Lemma \ref{lemcR} and Lemma \ref{lemsource2}. One can  find that some of them can indeed be better in the sense of some extra positive power of $T$ can show up.For instance, corresponding to \eqref{convection-zeta1}, we use identity \eqref{identity-convection}, the estimate  \eqref{zeta3energy1} for $\zeta_3$and the generalized Gargliardo-Nirenberg inequality to get that:
%\begin{align*}\|{(\tilde{u} \cdot \na)\zeta_{3}}\|_{m-3,t} &\lesssim    \|(\tlu_1, \tlu_2, \tlu_3/\phi, \cZ\zeta_{3})\|_{m-3,t}\lab \il(\tlu_1, \tlu_2, \tlu_3/\phi, \cZ\zeta_{3})\il_{0,\infty,t} \big)\lesssim T^{\f{1}{2}}\lat\cE_{m,t}.\end{align*}
%Although tedious, the proofs of the other properties estimates are very similar to those presented in Lemma \ref{lemcR} and Lemma \ref{lemsource2}. We thus omit the details. Nevertheless,
%We will not detail the other properties since it is tedious but direct. it is  worthy to point out that here the highest regularity $m-3$ is crucial especially in the estimate of the term $\cR_4^{\gamma}$ (corresponding to \eqref{es-cR4}) where one can take benefits of uniform control of $\il 1/r_0^{\Psi}\il_{m-3,\infty,t}.$ Nevertheless, when $|\gamma|=m-2,$ it requires in general the control of 
%$\il 1/r_0^{\Psi}\il_{m-2,\infty,t}$ which cannot be uniform in $\mu.$ This aspect will be more clear when we prove \eqref{cR-4} in the following.\end{proof}

\begin{proof}[\textbf{Proof of Lemma \ref{lemcR-zeta4}}]
The estimates \eqref{es-cR-124} \eqref{es-cR-5} follow from the product estimate \eqref{GN1} and the commutator estimate \eqref{GN2}. Let us show for instance the estimate for $\cR_{1}^{\gamma}=-[\cZ^{\gamma},\tlu \cdot\na]\zeta_4.$ %[\cZ^{\gamma},\tlu \cdot\na]
In view of the identity \eqref{com-convetion-id}, we can apply the commutator estimate \eqref{GN2} to obtain that:
\begin{align*}
   \| \cR_{1}^{\gamma}\|_{0,t}=\|[\cZ^{\gamma},\tlu \cdot\na]\zeta_5\|_{0,t}\lesssim \|(\tlu_1,\tlu_2,\tlu_3/\phi)\|_{m-2,t}\il\zeta_5\il_{1,\infty,t}+\|\zeta_5\|_{m-2,t}\il (\tlu_1,\tlu_2,\tlu_3/\phi)\il_{1,\infty,t}.
\end{align*}
Note that by \eqref{zeta3infty}, \eqref{zeta4infty}, one has:
\beqs 
\il\zeta_5\il_{1,\infty,t}\lesssim \il (\zeta, \zeta_3,\zeta_4) \il_{1,\infty,t}\lesssim \lat,
\eeqs
which, together with \eqref{zeta4-pre} \eqref{sec7:eq80} and the fact 
$$\|(\tlu,\na\tlu)\|_{m-2,t}\lesssim \|(u,\na u)\|_{L_t^2\cH^{m-2}}\lesssim T^{\f{1}{2}}\cE_{m,t},
$$
gives rise to:
\begin{align*}
   \| \cR_{1}^{\gamma}\|_{0,t} \lesssim T^{\f{1}{2}}\lat\cE_{m,t}.
\end{align*}

We now prove \eqref{cR-4}, which relies crucially on \eqref{zeta5-prop1}. 
By using \eqref{r0infty}, we can have directly: %the first term can be controlled directly:
\beq\label{sec7:eq98}
\begin{aligned}
   & \mu^{\f{1}{2}}\|\p_z\zeta_5/r_{0}^{\Psi}\|_{m-3,t}+\mu^{\f{1}{2}} \|[\cZ^{\gamma}, 1/r_{0}^{\Psi}]\p_z\zeta_5-\cZ^{\gamma}(1/r_{0}^{\Psi})\p_z\zeta_4\|_{0,t}\\
&\lesssim\mu^{\f{1}{2}}\|\p_z\zeta_5\|_{m-3,t} \il  1/r_{0}^{\Psi} \il_{m-3,\infty,t}\lesssim \mu^{\f{1}{2}}\|\p_z\zeta_5\|_{m-3,t}\lab \il(\theta,\ep\tsigma)\il_{m-3,\infty,t}\big).
\end{aligned}
\eeq
It now remains to control $\mu^{\f{1}{2}}\|\cZ^{\gamma}(1/r_{0}^{\Psi})\p_z\zeta_4\|_{0,t}$ which is the most problematic one. As a matter of fact, in general, the term $\mu^{\f{1}{2}}\p_z\zeta_5(s)$ %appearing in  $\cR_4^{\gamma}$ (defined in \eqref{defcR-4})
cannot be uniformly bounded in $L^P(\mR_{+}^3)$ for $2<p<+\infty,$ one is forced to attribute the $L^{\infty}(\mR_{+}^3)$ norm on $\cZ^{\gamma}(1/r_{0}^{\Psi}),$  which is not uniformly (in $\mu$) bounded. To get around this difficulty, we assume the relation \eqref{assumption-mukpa} holds and take benefit of \eqref{zeta5-prop1}-\eqref{propzeta52} to  obtain that:
\beq\label{sec7:eq99}
\begin{aligned}
  \mu^{\f{1}{2}}\|\cZ^{\gamma}(1/r_{0}^{\Psi})\p_z\zeta_5\|_{0,t}&\lesssim T^{\f{1}{2}} %\sup_{0\leq s\leq t}
  \| \cZ^{\gamma}(1/r_{0}^{\Psi})\|_{L_t^{\infty}L^2}
  \il\mu^{\f{1}{2}} \p_z\zeta_{51}\il_{0,\infty,t} 
  + \il\mu^{\f{1}{2}}/r_{0}^{\Psi}\il_{m-2,\infty,t}
  \|\p_{z}\zeta_{52}\|_{0,t}\\
  &\lesssim (T^{\f{1}{2}}+\ep)\big(\lab Y_m(0)\big)+\lae\big).
\end{aligned}
\eeq
Note that by the definition \eqref{defr0} for $r_0$
\begin{align*}
    \il\mu^{\f{1}{2}}/r_{0}^{\Psi}\il_{m-2,\infty,t}\lesssim \lab \il (\Id, \mu^{\f{1}{2}}\cZ)(\theta,\ep\tsigma) \il_{m-3,\infty,t}\big)\lesssim \lat.
\end{align*}
We thus finish the proof of \eqref{cR-4} by collecting \eqref{sec7:eq98} and \eqref{sec7:eq99}.
\end{proof}
\begin{proof}[\textbf{Proof of Lemma \ref{lemsource5}}]
We begin with the proof of the first aspect, namely the estimate \eqref{convection-zeta34}. 
It follows from the identity \eqref{identity-convection}, 
the commutator estimate \eqref{GN1} that:
\begin{align*}
\| ( \tlu \cdot\na)  (\zeta_3+\zeta_4)\|_{m-2,t}\lesssim \|(\zeta_3,\zeta_4)\|_{m-1,t} \il(\tlu, \tlu_3/\phi) \il_{0,\infty,t}+
\|(\tlu, \tlu_3/\phi)\|_{m-2,t}\il (\zeta_3,\zeta_4)\il_{1,\infty,t},
\end{align*}
which, combined with estimates \eqref{zeta3infty}, \eqref{zeta3energy1} for $\zeta_3$, \eqref{zeta4infty}, \eqref{eszeta4} for $\zeta_4$ and the fact $\|(\tlu, \tlu_3/\phi )\|_{m-2,t}\lesssim T^{\f{1}{2}}\|(u,\na u)\|_{L_t^{\infty}\cH^{m-2}},$ yields \eqref{convection-zeta34}.

Let us now sketch the second aspect concerning the fact
\eqref{cs34} which can be shown by following the similar arguments
as in the proof of \eqref{diffusion-zeta1}. In view of the definitions \eqref{defcS3}, \eqref{defcS4} for $\cS_3, \cS_4$ and the identity 
\eqref{zeta3-secnor},  let us set
\begin{align*}
    &  \cS_{34k}=  (1/r^{\Psi}_0)\big( g^{kl}\p_{y^l}(\zeta_3+\zeta_4)-\p_{y^k}\zeta_4\big),\, k=1,2, \quad  \\
   & \cS_{343}=\big((1/r^{\Psi}_0)- (1/r^{\Psi}_0)|_{z=0}\big)\p_z\zeta_3-\p_z(1/r^{\Psi}_0)\zeta_3+b_2\big(\exp(\ep^2\mu \bar{C}\Gamma\div u)-1\big)^{\Psi}|_{z=0}\p_z\zeta_3,\\
   & \cS_{344}= \mu\lambda_1\overline{\Gamma}\bigg(-\p_{y^k}(|g|^{-\f{1}{2}}/r^{\Psi}_0)
    g^{kl}|g|^{\f{1}{2}}\p_{y^l}(\zeta_3+\zeta_4)+\p_{y^k}(1/r^{\Psi}_0)\p_{y^k}\zeta_4 +\p_z^2  (1/r^{\Psi}_0) \zeta_3+(1/r^{\Psi}_0)(\pt\zeta_4-\zeta_4)\bigg)
-\pt\zeta_4.
\end{align*}
The desired estimates listed in \eqref{cs344-prop} are
the consequences of the product estimate \eqref{GN1},
Lemma \ref{lemze3} and Lemma \ref{lemze4} for 
$\zeta_3$ and $\zeta_4,$ we omit the details.

Finally, for the third aspect, we check the terms appearing in the definitions \eqref{defH}, \eqref{defHdag} for $(H^0)^{\Psi}$ and $ H^{0,\dag}$ to find that:
\beq\label{sec7:eq173}
\begin{aligned}
    &|g|^{\f{1}{4}} (H^0)^{\Psi}+ H^{0,\dag}
    = l.o.t+\cM_1-(\cS_3+\cS_4)\\%+\bigg(\curl \big( \f{\kpa}{C_v\gamma}u\big(\exp\big({\f{R\ep}{C_v\gamma}\tsigma}\big)\Gamma-\overline{\Gamma}\big)\Delta\theta \big)\times\chi\bn\bigg)^{\Psi}\\
    & +\bigg(\curl\bigg(  \f{\kpa}{C_v\gamma}u\big(\exp\big({\f{R\ep}{C_v\gamma}\tsigma}\big)\Gamma-\overline{\Gamma}\big)\Delta\theta+\ep\mu\lambda_1 \big(\big(r_1\Gamma +\ep^{-1}(\Gamma-\overline{\Gamma})\big)\curl\curl u \big)\bigg)\times\chi\bn\bigg)^{\Psi}
\end{aligned}
\eeq
where $\cM_1$ is defined in \eqref{defcM}
and $l.o.t$ stands for the terms that admit the property:
\begin{align*}
    \|l.o.t\|_{m-2,t}\lesssim %(T^{\f{1}{2}}+\ep)
    \lat\cE_{m,t}.
\end{align*}
However, in the proof of Lemma \ref{lemzeta5-prop},
we found that there exists $\cH_{1}-\cH_{3}$ such that
\begin{align*}
    \cM_1-(\cS_3+\cS_4)= \p_{y^1}\cH_{1}+\p_{y^2}\cH_{2}+\p_{z}\cH_{3}+l.o.t.
\end{align*}
Moreover, if the relation \eqref{assumption-mukpa} holds, then (see \eqref{es-ch1-3}):
\begin{align}\label{sec7:171}
    \mu^{-\f{1}{2}} \|(\cH_{1}, \cH_{2}, \cH_{3})\|_{m-2,t}\lesssim (T^{\f{1}{2}}+\ep)\lat\cE_{m,t}.
\end{align}
By using the identity \eqref{id-curlftimesn}, we can show again that 
there exists $\cH_5-\cH_7$, such that the last term of $\eqref{sec7:eq173}$ can be rewritten in the following form:
\begin{align*}
\p_{y^1}\cH_{5}+\p_{y^2}\cH_{6}+\p_{z}\cH_{7}+l.o.t.,
\end{align*}
where $\cH_5-\cH_7$ are smooth functions with respect to the quantities:
\beqs 
\bigg(\kpa u\big(\exp\big({\f{R\ep}{C_v\gamma}\tsigma}\big)\Gamma-\overline{\Gamma}\big)\Delta\theta \big)\bigg)^{\Psi},\quad \ep\mu\lambda_1 \bigg(\big(r_1\Gamma +\ep^{-1}(\Gamma-\overline{\Gamma})\big)\curl\curl u \bigg)^{\Psi}
\eeqs
we thus can conclude from  the product estimate \eqref{GN1} that:
\begin{align*}
    \mu^{-\f{1}{2}}\|\cH_{5},\cH_{6},\cH_{7}\|_{m-2,t}&\lesssim \big(\ep\| \kpa^{\f{1}{2}}\Delta\theta\|_{L_t^2\cH^{m-2}}
    +\ep^{\f{1}{3}}\|(\ep^{\f{2}{3}}\mu^{\f{1}{2}})\na^2 u \|_{L_t^2\cH^{m-2}}+T^{\f{1}{2}}\|(\sigma,\tsigma)\|_{L_t^{\infty}\cH^{m-2}}\big)\lat\\
    &\lesssim (T+\ep)^{\f{1}{3}}\lat\cE_{m,t}.
\end{align*}
This, together with \eqref{sec7:171}, enables us to 
finish the proof of \eqref{factHHdag} and \eqref{factHHdag1}  by setting $\cH_{51}=\cH_{1}+\cH_5, \cH_{52}=\cH_{2}+\cH_6, \cH_{53}=\cH_{3}+\cH_7.$
\end{proof}

\section{$L_{t,x}^{\infty}$ estimates for $\theta$}
In this section, we aim to bound the several $L_{t,x}^{\infty}$ norms for $\theta,$ which will be useful to recover higher order $L_t^{\infty}$ type conormal regularities of $(\na\sigma, \div u)$ in the next section.
\begin{lem}
%Let  $\mathcal{A}_{m,T}(\theta)$ be defined in  \eqref{defcAmt} and is recalled here again:
Recall the definition of $\mathcal{A}_{m,t}(\theta):$
\beq\label{amttheta}
\begin{aligned}
 \mathcal{A}_{m,t}(\theta)&= \il (\theta, \kpa \na\theta)\il_{m-3,\infty,t}
 +\kpa^{\f{1}{2}}\il \na\theta\il_{m-4,\infty,t}+\il
  (\nabla \theta, \kpa\na^2\theta)\il_{m-5,\infty,t}
 + \kpa^{\f{1}{2}}\il\na^2\theta\il_{m-6,\infty,t}.
\end{aligned}
\eeq
Under the assumption \eqref{preasption}, there exists a constant $\vartheta_6>0$ such that for any $\ep\in (0,1], (\mu,\kpa)\in A,$ any $0<t\leq T,$ the following estimate holds:
\begin{align}\label{nathetainfty}
    \cA_{m,t}(\theta)\lesssim  \lab Y_m(0)\big) +(T+\ep)^{\vartheta_6}\lae.
\end{align}
\end{lem}
\begin{proof}
First of all, it follows from the inequality \eqref{thetainfty} and  the estimates %\eqref{EE-theta-0},
\eqref{EE-theta-1}, \eqref{EE-theta-2-1} that:
\begin{align}\label{thetasob}
   \il \theta\il_{m-3,\infty,t}+ \kpa^{\f{1}{2}}\il \na\theta\il_{m-4,\infty,t}\lesssim \lab Y_m(0)\big)+(T+\ep)^{\f{\tilde{\vartheta}_0}{2}}\lae.
\end{align}
Similarly, by the Sobolev embedding \eqref{sobebd} and the estimates \eqref{EE-2}, \eqref{EE-theta-2-2},
\beq\label{thetasob-2}
\begin{aligned}
   \kpa\il\na \theta\il_{m-3,\infty,t}
  &\lesssim \kpa \|\na^2\theta\|_{\infco^{m-2}}+\kpa \|\na\theta\|_{\infco^{m-1}}
 \\
 &\lesssim \lab Y_m(0)\big)+(T+\ep)^{\f{\tilde{\vartheta}_0}{2}}\lae.
\end{aligned}
\eeq

It thus remains to control the last three terms in \eqref{amttheta}. We note however that away from the boundaries  where the conormal Sobolev norm is equivalent to the usual Sobolev norm,  these three terms can be bounded by the standard Sobolev embedding. Therefore, it suffices to control 
 \beq\label{inftytheta-tobebound}
 \il\chi_{i}\partial_{\bn} \theta\il_{m-5,\infty,t}, \quad \kpa^{\f{1}{2}}\il\chi_{i}\partial_{\bn} ^2\theta\il_{m-6,\infty,t}, \quad \kpa \il\chi_i \p_{\bn}^2 \theta\il_{m-5,\infty,t},
 \eeq
 where $\chi_i, (1\leq i\leq N)$ are smooth functions compactly supported in $\Omega_i.$ 
 
 Let us start with the derivation of $\chi_i\p_{\bn}\theta.$ The equation $\eqref{NCNS-S2}_3$ together with $\eqref{NCNS-S2}_1,$ allow us to rewrite the equation satisfied by $\theta:$
 \begin{align*}
 (\pt+u\cdot \na)\theta-\f{R\kpa}{C_v\gamma}(\Gamma\beta) \Delta\theta= \f{R}{ C_v\gamma}\big(\ep(\pt+u\cdot\na)\sigma+\ep\mathfrak{N}+\kpa\Gamma \na\beta\cdot\na\theta\big)=\colon \mathfrak{M}_1^i
 \end{align*}
 where $\mathfrak{N}$ is defined in \eqref{defgamma-beta}. 
 Taking the gradient of the above equation and multiplying the resultant by $\chi_i\bn,$ we find that 
 $\chi_i\p_{\bn}\theta$ solves the system:
 \begin{align}\label{eqpntheta}
     (\pt + w\cdot\na)(\chi_i\p_{\bn}\theta)-\f{R\kpa}{C_v\gamma}(\Gamma\beta) \Delta(\chi_i\p_{\bn}\theta)= \mathfrak{M}_2^i, \quad \chi_i\p_{\bn}\theta|_{\p\Omega}=0
 \end{align}
 where $w=u-\f{\gamma-1}{\gamma}\kpa\overline{\Gamma}\beta\na\theta$ and 
 \beq\label{defm2}
 \begin{aligned}
\mathfrak{M}_2^i&=\big( \na  \mathfrak{M}_1^i+\f{R\kpa}{C_v\gamma}\na(\Gamma\beta)\Delta\theta -\na u\cdot\na\theta\big)\cdot\chi_i\bn\\
&\quad+ u\cdot \na(\chi_i\bn)\cdot\na\theta+\f{R\kpa}{C_v\gamma}\Gamma\beta[\chi_i\p_{\bn}, \Delta]\theta - \f{\gamma-1}{\gamma}\kpa\overline{\Gamma}\beta\na\theta \cdot\na (\chi_i\p_{\bn}\theta).
 \end{aligned}
 \eeq
 The reason why we change the transport term $u\cdot \na $ to $w\cdot\na$ will be clear later.
 
The estimates of the three quantities in \eqref{inftytheta-tobebound} will rely on the careful studies on the transport-diffusion equation \eqref{eqpntheta}. Such an equation, when transformed in half space by using the normal geodesic coordinates \eqref{normal geodesic coordinates}, has the similar form as \eqref{eqheattransport}, whose solution admits the 'explicit' expression \eqref{green-heat}.
%Therefore, we will use  in which the Laplacian has favorable form \eqref{laplaceg}. 
Denote $\eta_i=\chi_i\p_{\bn}\theta\circ \Psi_i,\, f^{\Psi_i}=f\circ \Psi_i,$  where 
$\Psi_i$ is the normal geodesic coordinates defined in \eqref{normal geodesic coordinates}. By using \eqref{change-variable}-\eqref{laplaceg}, $\eta_i$ solves:
\begin{align}\label{eqetai-0}
\left\{
 \begin{array}{l}
      (\pt+\tilde{w}_i\cdot\na)\eta_i-\f{R\kpa}{C_v\gamma}
      ({\Gamma\beta})^{\Psi_i} \p_z^2\eta_i= \f{R\kpa}{C_v\gamma}
      \big({\Gamma\beta})^{\Psi_i} (\f{1}{2}\p_z\ln |g_i|\p_z+\Delta_{g_i}\big)\eta_i+({\mathfrak{M}_2^i})^{\Psi_i}, \\[3pt]
      \eta_i|_{t=0}=\chi_i\p_{\bn}\theta_0\circ \Psi_i,  \quad  \eta_i|_{z=0}=0,
 \end{array} 
 \right.
\end{align}
where $\na=(\p_{y^1}, \p_{y^2}, \p_{y^3})$ and $\div=(\na)^{*}$ 
stand for the gradient and the divergence in the new coordinates  and $\Delta_{g_i},$ defined in \eqref{defdeltag}, involves only tangential derivatives.
 Moreover, $\tilde{w}_i$ is defined as:
\beq\label{deftildew}
\tilde{w}_i=(\tilde{w}_{i,1},\tilde{w}_{i,2},\tilde{w}_{i,3})=( D\Psi_i|_{z=0})^{-*}(\chi_i w)^{\Psi_i}.
\eeq
Note that on the boundary $z=0,$ 
\beq%\label{tlu3-bd}
\tilde{w}_{i,3}|_{z=0}=\big[( D\Psi_i|_{z=0})^{-*}(\chi_i w)^{\Psi_i}\big]_3\big|_{z=0}=- ((\chi_i w\cdot \bN^i)|_{\p\Omega})^{\Psi_i}=0.
\eeq
In order to apply Lemma \ref{lemheatgreen} in the appendix, let us rewrite the system \eqref{eqetai-0} as:
\begin{align}\label{eqetai}
\left\{
 \begin{array}{l}
      (\pt+\tilde{w}_{i,1}^b\p_{y^1}+\tilde{w}_{i,2}^b\p_{y^2}+z(\p_z\tilde{w}_{i,3})^b\p_z)\eta_i-\f{R\kpa}{C_v\gamma}
     [ ({\Gamma\beta})^{\Psi_i}]^b \p_z^2\eta_i={\mathfrak{M}}^{\Psi_i}_{3}, %\f{R\kpa}{C_v\gamma} \big({\Gamma\beta})^{\Psi_i} (\f{1}{2}\p_z\ln |g_i|\p_z+\Delta_{g_i}\big)\eta_i+({\mathfrak{M}_2})^{\Psi_i}
     \\[3pt]
      \eta_i|_{t=0}=\chi_i\p_{\bn}\theta_0\circ \Psi_i,  \quad  \eta_i|_{z=0}=0,
 \end{array} 
 \right.
\end{align}
where we denote $f^b=f|_{z=0}$ and ${\mathfrak{M}}^{\Psi_i}_3={\mathfrak{M}}^{\Psi_i}_{3,1}+{\mathfrak{M}}^{\Psi_i}_{3,2}+{\mathfrak{M}}^{\Psi_i}_{3,3}+{\mathfrak{M}}^{\Psi_i}_{3,4}$ with
\beq\label{deffrakm3}
\begin{aligned}
  &{\mathfrak{M}}^{\Psi_i}_{3,1}=\f{R\kpa}{C_v\gamma}
      ({\Gamma\beta})^{\Psi_i} \big(\f{1}{2}(\p_z\ln |g_i|)\p_z+\Delta_{g_i}\big)\eta_i,\quad    {\mathfrak{M}}^{\Psi_i}_{3,2}=({\mathfrak{M}_2^i})^{\Psi_i}, \, ({\mathfrak{M}_2^i}
        \textnormal{ is defined in } \eqref{defm2})\\
        & {\mathfrak{M}}^{\Psi_i}_{3,3}=-\sum_{l=1}^2 (\tilde{w}_{i,l}-\tilde{w}_{i,l}^b)\p_{y^l}\eta_i- \big(\tilde{w}_{i,3}-\tilde{w}_{i,3}^b-z  (\p_z\tilde{w}_{i,3})^b\big)\p_z \eta_i,  \\
 & {\mathfrak{M}}^{\Psi_i}_{3,4}= \f{R\kpa}{C_v\gamma}
  \bigg(  ({\Gamma\beta})^{\Psi_i}- [ ({\Gamma\beta})^{\Psi_i}]^b \bigg)\p_z^2\eta_i.
\end{aligned}
\eeq
From now on, we  will drop the subscript $i$  for the sake of notational convenience.  
Applying the properties \eqref{infty1}, \eqref{infty3} in Lemma \ref{lemheatgreen} by setting $a_1= \tilde{w}_1^b, \, a_2= \tilde{w}_2^b, \, a_3= (\p_z \tilde{w}_3)^b, \, b=[({\Gamma\beta})^{\Psi}]^b, \, \nu=\f{R\kpa}{C_v\gamma},$ we find that: 
\beq\label{etainfty-pre}
\begin{aligned}
   & \il\eta\il_{m-5,\infty,t}\lesssim \Lambda_1^{m-5} \|\p_{\bn}\theta(0)\|_{m-5,\infty}\\
   & \qquad  \qquad \qquad+T^{\f{1}{2}}\Lambda_1^{m-5}
    \bigg( \Lambda_2^{m-5}\il\eta\il_{m-5,\infty,t}+ \big(\izt \| {\mathfrak{M}}^{\Psi}_3 \|_{m-5,\infty}^2\d s\big)^{\f{1}{2}} \bigg),\\
   & \kpa^{\f{1}{2}}\il\p_z\eta\il_{m-6,\infty,t}\lesssim \kpa^{\f{1}{2}}\Lambda_1^{m-6}\|\p_{\bn}^2\theta(0)\|_{m-6,\infty}+ T^{\f{1}{2}}\Lambda_1^{m-6}\big(\Lambda_2^{m-6}\il\eta\il_{m-6,\infty,t} +\il {\mathfrak{M}}^{\Psi}_3\il_{m-6,\infty,t}\big),\\ 
  & \kpa\il\p_z\eta\il_{m-5,\infty,t}\lesssim \kpa \Lambda_1^{m-5}\|\p_{\bn}^2\theta(0)\|_{m-5,\infty}+ T^{\f{1}{2}}\Lambda_1^{m-5}\Lambda_2^{m-5}
 \il(\eta,  \kpa^{\f{1}{2}} {\mathfrak{M}}^{\Psi}_3)\il_{m-5,\infty,t},
\end{aligned}
\eeq
where we denote
\begin{align*}
    \Lambda_1^{k}= \lab | T(\tilde{w}^b, (\p_z\tilde{w}_3)^b), [({\Gamma\beta})^{\Psi}]^b |_{k,\infty,t} \big), \quad \Lambda_2^k=\lab |(\tilde{w}^b, (\p_z\tilde{w}_3)^b), \pt [({\Gamma\beta})^{\Psi}]^b |_{k,\infty,t}
    \big) .
\end{align*}
We derive from the definition of $\tilde{w}$ in \eqref{deftildew} that:
\beq\label{es-Lambda12}
\begin{aligned}
   &  \Lambda_1^{m-5}\lesssim \lab \il T( w, \na w), \Gamma\beta \il_{m-5,\infty,t} \big) \lesssim 
     \lab \il (\ep\sigma, \theta) \il_{m-5,\infty,t}+T \il u, \na u, \kpa (\na\theta, \na^2\theta)\il_{m-5,\infty,t} \big), \\
     & \Lambda_2^{m-5}\lesssim \lab \il \ep\pt,\Id) (\sigma,\theta),u, \na u, \kpa (\na\theta, \na^2\theta) \il_{m-5,\infty,t} \big)\lesssim \lat,
\end{aligned}
\eeq
By using the inequality \eqref{thetainfty} %\eqref{tsigmainfty}
and the fact that $\Lambda$ is a polynomial with respect to its arguments, we can find a constant $\vartheta_7>0,$ such that
\begin{align}\label{es-Lambda1}
 \Lambda_1^{m-5}
 \lesssim \lab Y_m(0)\big) + (T+\ep)^{\f{\vartheta_7}{2}}\lae.
\end{align}
Therefore, it holds that:
\begin{align*}
   & \Lambda_1^{m-5}\big(\|\p_{\bn}\theta(0)\|_{m-5,\infty}+\kpa^{\f{1}{2}}\|\p_{\bn}^2\theta(0)\|_{m-6,\infty}+\kpa\|\p_{\bn}^2\theta(0)\|_{m-5,\infty} \big)\\
    &\lesssim 
    \lab Y_m(0)\big) + (T+\ep)^{\vartheta_7}\lae.
\end{align*}
In view of  \eqref{etainfty-pre}
and the fact $\il\eta\il_{m-5,\infty,t}\lesssim \il\na\theta\il_{m-5,\infty,t}\lesssim \lat,$ it suffices for us to control the three terms involving ${\mathfrak{M}}^{\Psi}_3$ which is done in the next lemma.  
Inserting \eqref{es-frakm3} into the above three inequalities and changing back to the original variable, 
we achieve that,  by defining $\vartheta_6=\min \{ \vartheta_7, \f{1}{2}, \f{\tilde{\vartheta}_0}{2}\},$
\begin{align*}
    \il\chi_{i}\partial_{\bn} \theta\il_{m-5,\infty,T}+\kpa^{\f{1}{2}}\il\chi_{i}\partial_{\bn} ^2\theta\il_{m-6,\infty,T}+ \kpa \il\chi_i \p_{\bn}^2 \theta\il_{m-5,\infty,T}\lesssim \lab Y_m(0)\big) + (T+\ep)^{\vartheta_6}\lae,
\end{align*}
which, together with \eqref{thetasob}, \eqref{thetasob-2},  leads to \eqref{amttheta}.
\end{proof}
\begin{lem}
Let ${\mathfrak{M}}^{\Psi}_3={\mathfrak{M}}^{\Psi}_{3,1}+{\mathfrak{M}}^{\Psi}_{3,2}+{\mathfrak{M}}^{\Psi}_{3,3}+{\mathfrak{M}}^{\Psi}_{3,4}$ be defined in \eqref{deffrakm3}, the following estimates hold:
\begin{align}\label{es-frakm3}
  \big( \izt \| {\mathfrak{M}}^{\Psi}_3 \|_{m-5,\infty}^2\d s\big)^{\f{1}{2}}+ \il{\mathfrak{M}}^{\Psi}_3\il_{m-6,\infty,t}+ \kpa^{\f{1}{2}}\il{\mathfrak{M}}^{\Psi}_3\il_{m-5,\infty,t}\lesssim \lae.
\end{align}
\end{lem}
\begin{proof}
First, in view of the very definitions of ${\mathfrak{M}}^{\Psi}_{3,1}, {\mathfrak{M}}^{\Psi}_{3,2}$ in \eqref{deffrakm3}, it is direct to verify that:
\beq\label{escm3-1,2}
\begin{aligned}
\il{\mathfrak{M}}^{\Psi}_{3,1}, {\mathfrak{M}}^{\Psi}_{3,2}\il_{m-5,\infty,t}
\lesssim \lab & \, \kpa \il\na\theta\il_{m-3,\infty,t}+\il\na\sigma\il_{m-4,\infty,t}\\
&+ \il(\Id, \na)(\sigma, u, \theta), \kpa\na^2\theta, \ep\mu\na^2 u\il_{m-5,\infty,t} \big)\lesssim \lat.
\end{aligned}
\eeq
We now focus on the control of 
${\mathfrak{M}}^{\Psi}_{3,3}, {\mathfrak{M}}^{\Psi}_{3,4}.$ For the term ${\mathfrak{M}}^{\Psi}_{3,3},$ we bound $(\tilde{w}_{l}-\tilde{w}_{l}^b)\p_{y^l}\eta, \, l=1,2,$ as: 
\begin{align*}
    \il (\tilde{w}_{l}-\tilde{w}_{l}^b)\p_{y^l}\eta \il_{m-5,\infty,t}&\lesssim 
    \il z^{-1}(\tilde{w}_{l}-\tilde{w}_{l}^b) \il_{m-5,\infty,t}
    \il z\eta\il_{m-4,\infty,t}\\
    &\lesssim  \il \na w\il_{m-5,\infty,t} \il\theta\il_{m-3,\infty,t}.
\end{align*}
Note that the last inequality results from the fundamental theorem of calculus and the definitions of $\tilde{w}$ in \eqref{deftildew}, 
as well as the fact $\eta=\chi^{\Psi}\p_z \theta^{\Psi}.$ Since by the definition of
${w}$ in \eqref{def-w}, 
\begin{align*}
\il \na w\il_{m-5,\infty,t}\lesssim \lab \il (\na u, \theta, \na\theta, \kpa\Delta\theta) \il_{m-5,\infty,t}\big),
\end{align*}
we find that:
\begin{align}\label{mathfrakm33-1}
     \sum_{l=1}^2\il (\tilde{w}_{l}-\tilde{w}_{l}^b)\p_{y^l}\eta \il_{m-5,\infty,t}\lesssim \lat.
\end{align}
Next, let us control the second term in 
${\mathfrak{M}}^{\Psi}_{3,3}$ in the following way:
\begin{align*}
   \il \big(\tilde{w}_{3}-\tilde{w}_{3}^b-z  (\p_z\tilde{w}_{3})^b\big)\p_z \eta_i\il_{m-5,\infty,t}&\lesssim \il z^{-2}(\tilde{w}_{3}-\tilde{w}_{3}^b-z  (\p_z\tilde{w}_{3})^b)\il_{0,\infty,t}\il z^2\p_z\eta\il_{m-5,\infty,t}\\
 &+\il z^{-1}(\tilde{w}_{3}-\tilde{w}_{3}^b-z  (\p_z\tilde{w}_{3})^b)\il_{m-5,\infty,t}\il z\p_z\eta\il_{m-6,\infty,t}  \\
 &\lesssim \il\p_z^2 \tilde{w}_{3}\il_{0,\infty,t} 
 \il\theta\il_{m-3,\infty,t}+\il \p_z \tilde{w}_3\il_{m-5,\infty,t}\il\na\theta\il_{m-5,\infty,t}.
\end{align*}
By the virtue of the definition \eqref{deftildew}
and the change of variable \eqref{change-variable},  one can write that:
\beqs 
\div \tilde{w}=\div ((D\Psi)^{-\star}(\chi w)^{\Psi})
=(\div(\chi w))^{\Psi}-\f{\na d}{d}(D\Psi)^{-\star}(\chi w)^{\Psi}, 
\eeqs
where $d=\det(D\Psi).$ Therefore, 
using the identity $\p_z^2\tilde{w}_3=\p_z \div \tilde{w}- \sum_{l=1}^2 \p_z\p_{y^l}\tilde{w}_l, $
we derive that:
\begin{align*}
    \il\p_z^2 \tilde{w}_3\il_{0,\infty,t}\lesssim 
    \il \na w\il_{1,\infty,t}+\il\na\div w\il_{0,\infty,t}.
\end{align*}
In light of the equations $\eqref{NCNS-S2}_1, \eqref{NCNS-S2}_3,$ we can write:
\begin{align*}
    \div w&=\div \big(u-\f{\gamma-1}{\gamma}\overline{\Gamma}\kpa\beta\na\theta\big)\\
    &=-\f{\ep}{\gamma}(\pt+u\cdot\na)\sigma+\f{\gamma-1}{\gamma}\ep \mathfrak{N}+\f{\gamma-1}{\gamma}\f{\Gamma-\overline{\Gamma}}{\ep\Gamma}\big[ (\ep\pt+\ep u\cdot\na)\big(\f{C_v\gamma}{R}\theta-\ep\sigma\big)-\ep^2 \mathfrak{N}\big]
\end{align*}
which leads to that:
\begin{align*}
    \il\na\div w\il_{0,\infty,t}\lesssim \lab \il(\sigma,u,\theta)\il_{2,\infty,t}+ \il\na(\sigma, u, \theta, \ep\mu\na^2 u)\il_{1,\infty,t} \big)\lesssim \lat. 
\end{align*}
Consequently, we find $\il \p_z^2 \tilde{w}_3\il_{0,\infty,t}\lesssim \lat$ and thus 
\begin{align*}
    \il \big(\tilde{w}_{3}-\tilde{w}_{3}^b-z  (\p_z\tilde{w}_{3})^b\big)\p_z \eta\il_{m-5,\infty,t}\lesssim \lat
\end{align*}
which, together with \eqref{mathfrakm33-1}, finishes the control of $\mathfrak{M}_{3,3}^{\Psi}:$ 
\beq\label{escm3-3}
\il\mathfrak{M}_{3,3}^{\Psi}\il_{m-3,\infty,t}\lesssim \lat.
\eeq
% yields that \eqref{}. 

Finally, let us detail the estimate of  $\mathfrak{M}_{3,4}^{\Psi}$ defined in \eqref{deffrakm3}.  By the Taylor expansion and the Neumann boundary condition: $\p_z\theta^{\Psi}|_{z=0}= (\p_{\bn}\theta|_{\p\Omega})^{\Psi}=0,$ one can write:
\begin{align*}
     ({\Gamma\beta})^{\Psi}- [ ({\Gamma\beta})^{\Psi}]^b =\ep(\beta^{\Psi})^b z\cU_1+\Gamma^{\Psi} z^2\cU_2,
\end{align*}
where 
\begin{align*}
    \cU_1=\int_0^1 ((\Gamma')^{\Psi}\p_z\sigma^{\Psi})(t,y,\tau z)\d \tau,\qquad \cU_2=\int_0^1 \p_z^2(\beta)^{\Psi}(t,y, \tau z)(1-\tau)\, \d \tau.
\end{align*}
It can be verified directly that:
\begin{align}
& \il\cU_1\il_{m-5,\infty,t}\lesssim \lab \il\na\sigma\il_{m-5,\infty,t}\big) \lesssim \lat, \notag\\
& \kpa \il\cU_2\il_{m-5,\infty,t}+\kpa^{\f{1}{2}}\il\cU_2\il_{m-6,\infty,t} \notag\\
&\lesssim \lab \kpa \il(\Id,\na,\na^2)\theta\il_{m-5,\infty,t}+\kpa^{\f{1}{2}}\il(\Id,\na,\na^2)\theta\il_{m-6,\infty,t}\big)\lesssim \lat.\label{escU2}
\end{align}
We thus can bound $(\beta^{\Psi})^b \cU_1 (\ep\kpa z)\p_z^2\eta$ as: %+\Gamma^{\Psi} z^2\cU_2 $
\begin{align}\label{escm341}
    \il(\beta^{\Psi})^b \cU_1 (\ep\kpa z)\p_z^2\eta \il_{m-5,\infty,t}&\lesssim \lab\il(\sigma,\na\sigma)\il_{m-5,\infty,t}+\ep\kpa \il\na^2\theta\il_{m-4,\infty,t} \big) \lesssim \lat.
\end{align}
Note that by the equation \eqref{sec4:eq0}, it holds that: 
\begin{align*}
    \ep\kpa \il\na^2\theta\il_{m-4,\infty,t}\lesssim \lab \il(\ep\sigma, u, \theta)\il_{m-3,\infty,t}+  \il(\ep\kpa\na\theta,\ep\mu\na u)\il_{m-4,\infty,t} \big)\lesssim \lat.
\end{align*}
For the term $\kpa \Gamma^{\Psi} \cU_2 z^2\p_z^2\eta,$  
it is not easy to control it in $\il\cdot\il_{m-5,\infty,t},$ we thus need to control it in the three norms appearing in \eqref{es-frakm3}, which, thanks to \eqref{escU2}, can be done in the following way:
\begin{align}
\il \kpa \Gamma^{\Psi} \cU_2 z^2\p_z^2\eta \il_{m-6,\infty,t}\lesssim %\lab \il(\sigma, \kpa^{\f{1}{2}}\na^2\theta)\il_{m-6,\infty,t}+\kpa^{\f{1}{2}}\il\na\theta\il_{m-4,\infty,t}\big)
(1+\il \sigma \il_{m-6,\infty,t})\il\kpa^{\f{1}{2}}\na\theta\il_{m-4,\infty,t}\il\kpa^{\f{1}{2}}\cU_2\il_{m-6,\infty,t} 
\lesssim \lat.
\end{align}
Next, the other two norms can be dealt with 
\beq\label{es-cm34-2}
\begin{aligned}
 &\big(\int_0^t \| \kpa \Gamma^{\Psi} \cU_2 z^2\p_z^2\eta (s)\|_{m-5,\infty}^2\d s\big)^{\f{1}{2}} + \kpa^{\f{1}{2}}\il \kpa \Gamma^{\Psi} \cU_2 z^2\p_z^2\eta \il_{m-5,\infty,t}\\
& \lesssim %\il\theta\il_{m-5,\infty,t}
\big(1+\il \sigma \il_{m-5,\infty,t}\big)\bigg(\il \na\theta\il_{2,\infty,t} \il\kpa %\na^2\theta
\cU_2\il_{m-5,\infty,t}\\
&\qquad\qquad\qquad+  \il\kpa^{\f{1}{2}}%\p_z^2\theta
\cU_2\il_{m-6,\infty,t}\big(\il\kpa\na\theta\il_{m-3,\infty,t}+ \big(\int_0^t \| \kpa^{\f{1}{2}}\p_z\theta (s)\|_{m-3,\infty}^2\d s\big)^{\f{1}{2}}\big)\bigg)\\
&\lesssim \lae.
\end{aligned}
\eeq
We remark that again from Sobolev embedding, 
\begin{align*}
  \big(\int_0^t \| \kpa^{\f{1}{2}}\p_z\theta (s)\|_{m-3,\infty}^2\d s\big)^{\f{1}{2}}\lesssim \kpa^{\f{1}{2}}  \big(\|\na^2\theta\|_{\hco^{m-2}}+\|\na\theta\|_{\hco^{m-1}}\big)\lesssim \cE_{m,t}.
\end{align*}
Collecting \eqref{escm341}-\eqref{es-cm34-2}, we obtain that: 
\begin{align*}
     \big( \izt \il {\mathfrak{M}}^{\Psi}_{3,4} \il_{m-5,\infty}^2\d s\big)^{\f{1}{2}}+ \il{\mathfrak{M}}^{\Psi}_{3,4}\il_{m-6,\infty,t}+ \kpa^{\f{1}{2}}\il{\mathfrak{M}}^{\Psi}_{3,4}\il_{m-5,\infty,t}\lesssim \lae,
\end{align*}
which, together with \eqref{escm3-1,2}, \eqref{escm3-3}, 
yields \eqref{es-frakm3}.
\end{proof}
\begin{rmk}
It follows from the property \eqref{infty4} that, for any $2\leq q <+\infty,$ 
\begin{align*}
\big( \int_0^t \|\kpa^{\f{1}{2}}\p_z \eta\|_{m-5,\infty}^q\d s\big)^{\f{1}{q}}\lesssim  %\Lambda_{a_1,a_2,a_3,b}^k 
T^{\f{1}{q}}\Lambda_1^{m-5}\Lambda_2^{m-5}
 \bigg(\|\kpa^{\f{1}{2}} \p_{\bn}^2\theta(0)\|_{m-5,\infty}+\big(\izt \|(\eta, {\mathfrak{M}}^{\Psi}_3) (s)\|_{m-5,\infty}^2\d s\big)^{\f{1}{2}} \bigg).
 \end{align*}
 Therefore, as long as $\|\kpa^{\f{1}{2}} \p_{\bn}^2\theta(0)\|_{m-5,\infty}$ being bounded, 
 we have by \eqref{es-Lambda12}, \eqref{es-frakm3} that:
 \begin{align}\label{na2theta-Lp}
\big( \int_0^t \|\kpa^{\f{1}{2}}\na^2\theta\|_{m-5,\infty}^q\d s\big)^{\f{1}{q}}\lesssim T^{\f{1}{q}}\big( Y_m(0)+\lae\big).
 \end{align}
\end{rmk}

\section{Uniform estimates for the compressible part-II}
\begin{prop}
Under the assumption of \eqref{preasption}, there exists $\vartheta_8>0,$ such that  for any $\ep\in(0,1], (\mu,\kpa)\in A,$
 any $0<t\leq T,$ the following estimate holds:
\begin{align}\label{com-uniLinfty}
\|(\na\sigma,\div u)\|_{\infco^{m-2}}\lesssim \lab Y_m(0)\big)+
(T+\ep)^{{\vartheta_8}}\lae. 
\end{align}
\end{prop}
\begin{proof}
The estimate can be obtained again by using the equations $\eqref{NCNS-S2}$ and the induction arguments. 
Nevertheless, as so far only the control of $\bp (r_0u)$ (rather than $\bp u$) is shown, in the process of induction, we need to deal with some
product terms like $\na r_0\cdot\ep\pt u$ in the $L_t^{\infty}H_{co}^{m-3}$ norm. This requires the control of $\il u\il_{[\f{m}{2}]-1,\infty,t},$ which have not been controlled. 
Therefore,  we need to do induction in the following two steps:

$\bullet$ Control $(\na\sigma,\div u)$ in a lower order regularity space $\infco^{m-4},$ 

$\bullet$ Control $(\na\sigma,\div u)$ in higher order regularity space $\infco^{m-2}.$\\[3pt]
We remark that after the first step, we can bound $\il u\il_{m-5,\infty,t} \gtrsim \il u\il_{[\f{m}{2}]-1,\infty,t},$
which enables us to proceed the second step.

As a preparation, before going into the detail for the first step, we show two estimates concerning the higher order time derivatives for $(\na\sigma,\div u)$ 
and the higher order conormal norms for $\div u.$
Namely,  for any $k\leq m-2,$
\beq\label{sec5:eq13}
\begin{aligned}
\|(\na\sigma,\div u)\|_{\infcoch^{k,0}}&\lesssim \|(\sigma,u)\|_{\infcoch^{k+1,0}}
%+\|\kpa\na\theta\|_{\infco^{m-1}}
+\ep \lae\\
&\quad+ \|\big(\mu\ep\pt\na(\theta,\ep\sigma),  \kpa\div(\beta\na\theta)\big)\|_{\infco^{k}},
\end{aligned}
\eeq
and for any $j+l\leq m-2, l\geq 1,$
\begin{align}\label{sec5:eq20}
 \|\div u\|_{\infcoch^{j,l}}\lesssim \|\sigma\|_{\infcoch^{j+1}}+\|\na\sigma\|_{\infcoch^{j+1,l-1}} +\|\kpa\div(\beta\na\theta)\|_{\infco^{j+l}}+\ep\lat.
\end{align}

We start with the proof of \eqref{sec5:eq13}. 
By using the equation \eqref{rew-sigma}, one can obtain that:
\begin{align}\label{sec5:eq17}
 \|\div u\|_{\infcoch^{k,0}}&\lesssim \|\sigma\|_{\infcoch^{k+1,0}}+\|\kpa\overline{\Gamma}\div(\beta\na\theta)\|_{\infco^{k}}
 + \ep \|\big(u\cdot\na\sigma, \mathfrak{N},
 \ep^{-1}(\Gamma-\overline{\Gamma})\kpa\div(\beta\na\theta)\big)\|_{\infco^{k}}.\notag
 %&\lesssim \|\sigma\|_{\infcoch^{k+1,0}} +\|\kpa\div(\beta\na\theta)\|_{\infco^{k}}+\ep\lat.
\end{align}
By the product estimate \eqref{roughproduct1}, we find that the last term in the above can be bounded by $\ep\lat\cE_{m,t},$ which, in turn together with the above estimate, yields \eqref{sec5:eq13} for $\div u.$

Let us now control $\na\sigma$. We rewrite the equation \eqref{eq-tiltasigma} satisfied by $\tilde{\sigma}=\sigma-
\ep\mu(2\lambda_1+\lambda_2){\Gamma}\div u$ as:
\beq\label{eq-tiltasigma-2}
\left\{
\begin{array}{l}
      \Delta\tilde{\sigma}=\div\big( f+\mu\ep\lambda_1\overline{\Gamma}\curl\curl u\big) \text{ in } \Omega,\\[4pt]
      \p_{\bn} \tilde{\sigma}=f\cdot\bn \text{ on } \p\Omega,
\end{array}
   \right.
\eeq 
where  $$f=-\f{1}{R\beta}(\ep\pt+\ep u\cdot\na)u-\ep\mu\lambda_1\Gamma\curl\curl u-\ep\mu(2\lambda_1+\lambda_2)(\div u)\na\Gamma.$$
By using the elliptic estimate 
\eqref{highconormal},  we find that:
\begin{align}\label{sec5:eq15}
    \|\na\tilde{\sigma}\|_{\infcoch^{k,0}}&\lesssim \|f+\mu\ep\lambda_1\overline{\Gamma}\curl\curl u\|_{\infcoch^{k,0}}+\mu\ep\sum_{j\leq k}|(\ep\pt)^j\curl\curl u\cdot\bn|_{H^{-\f{1}{2}}(\p\Omega)}. 
\end{align}
On the one hand, it can be verified that:
\begin{align*}
\|f+\mu\ep\lambda_1\overline{\Gamma}\curl\curl u+\f{1}{R\beta}\ep\pt u\|_{\infcoch^{m-2,0}}\lesssim \ep\lae.
\end{align*}
On the other hand, for any $1\leq  k\leq m-2,$
\begin{align*}
\big\|\big[(\ep\pt)^k, \f{1}{R\beta}\big]\ep\pt u\big\|_{L_t^{\infty}L^2}&\lesssim \ep \big(\|\pt(\f{1}{R\beta})\|_{\infcoch^{k-1,0}}+\il u\il_{\infcoch^{k,0}}\big)\lat\\
&\lesssim \ep\lae.
\end{align*}
Consequently, we obtain, for any $k\leq m-2,$
\begin{align*}
\|f+\mu\ep\lambda_1\overline{\Gamma}\curl\curl u\|_{\infcoch^{k,0}}   &\lesssim \|u\|_{\infcoch^{k+1,0}}+ \ep\lae.
\end{align*}
Inserting this inequality and \eqref{curlcurlu-infty} into \eqref{sec5:eq15}, one finds that:
\beqs
 \|\na\tilde{\sigma}\|_{\infcoch^{k,0}}\lesssim\|u\|_{\infcoch^{k+1,0}}+\mu\ep\|\na\div u\|_{\infco^{k}}+\ep\lae.
\eeqs
Since $\sigma=\tilde{\sigma}+\ep\mu(2\lambda_1+\lambda_2){\Gamma}\div u, $ we derive further that 
\beqs
\|\na{\sigma}\|_{\infcoch^{k,0}}\lesssim\|u\|_{\infcoch^{k+1,0}}+\mu\ep\|\na\div u\|_{\infco^{k}}+\ep\lae.
\eeqs
By the equation $\eqref{newsys}_1$ for 
$\vr=\f{\ep}{\gamma}\sigma-\f{(\gamma-1)C_v}{\gamma R}\theta,$ we can bound 
$\mu\ep\na\div u$ in the following way:
\begin{align}\label{nadiv-infty}
    \mu\ep\|\na\div u\|_{\infco^k}&\lesssim \mu \|\ep\pt\na(\theta,\ep\sigma)\|_{\infco^k}+\ep\lae,
    %\notag\\ &\lesssim \mu\|\ep\pt\na\theta\|_{\infco^k}+\ep\lae,
\end{align}
which, together with the previous estimate, leads to that for any $0\leq k\leq m-2,$
\beq\label{sec5:eq16}
\|\na{\sigma}\|_{\infcoch^{k,0}}\lesssim\|u\|_{\infcoch^{k+1,0}}+\mu\|\ep\pt\na (\theta, \ep\sigma)\|_{\infco^{k}}+\ep\lae.
\eeq
We thus finish the proof of \eqref{sec5:eq13} for $\na\sigma.$ %by remembering the assumption $\mu\lesssim \kpa.$ %by noticing \eqref{sec5:eq17} and \eqref{sec5:eq16} and 

As for the proof of \eqref{sec5:eq20}, we can use again the equation \eqref{rew-sigma} to derive that for any $j+l\leq m-2, l\geq 1$
\begin{align*}
 \|\div u\|_{\infcoch^{j,l}}&\lesssim \|\sigma\|_{\infcoch^{j+1,l}}+\|\kpa\overline{\Gamma}\div(\beta\na\theta)\|_{\infco^{j+l}}+\ep\lae,
\end{align*}
which, combined with the fact 
$\|\sigma\|_{\infcoch^{j+1,l}}\lesssim\|\sigma\|_{\infcoch^{j+1}}+\|\na\sigma\|_{\infcoch^{j+1,l-1}}, $
yields \eqref{sec5:eq20}.

$\bullet$\underline{ Control of $\|(\na\sigma,\div u)\|_{\infco^{m-4}}$.}
We will show that: there exists $\vartheta_9>0,$ such that: 
\beq\label{nasidivu-low}
\|(\na\sigma,\div u)\|_{\infco^{m-4}}\lesssim \lab Y_m(0)\big) + (T+\ep)^{\vartheta_9}\lae.
\eeq
By noticing the estimates \eqref{EE-highest}, \eqref{EE-3}, \eqref{EE-theta-final},  \eqref{divbetanatta}, \eqref{EE-v}, \eqref{thetainfty}, \eqref{nathetainfty},  we find that the above property is the consequence of the following estimate:
\begin{align*}
\|(\na\sigma,\div u)\|_{\infco^{m-4}}
&\lesssim\ep\lae+\Lambda\big(\f{1}{c_0}, \il \theta \il_{m-3,\infty,t}+\il\na\theta\il_{m-5,\infty,t} \big)\cdot\\
&\big(\|(\sigma, u)\|_{\infcoch^{m-3,0}}%\|(\na\sigma,\div u)\|_{\infco^{j+1,l-1}}
+\|v\|_{\infco^{m-3}}+\|\big(\mu\ep\pt\na\theta, \kpa\div(\beta\na\theta)\big)\|_{\infco^{m-4}}\big),
\end{align*}
whose proof relies further on the following estimate:
for any $j+l\leq m-4, l\geq 1,$ 
\begin{align}\label{sec5:eq21}
\|(\na\sigma,\div u)\|_{\infcoch^{j,l}}
&\lesssim\big(\|(\sigma, u)\|_{\infcoch^{m-3,0}}+\|(\na\sigma,\div u)\|_{\infcoch^{j+1,l-1}}
+\|v\|_{\infco^{m-3}}\big)\notag\\
&\qquad\qquad\cdot\Lambda\big(\f{1}{c_0}, \il \theta \il_{m-3,\infty,t}+\il\na\theta\il_{m-5,\infty,t} \big)\\
&\quad+\|\big(\mu\ep\pt\na\theta, \kpa\div(\beta\na\theta)\big)\|_{\infco^{m-4}}+\ep\lae.\notag
\end{align}
In view of \eqref{sec5:eq13}, \eqref{sec5:eq20}, 
 to prove \eqref{sec5:eq21}, it suffices to show the following estimate for $\|\na\sigma\|_{\infcoch^{j,l}}:$ %(j+l\leq m-4, l\geq 1):$ 
\begin{align}\label{sec5:eq22}
   & \|\na\sigma\|_{\infcoch^{j,l}}\lesssim \mu\|\ep\pt\na\theta\|_{\infco^{m-4}}+\ep\lae\notag\\
  &\quad  +\big(\|u\|_{\infcoch^{j+1,0}}+\|v\|_{\infco^{m-3}} +\|\div u\|_{\infcoch^{j+1,l-1}}\big)\Lambda\big(\f{1}{c_0}, \il \theta \il_{m-3,\infty,t}+\il\na\theta\il_{m-5,\infty,t} \big),
\end{align}
which is the task of what follows. 

 % By using equation \eqref{sec7:eq1}, it can be found 
 Let us first notice that by $\eqref{NCNS-S2}_2$, $\sigma$ solves the following elliptic problem:
\beq\label{sigma-elliptic}
\left\{
\begin{array}{l}
      \Delta{\sigma}=\div\big( -\ep\pt(r_0 u)+\mu\ep(2\lambda_1+\lambda_2)\na\div u+\ep d\big) \text{ in } \Omega,\\[4pt]
      \p_{\bn}{\sigma}=(-\ep\pt(r_0 u)+\mu\ep(2\lambda_1+\lambda_2)\na\div u+\ep d)\cdot\bn-\mu\ep\lambda\overline{\Gamma}\curl\curl u\cdot\bn \text{ on } \p\Omega,
\end{array}
   \right.
\eeq 
where
$$r_0=\f{1}{R\beta}\exp (\ep R\sigma/C_v\gamma)%=\f{1}{R\overline{\Gamma}}\exp(-R\tilde{\theta}/{C_v})
,\quad r_1= \ep^{-1}(\exp (\ep R\sigma/C_v\gamma)-1),$$
$$d=\colon u\pt r_0-r_0 u\cdot\na u+r_1(-\na\sigma+\ep\mu\div\cL u)+\f{\Gamma-\overline{\Gamma}}{\ep}\ep\mu \div \cL u.  $$
By using Proposition \ref{prop-prdcom}, Corollary \ref{cor-gb} as well as the definition of $\cN_{m,t}$ (see \eqref{defcalNmt}), one can verify that, %for any $t\leq T,$ 
\beqs
\|d\|_{\infco^{m-2}}\lesssim \lae.
\eeqs
We thus get from the elliptic estimate \eqref{highconormal} that:
%By using \eqref{} and \eqref{}, 
for any $j,l,$ with $j+l\leq 4, l\geq 1,$
\begin{align*}
    \|\na\sigma\|_{\infcoch^{j,l}}&\lesssim \|\ep\pt(r_0 u)\|_{\infcoch^{j,l}}
    + \mu\ep\|\na\div u\|_{\infco^{m-4}}\notag\\
    &\qquad+ \mu\ep |\curl\curl u\cdot\bn|_{L_t^{\infty}\tilde{H}^{m-\f{9}{2}}}+\ep\lae,
\end{align*}
which, together with the estimates \eqref{curlcurlu-infty} and \eqref{nadiv-infty}, 
yields that:
\begin{align}\label{sec5:eq39}
   \|\na\sigma\|_{\infcoch^{j,l}}&\lesssim \|r_0 u\|_{\infcoch^{j+1,l}}
    +\mu\|\ep\pt\na\theta\|_{\infco^{m-4}}+\ep\lae.
\end{align}
The proof of \eqref{sec5:eq22} is thus complete if we show that, for $j+l\leq m-4,$
\begin{align}\label{sec5:eq23}
\|r_0 u\|_{\infcoch^{j+1,l}}&\lesssim \big(\|u\|_{\infcoch^{j+1,0}}+\|v\|_{\infco^{m-3}} +\|\div u\|_{\infcoch^{j+1,l-1}}\big)\notag\\
&\qquad\cdot\Lambda\big(\f{1}{c_0}, \il \theta \il_{m-3,\infty,t}+\il\na\theta\il_{m-5,\infty,t} \big),
\end{align}
which can be done in the following way. 
It follows from the elliptic estimate \eqref{secderelliptic-1}, the definition $v=\mathbb{P}(r_0 u)$ and the condition $r_0u\cdot\bn|_{\p\Omega}=0$ that:
\beq\label{sec5:eq24}
\begin{aligned}
    \|r_0 u\|_{\infcoch^{j+1,l}}&\lesssim \|v\|_{\infco^{m-3}}+\|\div(r_0 u)\|_{\infcoch^{j+1,l-1}}\\
    &\lesssim \|v\|_{\infco^{m-3}}+\|\na r_0\cdot u\|_{\infcoch^{j+1,l-1}}+ \ep \lae\\
    &\quad +\lab \il\theta\il_{m-3,\infty,t}\big)\|\div u\|_{\infcoch^{j+1,l-1}}.
\end{aligned}
\eeq
Noticing the fact
$$\|\pt(\na r_0/r_0)\cdot r_0 u\|_{\infco^{m-4}}\lesssim \lae,$$ we can bound the last term in the last line as:
\begin{align*}
    \|\na r_0\cdot u\|_{\infcoch^{j+1,l-1}}\lesssim \lab \|\na\theta\|_{m-5,\infty,t} \big) \|r_0 u\|_{\infcoch^{j+1,l-1}}+\ep\lae.
\end{align*}
Inserting the above estimate into \eqref{sec5:eq24}, we find that 
\begin{align*}
    \|r_0 u\|_{\infcoch^{j+1,l}}&\lesssim  \lab \|\na\theta\|_{m-5,\infty,t} \big) \|r_0 u\|_{\infcoch^{j+1,l-1}}\\
    &+\|v\|_{\infco^{m-3}}+\ep\lae+ \lab \il\theta\il_{m-3,\infty,t}\big) \|\div u\|_{\infcoch^{j+1,l-1}}.
\end{align*}
Therefore,  we eventually find \eqref{sec5:eq23} by induction on $l.$

$\bullet$\underline{ Control of $\|(\na\sigma,\div u)\|_{\infco^{m-2}}$.}

In the previous step, we have controlled $\|(\div u,\na\sigma)\|_{\infco^{m-4}}.$ In this step, we recover the spatial conormal derivatives to the higher order $m-2.$ Before going into the detail, let us first note that 
%$\|\div (r_0 u)\|_{\infco^{m-4}}$ 
$\il r_0 u\il_{m-5,\infty,t}$ has essentially  been bounded by the quantities that we have controlled. 
Indeed, by using the identity \eqref{normalofnormalder} and the fact that near the boundaries $$u^j\cdot\na =u_1^j\p_{y^1}+ u_2^j\p_{y^2}+ \f{u^j\cdot\bn}{\phi} \phi\p_z, \text{ where} \,\, u^j=(\ep\pt)^j u,$$
 we find that: for any $0\leq j\leq m-4, 0\leq k\leq m-4-j,$
\begin{align*}
    \|\div (r_0 u^j)\|_{L_t^{\infty}\cH^{0,k}}&\lesssim \lab \il\theta\il_{m-4,\infty,t}\big)\|\div u^j\|_{L_t^{\infty}\cH^{0,k}}+\|u^j\cdot\na r_0\|_{L_t^{\infty}\cH^{0,k}}
    +\ep\lae\notag\\
    &\lesssim  \lab\il\theta\il_{m-3,\infty,t}+ \|\na\theta\|_{m-5,\infty,t} \big)(\|\div u\|_{\infco^{m-4}}+\|r_0 u^j\|_{L_t^{\infty}\cH^{0,k}})+\ep\lae\notag\\
    &\lesssim \ep \lae+\lab\il\theta\il_{m-3,\infty,t}+ \|\na\theta\|_{m-5,\infty,t}\big) \notag \\
   & \quad \cdot \big(\|\div u\|_{\infco^{m-4}}+\|\div (r_0 u^j)\|_{L_t^{\infty}\cH^{0,k-1}}\mathbb{I}_{\{k\geq 1\}}+\|v^j\|_{L_t^{\infty}\cH^{0,m-4-j}} +\|u\|_{L_t^{\infty}\cH^{m-4,0}}\big),
\end{align*}
where we have used the notation $v^j=\bp(r_0 u^j).$
Therefore, one obtains by induction on $0\leq k\leq m-4-j$ that: for any $0\leq j\leq m-4,$
\begin{align*}
     \|\div (r_0 u^j)\|_{L_t^{\infty}\cH^{0,m-4-j}}&\lesssim \ep \lae+\lab\il\theta\il_{m-3,\infty,t}+ \|\na\theta\|_{m-5,\infty,t}\big) \notag\\
    &\qquad \qquad \qquad\cdot \big(\|\div u\|_{\infco^{m-4}}+\|v^j\|_{L_t^{\infty}\cH^{0,m-4-j}} +\|u\|_{L_t^{\infty}\cH^{m-4,0}}\big).
\end{align*}
%This, in turn, yields that:
Moreover, by again writing $r_0 u^j=\bq (r_0 u^j)+v^j,$ we have:
\begin{align*}
    \|r_0 u^j\|_{L_t^{\infty}\cH^{0,m-3-j}}&\lesssim \|r_0 u^j\|_{L_t^{\infty}\cH^{m-3-j,0}}+ \|\div (r_0 u^j)\|_{L_t^{\infty}\cH^{0,m-4-j}}+ \|v^j\|_{L_t^{\infty}\cH^{0,m-3-j}}.
\end{align*}
The previous estimates, together with the  Sobolev embedding \eqref{sobebd} and the fact $\il\pt r_0\il_{m-5,\infty,t}\lesssim \lat,$
 enable us to get that 
\beqs
\begin{aligned}
\il r_0 u\il_{m-5,\infty,t}&\lesssim 
\sum_{j=0}^{m-5}\il r_0 u^j\il_{\star, m-5-j,\infty,t}+\ep\lat\\
&\lesssim \sum_{j=0}^{m-5}\big( \|r_0 u^j\|_{L_t^{\infty}\cH^{0,m-3-j}}+\|(\div (r_0 u^j), \na v^j)\|_{L_t^{\infty}\cH^{0,m-4-j}}\big)+\ep\lat\\
&\lesssim  \ep \lae+\lab\il\theta\il_{m-3,\infty,t}+ \|\na\theta\|_{m-5,\infty,t}\big) \\
    &\qquad \cdot \big(\|\div u\|_{\infco^{m-4}} +\|u\|_{L_t^{\infty}\cH^{m-4,0}}+\sum_{j=0}^{m-5}(\| \na v^j\|_{L_t^{\infty}\cH^{0,m-4-j} }+\|v^j\|_{L_t^{\infty}\cH^{0,m-3-j}})\big).
\end{aligned}
\eeqs
Note that we used a slightly different notation here $\il\cdot\il_{\star,l,\infty,t}$ to stress that only  spatial conormal derivatives  involved in the norm. By \eqref{EE-3}, %\eqref{thetainfty},
 \eqref{EE-vj}, \eqref{nablavLinfty}, \eqref{nasidivu-low},  \eqref{thetainfty}, \eqref{nathetainfty}, we can thus find from the above estimate that, there exists a constant $\vartheta_{10}>0,$ such that:
\begin{align}\label{ruinfty}
    \il r_0 u\il_{m-5,\infty,t}&\lesssim %\lab\il\theta\il_{m-3,\infty,t}+ \|\na\theta\|_{m-5,\infty,t}\big) (
    \lab Y_m(0)\big)+ (T+\ep)^{\vartheta_{10}}\lae.
\end{align}

We now are able to control $\|(\div u,\na\sigma)\|_{\infco^{m-2}}.$ Our ultimate goal is to show that 
\eqref{com-uniLinfty} holds, which relies on the following estimate: 
 \begin{align}\label{sec5:eq40}
\|(\na\sigma, &\div u)\|_{\infco^{m-2}}
\lesssim\ep\lae+\Lambda\big(\f{1}{c_0}, \il \theta \il_{m-3,\infty,t}+\il(r_0 u,\na\theta)\il_{m-5,\infty,t} \big)\cdot\notag\\
&\big(\|(\sigma, u)\|_{\infcoch^{m-1,0}}
+\|v\|_{\infco^{m-1}}+\|\na\theta\|_{\infco^{m-2}}+\|\big(\mu\ep\pt\na(\theta, \ep\sigma), \kpa\div(\beta\na\theta)\big)\|_{\infco^{m-2}}\big).
\end{align}
Indeed, by \eqref{EE-highest}, \eqref{EE-3}, \eqref{EE-theta-final}, \eqref{divbetanatta},  \eqref{EE-v}; %\eqref{EE-v} 
\eqref{thetainfty}, \eqref{nathetainfty}, \eqref{ruinfty}, one derives \eqref{com-uniLinfty} 
from the above estimate. 
However, by induction, we find that to show \eqref{sec5:eq40}, it suffices to prove that, for any $j+l\leq m-2, l\geq 1,$
\begin{align*}
\|(\na\sigma,\div u)\|_{\infcoch^{j,l}}
&\lesssim\Lambda\big(\f{1}{c_0}, \il \theta \il_{m-3,\infty,t}+\il(r_0 u, \na\theta)\il_{m-5,\infty,t} \big)\notag\\
&\cdot \big(\|(\sigma, u)\|_{\infcoch^{m-2,0}}+\|(\na\sigma,\div u)\|_{\infcoch^{j+1,l-1}}
+\|v\|_{\infco^{m-1}}+\|\na\theta\|_{\infco^{m-2}}\big)\\
&\quad +\|\big(\mu\ep\pt\na(\theta, \ep\sigma), \kpa\div(\beta\na\theta)\big)\|_{\infco^{m-2}}
%+\|\kpa\big(\ep\pt\na\theta,\div(\beta\na\theta)\big)\|_{\infco^{m-2}}
+\ep\lae.\notag
\end{align*}
This, in turn, would be the consequence of  \eqref{sec5:eq13}, \eqref{sec5:eq20},
and the following estimate for $\|\na\sigma\|_{\infcoch^{j,l}}:$ %(j+l\leq m-4, l\geq 1):$ 
\begin{align}\label{sec5:eq41}
   \|\na\sigma\|_{\infcoch^{j,l}}&\lesssim   \mu\|\ep\pt\na(\theta, \ep\sigma)\|_{\infco^{m-2}}+\ep\lae+\Lambda\big(\f{1}{c_0}, \il \theta \il_{m-3,\infty,t}+\il(r_0 u, \na\theta)\il_{m-5,\infty,t} \big)\notag\\
  &\qquad \qquad\qquad\cdot\big(\|u\|_{\infcoch^{j+1,0}}+\|\div u\|_{\infcoch^{j+1,l-1}}+\|v\|_{\infco^{m-1}}+\|\na\theta\|_{\infco^{m-2}}\big).
\end{align}
The remaining paragraph is thus devoted to the proof of \eqref{sec5:eq41}.

Similar to \eqref{sec5:eq39}, one derives from  \eqref{curlcurlu-infty}, \eqref{nadiv-infty} and  the elliptic estimate \eqref{highconormal} that for any $j,l,$ with $j+l\leq m-2, l\geq 1,$
\begin{align}\label{sec5:eq42}
   \|\na\sigma\|_{\infcoch^{j,l}}&\lesssim \|r_0 u\|_{\infcoch^{j+1,l}}
    +\mu\|\ep\pt\na(\theta, \ep\sigma)\|_{\infco^{m-2}}+\ep\lae.
\end{align}
 It thus remains to control the term $\|r_0 u\|_{\infcoch^{j+1,l}},$ which can be bounded by using the decomposition $r_0 u=\bq (r_0 u)+v:$
 \beq\label{sec5:eq43}
\|r_0 u\|_{\infcoch^{j+1,l}} \lesssim \|u\|_{\infcoch^{j+1,0}}+\|\div (r_0 u)\|_{\infcoch^{j+1,l-1}}+\|v\|_{\infco^{m-1}}+\ep\lae.
\eeq
 Let us estimate $\|\div (r_0 u)\|_{\infcoch^{j+1,l-1}}.$
On the one hand, since $j+l\leq m-2,$ and 
\begin{align*}
    \|(\ep\pt r_0)\div u\|_{\infco^{m-3}}\lesssim 
    \ep \|(\pt r_0,\div u)\|_{\infco^{m-3}} \Lambda\big(\f{1}{c_0}, \il (\pt r_0, \div u)\il_{[\f{m}{2}]-1,\infty,t}\big)\lesssim
    \ep \lae,
\end{align*}
we have that: 
\begin{align}\label{sec5:eq44}
    \|r_0 \div u\|_{\infcoch^{j+1,l-1}}\lesssim \Lambda\big(\f{1}{c_0}, \il\theta\il_{m-3,\infty,t}\big)\|\div u\|_{\infcoch^{j+1,l-1}}+\ep\lae.
\end{align}
On the other hand, since $\na r_0=r_0(-\na\theta+\ep\f{R}{C_v\gamma}\na\tsigma)$
 and $[\f{m}{2}]-1\leq m-5$ by the assumption $m\geq 7,$ the term $\|\na r_0\cdot u\|_{\infcoch^{j+1,l-1}}$ can be estimated  as follows:
\begin{align}\label{sec5:eq45}
    \|\na r_0\cdot u\|_{\infcoch^{j+1,l-1}}&\lesssim  \|\na \theta \cdot r_0 u\|_{\infcoch^{j+1,l-1}}+\ep\lae \notag\\
   &\lesssim  \lab \il (r_0 u, \na\theta)\il_{m-5,\infty,t} \big) (\|r_0 u\|_{\infcoch^{j+1,l-1}}+\|\na\theta\|_{\infco^{m-2}})+\ep\lae.
\end{align}
Inserting estimates \eqref{sec5:eq44} and \eqref{sec5:eq45} into \eqref{sec5:eq43}, we find that:
\begin{align*}
    \|r_0 u\|_{\infcoch^{j+1,l}}  &\lesssim  \lab \il (r_0 u, \na\theta)\il_{m-5,\infty,t} \big) (\|r_0 u\|_{\infcoch^{j+1,l-1}}+\|\na\theta\|_{\infco^{m-2}})\\
    &\qquad+\Lambda\big(\f{1}{c_0}, \il\theta\il_{m-3,\infty,t}\big)\|\div u\|_{\infcoch^{j+1,l-1}}+\|v\|_{\infco^{m-1}}+\ep\lae.
\end{align*}
Using the above estimate recursively $l$ times, we obtain
that:
\beq\label{r0uch}
\begin{aligned}
     \|r_0 u\|_{\infcoch^{j+1,l}}  &\lesssim \ep\lae+\Lambda\big(\f{1}{c_0}, \il\theta\il_{m-3,\infty,t}+\il (r_0 u, \na\theta)\il_{m-5,\infty,t}  \big)
    \\
    &\qquad \qquad\cdot \big(\|u\|_{L_t^{\infty}\cH^{j+1}}+\|\div u\|_{\infcoch^{j+1,l-1}}+\|v\|_{\infco^{m-1}}+\|\na\theta\|_{\infco^{m-2}}\big).
\end{aligned}
\eeq
Plugging the above estimate into \eqref{sec5:eq42}, we eventually get \eqref{sec5:eq41} and thus finish the proof.
\end{proof}
 \begin{rmk}
 Thanks to the relation $\f{1}{R\beta r_0}=\exp({-\f{R\ep}{C_v\gamma}\tsigma})$ 
 and the fact $\|\pt r_0\|_{\infco^{m-2}}\lesssim \lae,$
 %and \eqref{r0uch},
it holds that:
 \begin{align*}
    \big\| \f{1}{R\beta}\ep\pt u\big\|_{\infco^{m-2}}\lesssim \|(\ep\pt)(r_0 u)\|_{\infco^{m-2}}+\ep\lae,
 \end{align*}
 inserting which into \eqref{secnoru-LinftyL2}, and using the estimates \eqref{na2uhigh}, \eqref{EE-theta-final}, \eqref{com-uniLinfty}, \eqref{r0uch}, we can find a constant $\vartheta_{11}$ such that:
 \begin{align}\label{secnoru-LinftyL2-final}
     \ep\mu \|\na^2 u \|_{\infco^{m-2}}\lesssim\lab Y_m(0)\big)+ (T+\ep)^{\vartheta_{11}}\lae.
 \end{align}
 \end{rmk}
 \begin{rmk}
In view of estimates \eqref{na2sigmauni-0}, \eqref{EE-theta-final}
and \eqref{com-uniLinfty}, 
we have, by choosing $\vartheta_{12}=\min\{ \f{\tilde{\vartheta}_0}{2}, \vartheta_8, \f{1}{12}\}$ that
 \begin{align}\label{na2sigmauni}   \|\na^2\sigma\|_{\infco^{m-3}}\lesssim \lab Y_m(0)\big)+ (T+\ep)^{\vartheta_{12}}\lae.
 \end{align}
 \end{rmk}
\begin{prop}
Assume \eqref{preasption} holds true, then there exists a constant $\vartheta_{13}>0,$ such that for any $\ep\in(0,1], (\mu,\kpa)\in A,$
 any $0<t\leq T:$
\begin{align}
&\|\na u\|_{L_t^{\infty}H_{co}^{m-3}}\lesssim  \lab Y_m(0) \big)+ (T+\ep)^{\vartheta_{13}}\lae, \label{nauhighLinftyL2-m-3}\\
&\|\na u\|_{L_t^{\infty}H_{co}^{m-2}}\lesssim  \lab Y_m(0) \big)+ (T+\ep)^{\vartheta_{13}}\lae+ \|\theta\|_{\infco^{m-1}}\lab \il \na u\il_{0,\infty,t}\big). \label{nauhighLinftyL2-m-2}
\end{align}
\end{prop}
\begin{proof}
For any $0\leq j\leq m-2,$ we write $(\ep\pt)^j u=u^j=\bp u^j+\bq u^j.$ 
First, by the elliptic estimate \eqref{secderelliptic-1} and the estimate \eqref{com-uniLinfty}, it holds that for any $j\leq k\leq m-2,$
\beq\label{nabquj}
\|\na \bq u^j\|_{L_t^{\infty}\cH^{k-j}}\lesssim \|\div u^j\|_{L_t^{\infty}\cH^{m-2-j}}%\lesssim \|\div u\|_{\infco^{m-2}}
\lesssim \lab Y_m(0)\big)+
(T+\ep)^{{\vartheta_8}}\lae.
\eeq
Note that we denote $\cH^l=\cH^{0,l}$ for spaces that involves only spatial conormal derivatives.
It remains to control $\na \bp u^j.$ We will show later that for
$0\leq j\leq k\leq m-2,$
\begin{align}\label{sec8:eq25}
  \|r_0 \na \bp u^j\|_{L_t^{\infty}\cH^{k-j}}&\lesssim 
  \|\na v^j\|_{L_t^{\infty}\cH^{k-j}}+\|\na r_0  (u^j, \bq u^j)\|_{L_t^{\infty}\cH^{k-j}}
 % &\lesssim  \|\na v^j\|_{L_t^{\infty}\cH^{m-2-j}}+ \|(\na\theta, \theta)\|_{L_t^{\infty}\cH^{m-2-j}}\lab \il r_0 u ,\na\Psi\il_{m-5,\infty,t}\big)+\|(r_0 u^j, \div u^j)\|_{L_t^{\infty}\cH^{m-2-j}}\lab \il(\theta,\na\theta)\il_{2,\infty,t}\big)
\end{align}
where $v^j=\bp (r_0 u^j).$ The last term in the above estimate can be bounded as:
\begin{align*}
\|(\theta, \na\theta, r_0 u, \div u)\|_{L_t^{\infty}\cH^{m-2}}\lab \il (\theta, \na\theta, r_0 u, \div u)\il_{m-5,\infty,t} \big).
  %  \|(\na\theta, \theta)\|_{L_t^{\infty}\cH^{m-2-j}}\lab \il r_0 u ,\bq u\il_{m-5,\infty,t}\big)+\|(r_0 u, \div u)\|_{L_t^{\infty}\cH^{m-2}}\lab \il(\theta,\na\theta)\il_{2,\infty,t}\big).
\end{align*}
We observe that the quantities $\il r_0 u \il_{m-5,\infty,t}, \|r_0 u\|_{L_t^{\infty}\cH^{m-2}}$ have been controlled in \eqref{ruinfty} and \eqref{r0uch}, the quantity  $\il\bq u\il_{m-5,\infty,t}$ can be estimated thanks to the Sobolev embedding \eqref{sobebd} by $\|\div u\|_{\infco^{m-4}}$ which has been controlled in \eqref{com-uniLinfty}.
These facts, together with the estimate \eqref{nablavLinfty}, 
enable us to find a constant  $\vartheta_{14}>0$
%have, by choosing  $\vartheta_4>0$ smaller if necessary,
\begin{align}\label{r0napuj}
     \|r_0 \na \bp u^j\|_{L_t^{\infty}\cH^{k-j}}&\lesssim \lab Y_m(0) \big)+ (T+\ep)^{\vartheta_{14}}\lae.
\end{align}
Therefore, by noticing 
 \eqref{thetainfty}, \eqref{tsigmainfty}, we have by by choosing  $\vartheta_{14}>0$ smaller if necessary,
 %leads to that
%from which, we derive that (choosing $\vartheta_4$ smaller if necessary),
for any $1\leq j\leq m-2,$
\begin{align*}
 \|\na \bp u^j\|_{L_t^{\infty}\cH^{m-2-j}}&\lesssim   \|r_0 \na \bp u^j\|_{L_t^{\infty}\cH^{m-2-j}}\lab \il r_0^{-1}\il_{m-3,\infty,t}\big)\\
 &\lesssim \lab Y_m(0) \big)+ (T+\ep)^{\vartheta_{14}}\lae,
\end{align*}
and for any $0\leq j\leq m-3,$
\begin{align*}
    \|\na \bp u^j\|_{L_t^{\infty}\cH^{m-3-j}}\lesssim \lab Y_m(0) \big)+ (T+\ep)^{\vartheta_{14}}\lae.
\end{align*}
The last estimate, together with \eqref{nabquj}, yields
\eqref{nauhighLinftyL2-m-3} by letting $\vartheta_{13}=\min\{\vartheta_{8}, \vartheta_{14} \}.$
%\beq\label{naulow}\|\na u\|_{L_t^{\infty}H_{co}^{m-3}}\lesssim \lab Y_m(0) \big)+ (T+\ep)^{\vartheta_4}\lae.\eeq
To prove \eqref{nauhighLinftyL2-m-2}, we still need to control $\na \bp u$ in $L_t^{\infty}\cH^{m-2}:$
%It follows from the generalized GN inequality  that:
\begin{align*}
\|\na \bp u\|_{L_t^{\infty}\cH^{m-2}}\lesssim \|r_0\na\bp u\|_{L_t^{\infty}\cH^{m-2}}\lab \il r_0^{-1}\il_{m-3,\infty,t}\big) + 
%+ \|\na\bp u\|_{L_t^{\infty}\cH^{m-3}} \il Z r_0\il_{1,\infty,t}+
\|Z r_0^{-1}\|_{L_t^{\infty}\cH^{m-2}}\il r_0\na \bp u\il_{0,\infty,t}.
\end{align*}
Since $\il r_0\na \bp u\il_{0,\infty,t}\lesssim \il(\na u, \na\mathbb{Q} u)\il_{0,\infty,t}\lesssim \il\na u\il_{0,\infty,t}+\|(u ,\div u)\|_{\infco^{2}},$ we can use \eqref{EE-theta-final}, \eqref{com-uniLinfty}, \eqref{nabquj}, \eqref{r0napuj} to derive \eqref{nauhighLinftyL2-m-2} (choosing $\vartheta_{13}$ smaller if necessary).

It now remains to show \eqref{sec8:eq25}. Let us write:
\begin{align}\label{sec8:eq27}
    r_0 \na\bp u^j-\na \bp(r_0 u^j)=-\bp u^j\cdot\na r_0 +\na \cW^j, \quad \cW^j=\colon \bq(r_0 u^j)-(r_0\bq u^j),
\end{align}
 By the definition of $\bq,$ one has $\cW^j\cdot\bn|_{\p\Omega}=0.$ Consequently, it follows from the  div-curl lemma that:
\begin{align*}
    \|\na \cW^j\|_{L_t^{\infty}\cH^{m-2-j}}&\lesssim \|(\div \cW^j,\curl \cW^j) \|_{L_t^{\infty}\cH^{m-2-j}}\\
    &\lesssim 
    \|\na r_0 \times \bq u^j, \na r_0 \cdot( u^j, \bq u^j)\|_{L_t^{\infty}\cH^{m-2-j}},
\end{align*}
which, together with \eqref{sec8:eq27}, yields \eqref{sec8:eq25}. 
%We first show that for any $0\leq j\leq m-2,$ 
%\beq\|\na (r_0 u^j)\|_{L_t^{\infty}\cH^{m-2-j}}\lesssim  \lab Y_m^2(0) \big)+ (T+\ep)^{\vartheta_2}\lat\cE_{m,t}.\eeq
%Note that we use here the short notation $\cH^{k}=\cH^{0,k}$ for simplicity. 
%By writing $r_0 u^j= \bq(r_0 u^j)+v^j,$ we have:
%\begin{align} \|\na (r_0 u^j)\|_{L_t^{\infty}\cH^{m-2-j}}\lesssim \|\div (r_0 u^j)\|_{L_t^{\infty}\cH^{m-2-j}}+\|\na v^j\|_{L_t^{\infty}\cH^{m-2-j}}+\|v^j\|_{L_t^{\infty}\cH^{m-1-j}}\end{align}
%In view of estimates \eqref{EE-vj} and \eqref{nablavLinfty}, it suffices to bound $\div(r_0 u^j).$ 
\end{proof}
We are now in position to summarize the obtained estimates concerning the energy norms $\cE_{m,T}(\sigma, u):$
\begin{prop}
Denote 
\begin{align}\label{tildecE}
\tilde{\cE}_{m,t}(\sigma,u)=\cE_{m,t}(\sigma,u) -\|\na u\|_{\infco^{m-2}}+\|\na u\|_{\infco^{m-3}}.
\end{align}
Suppose that the assumption \eqref{preasption} holds, 
then there exists $\vartheta_{15}>0,$ such that:
  any $\ep \in (0,1], (\mu,\kpa)\in A,$
any $ 0<t\leq T,$ 
\begin{align}\label{es-tildecE}
    \tilde{\cE}_{m,t}(\sigma,u)\lesssim \lab Y_m(0)\big)+ (T+\ep)^{\vartheta_{15}} \lae.
\end{align}
\end{prop}
\begin{proof}
Collecting \eqref{EE-highest}, \eqref{na2uhigh}, \eqref{secnoru-LinftyL2-final} for the control of  $\ep-$dependent norms appearing in the first two line of $\cE_{m,t}$ (see definition \eqref{defcEmt})
%$$\ep\mu^{\f{1}{2}}\|\na(\sigma, u)\|_{L_T^{\infty}H_{co}^{m-1}},\quad \mu\|\na^2 u \|_{L_T^{2}H_{co,\sqrt{\ep}}^{m-1}}, \quad \ep^{\f{2}{3}}\mu^{\f{1}{2}}\|\na^2 u \|_{L_T^{2}H_{co}^{m-2}},$$
and \eqref{EE-3}  \eqref{com-uniL2}, \eqref{nadivuL2} \eqref{na2sigma},  \eqref{nauL2-uniform}, 
\eqref{com-uniLinfty}, \eqref{na2sigmauni}, \eqref{nauhighLinftyL2-m-3} for the uniform estimates: 
$$\|(\sigma,u)\|_{\infco^{m-1}}, \quad \|\na(\sigma, u) \|_{L^2H_{co,\sqrt{\mu}}^{m-1}}, \quad \kpa^{\f{1}{2}}\|\na (\na\sigma,\div u)\|_{\hco^{m-2}},
\quad \|\na\sigma\|_{\infco^{m-2}}, \quad 
\|(\na u, \na^2\sigma)\|_{\infco^{m-3}}, $$
we can find a constant $\vartheta_{15}>0,$ such that \eqref{es-tildecE} holds.
%\begin{align*}   \tilde{\cE}_{m,t}(\sigma,u)\lesssim \lab Y_m(0)\big)+ (T+\ep)^{\vartheta_{14}}\lae. \end{align*}
\end{proof}

\section{$L_{t,x}^{\infty}$ estimates for $(\sigma, u)$}
In this section, we focus on  the a-priori estimates for $\cA_{m,t}(\sigma, u):$ 
\begin{align*}
    \mathcal{A}_{m,t}(\si, u)
  =\colon \il (\si, u)\il_{m-3,\infty, t}+ \il (\na\si, (\ep\mu)^{\f{1}{2}}\na u)\il_{m-4,\infty,t}
 + \il (\nabla u, \ep\mu^{\f{1}{2}}\na^2(\sigma, u) %(\ep\mu)^{\f{1}{2}}\na^2 \sigma
 \big)\il_{m-5,\infty,t}. 
\end{align*}
%\begin{equation}\label{defcAmt-1}
 %\begin{aligned} \mathcal{A}_{m,T}(\sigma,u,\theta)&=\il (\sigma, u, \theta, \kpa\na\theta)\il_{m-3,\infty,T}+\il (\na \sigma, \kpa^{\f{1}{2}}\na\theta)\il_{m-4,\infty,T}\\&\,\,+ \il (\nabla (u,\theta), \ep\mu^{\f{1}{2}}\na^2 u)\il_{m-5,\infty,T}+\kpa\il\na^2\theta\il_{m-5,\infty,T}+ \kpa^{\f{1}{2}}\il\na^2\theta\il_{m-6,\infty,T}. \end{aligned} \end{equation}
 Before going into the details, we begin with the following lemma which will be useful later.
 \begin{lem}
 Define the following non-uniform $L^{\infty}_{t,x}$ norms:
 \begin{align}
     \mathcal{I}_{m,t}(u)=\colon %\ep\bigg(
     \il\ep\mu\na u\il_{m-3,\infty,t}
    + \il%((\ep\mu)^{\f{1}{2}}\na u,
    \ep\mu\na^2 u\il_{m-4,\infty,t}.
 \end{align}
 Then it holds that:
 \begin{align}\label{epmuna2u}
  \mathcal{I}_{m,t}(u)\lesssim \lae.   
 \end{align}
 \end{lem}
 \begin{proof}
 %\begin{align}
  % \ep\mu \big(\il\na u\il_{m-3,\infty,t}+\il\na^2 u\il_{m-4,\infty,t}\big)\lesssim \lae.\end{align}
 %Indeed, 
 By the Sobolev embedding \eqref{sobebd},
 \begin{align*}
    & \ep\mu \il\na u\il_{m-3,\infty,t}\lesssim \ep\mu\big( \|\na^2 u\|_{\infco^{m-2}}+ \|\na u\|_{\infco^{m-1}}\big)\lesssim \cE_{m,t}. %\\
   % & (\ep\mu)^{\f{1}{2}}\il\na u\il_{m-4,\infty,t}\lesssim   \ep \mu\|\na^2 u\|_{\infco^{m-3}}+\|\na u\|_{\infco^{m-2}}\lesssim \lae.
 \end{align*}
 Moreover, we use successively the equation  $\eqref{newsys}_1,$ the Sobolev embedding \eqref{sobebd} to get that:
 \begin{align*}
     \ep\mu \il\na\div u\il_{m-4,\infty,t}&\lesssim 
     \mu \il\na(\ep\pt+\ep u\cdot\na)(\ep\sigma,\theta)\il_{m-4,\infty,t}\\
     &\lesssim \mu \il \na (\ep\sigma,\theta)\il_{m-3,\infty,t}+\ep\lat  \lesssim \lat.
 \end{align*}
 Finally, it follows from %$\eqref{NCNS-S2}_2$ to rewrite:
% \begin{align*}   \ep \mu \lambda_1 \Gamma \Delta u= -\ep\mu(\lambda_1+\lambda_2)\Gamma\na\div u+\f{\ep}{R\beta}(\pt+u\cdot\na)u+\na\sigma,\end{align*}
\eqref{rewrite-equ} and the
 local expression of the Laplacian \eqref{Laplace-local} that:
 \begin{align*}
     \ep\mu \il\na^2 u\il_{m-4,\infty,t}&\lesssim \il(\ep\mu\na u, u)\il_{m-3,\infty,t}+\il(\ep\mu \na\div u, \na\sigma, \theta)\il_{m-4,\infty,t}+\ep\lat \\
     &\lesssim \lat.
 \end{align*}
\end{proof}
 \begin{prop}\label{prop-amtsiu}
 Assume that \eqref{preasption} holds true. 
There exists a constant $\vartheta_{16}>0,$ such that for any $\ep\in(0,1], (\mu,\kpa)\in A,$ any $0<t\leq T,$
the following property holds:
  \begin{align}\label{amtsi-u}
      \cA_{m,t}(\sigma, u)%-\il u \il_{m-3,\infty,t}
      \lesssim  \lab Y_m(0)\big) +(T+\ep)^{\vartheta_{16}}\lae.
  \end{align}
 \end{prop}
 %\begin{rmk}As we have \end{rmk}
 \begin{proof}
 We first observe that it suffices to show that there exists a constant $\vartheta_{17}>0,$ such that:
 \beq\label{catildemt}
 \cA_{m,t}-\il u\il_{m-3,\infty,t}-\il(\ep\mu)^{\f{1}{2}}\na u\il_{m-4,\infty,t}\lesssim \lab Y_m(0)\big) +(T+\ep)^{\vartheta_{17}}\lae.
 \eeq
 Note that $\il\na u\il_{m-5,\infty,t}$ is contained in 
 the definition of $\cA_{m,t}.$ 
 Inserting this estimate into
 \eqref{nauhighLinftyL2-m-2}, we can find, by using 
 \eqref{EE-theta-0}, 
a constant $\vartheta_{18}>0$ such that:
 \begin{align}\label{nauinftyL2-final}
  \|\na u\|_{\infco^{m-2}}\lesssim \lab Y_m(0)\big)+ 
  (T+\ep)^{\vartheta_{18}}\lae.
 \end{align}
 We then get by the Sobolev embedding $\eqref{sobebd}$ and the estimate \eqref{es-tildecE} that
 \begin{align*}
    \il u\il_{m-3,\infty,t}+\il(\ep\mu)^{\f{1}{2}}\na u\il_{m-4,\infty,t}&\lesssim \| u\|_{\infco^{m-1}}+\|\na u\|_{\infco^{m-2}}+ \|\ep \mu\na^2 u\|_{\infco^{m-3}}\\
    &\lesssim \lab Y_m(0)\big)+ 
  (T+\ep)^{\vartheta_{16}}\lae,
  %\\(\ep\mu)^{\f{1}{2}}\il\na u\il_{m-4,\infty,t}&\lesssim  +\|\na u\|_{\infco^{m-2}}\lesssim \lae.
 \end{align*}
 where $\vartheta_{16}=\min \{\vartheta_{15}, \vartheta_{17}, \vartheta_{18} \}.$
This, combined \eqref{catildemt}, gives rise to \eqref{amtsi-u}. 
 %Let us first reduce the problem to the estimate of 
 %Thanks to \eqref{nathetainfty} for the $L_{t,x}^{\infty}$ estimates of $\theta,$ it remains for us to estimate
%\begin{align*}  \mathcal{A}_{m,T}(\sigma,u)=\colon \il (\sigma, u)\il_{m-3,\infty,T}+\il \na \sigma\il_{m-4,\infty,T}+ \il(\nabla u, \ep\mu^{\f{1}{2}}\na^2 u)\il_{m-5,\infty,T}.\end{align*}

Let us now prove \eqref{catildemt}. Thanks to the 
 Sobolev embedding \eqref{sobebd}, some quantities
 in $\cA_{m,t}(\sigma, u)$ can be controlled directly 
 by $\tilde{\cE}_{m,t}$ which is defined in \eqref{tildecE}. For instance:
\begin{align*}
 \il u\il_{m-4,\infty,t}&\lesssim \|\na u\|_{\infco^{m-3}}+\|u\|_{\infco^{m-2}}\lesssim \tilde{\cE}_{m,t}(\sigma, u),\\
    \il\sigma\il_{m-3,\infty,t}&\lesssim \|\na\sigma\|_{\infco^{m-2}}+\|\sigma\|_{\infco^{m-1}}\lesssim \tilde{\cE}_{m,t}(\sigma, u),\\
     \il\na \sigma\il_{m-4,\infty,t}&\lesssim \|\na^2\sigma\|_{\infco^{m-3}}+\| \na\sigma\|_{\infco^{m-2}}\lesssim \tilde{\cE}_{m,t}(\sigma, u). 
\end{align*}
%and enjoys the property \eqref{}
Consequently, in view of \eqref{es-tildecE}, we are left to control $\il(\nabla u, \ep\mu^{\f{1}{2}}\na^2 u )\il_{m-5,\infty,T}.$ %which is the aim of the following.  

As before, in the interior domain $\Omega_0$, by using the Sobolev embedding as well as the equivalence of the conormal  spaces and
usual Sobolev spaces, we have directly that:
\begin{align*}
    \il\chi_0(\nabla u, \ep\mu^{\f{1}{2}}\na^2 u , (\ep\mu)^{\f{1}{2}}\na^2 \sigma)\il_{m-5,\infty,t}\lesssim \|\chi_0 (u,\sigma)\|_{L^{\infty}H^{m-1}}\lesssim \cE_{m,t},
\end{align*}
where $\chi_0$ is a smooth function supported on $\Omega_0.$ Therefore, we focus on the estimate of 
$$\il\chi_i(\nabla u, \ep\mu^{\f{1}{2}}\na^2 u, (\ep\mu)^{\f{1}{2}}\na^2 \sigma )\il_{m-5,\infty,t},$$ where $\chi_i (i=1,\cdots, N)$ are the cut-off functions supported on the charts $\Omega_i.$

It can be deduced from the identity \eqref{normalofnormalder}, the fact
\beqs%\label{tan-nor}
\Pi(\partial_{\bn} u)
=\Pi(\curl u \times \bn)+\Pi \na (u\cdot\bn) +\Pi(-(D \bn) u)
\eeqs
and the local expression of the Laplacian \eqref{Laplace-local} that:
\begin{align*}
\il \chi_i \na u \il_{m-5,\infty,t}&\lesssim \il u\il_{m-4,\infty,t}+\il\div u\il_{m-5,\infty,t}+ \il\chi_i \curl u\times\bn\il_{m-5,\infty,t},\\
\ep\mu^{\f{1}{2}}\il \chi_i \na^2 u \il_{m-5,\infty,t}&\lesssim \ep^{\f{1}{2}}\il (\ep\mu)^{\f{1}{2}}\na u\il_{m-4,\infty,t}+\ep\mu^{\f{1}{2}}\il\na\div u\il_{m-5,\infty,t}\\
&\qquad\qquad\qquad\qquad+\ep\mu^{\f{1}{2}} \il\chi_i\p_{\bn} (\curl u\times\bn)\il_{m-5,\infty,t}, \\
\ep\mu^{\f{1}{2}}\il \chi_i \na^2 \sigma \il_{m-5,\infty,t}&\lesssim \ep\mu^{\f{1}{2}}\il \na \sigma \il_{m-4,\infty,t}+\ep\mu^{\f{1}{2}} \il \chi_i \Delta \sigma \il_{m-5,\infty,t}.
\end{align*}
In light of \eqref{rew-sigma}, 
%$\eqref{NCNS-S2}_1$ for $\sigma:$
%\beqs \div u=-\f{\ep}{\gamma} (\pt+u\cdot\na)\sigma+\f{\gamma-1}{\gamma}\kpa\Gamma\div(\beta\na\theta),\eeqs
we control the $L_{t,x}^{\infty}$ of $\div u$ by those having been controlled:
\begin{align*}
\il\div u\il_{m-5,\infty,t}&\lesssim \il\sigma\il_{m-4,\infty,t}+\kpa \il\Delta\theta\il_{m-5,\infty,t}+\kpa \il\na\theta\il_{m-5,\infty,t}^2+\ep \lat. %\\
%&\lesssim \lab Y_m(0)\big) + (T+\ep)^{\vartheta_6} \lae.
\end{align*}
Moreover, by using $\eqref{newsys}_1,$
we also see that:
\begin{align*}
\ep\mu^{\f{1}{2}}\il\na \div u\il_{m-5,\infty,t}&\lesssim %\il\sigma\il_{m-4,\infty,t}+\kpa \il\Delta\theta\il_{m-5,\infty,t}+\kpa \il\na\theta\il_{m-5,\infty,t}^2+\ep \lat \\&\lesssim \lab Y_m(0)\big) + (T+\ep)^{\vartheta_6} \lae.
\mu^{\f{1}{2}} \il\na(\ep\pt+\ep u\cdot\na)(\ep\sigma,\theta)\il_{m-5,\infty,t}\\
     &\lesssim \mu^{\f{1}{2}} \il \na \ep\theta\il_{m-4,\infty,t}+\ep\lae. 
\end{align*}
It thus remains for us to control $$\il\chi_i \curl u\times\bn\il_{m-5,\infty,t}, \quad  \ep\mu^{\f{1}{2}} \il\chi_i\p_{\bn} (\curl u\times\bn)\il_{m-5,\infty,t}, \quad \ep\mu^{\f{1}{2}} \il \chi_i \Delta \sigma \il_{m-5,\infty,t}, $$
which will be presented in the following three lemmas.
\end{proof}
\begin{lem}
Under the same assumption as in Proposition \ref{prop-amtsiu}, there exists a constant $\vartheta_{19}>0$ such that, for any $\ep \in (0,1], (\mu,\kpa)\in A,$ any $0<t\leq T,$ 
\begin{align*}
     \il \chi_i\curl u\times\bn\il_{m-5,\infty,t}\lesssim   \lab Y_m(0)\big) +(T+\ep)^{\vartheta_{19}}\lae, \qquad  \forall \, i=1,\cdots N.
\end{align*}
\end{lem}
\begin{proof}
Let us write:
\begin{align*}
    \omega\times\chi_i\bn=\colon\curl u\times  \chi_i\bn=\f{1}{r_0}\curl(r_0 u)\times\chi_i\bn-\f{\na r_0\times u}{r_0}\times \chi_i\bn,
\end{align*}
where $r_0$ is defined in \eqref{defr0}. Consequently, 
we have that:
\begin{align*}
  \il \omega \times \chi_i\bn \il_{m-5,\infty,t}\lesssim 
 \il r_0^{-1}\il_{m-5,\infty,t} \il\curl(r_0 u)\il_{m-5,\infty,t}+  \il (r_0^{-1}\na r_0, u)\il_{m-5,\infty,t}^2.
\end{align*}
It can be derived from the definition \eqref{defr0} for $r_0,$ the estimate \eqref{Es-tsigma4} for $\tsigma$
that
\begin{align*}
   \il( r_0^{-1}, r_0^{-1}\na r_0) \il_{m-5,\infty,t}&\lesssim \lab \il(\Id, \na)(\ep\tsigma, \theta) \il_{m-5,\infty,t}%+ \il\mu\na(\sigma, \theta)\il_{m-4,\infty,t}
   \big) \\
   &\lesssim \lab \ep\lat+ \cA_{m,t}(\theta)\big)
\end{align*}
where $\cA_{m,t}(\theta)$ has been controlled already in \eqref{nathetainfty}. Therefore, it suffices to prove that, there exists a constant $\vartheta_{20}>0,$ such that:
\begin{align}\label{toprove}
 \il \chi_i\curl(r_0 u)\times\bn\il_{m-5,\infty,t}\lesssim   \lab Y_m(0)\big) +(T+\ep)^{\vartheta_{20}}\lae, \qquad  \forall \, i=1,\cdots N.
\end{align}
Denote $\zeta_i=|g|^{\f{1}{4}}\big(\curl(r_0 u)\times \chi_i\bn\big)^{\Psi_i},$ where $\Psi_i$ is the transformation associated to the normal geodesic coordinates defined in \eqref{normal geodesic coordinates}, we are lead to  control $\il \zeta_i\il_{m-5,\infty,t}, \, i=1,\cdots N.$ In what follows, we shall drop the  subscript  $i$ for the notational clarity.

By \eqref{eqzeta}, $\zeta$ solves the system:
\begin{equation}\label{eqzetainfty}
\left\{
    \begin{array}{l}
   (\pt+ \tilde{u} \cdot \na){\zeta}-\mu \lambda_1\overline{\Gamma}
  ( \f{{1}}{r_0})^{\Psi}(\partial_z^2  +\Delta_{{g}})\zeta =|g|^{\f{1}{4}} H^{\Psi}+ H^{\dag},  \quad (t,x) \in \mR_{+}\times \mR^3,  \\[2.5pt]
\zeta|_{z=0}= \big[\f{R\ep}{C_v\gamma}r_0\p_{\bn}\tsigma |g|^{\f{1}{4}}\chi\Pi u+2r_0|g|^{\f{1}{4}}\chi\Pi\big((D\bn)u-a u\big)\big]^{\Psi}|_{z=0},  \\[2.5pt]
\zeta|_{t=0}=|g|^{\f{1}{4}}\big(\curl(r_0 u)\times \chi\bn\big)^{\Psi}|_{t=0},
    \end{array}
\right.
\end{equation}
where $H^{\dag}$ is defined in \eqref{defHdag}, $H^{\Psi}=H^0\circ \Psi$ with $H^0$ defined in \eqref{defH}, $\Delta_g$ defined in \eqref{defdeltag} involves only tangential derivatives. Moreover, the definitions of $r_0, \tsigma, \tilde{u} $ can be found in \eqref{defr0},  \eqref{deftildeu}.
In order to take benefit of the property \eqref{infty2}
for the diffusion-transport equation \eqref{eqheattransport}, 
let us rewrite the system in the following way:
\beq\label{eqzeta-infty}
\begin{aligned}
  & \big(\pt+\tilde{w}_1^b\p_{y^1}+\tilde{w}_2^b\p_{y^2}+  z (\p_z\tilde{w}_3)^b\p_z\big){\zeta}-\mu \lambda_1\overline{\Gamma}
  [(r_0^{-1})^{\Psi}]^b\partial_z^2  \zeta\\
  &=|g|^{\f{1}{4}} H^{\Psi}+ H^{\dag}-\sum_{l=1}^2 (\tilde{w}_{l}-\tilde{w}_{l}^b)\p_{y^l}\zeta- \big(\tilde{w}_{3}-\tilde{w}_{3}^b-z  (\p_z\tilde{w}_{3})^b\big)\p_z \zeta\\
 &\quad+\mu \lambda_1\overline{\Gamma}
 \bigg( (r_0^{-1})^{\Psi}\Delta_g \zeta+  ((r_0^{-1})^{\Psi}-[(r_0^{-1})^{\Psi}]^b) \partial_z^2  \zeta \bigg)- \kpa\overline{\Gamma}(D\Psi)^{-*} (\beta\na\theta)^{\Psi}\cdot\na\zeta
 %&= \colon \mathfrak{F}_1+\p_z\mathfrak{F}_2
\end{aligned}
\eeq
where $\tilde{w}=\tilde{u}-\kpa\overline{\Gamma}(D\Psi)^{-*} (\beta\na\theta)^{\Psi}$ is defined in \eqref{deftildew}.  
To control $\il\zeta\il_{m-5,\infty,t},$ it suffices to 
show the following proposition: 
\begin{prop}\label{lem-source}
The right hand side of \eqref{eqzeta-infty} can be written into the form $\mathfrak{F}_1+\mu^{\f{1}{2}}\p_z\mathfrak{F}_2,$ where 
$\mathfrak{F}_1, \mathfrak{F}_2$ admit the property: 
\begin{align}\label{mainprop}
    \big(\int_0^t \|\mathfrak{F}_1(s)\|_{m-5,\infty}^2\d s\big)^{\f{1}{2}}+  \big(\int_0^t \|\mathfrak{F}_2(s)\|_{m-5,\infty}^4\d s\big)^{\f{1}{4}}\lesssim \lae. 
\end{align}
\end{prop}
Indeed, once this proposition is proved, it then follows from \eqref{infty2} (by setting $\nu=\mu, b=[(r_0^{-1})^\Psi]^b$) that,
\beq\label{zetainfty-pre}
\begin{aligned}
    \il\zeta\il_{m-5,\infty,t}\lesssim \Lambda_3^{m-5}\bigg(\|\zeta|_{t=0}\|_{m-5,\infty}+ 
    |\zeta|_{z=0}|_{m-5,\infty,t}+T  \Lambda_4^{m-5} \il\zeta\il_{m-5,\infty,t}\\
 \qquad \quad  + T^{\f{1}{2}} \big(\int_0^t \|\mathfrak{F}_1(s)\|_{m-5,\infty}^2\d s\big)^{\f{1}{2}}+ T^{\f{1}{4}} \big(\int_0^t \|\mathfrak{F}_2(s)\|_{m-5,\infty}^4\d s\big)^{\f{1}{4}}    \bigg)
\end{aligned}
\eeq
where we denote
\begin{align*}
    \Lambda_3^{k}= \lab | T(\tilde{w}^b, (\p_z\tilde{w}_3)^b), [(r_0^{-1})^{\Psi}]^b |_{k,\infty,t} \big), \quad \Lambda_4^k=\lab |(\tilde{w}^b, (\p_z\tilde{w}_3)^b), \pt [(r_0^{-1})^{\Psi}]^b |_{k,\infty,t}
    \big) .
\end{align*}
Analogues to the estimate \eqref{es-Lambda12} and 
\eqref{es-Lambda1}, we can find a positive constant %that is still denoted as  
$\vartheta_{21}>0,$ such that:
\begin{align}\label{eslembda35}
\Lambda_3^{m-5} \lesssim \lab Y_m(0)\big) + (T+\ep)^{\f{\vartheta_{21}}{2}}\lae, \,\qquad  \Lambda_4^{m-5}\lesssim \lat,
\end{align}
which yields that:
\begin{align}\label{zetainitial}
    \Lambda_3^{m-5}\|\zeta|_{t=0}\|_{m-5,\infty}
 \lesssim \lab Y_m(0)\big)+(T+\ep)^{\vartheta_{21}}\lae.
\end{align}
Moreover, by the definition $\eqref{eqzetainfty}_2$ and the Sobolev embedding  \eqref{sobebd}, 
we can have that:
\begin{align*}
  |\zeta|_{z=0}|_{m-5,\infty,t}&\lesssim \lab \il(\theta, u)\il_{m-5,\infty,t}+ \ep \il(\na \sigma, \ep\mu \na\div u)\il_{m-5,\infty,t}\big) \\
  &\lesssim \lab \tilde{\cE}_{m,t}+ \ep \cA_{m,t}\big),
\end{align*}
which, combined with \eqref{es-tildecE} for $\tilde{\cE}_{m,t}$, allows us to find a constant $\vartheta_{22}>0$ such that:
\begin{align*}
|\zeta|_{z=0}|_{m-5,\infty,t}\lesssim  \lab Y_m(0)\big) + (T+\ep)^{\vartheta_{22}}\lae.
\end{align*}
Collecting \eqref{eslembda35} and the above estimate, we then find a constant $\vartheta_{23}>0,$ such that:
\begin{align}\label{zetabdry}
\Lambda_3^{m-5}|\zeta|_{z=0}|_{m-5,\infty}\lesssim  \lab Y_m(0)\big) + (T+\ep)^{\vartheta_{23}}\lae.
\end{align}
Inserting now \eqref{zetainitial}, \eqref{zetabdry}, 
\eqref{eslembda35} and \eqref{mainprop} into \eqref{zetainfty-pre}, we can ensure the existence of some constant $\vartheta_{20}>0,$ such that:
\begin{align}
    \il\zeta\il_{m-5,\infty,t}\lesssim\lab Y_m(0)\big) + (T+\ep)^{\vartheta_{20}}\lae,
\end{align}
which, by the change of variable, lead to \eqref{toprove} and thus finishes the proof.
 \end{proof}
%Consequently, the estimate
%The further application of the Sobolev embedding \eqref{sobebd} and 
\begin{proof}[Proof of Proposition $\ref{lem-source}$]
Let us come back to prove Proposition \ref{lem-source}. 
We sort up the right hand side of the above equations and clarify the definition of $\mathfrak{F}_1$ and $\mathfrak{F}_2.$ First of all, as $\zeta=(\curl(r_0 u)\times\chi\bn)^{\Psi},$ we can rewrite the term 
$\lambda_1\overline{\Gamma}
 (r_0^{-1})^{\Psi}\Delta_g \zeta$ as:
 \begin{align}\label{defFR12}
  \mu\lambda_1\overline{\Gamma}
 (r_0^{-1})^{\Psi}\Delta_g \zeta= \mu\p_z\mathfrak{R}_1\bigg(
 (r_0^{-1})^{\Psi}, (\p_y^2, \p_y, \Id)(r_0 u)^{\Psi} \bigg)+\mu \mathfrak{R}_2 \bigg( (\p_z,\Id)  (r_0^{-1})^{\Psi}, (\p_y^3, \p_y^2, \p_y, \Id)(r_0 u)^{\Psi}
 \bigg)
 \end{align}
 where $\mathfrak{R}_1$ and $\mathfrak{R}_2$ are smooth functions concerning their arguments with coefficients depending on $(\chi\bn)^{\Psi}, D\Psi$ and their derivatives up to the second order. 
 Similarly, let us split the term $\kpa\overline{\Gamma}(D\Psi)^{-*} (\beta\na\theta)^{\Psi}\cdot\na\zeta$ as
 \begin{align}\label{defFR34}
   \kpa\overline{\Gamma}(D\Psi)^{-*} (\beta\na\theta)^{\Psi}\cdot\na\zeta=\kpa \p_z 
   \mathfrak{R}_3 \big((\beta\na\theta)^{\Psi}, \na (r_0 u)^{\Psi}\big)
   + \kpa \mathfrak{R}_4\big((\Id ,\p_z) (\beta\na\theta)^{\Psi}, (\na, \p_{y}^2) (r_0 u)^{\Psi} \big).
 \end{align}
 Next, let us write: 
 \begin{align*}
     \mu \lambda_1\overline{\Gamma}
  \big((r_0^{-1})^{\Psi}-[(r_0^{-1})^{\Psi}]^b\big) \partial_z^2 \zeta=\mu \p_z \mathfrak{R}_5+  \mu\mathfrak{R}_6
 \end{align*}
 with 
 \begin{align}\label{defFR56}
    \mathfrak{R}_5=\lambda_1\overline{\Gamma} 
 \bigg( \big((r_0^{-1})^{\Psi}-[(r_0^{-1})^{\Psi}]^b\big)\p_z\zeta- \p_z (r_0^{-1})^{\Psi} \zeta\bigg), \quad
 \mathfrak{R}_6=\lambda_1\overline{\Gamma} \p_z^2 (r_0^{-1})^{\Psi}  \zeta.
 \end{align}
 Finally, for the term $H^{\Psi}$,  there are also some
 terms that may involve one more normal derivative than what can be controlled, namely the terms $(1)^{\Psi}, (2)^{\Psi}:$
 \begin{align*}
     %&H_{11}=\colon     H_{12}=\colon H_{21}=\colon
(1)= \curl G\times\chi\bn,\qquad 
(2)= -\mu\lambda_1\overline{\Gamma}\big(\na (r_0^{-1})\times \curl \omega_{r_0}\big)\times\chi\bn- 2\mu\lambda_1\overline{\Gamma} r_0^{-1}  \div\big(\omer\times\na (\chi\bn)\big),
 \end{align*}
where $\omega_{r_0}=\curl(r_0 u)$ and $G$ is defined in \eqref{def-G}. These two terms could be split into the following forms respectively:
\begin{align}\label{defFR78}
(1)^{\Psi}=\mu^{\f{1}{2}}\p_z\mathfrak{R}_7\big(\mu^{-\f{1}{2}}G^{\Psi}\big)+\mathfrak{R}_8\big((\Id, \p_y)G^{\Psi}\big),
\end{align}
\begin{align}\label{defFR910}
   (2)^{\Psi}=\mu  \p_z\mathfrak{R}_9\big(\omega_{r_0}^{\Psi}, \p_y (r_0 u)^{\Psi}, (\Id, \na) (r_0^{-1})^{\Psi}\big)
   + \mu  \mathfrak{R}_{10} \big( \omega_{r_0}^{\Psi}, (\p_y,\p_y^2)(r_0 u)^{\Psi},(\Id,\na, \na^2)(r_0^{-1})^{\Psi}\big).
\end{align}

 We now can define $ \mathfrak{F}_1, \mathfrak{F}_2$ as:
 \begin{align*}
    & \mathfrak{F}_1=\big(H-(1)-(2)\big)^{\Psi}+ H^{\dag}+\mu (\mathfrak{R}_2+\mathfrak{R}_6+\mathfrak{R}_{10})+\kpa \mathfrak{R}_4+\mathfrak{R}_8\\
    &\qquad \qquad+\sum_{l=1}^2 (\tilde{w}_{l}-\tilde{w}_{l}^b)\p_{y^l}\zeta- \big(\tilde{w}_{3}-\tilde{w}_{3}^b-z  (\p_z\tilde{w}_{3})^b\big)\p_z \zeta, \\
 & \mathfrak{F}_2=\mu^{\f{1}{2}}(\mathfrak{R}_1+\mathfrak{R}_5+\mathfrak{R}_9)+\kpa\mu^{-\f{1}{2}}\mathfrak{R}_3
+\mathfrak{R}_7(\mu^{-\f{1}{2}}G),
 \end{align*}
 where $\mathfrak{R}_1-\mathfrak{R}_{10}$ are defined just before in 
 \eqref{defFR12}- \eqref{defFR910} 
 and 
 \beqs
\begin{aligned}
H-(1)-(2)=  -\sum_{l=1}^3\na u_{l}\times \p_{l}(r_0 u) \times \chi \bn+\omer\times (u\cdot \na (\chi\bn))+\mu\lambda_1\overline{\Gamma} r_0^{-1}
\omer \times\Delta(\chi\bn),
\end{aligned}
\eeqs
 see also  \eqref{defHdag} for the definitions of $H^{\dag}.$
 
It now remains for us to show that $ \mathfrak{F}_1, \mathfrak{F}_2$ satisfy the property \eqref{mainprop}.
Most of the terms appearing in $\mathfrak{F}_1, \mathfrak{F}_2$ can be controlled directly in view of their expressions. For instance, by using the assumption 
$\kpa \approx \mu,$
\begin{align}\label{frakF-easyterm}
    \il(H-(1)-(2))^{\Psi}, H^{\dag}, \mathfrak{R}_1, \kpa \mu^{-\f{1}{2}} \mathfrak{R}_3,  \kpa \mathfrak{R}_4, \mu \mathfrak{R}_6,  \mathfrak{R}_9, \mu\mathfrak{R}_{10} )\il_{m-5,\infty,t}\lesssim \lae.
\end{align}
Moreover, it follows similarly from the derivation of \eqref{escm3-3} that:
\beq\label{es-corectedconvection}
\begin{aligned}
&\il \sum_{l=1}^2 (\tilde{w}_{l}-\tilde{w}_{l}^b)\p_{y^l}\zeta+ \big(\tilde{w}_{3}-\tilde{w}_{3}^b-z  (\p_z\tilde{w}_{3})^b\big)\p_z \zeta \il_{m-5,\infty,t}\\
&\lesssim \lab \il(\na u, \na\tilde{w})\il_{m-5,\infty,t} +\il\p_z^2\tilde{w}_3\il_{0,\infty,t}+\il u \il_{m-3,\infty,t}\big) \lesssim \lae.
\end{aligned}
\eeq
It remains to control $\mu\mathfrak{R}_2, \mu^{\f{1}{2}}\mathfrak{R}_5, \mathfrak{R}_7, \mathfrak{R}_8$  which are the tasks of the following.
First, by the Sobolev embedding \eqref{sobebd},
it holds that:
\beq
 \begin{aligned}
 \mu^{\f{1}{2}}\big(\int_0^t &\|\mathfrak{R}_2(s)\|_{m-5,\infty}^2 \d s\big)^{\f{1}{2}}\lesssim \mu^{\f{1}{2}} 
   \big(\int_0^t \|(r_0 u)(s)\|_{m-2,\infty}^2\d s\big)^{\f{1}{2}}+T^{\f{1}{2}}\lat\\
  & \lesssim \mu^{\f{1}{2}} \big(\|\na(r_0 u)\|_{\hco^{m-1}}+ \|r_0 u\|_{\uhco^m}\big)+T^{\f{1}{2}}\lat \lesssim \lae.
 \end{aligned}
 \eeq
 Let us proceed to estimate  $\mu^{\f{1}{2}}\mathfrak{R}_5$ defined in \eqref{defFR56}. Thanks to \eqref{r0-r0b}, \eqref{defbetaK-betaKb}, we have that:
 \begin{align*}
   \mu^{\f{1}{2}}  \big((r_0^{-1})^{\Psi}-[(r_0^{-1})^{\Psi}]^b\big)\p_z\zeta = \mu^{\f{1}{2}} (\ep z\mathcal{V}_1\p_z\zeta+z^2\mathcal{V}_2 \zeta )
 \end{align*}
 where, by the property \eqref{na2theta-Lp}  (for $p=4$),
 \begin{align*}
 \il \mathcal{V}_1\il_{m-5,\infty,t}\lesssim \lat, \quad
 \big(\int_0^t\|\mu^{\f{1}{2}}\mathcal{V}_2 \|_{m-5,\infty}^4\d s\big)^{\f{1}{4}}\lesssim  \big(\int_0^t\|\mu^{\f{1}{2}}(\Id,\na,\na^2)\theta \|_{m-5,\infty}^4\d s\big)^{\f{1}{4}}\lesssim \lae.
 \end{align*}
 Therefore, by noting the following estimates:
 \begin{align*}
& \il\ep\mu^{\f{1}{2}}z \p_z\zeta\il_{m-5,\infty,t}+\il z^2\zeta\il_{m-5,\infty,t}\lesssim \ep\mu^{\f{1}{2}}\il\na u\il_{m-4,\infty,t}+\
 \il u\il_{m-3,\infty,t}\lesssim \lat,\\
&\il\p_z (r_0^{-1})^{\Psi} \zeta\il_{m-5,\infty,t}\lesssim \lab \il(\Id,\na)(\theta,\ep\tsigma, u)\il_{m-5,\infty,t}\big)
\lesssim \lat,
 \end{align*}
 we can conclude that:
 \begin{align}\label{esFr5}
      \big(\int_0^t\|\mu^{\f{1}{2}}\mathfrak{R}_5 \|_{m-5,\infty}^4\d s\big)^{\f{1}{4}}\lesssim 
\lae.
 \end{align}
 Next, we control $\mu^{\f{1}{2}}\mathfrak{R}_7, \mathfrak{R}_8$ as:
 \begin{align}\label{esFr78-pre}
    \big( \int_0^t \|\mathfrak{R}_7(s)\|_{m-5,\infty}^4 \d s\big)^{\f{1}{4}}\lesssim  \big( \int_0^t \|\mu^{-\f{1}{2}}G(s)\|_{m-5,\infty}^4 \d s\big)^{\f{1}{4}},\quad 
 \big(\int_0^t \|\mathfrak{R}_8(s)\|_{m-5,\infty} ^2\d s\big)^{\f{1}{2}} \lesssim \big(\int_0^t \|G(s)\|_{m-4,\infty} ^2\d s\big)^{\f{1}{2}}
 \end{align}
 where $G=G_1+G_2$ with
 \begin{align*}
   &  G_1=u(\pt+u\cdot\na)r_0+\mu\lambda_1\overline{\Gamma}r_0^{-1}[\curl\curl, r_0]u,\\
   & G_2=-\ep\mu r_1\lambda_1\Gamma\curl\curl u-\mu\lambda_1(\Gamma-\overline{\Gamma})\curl\curl u-\mu(2\lambda_1+\lambda_2)(1+\ep r_1)\div u \na\Gamma.
 \end{align*}
By the definition of $\cA_{m,t}$ and the estimate \eqref{epmuna2u},  we see that:
\beq\label{propG2}
\begin{aligned}
 &\mu^{-\f{1}{2}}\il G_2 \il_{m-5,\infty,t}\lesssim 
 \lab \il(\sigma, \na\sigma, \ep\mu \div u, \ep\mu^{\f{1}{2}}\na^2 u)\il_{m-5,\infty,t}\big)\lesssim \lat ,\\
 &\il G_2\il_{m-4,\infty,t}\lesssim \lab \il(\sigma, \na\sigma, \ep\mu \div u, \ep\mu\na^2 u)\il_{m-4,\infty,t}\big) \lesssim \lae.
\end{aligned}
\eeq
Moreover, by \eqref{identity1}, \eqref{timeder-r0-1}, 
\begin{align*}
u(\pt+u\cdot\na)r_0=-\f{\kpa}{C_v\gamma} \exp\big(\f{R\ep}{C_v\gamma} \tsigma\big)\Gamma u\Delta\theta + l.o.t, \quad  \mu\lambda_1\overline{\Gamma}r_0^{-1}[\curl\curl, r_0]u
  = \mu\lambda_1\overline{\Gamma}(\Delta\theta%\f{R\ep}{C_v\gamma}\Delta\tsigma
  )u+ l.o.t. \, ,
\end{align*}
%\beqs -r_0^{-1}\Delta{r_0}=\Delta\theta-\f{R\ep}{C_v\gamma}\Delta\tsigma+|\f{R\ep}{C_v\gamma}\na\tsigma-\na\theta|^2,\eeqs
where we denote $l.o.t.$ for the terms which satisfy: %{\color{red}we need to show $\ep\il\mu^{\f{1}{2}}\Delta\sigma\il_{m-5,\infty,t}, \ep\mu \il\Delta\sigma\il_{m-4,\infty,t} $}
\begin{align*}
 % \big(\int_0^t \| \mu^{-\f{1}{2}}  (l.o.t)(s)\|_{m-5,\infty}^2 \d s\big)^{\f{1}{2}}
\il \mu^{-\f{1}{2}}  (l.o.t)\il_{m-5,\infty,t} +\big(\int_0^t \| (l.o.t)(s)\|_{m-4,\infty}^2\d s\big)^{\f{1}{2}}\lesssim \lae. 
\end{align*}
Noticing the property \eqref{na2theta-Lp} and the fact
$$(\mu+\kpa)\il\Delta\theta\il_{m-5,\infty,t}\lesssim \lat,\quad  \kpa\big(\int_0^t \| \Delta\theta(s)\|_{m-4,\infty}^2\d s\big)^{\f{1}{2}}\lesssim \kpa(\|\na\Delta\theta\|_{\hco^{m-3}}+\|\Delta\theta\|_{\hco^{m-2}}) \lesssim \cE_{m,t},$$ we can conclude that:
\begin{align*}
    \big( \int_0^t \|\mu^{-\f{1}{2}}G_1(s)\|_{m-5,\infty}^4 \d s\big)^{\f{1}{4}}+ \big(\int_0^t \|G_1(s)\|_{m-4,\infty} ^2\d s\big)^{\f{1}{2}}\lesssim \lae,
 \end{align*}
 which, combined with  \eqref{propG2} and \eqref{esFr78-pre}, yields
 \begin{align}\label{frakR78}
    \big( \int_0^t \|\mathfrak{R}_7(s)\|_{m-5,\infty}^4 \d s\big)^{\f{1}{4}}+
 \big(\int_0^t \|\mathfrak{R}_8(s)\|_{m-5,\infty} ^2\d s\big)^{\f{1}{2}} \lesssim \lae.
 \end{align}
We now finish the proof of \eqref{mainprop} by 
collecting \eqref{frakF-easyterm}-\eqref{esFr5} and \eqref{frakR78}. 
 %corrected convection term.
 \end{proof}
\begin{lem}
Under the same assumption as in Proposition \ref{prop-amtsiu}, there exists a constant $\vartheta_{24}>0$ such that, for any $\ep \in (0,1], (\mu,\kpa)\in A,$ any $0<t\leq T,$ %There exists a constant $\vartheta_{24}>0$ such that
\begin{align*}
   \ep\mu^{\f{1}{2}}  \il \p_{\bn}(\curl u\times\chi_i\bn)\il_{m-5,\infty,t}\lesssim   \lab Y_m(0)\big) +(T+\ep)^{\vartheta_{24}}\lae, \qquad  \forall \, i=1,\cdots N.
\end{align*}
\end{lem}
\begin{proof}
Since $\il\na u\il_{m-5,\infty,t}$ has been controlled, 
it suffices for us to control $\ep\mu^{\f{1}{2}}\il\p_{\bn}\omega_{\bn i}\il_{m-5,\infty,t},$ where 
%\beqs%\label{defomegani}
$\omega_{\bn i}=\colon \curl u\times\chi_i \bn -2 \Pi (\chi_i(-a u+D\bn\cdot u))$
%\eeqs
solves the equations:
\begin{align*}
    (\pt +u\cdot\na)\omega_{\bn i}-\lambda_1 R\mu (\beta\Gamma)\Delta \omega_{\bn i}=F_i^{\omega}+L_i, \, \text{ in } \Omega_i\cap \Omega 
    \quad \omega_{\bn i}|=0, \, \text{ on } {\p(\Omega_i\cap \Omega)} ,
\end{align*}
we refer to \eqref{defFomegaL} for the definitions of 
$F_i^{\omega}$ and $L_i.$ 
As before, we define $  \varsigma_i=\colon \omega_{\bn i}\circ \Psi_i,$
($\Psi_i$ is defined in \eqref{normal geodesic coordinates}) and extend it  by zero from 
$\Psi^{-1}(\overline{\Omega_i\cap \Omega})$ to $\overline{\mathbb{R}_{+}^3}$ without changing the notation.
%and still denote it by $ \varsigma_i.$ 
Thanks to \eqref{change-variable}-\eqref{laplaceg}, we can find that $ \varsigma_i$ is governed by the system: (the subscript is dropped)
\begin{align*}
\left\{ 
\begin{array}{l}
   \big(\pt+\tilde{w}_1^b\p_{y^1}+\tilde{w}_2^b\p_{y^2}+  z (\p_z\tilde{w}_3)^b)\p_z\big) \varsigma-\mu \lambda_1 R
  [(\Gamma\beta)^{\Psi}]^b\partial_z^2  \varsigma= F^{\varsigma} \text{ in } \mR_{+}^3,\\[2.5pt]
   \varsigma|_{t=0}=\omega_{\bn}|_{t=0}\circ \Psi, \quad  \varsigma|_{z=0}=0,
\end{array}
\right.
\end{align*}
  where $\tilde{w}$ is defined in \eqref{deftildew} and 
  \beq\label{defFvarsigma}
  \begin{aligned}
 F^{\varsigma}& =(F^{\omega}+L)^{\Psi}-\sum_{l=1}^2 (\tilde{w}_{l}-\tilde{w}_{l}^b)\p_{y^l}\varsigma- \big(\tilde{w}_{3}-\tilde{w}_{3}^b-z  (\p_z\tilde{w}_{3})^b\big)\p_z \varsigma\\
 &\quad+\mu \lambda_1 R
 \bigg( (\Gamma\beta)^{\Psi}(\f{1}{2}\p_z \ln |g|\p_z+\Delta_g) \varsigma+  ((\Gamma\beta)^{\Psi}-[(\Gamma\beta)^{\Psi}]^b) \partial_z^2  \varsigma \bigg)- \kpa\overline{\Gamma}(D\Psi)^{-*} (\beta\na\theta)^{\Psi}\cdot\na\varsigma.
\end{aligned}
\eeq
We can now apply \eqref{infty3} in appendix (setting $\nu=\mu$)
to obtain that:
\begin{align}\label{varsigma-pre}
    \ep\mu^{\f{1}{2}}\il\p_z \varsigma\il_{m-5,\infty,t}\lesssim \Lambda_3^{m-5} \|\ep \mu^{\f{1}{2}}\varsigma|_{t=0}\|_{m-5,\infty}+
   T^{\f{1}{2}} \Lambda_3^{m-5}\Lambda_4^{m-5} \il\ep F^{\varsigma}\il_{m-5,\infty,t}
\end{align}
where $\Lambda_3^{m-5}, \Lambda_4^{m-5}$ admit the property \eqref{eslembda35}. Similar to \eqref{zetainitial}, we have directly by \eqref{eslembda35} that:
\begin{align}\label{varsigma-intial}
    \Lambda_3^{m-5} \|\ep \mu^{\f{1}{2}}\varsigma|_{t=0}\|_{m-5,\infty}\lesssim 
\Lambda_3^{m-5}Y_m(0)
 \lesssim \lab Y_m(0)\big)+(T+\ep)^{\vartheta_{21}}\lae.
\end{align}
It thus remains to control $\il \ep F^{\varsigma}\il_{m-5,\infty,t}.$ %In view of the 
By checking every term in \eqref{defFomegaL} where $F^{\omega}, L$ are defined, one can have that:
\begin{align*}
    \il\ep (F^{\omega}+L)^{\Psi}\il_{m-5,\infty,t}\lesssim \lab \il\big((\Id,\na, \ep\pt)(\sigma, u, \theta), \ep\mu\na^2 u\big)\il_{m-5,\infty,t}\big) \lesssim \lat.
\end{align*}
Next, it follows almost identically as in \eqref{es-corectedconvection} that:
\begin{align*}
   \il \sum_{l=1}^2 (\tilde{w}_{l}-\tilde{w}_{l}^b)\p_{y^l}\varsigma+ \big(\tilde{w}_{3}-\tilde{w}_{3}^b-z  (\p_z\tilde{w}_{3})^b\big)\p_z \varsigma\il_{m-5,\infty,t}\lesssim \lae.
\end{align*}
Finally, it is directly to verify from their expression that the multiplication of $\ep$ and the last line in \eqref{defFvarsigma} can be controlled as:
\begin{align*}
    \il\ep (\textnormal{last line of \eqref{defFvarsigma}})\il_{m-5,\infty,t}&\lesssim \lab \ep\mu(\il\na u\il_{m-3,\infty,t}+\il\na^2 u\il_{m-4,\infty,t})+\il\na(\sigma,\theta)\il_{m-5,\infty,t}\big)\\
    &\lesssim \lae.
\end{align*}
Note that the intermediate estimate \eqref{epmuna2u} has been used in the derivation of the second inequality. 
The previous three estimates now give rise to:
$ \il \ep F^{\varsigma}\il_{m-5,\infty,t}\lesssim \lae,
$ inserting which and \eqref{varsigma-intial} into 
\eqref{varsigma-pre} stems
\begin{align*}
    \ep\mu^{\f{1}{2}}\il\p_z\varsigma\il_{m-5,\infty,t}\lesssim \lab Y_m(0)\big)+(T+\ep)^{\vartheta_{24}}\lae, \quad
\end{align*}
where $ \vartheta_{24}=\min \{\vartheta_{21},\f{1}{2}\}.$
The proof of this lemma is thus completed.
\end{proof}
\begin{lem}
Under the assumption \eqref{preasption}, it holds that, for any  $\ep \in (0,1], (\mu,\kpa)\in A,$ any $0<t\leq T,$
\begin{align}\label{Deltasigmainftytx}
\ep\mu^{\f{1}{2}} \il \Delta \sigma \il_{m-5,\infty,t}\lesssim Y_m(0)+\ep^{\f{1}{2}}\lae.
\end{align}
\end{lem}
\begin{proof}
It has been computed in Section 6 that $\Delta\sigma$ solves the damped transport equation \eqref{eqDeltasigma}. %By using the Lagranian coordinates,
One can verify, by using the similar arguments as in [Lemma 3.3, \cite{MR4403626}] that:
\begin{align}\label{Deltasigmainftytx-1}
  \ep\mu^{\f{1}{2}}   \il \Delta\sigma\il_{m-5,\infty,t}\lesssim \ep\mu^{\f{1}{2}}  \| \Delta\sigma(0)\|_{m-5,\infty}+ \ep\mu^{\f{1}{2}} \il \mathfrak{H}_3\il_{m-5,\infty,t}+ \ep\lat,
\end{align}
where %we denote 
$\mathfrak{H}_3=\div \mathfrak{H}_2-\ep^2\mu(2\lambda_1+\lambda_2) \big(\na\Gamma\cdot\pt\na\sigma+\na(\Gamma u)\cdot\na\sigma\big)$
with $\mathfrak{H}_2$ defined in
\eqref{deffrakh12}. By checking every term in the expression of $\mathfrak{H}_3,$ we find that:
\begin{align*}
   \ep\mu^{\f{1}{2}} \il \mathfrak{H}_3\il_{m-5,\infty,t}\lesssim \ep^{\f{1}{2}}\lat,
\end{align*}
which, together with \eqref{Deltasigmainftytx-1}, yields \eqref{Deltasigmainftytx}.
\end{proof}
\section{Uniform estimates-proof of Proposition \ref{prop-uniform es}}
The estimates established in the previous sections allow us to give the proof of Proposition \ref{prop-uniform es}:
\begin{proof}[Proof of \ref{prop-uniform es}]
In light of the estimates \eqref{es-tildecE}, \eqref{nauinftyL2-final}, we find that 
\begin{align}\label{es-cemt-final}
     \cE_{m,T}(\sigma,u)\lesssim \lab Y_m(0)\big) + (T+\ep)^{\vartheta_{25}} \lab \cN_{m,T}\big),
\end{align}
where $\vartheta_{25}=\min\{ \vartheta_{15}, \vartheta_{18} \}>0.$ 
This estimate, together with the estimate \eqref{EE-theta-final} for $\cE_{m,T}(\theta),$ yields that:
\begin{align*}
    \cE_{m,T}(\sigma, u, \theta)\lesssim \lab Y_m(0)\big) + (T+\ep)^{\vartheta_{26}}\lab \cN_{m,T}\big),
\end{align*}
where $\vartheta_{26}=\min\{ \vartheta_{25}, \f{\tilde{\vartheta}_0}{2}\}>0.$ 
Moreover, by the virtue  of \eqref{nathetainfty}, \eqref{amtsi-u}, we conclude the following estimate for 
$\cA_{m,t}(\sigma, u, \theta):$ 
\begin{align}\label{escamt-final}
    \cA_{m,T}(\sigma, u, \theta)\lesssim \lab Y_m(0)\big)+(T+\ep)^{\vartheta_{27}} \lab \cN_{m,T}\big),
\end{align}
where $\vartheta_{27}=\min\{ \vartheta_{6},  \vartheta_{16}\}>0.$ 
By choosing $\vartheta_0=\min\{\vartheta_{26}, \vartheta_{27} \},$ the desired estimate \eqref{enerineq} then follows from \eqref{es-cemt-final} and \eqref{escamt-final}. 
\end{proof}

\section{Proof of Theorem \ref{thm1}}
%In this section, we prove the Theorem \ref{thm1}
This section is devoted to the proof of Theorem  \ref{thm1}, which relies on Proposition \ref{prop-uniform es} and
the local well-posedness result stated in below: %following result on the local well-posedness result:
\begin{thm}[Local existence for fixed $\ep,\mu,\kpa$]\label{classical-local}
Given the initial data $(\sigma_0^{\ep}, u_0^{\ep}, \theta_0^{\ep})$ satisfying the compatibility condition up to order 2 and
\begin{align*}
\pt^l (\sigma^{\ep}, u^{\ep}, \theta^{\ep})(0) \in H^{2-l}(\Omega),  \quad l=0,1,2,
\end{align*}
$$-\bar{c}
\leq \ep \sigma_0^{\ep}(x), %\, \ep (\sigma_0^{\ep}-\ep\mu{\Gamma}(\ep\sigma_0^{\ep})\div u_0^{\ep}),
\,
\theta_0^{\ep}(x)\leq 
\bar{c}, \quad \forall x\in\Omega,$$
where $\bar{c}>0$ is a fixed constant. %andwhere $\tsigma_0^{\ep}=\sigma_0^{\ep}-\ep\mu{\Gamma}(\ep\sigma_0^{\ep})\div u_0^{\ep}.$
Then %for any $\ep \in(0,1],$
there exists $T_{\mu,\kpa}^{\ep}>0,$ such that
system \eqref{NCNS-S2}, \eqref{bdryconditionofu} has a unique strong solution 
\beq\label{prop-strongsol}
(\si^{\ep},u^{\ep},\theta^{\ep})\in C^{l}\big([0, T_{\mu,\kpa}^{\ep}], H^{2-l}\big),\quad \pt^l (u^{\ep}, \theta^{\ep}) \in L^2\big([0, T_{\mu,\kpa}^{\ep}], H^{3-l}\big), \quad l=0,1,2.
\eeq
Moreover, the following property holds:
 \beqs
 -3{\bar{c}}\leq \ep\sigma^{\ep}(t,x),\, 
 %\ep (\sigma^{\ep}-\ep\mu{\Gamma}(\ep\sigma^{\ep})\div u^{\ep})(t,x), 
  \theta^{\ep}(t,x)\leq 3\bar{c},\quad  \forall (t,x)\in [0, T_{\mu,\kpa}^{\ep}]\times\Omega.
 \eeqs
\end{thm}
Such a result can be established  in a similar way (Galerkin method and fixed point arguments) as in \cite{MR1667031} or \cite{MR2914244}. %, where the isentropic Navier-Stokes system with Navier-slip boundary condition \eqref{bdryconditionofu} is shown.

\begin{proof}[Proof of Theorem \ref{thm1}]
On the one hand, by the assumption $(\sigma_0^{\ep},u_0^{\ep}, \theta_0^{\ep})\in Y_m(0) \, (\text{defined in} \eqref{initialnorm}),$ we have that:
\begin{align*}
   \ep \mu u_0^{\ep}, \sigma_0^{\ep}, \kpa \theta_0^{\ep} \in H^2(\Omega), \,\quad 
   \ep\pt (\sigma^{\ep}, u^{\ep}, \theta^{\ep})(0),\,  \in H^1(\Omega),\quad (\ep\pt)^2 (\sigma^{\ep}, u^{\ep})(0),\, (\ep\pt^2\theta^{\ep})(0) \in L^2(\Omega).
\end{align*}
Therefore, by Theorem \ref{classical-local}, one can find some $T_{\mu,\kpa}^{\ep}>0$ such that there is a unique solution of \eqref{NCNS-S2}, \eqref{bdryconditionofu}  satisfying the %property 
\eqref{prop-strongsol} and \eqref{preasption}.
On the other hand, as $(\sigma_0^{\ep},u_0^{\ep}, \theta_0^{\ep})\in Y_m,$ 
a higher regularity space, 
by propagation of regularity arguments (for example based on applying  finite difference instead of derivatives)
 in the estimates of Section 2-11, we find that
 the estimates of  Proposition \ref{prop-uniform es} hold on $[0, T_{\mu,\kpa}^\ep]$.
More specifically, we can find 
two increasing polynomials $\Lambda_0, \Lambda_1$ that are independent of $\ep$ and $T_{\mu,\kpa}^{\ep},$ and a constant $\vartheta_0>0,$
such that for any $0\leq T\leq \min\{1, T_{\mu,\kpa}^{\ep}\}, 0< \ep\leq 1, (\mu,\kpa)\in A,$
\beq\label{sec6:eq1}
 \mathcal{N}_{m,T}^{\mu,\kpa}(\si^{\ep}, u^{\ep}, \theta^{\ep})\leq \Lambda_0\big( \f{1}{c_0}, 
 Y_{m, \mu,\kpa}^{\ep}(0)\big)+
 (T+\ep)^{\vartheta_0} \Lambda_1\big(\f{1}{c_0}
,\mathcal{N}_{m,T}^{\mu,\kpa}(\sigma^{\ep},u^{\ep},\theta^{\ep})\big).
%\cN_{m,T}(\sigma^{\ep},u^{\ep}, \theta^{\ep})\leq \Lambda_0\big(\f{1}{c_0}, Y_m^{\ep}(0)\big)+(T+\ep)^{\vartheta_0}\Lambda_1\big(\f{1}{c_0}, \cN_{m,T}(\sigma^{\ep},u^{\ep}, \theta^{\ep})\big).
\eeq
Moreover, by using the characteristics method, we have  that $\ep\sigma$ can be expressed as,
\begin{align}%\label{lagranian}
\ep\sigma^{\ep}(t,x)=\ep\sigma_0^{\ep}\big(X^{-1}(t,x)\big)
-\int_0^t S_{\sigma^{\ep}}\big(X(s,X^{-1}(t,x))\big)\,\d s,\label{lagranian} \\
\theta^{\ep}(t,x)=\theta_0^{\ep}\big(X^{-1}(t,x)\big)
-\int_0^t S_{u^{\ep}} \big(X(s,X^{-1}(t,x))\big)\,\d s , \label{lagranian1}
\end{align}
where $X(t,\cdot)$ is the flow associated to $u^{\ep}$
and $$S_{\sigma^{\ep}}=\gamma\div u^{\ep}-(\gamma-1)\big(\kpa\Gamma\div(\beta\na\theta^{\ep})+\mathfrak{N} \big), \qquad S_{u^{\ep}}=\f{R}{C_v}\big(\div u^{\ep}-\kpa\Gamma\div(\beta\na\theta^{\ep})-\ep\mathfrak{N}\big).$$
From the definition of $\cA_{m,T}^{\mu,\kpa}(\sigma^{\ep},u^{\ep}, \theta^{\ep})$ in \eqref{defcAmt}, we can find a polynomial 
$\Lambda_2,$ such that as long as \eqref{preasption} and \eqref{preasption1} hold,
\begin{align}\label{souceinftysi-u}
    \|(S_{\sigma^{\ep}},  S_{u^{\ep}})\|_{L^{\infty}([0,T]\times\Omega)}\leq \Lambda_2\big( \f{1}{c_0}, \cA_{m,T}^{\mu,\kpa}(\sigma^{\ep},u^{\ep}, \theta^{\ep})\big).
\end{align}

Let us define for any $\ep\in [0,1], (\mu, \kpa) \in A,$     
\beqs
T^{\ep}_{*,\mu,\kpa}=\sup\big\{T\big| (\sigma^{\ep},u^{\ep}, \theta^{\ep})\in C^l([0,T],H^{2-l}), \, \pt^l (u^{\ep}, \theta^{\ep}) \in L^2([0,T],H^{3-l}),\, l=0,1,2\big\},
\eeqs
\beqs
\begin{aligned}
T_{0,\mu,\kpa}^{\ep}=\sup\big\{&T\leq\min\{T^{\ep}_{*,\mu,\kpa},1\}\big| \cN_{m,T}^{\mu,\kpa}(\sigma^{\ep}, u^{\ep})\leq 2\Lambda_0\big(\f{1}{c_0}, M \big), \\
&-2{\bar{c}}\leq \ep\sigma^{\ep}(t,x),\, \ep(\sigma^{\ep}-\ep\mu(2\lambda_1+\lambda_2)\Gamma(\ep\sigma^{\ep})\div u^{\ep})(t,x),\, \theta^{\ep}(t,x)\leq 2\bar{c}, \,\, \forall (t,x)\in[0,T]\times\Omega 
\big\}
\end{aligned}
\eeqs
where $M>\sup_{(\mu,\kpa)\in A, \ep\in(0,1]}Y_{m,\mu,\kpa}^{\ep}(0).$ 

 We now choose successively two  constants $0<\ep_0\leq 1$ and $0<T_0\leq 1$ (uniform in $\ep\in(0,\ep_0]$) which are small enough, such that:
 \begin{align*}
& (2\lambda_1+\lambda_2)\f{\ep_0^2}{c_0}\Lambda_0\big(\f{1}{c_0}, M \big)\leq \f{1}{4}\bar{c}, \\
 ({T_0}+\ep_0)^{\vartheta_0}\Lambda_1\big( \f{1}{c_0}, 2\Lambda_0&\big(\f{1}{c_0}, M \big)\big)<\f{1}{2}\Lambda_0\big(\f{1}{c_0}, M \big),\quad
 T_0\Lambda_2\big(\f{1}{c_0}, 2\Lambda_0\big(\f{1}{c_0}, M \big)\big)< \f{1}{2}\bar{c}.
 \end{align*}
In order to prove Theorem \ref{thm1}, it suffices  to show that $T_{0,\mu,\kpa}^{\ep}\geq {T_0}$ for every  $0<\ep\leq \ep_0, (\mu,\kpa)\in A.$ Suppose otherwise $T_{0,\mu,\kpa}^{\ep}<{T_0}$ for some $0<\ep\leq \ep_0, (\mu,\kpa)\in A,$ then in view of the inequality 
\eqref{sec6:eq1}, the formulas \eqref{lagranian}, \eqref{lagranian1}  and the estimate \eqref{souceinftysi-u},
we have by the definition of $\ep_0$ and $T_0$ that, for any $0<\ep \leq \ep_0,$
\beq\label{sec6:eq3}
\cN_{m,T}^{\mu,\kpa}(\sigma^{\ep},u^{\ep}, \theta^{\ep})\leq \f{3}{2}\Lambda_0\big(\f{1}{c_0}, M \big), \qquad \forall\, T\leq \tilde{T}=\min\{T_0, T^{\ep}_{*,\mu,\kpa}\},
\eeq
\beq\label{sec6:eq4}
-2{\bar{c}}\leq \ep\sigma^{\ep}(t,x),\, %\ep\tsigma^{\ep}(t,x),
 \ep (\sigma^{\ep}-\ep\mu(2\lambda_1+\lambda_2){\Gamma}(\ep\sigma^{\ep})\div u^{\ep})(t,x), \, \theta^{\ep}(t,x)\leq 2\bar{c},\quad \forall\, (t,x)\in [0,\tilde{T}]\times\Omega.
\eeq
We will prove that $\tilde{T}=T_0\leq T_{*,\mu,\kpa}^{\ep}.$ This fact, combined with the definition of $T_{0,\mu,\kpa}^{\ep}$ and the estimates \eqref{sec6:eq3}, \eqref{sec6:eq4},
yields $T_{0,\mu,\kpa}^{\ep}\geq T_0,$ which is a contradiction with the  assumption $T_{0,\mu,\kpa}^{\ep}<T_0.$ To continue, we shall need the claim stated and proved below.
 Indeed, once the following claim holds, we have by \eqref{sec6:eq3} that 
 $$\|\pt^l(\sigma^{\ep},u^{\ep}, \theta^{\ep})(T_0)\|_{H^{2-l}(\Omega)}<+\infty,\quad \forall \, l=0,1,2, $$
 which, combined with the local existence result stated in 
 Theorem \ref{classical-local},  yields that $T_{*,\mu,\kpa}^{\ep}>T_0.$
\end{proof}
 $\textbf{Claim.}$ 
For all $\ep\in(0,1], (\mu,\kpa)\in A,$
if $\cN_{m,T}^{\mu,\kpa}(\sigma^{\ep},u^{\ep}, \theta^{\ep})<+\infty,$ then $$(\ep\pt)^l(\sigma^{\ep},u^{\ep}, \theta^{\ep})\in C^{l}([0,T], H^{2-l}), \quad (\ep\pt)^l(u^{\ep},\theta^{\ep})\in L^2([0,T],H^3), \quad l=0,1,2.$$
\begin{proof}[Proof of the claim]
We see from the definition of $\cN_{m,T}^{\mu,\kpa}$ in \eqref{defcalNmt} and Remark \ref{rmk-thirdordervelocity} that:
%and the estimate \eqref{threeorderofu} %for any $0\leq l\leq 2,$
\begin{align*}
 & (\ep \mu u^{\ep}, \sigma^{\ep}, \kpa \theta^{\ep}) \in L^{\infty}([0,T],H^{2}), \quad  (\ep\pt, (\ep\pt)^2) ( u^{\ep}, \sigma^{\ep}, \theta^{\ep}) %\in L^{\infty}([0,T],H^{1}), \quad 
 %(\ep\pt)^2 (\sigma^{\ep}, u^{\ep}),\, \ep\pt^2\theta^{\ep} 
 \in L^{\infty}([0,T], H^1),\\
&(\ep\mu  u^{\ep}, \kpa \theta^{\ep})\in 
L^2([0,T],H^{3}), \quad  
\ep\pt((\ep\mu)^{\f{1}{2}}u^{\ep}, \kpa^{\f{1}{2}} \theta^{\ep})\in  L^2([0,T], H^{2}), \quad \ep\pt^2 (\ep u^{\ep}, \theta^{\ep})\in  L^2([0,T], H^{1}).
\end{align*}
Moreover, as $m\geq 7,$
\begin{align*}
    \pt^3(\ep^3 u^{\ep},  \ep^3\sigma^{\ep}, \ep^2\theta^{\ep}), \quad \ep\pt\na^2\sigma^{\ep}\in  L^{\infty}([0,T],L^2),
\end{align*}
 one deduces from the interpolation that $$\pt^2(\ep^3 u^{\ep}, \ep^3 \sigma^{\ep}, \ep^2\theta^{\ep})\in C([0,T], L^2), \quad \, \pt(\ep^2 u^{\ep}, \ep^2\sigma^{\ep}, \ep \theta^{\ep})\in C([0,T],H^1), \quad (\ep\mu u^{\ep}, \ep\sigma^{\ep}, \ep\kpa\theta^{\ep})\in C([0,T],H^2).$$
 This ends  the proof of the claim.
Note that at this stage we do not require the norms listed in the last line to be bounded uniformly in $\ep, \mu,\kpa$.
\end{proof}

\section{Proof of Theorem \ref{thm-conv1}}
In this section, we sketch the proof of  the convergence result stated in Theorem \ref{thm-conv1}, which is essentially adapted from  \cite{MR2106119,MR1834114}. 

Before going into details, we first verify, 
based on the uniform estimates for $\theta^{\ep},$  that:
\begin{align}\label{conv-theta}
  \theta^{\ep}\rightarrow \theta^0\quad  \text{ in } C([0,T_1], H_{loc}^1(\Omega))\cap C([0,T_1]\times\Omega).
\end{align}
Indeed, it follows from the uniform estimate \eqref{unies-fixmukap} that 
\begin{align*}
\|\theta^{\ep}\|_{L_{T_1}^{\infty}H_{co}^{m-1}}+ \|(\na\theta^{\ep}, \kpa \na^2\theta^{\ep}, \pt\theta^{\ep})\|_{L_{T_1}^{\infty}H_{co}^{m-2}}+
\|\pt\na\theta^{\ep}\|_{L_{T_1}^2H_{co}^{m-3}}<+\infty,
 % \|\theta^{\ep}\|_{L_{T_1}^{\infty}H^1}+ \| \kpa \na^2\theta^{\ep}\|_{L_{T_1}^{\infty}L^2}+\|\pt\theta^{\ep}\|_{L_{T_1}^2H^1}<+\infty
\end{align*}
which stems the convergence (up to extraction of the subsequences) of $(\theta^{\ep}, \na\theta^{\ep})$ to $(\theta^0, \na\theta^0)$ in 
$C([0,T_1], H_{co, loc}^{m-3}(\Omega)).$ 
By using further the Sobolev embedding \eqref{sobebd}, 
we have also that $\il\theta^{\ep}-\theta^0\il_{m-5,\infty,T_1}\rightarrow 0$ and thus finish the proof of \eqref{conv-theta}
by remembering the assumption $m\geq 7.$ %We comment that we have not used so far the result in Proposition \ref{prop-conv}.

The delicate part concerning the proof of Theorem \ref{thm-conv1} 
is the first half which is re-summarize in the following proposition for convenience:  
\begin{prop}\label{prop-conv}
 Let $(\sigma^{\ep}, u^{\ep}, \theta^{\ep})\in C([0,T_1], L^2)$ be the solution to system \eqref{NCNS-S2}, \eqref{bdryconditionofu} satisfying  the uniform estimates \eqref{unies-fixmukap}. Under the assumptions made in Theorem  \ref{thm-conv1},  the following strong convergence for the penalized terms holds:
\beq\label{conv-compressible}
\sigma^{\ep}\rightarrow 0, \,\quad \div u^{\ep}-\f{\gamma-1}{\gamma}\kappa\Gamma(\ep\sigma^{\ep})\div(\beta(\theta^{\ep})\na\theta^{\ep})\rightarrow 0 \qquad \text{ in } \, L^2([0,T_1], L_{loc}^2(\Omega)).
\eeq
\end{prop}

Let us first assume that the above proposition holds true and verify that the limit $(u^0, \theta^0)$ of the sequences $(u^{\ep}, \theta^{\ep})$ solves the equations \eqref{NINS-1}.  To proceed, we need the following claim: %is very useful:
%Let us begin with the following claim:

\textbf{Claim}: Let $(\sigma^{\ep}, u^{\ep}, \theta^{\ep})$ be the solution to system \eqref{NCNS-S2}, \eqref{bdryconditionofu} satisfying  the uniform estimates \eqref{unies-fixmukap}. The following convergence properties hold:
up to extraction of the subsequences, 
\begin{align}
    u^{\ep}\rightarrow u^0 \text{ in } L^2([0,T_1], L_{loc}^2(\Omega)), \quad \mathbb{Q}u^0=\f{\gamma-1}{\gamma}\kappa\overline{\Gamma}\beta(\theta^0)\na\theta^0, \label{conv-u}
\end{align}
where $\mathbb{Q}$ is the projection defined in \eqref{proj-Q} and $\overline{\Gamma}=1/\overline{P}.$ Note that by the definition of $\beta$ in \eqref{defgamma-beta}, the term $\beta(\theta^0)\na\theta^0=\na (\beta(\theta^0)-\beta(0)).$

Let us postpone the proof of this claim and justify the convergence of the system \eqref{NCNS-S2} to \eqref{NINS-1}. Thanks to the facts \eqref{conv-theta}, \eqref{conv-u},
 the uniform estimates \eqref{unies-fixmukap}, 
it is easy to verify from the equation of 
$\theta^{\ep}$ in $\eqref{NCNS-S2}_3$ that 
\begin{align*}
    \f{C_v}{R} (\pt+u^0\cdot\nabla)\theta^0%+\div u^0
    -\f{1}{\gamma}\kappa \overline{\Gamma}\div(\beta(\theta^0)\na\theta^0)=0
\end{align*}
holds in the sense of distribution. However, as every term in the above equation belongs to $C([0,T_1], L^2(\Omega)),$ it also holds in the strong sense. It now remains to find $\na\pi^0\in L^2( [0,T_1], L^2(\Omega))$ 
and show that $(u^0, \na\pi^0)$ solves $\eqref{NINS-1}_1,$ \eqref{bc-limit} in the sense of \eqref{incom-EI}. %It turns out to be convenient to 
To start, it is convenient to look at the equations satisfied by $w^{\ep}=\colon u^{\ep}-\f{\gamma-1}{\gamma}\kappa \overline{ \Gamma} \beta(\theta^{\ep})\nabla\theta^{\ep},$ which, in light of $\eqref{sys-uni}_2,$ takes the form:
\begin{align*}
    \f{1}{R\beta(\theta^{\ep})}(\pt+u^{\ep}\cdot\nabla)w^{\ep} +\nabla\pi^{\ep}-\mu\overline{\Gamma}\div\cL u^{\ep}-F(\theta^{\ep}, u^{\ep})=\ep \tilde{F}_{w}^{\ep}  
\end{align*}
where $$\pi^{\ep}=\sigma^{\ep}/\ep- \f{(\gamma-1)}{C_v\gamma}\overline{\Gamma}\kpa\big(-\div w^{\ep}+\f{\kpa}{\gamma}\overline{\Gamma} \div(\beta(\theta^{\ep})\nabla\theta^{\ep})\big),$$
$$F(\theta^{\ep}, u^{\ep})=\kpa\overline{\Gamma}\f{\gamma-1}{R\gamma}(\na u^{\ep}\cdot \na \theta^{\ep}-\na\theta^{\ep}(\pt+u^{\ep}\cdot\na)\theta^{\ep}), \quad \sup_{\ep\in(0,\ep_1]}\ep^{\f{1}{2}}\|\tilde{F}_{w}^{\ep}\|_{L^2( [0,T_1], L^2(\Omega))}<+\infty.$$
With the help of  \eqref{conv-theta}, \eqref{conv-u},
and the uniform estimates \eqref{unies-fixmukap}, we deduce by testing the divergence-free, compactly supported test function that
\begin{align}\label{bpw}
    \mathbb{P}\bigg( \f{1}{R\beta(\theta^{0})}(\pt+u^{0}\cdot\nabla)w^{0} -F(\theta^{0}, u^{0}) +\mu\lambda_1\overline{\Gamma}\curl\curl u^{0}\bigg)=0
\end{align}
holds in the sense of distribution. 
That is, for any  $\psi\in C^{\infty}([0,T_1], C_c^{\infty}(\ \overline{\Omega}))$ with
 $\div \psi=0,\,\psi\cdot\bn|_{\p\Omega}=0,$ the following identity holds: for every  $0<t\leq T_1,$
 \beqs \begin{aligned} &\f{1}{R}\int_{\Omega}\big(\f{w^0}{\beta(\theta^0)}\cdot\psi\big)(t,\cdot) \ \d x+\mu\lambda_1\overline{\Gamma}\izt\int_{\Omega}\curl w^0\cdot\curl \psi \ \d x\d s\\&=\f{1}{R}\int_{\Omega}\big(\f{w^0}{\beta(\theta^0)}\cdot\psi\big)(0,\cdot) \ \d x+\izt\int_{\Omega}\bigg(F(\theta^0, u^0)-\f{u^0}{R\beta(\theta^0)}\cdot\nabla w^0\bigg)\cdot\psi  \ \d x\d s\\&\quad+\f{1}{R}\izt\int_{\Omega}\f{w^0}{\beta(\theta^0)}\cdot\pt\psi+ \pt(\f{1}{\beta(\theta^0)})w^0\cdot\psi \ \d x\d s+\mu\lambda_1\overline{\Gamma}\int_0^t\int_{\p\Omega}2\Pi\big(-a u^0+(D\bn) u^0\big)\cdot\psi \ \d S_y \d s.\end{aligned} \eeqs
The identity \eqref{bpw} lead us to find
a $\na \tilde{\pi}^0\in L^2( [0,T_1], L^2(\Omega))$ by solving the elliptic problem:
 \begin{align}\label{def-pi0}
     %\left\{ \begin{array}{l}
          \Delta \tilde{\pi} ^0 =\div H^0 \text{ in } \Omega\qquad \qquad
          \p_{\bn} \tilde{\pi} ^0|_{\p\Omega}=H^0\cdot\bn-\mu\overline{\Gamma}\lambda_1\curl\curl u^0\cdot\bn, \quad  \tilde{\pi} ^0 \xlongrightarrow{|x|\rightarrow+\infty} 0
    % \end{array}
    % \right.
 \end{align}
 with 
 $H^0=-\f{1}{R\beta(\theta^{0})}(\pt+u^{0}\cdot\nabla)w^{0} +F(\theta^{0}, u^{0}).$ %-\mu\overline{\Gamma}(2\lambda_1+\lambda_2)\na\div u^{0}.$ 
 Indeed, as $$\div (\f{1}{R\beta(\theta^{0})} \pt w^0)=\na\big(\f{1}{R\beta(\theta^{0})}\big)\pt w^0, $$ %\qquad \mu\Delta\div u^0\in  L^2([0,T_1], H^{-1}(\Omega)),$$
 it holds that $\div H^0\in L^2([0,T_1], H^{-1}(\Omega)).$ Moreover, it follows from \eqref{curlcurludotn}, \eqref{omegatimesn}, \eqref{conv-u} that: 
 \begin{align*}
     \mu\curl\curl u^0\cdot\bn \in L^2\big([0, T_1], H^{-\f{1}{2}}(\p\Omega)\big).
 \end{align*}

Next, direct computations show that:
\begin{align*}
    \f{1}{R\beta(\theta^{0})}(\pt+u^{0}\cdot\nabla)w^{0} -F(\theta^{0}, u^{0})= \f{1}{R\beta(\theta^{0})}(\pt+u^{0}\cdot\nabla)u^{0} -\kpa\overline{\Gamma}\f{\gamma-1}{R\gamma}\na (\pt+u^0\cdot\na)\theta^0,
\end{align*}
%which inspires us to define %,
which, together with \eqref{bpw}, implies that the following  holds in the sense of distribution:
\begin{align*}%\label{bpu}
    \mathbb{P}\bigg( \f{1}{R\beta(\theta^{0})}(\pt+u^{0}\cdot\nabla)u^{0}  -\mu\overline{\Gamma}\div\cL u^{0}\bigg)=0
\end{align*}

 Define $\na\pi^0=\na\big(\tilde{\pi}^0-\kpa\overline{\Gamma}\f{\gamma-1}{R\gamma} (\pt+u^0\cdot\na)\theta^0\big)\in  L^2( [0,T_1], L^2(\Omega)),$ we thus find that $(u^0, \na\pi^0)$ solves 
 $\eqref{NINS-1}_2$ in the sense of \eqref{incom-EI}.  Finally, it results from the uniform bounds
 \begin{align*}
   \sup_{\ep\in (0,\ep_1] } \big( \| ( u^{\ep} ,\theta^{\ep})\|_{  L_{T_1}^{\infty}\cH^{0,m-1}}+\|(\nabla  u^0, \nabla \theta^0, \pt\theta^0)\|_{
  L_{T_1}^{\infty}\cH^{0,m-2}
\cap L^{\infty}([0,T_1]\times\Omega)}\big)<+\infty
 \end{align*}
 that $(u^0, \theta^0)$ satisfy the additional regularity \eqref{addition-reg}. The uniqueness of the solution to the limit system stems from the boundedness of the
Lipschitz norm, which in turn, means that the convergence of $(u^{\ep},\theta^{\ep})$ holds for the whole sequence.

%\eqref{bdryconditionofu}
%we recall the deBy the equation $\eqref{newsys}_2$

\begin{proof}[Proof of the claim:]
Let us prove \eqref{conv-u}. Recall that $w^{\ep}=u^{\ep}-\f{\gamma-1}{\gamma}\kappa \overline{ \Gamma} \beta(\theta^{\ep})\nabla\theta^{\ep}.$ The property 
\eqref{conv-compressible} together with
uniform boundedness of $\div\big(\beta(\theta^{\ep})\nabla\theta^{\ep}\big)$
in $L^2([0,T_1], L^2(\Omega))$ imply that $\div w^{\ep}\rightarrow 0$ in $L^2([0,T_1], L_{loc}^2(\Omega)),$ which combined with the weak convergence of $\mathbb{Q}w^{\ep},$ yields %we see that
%enables us to conclude
$\mathbb{Q} w^{\ep}\rightarrow 0$ in $L^2([0,T_1], L_{loc}^2(\Omega)).$ Therefore, in view of \eqref{conv-theta},
we find that 
$\mathbb{Q} u^{\ep}$ converges strongly in $L^2([0,T_1], L_{loc}^2(\Omega))$ to $\f{\gamma-1}{\gamma}\kappa\overline{\Gamma}\beta(\theta^0)\na\theta^0.$

It thus remains to show $\mathbb{P}u^{\ep}\rightarrow \mathbb{P} u^0 \text{ in } L^2([0,T_1], L_{loc}^2(\Omega)).$ As $\theta^{\ep}, u^{\ep}$ are uniformly bounded in $(L^2\cap L^{\infty})([0,T_1]\times\Omega),$ we easily see that
%$\f{1}{R\beta(\theta^{\ep})}u^{\ep}$
\begin{align}\label{bpu-weak}
   \f{ u^{\ep}}{\beta(\theta^{\ep})} \rightharpoonup
    \f{u^{0}}{\beta(\theta^{0})} \Rightarrow 
    \mathbb{P}\big( \f{ u^{\ep}}{\beta(\theta^{\ep})}\big) \rightharpoonup
     \mathbb{P}\big( \f{u^{0}}{\beta(\theta^{0})}\big)
    \quad \text{ in }
 L^2([0,T_1]\times\Omega).
\end{align}
On the other hand, it follows from the identity %in light of the identity %by the virtue of the fact:
\begin{align*}
  \pt \mathbb{P} \big(\f{ u^{\ep}}{\beta(\theta^{\ep})}\big)= \mathbb{P}\bigg(u^{\ep}\pt \big(\f{1}{\beta(\theta^{\ep})}\big)- \f{ u^{\ep}}{\beta(\theta^{\ep})}\cdot\na u^{\ep}+\mu\Gamma\div\cL u^{\ep}\bigg)
\end{align*}
that $\pt \mathbb{P} \big(\f{ u^{\ep}}{\beta(\theta^{\ep})}\big)$ is uniformly bounded in $L^{\infty}([0,T_1], H^{-1}(\Omega)),$ which, combined with the uniform boundedness of 
$\f{ u^{\ep}}{\beta(\theta^{\ep})}$ in  $L^{\infty}([0,T_1], H^{1}(\Omega))$ and the Aubin-Lions Lemma that $\mathbb{P} \big(\f{ u^{\ep}}{\beta(\theta^{\ep})}\big)$ converges in 
$C([0,T_1], L_{loc}^2(\Omega)).$ We thus find by noticing \eqref{bpu-weak} that:
\begin{align*}
  \mathbb{P} \big(\f{ u^{\ep}}{\beta(\theta^{\ep})}\big) \rightarrow \mathbb{P}\big( \f{u^{0}}{\beta(\theta^{0})}\big) \text{ in } C([0,T_1], L_{loc}^2(\Omega)).
\end{align*}
Next, we derive  from 
the convergence results \eqref{conv-theta} \eqref{conv-u}  and the uniform estimates in $L^{\infty}([0,T_1]\times\Omega)$
%uniform estimates
for $(\theta^{\ep}, u^{\ep})$ that:
\begin{align*}
  \mathbb{P} \big(\f{  \mathbb{Q}u^{\ep}}{\beta(\theta^{\ep})}\big) \rightarrow \mathbb{P}\big( \f{\mathbb{Q}u^{0}}{\beta(\theta^{0})}\big) \text{ in } L^2([0,T_1], L_{loc}^2(\Omega)).
\end{align*}
The previous two properties yield further that:
\begin{align*}
  \mathbb{P} \big(\f{  \mathbb{P}u^{\ep}}{\beta(\theta^{\ep})}\big) \rightarrow \mathbb{P}\big( \f{\mathbb{P}u^{0}}{\beta(\theta^{0})}\big) \text{ in } L^2([0,T_1], L_{loc}^2(\Omega)),
\end{align*}
which, combined with the fact 
$\f{1}{\beta(\theta^{\ep})}\rightarrow \f{1}{\beta(\theta^{0})} $ in $ L^2([0,T_1], L_{loc}^2(\Omega)),$ imply that:
\begin{align*}
    \mathbb{P}\bigg( \f{\mathbb{P}(u^{\ep}-u^{0})}{\beta(\theta^{0})}\bigg)\rightarrow 0 \text{ in } L^2\big([0,T_1], L_{loc}^2(\Omega)\big).
\end{align*}
Consequently, 
we find that, by noticing $\beta(\theta^0)(t,x)\geq c_0$ for all $(t,x)\in[0,T_1]\times\Omega,$ $\mathbb{P}(u^{\ep}) $ strongly converges to $\mathbb{P}(u^{0}) $ in $ L^2\big([0,T_1], L_{loc}^2(\Omega)\big).$ The proof of \eqref{conv-u} is now complete. 
%$\mathbb{P}(u^{\ep})\rightarrow $
%Consequently, by using the f
%It follows from the uniform estimates 
%By the uniform boundedness of 
%Let $r_0^{\ep}$ be defined in \eqref{defr0}. 
%It follows from the uniform boundedeness of $(\theta_0)$
%In light of the %convergence result \eqref{conv-theta} for $\theta^{\ep},$
\end{proof}

To finish the proof of Theorem \ref{thm-conv1}, it now remains to show Proposition \ref{prop-conv} 
which is based on the following Theorem:
\begin{thm}
Assume that 
\beq\label{asmption-theta}
\begin{aligned}
   & \theta^{\ep}, \pt\na\theta^{\ep} \text{ are uniformly bounded respectively  in } C^1([0,T_1]\times\Omega), L^2([0,T_1]\times\Omega) \\
    & \theta^{\ep}\rightarrow \theta^0\quad  \text{ in } C([0,T_1], H_{loc}^1(\Omega))\cap C([0,T_1]\times\Omega),
\end{aligned}
\eeq
where $\theta^0$ enjoys the decay property \eqref{decaytheta0}. Let $q^{\ep}$ be a bounded sequence in $C([0,T_1], H^2(\Omega))$ satisfying $\sup_{\ep\in (0,\ep_1]}\|\na^2 q^{\ep}\|_{H_{co}^1}<+\infty$ and
solving the equations 
\begin{align}\label{def-acousticeq}
\f{1}{\gamma}(\ep\pt)^2 q^{\ep}-\div\big(R\beta(\theta^{\ep})\na q^{\ep}\big)= G_1^{\ep}+ \div\big(R\beta(\theta^{\ep})G_2^{\ep}\big)=\colon G^{\ep}, \quad \p_{\bn}q^{\ep}=g^{\ep},
\end{align}
where 
\beq\label{asption-source}
%\ep G_1^{\ep}, \ep G_2^{\ep}, \ep\div G_2^{\ep}
G^{\ep}\rightarrow 0 \text{ in } L^2([0,T_1]\times\Omega),
\qquad g^{\ep}\rightarrow 0 \text{ in } \, L^2([0,T_1], H^{\f{1}{2}}(\p\Omega)),
\eeq
then it holds that:  
$$q^{\ep}\rightarrow 0 \text{ in } \, L^2([0,T_1], L_{loc}^2(\Omega)).$$
\end{thm}
\begin{proof}
This theorem has essentially been proved in Proposition 3.1 of \cite{MR2106119}. %Nevertheless, th
For the reader's convenience, we outline the proof detailed in \cite{MR2106119} and more importantly, make some comments  on the assumptions imposed in \eqref{asmption-theta}, \eqref{asption-source} on 
$\theta^{\ep}$ and $G^{\ep}, g^{\ep}$ as they are slightly different from those in 
[Prop 3.1, \cite{MR2106119}]. 
%clarify 
To get the strong compactness in time, the idea is to construct the defect measures associated to the sequence $q^{\ep}$ and then prove it vanishes. More precisely, one first extends the functions to $t\in \mR,$ (without changing the notation) %(denote it $\tilde{q}^{\ep}, \tilde{G}^{\ep}$) 
and then uses the wave packet transform: 
\begin{align*}
    W^{\ep}f(t,\tau, x)=(2\pi^3)^{-\f{1}{4}}\ep^{-\f{3}{4}}\int_{\mR} e^{\big(i(t-s)\tau-(t-s)^2\big)/\ep} f(s,x)\,\d s
\end{align*}
to write the equation \eqref{def-acousticeq} in $\mR^2\times\Omega:$
\begin{align*}
    P^{\ep}(t,\tau,\na)(W^{\ep}q^{\ep})=W^{\ep} G^{\ep}+\div\big([W^{\ep}, R\beta(\theta^{\ep})]\na q^{\ep} \big)-\f{1}{\gamma}\big(W^{\ep}\big((\ep\pt)^2 q^{\ep}\big)+\tau^2W^{\ep}q^{\ep}\big)
\end{align*}
where $ P^{\ep}(t,\tau,\na)=\colon- \f{1}{\gamma}\tau^2\cdot %W^{\ep}q^{\ep}
    -\div(R\beta(\theta^{\ep})\na\cdot).$
Next, by defining $\Theta^{\ep}=(\Id-\Delta)W^{\ep}q^{\ep}\in L^2(\mR_{t,\tau}^2\times\Omega) $
following the similar arguments in [Lemma 4.3, \cite{MR1834114}], one shows the existence of the defect measure $\mu$ in $\mR^2$ associated to the sequence $\Theta^{\ep}$ and the existence of $M\in L^1(\mR^2, \mathcal{L}_{+}^1,\mu),$ such that, up to subsequence:
for any $A\in C_0(\mR^2, \mathcal{K}),$
\begin{align}\label{conv-measure}
    \int_{\mR^2}(A(t,\tau)\Theta^{\ep}, \Theta^{\ep})_{L^2(\Omega)}\,\d t\d \tau\xlongrightarrow{\ep\rightarrow 0}  \int_{\mR^2}\text{tr} \big(A(t,\tau) M(t,\tau)\big) \mu(\d t\d\tau)
\end{align}
Note that in the above $ \mathcal{K}, \mathcal{L}_{+}^1$ denote the spaces of compact operators and non-negative trace class operators in $L^2(\Omega).$ 
The remaining task is to show the following two facts:
\begin{align}
   & P^{0}(t,\tau,\na)(1-\Delta_{N})^{-1}M(t,\tau)=0 \quad \mu-a.e, \label{supp-measure}\\
    &\text{Ker}_{L^2(\Omega)} \big(P^{0}(t,\tau,\na)(1-\Delta_{N})^{-1}\big)={0},\,\, \forall (t,\tau)\in\mR^2 \label{zero-kernel}
\end{align}
\text{ where}   $P^{0}(t,\tau,\na)=\f{1}{\gamma}\tau^2\cdot 
    -\div(R\beta(\theta^{0})\na\cdot)$ and $(1-\Delta_{N})^{-1}$ is defined as:
    \begin{align*}
        u=(1-\Delta_{N})^{-1}f \Leftrightarrow (1-\Delta) u=f, \quad \p_{\bn} u=0.
    \end{align*}
    Under the assumption \eqref{decaytheta0} on $\theta^0,$ \eqref{zero-kernel} has been shown more or less in [Lemma 5.1, \cite{MR1834114}]. The matter is thus reduced to the justification of \eqref{supp-measure} where the assumptions \eqref{asmption-theta}, \eqref{asption-source} on 
$\theta^{\ep}$ and $G^{\ep}, g^{\ep}$ are used. Thanks to \eqref{conv-measure}, it suffices to show that for any $\varphi\in C_0(\mR^2), K\in \mathcal{K},$ 
\begin{align*}
    \varphi K P^0(t,\tau,\na)(1-\Delta_{N})^{-1}\Theta^{\ep}\rightarrow 0 \text{ in } L^2(\mR^2\times\Omega),
\end{align*}
which is in turn the consequence of the following:
\begin{align}
     \varphi K (P^{\ep}-P^0)(t,\tau,\na)(1-\Delta_N)^{-1}\Theta^{\ep}\rightarrow 0  \text{ in } L^2(\mR^2\times\Omega), \label{toprove-1}\\
     \quad \varphi K P^{\ep} (t,\tau,\na)\big(W^{\ep}q^{\ep}-(1-\Delta_{N})^{-1}N(W^{\ep} g^{\ep})\big)\rightarrow 0 \text{ in } L^2(\mR^2\times\Omega), \label{toprove-2}
\end{align}
where $N(h)$ is defined as:
\begin{align*}
    v=N(h) \Leftrightarrow (1-\Delta) u=0, \quad \p_{\bn} u=h.
\end{align*}
By the definition of $P^{\ep}(t,\tau,\na)$ and $P^{0}(t,\tau,\na),$ we have that:
\begin{align*}
   & \|\varphi K (P^{\ep}-P^0)(t,\tau,\na)(1-\Delta_N)^{-1}\Theta^{\ep}\|_{L^2(\mR^2\times\Omega)}\\
    &\lesssim \|(\theta^{\ep}-\theta^0)\|_{L^{\infty}\big([0,T_1], \, H^1_{loc}\cap L^{\infty}(\Omega)\big)} \|(1-\Delta_N)^{-1}\Theta^{\ep}\|_{L^2\big(\mR^2,\,  W^{1,\infty}\cap H^2(\Omega)\big)},%\rightarrow 0
\end{align*}
which tends to $0$.  Therefore, by noticing \eqref{asmption-theta} and 
the fact $W^{\ep}$ commutes with the spatial derivatives, 
and the properties
\begin{align*}
   \| W^{\ep} f\|_{L^2(\mR^2, L^p(\Omega))}\lesssim \|f\|_{L^2([0,T_1], L^p(\Omega))}, \,
   ( p=2,+\infty),\quad
   \|(1-\Delta_{N})^{-1}f\|_{W^{1,\infty}\cap H^2(\Omega)}\lesssim \|f\|_{H_{co}^1(\Omega)},
\end{align*}
one finds \eqref{toprove-1}.
Next, taking the benefits of the properties on $\theta^{\ep}, q^{\ep}$ %\eqref{asmption-theta},
we can follow the similar arguments as in [Lemma 3.3, \cite{MR2106119}] that: 
\begin{align*}
\varphi K \div\big([W^{\ep}, R\beta(\theta^{\ep})]\na q^{\ep} \big),\quad  \varphi K \big(W^{\ep}\big((\ep\pt)^2 q^{\ep}\big)+\tau^2W^{\ep}q^{\ep}\big) \longrightarrow 0 \text{ in } L^2(\mR^2\times\Omega),
\end{align*}
which, combined with the assumption \eqref{asption-source}, leads to \eqref{toprove-2}.
Note that in \cite{MR2106119}, it is assumed that 
 $\theta^{\ep}\rightarrow \theta^0$ in $C([0,T_1], H^s), s>1+{d}/{2}$ and $\|G_2^{\ep}\|_{L^2([0,T_1]\times H^1(\Omega))}\rightarrow 0.$ But it turns out from the above computations that \eqref{asmption-theta}, \eqref{asption-source} are enough to obtain the desired convergence result.
\end{proof}
\begin{proof}[Proof of Proposition \ref{prop-conv}]
Let $\tilde{\sigma}^{\ep}=\sigma^{\ep}-\ep\mu(2\lambda_1+\lambda_2)\Gamma(\ep\sigma^{\ep})\div u^{\ep},$ we derive from the equations \eqref{NCNS-S2} that: 
\begin{align*}
    \f{1}{\gamma}\ep\pt\tilde{\sigma}^{\ep}+\div u^{\ep}=F_1^{\ep},
    \quad \f{1}{R\beta(\theta^{\ep})}\ep\pt u^{\ep}+\na \tilde{\sigma}^{\ep}= F_2^{\ep},
\end{align*}
where 
\begin{align*}
F_1^{\ep}&=\f{\gamma-1}{\gamma}\kappa \Gamma(\ep\sigma^{\ep})\div(\beta(\theta^{\ep})\nabla\theta^{\ep})+\f{\gamma-1}{\gamma}\mu\ep^2\Gamma(\ep\sigma^{\ep})\cL u^{\ep}\cdot \mathbb{S}u^{\ep}\\
&\quad -\f{1}{\gamma}\big(\ep u^{\ep}\cdot\na\sigma^{\ep}+\ep\mu(2\lambda_1+\lambda_2)\ep\pt(\Gamma\div u^{\ep})\big),\\
F_2^{\ep}&=-\f{\ep}{R\beta (\theta^{\ep})}u^{\ep}\cdot \na u^{\ep}-\ep\mu \big(\lambda_1\Gamma \curl\curl u^{\ep}+\na\Gamma\div u^{\ep} \big).
\end{align*}
Consequently, we find that $\tsigma^{\ep}$ solves \eqref{def-acousticeq} with $G_1^{\ep}=\ep\pt F_1^{\ep}, G_2^{\ep}=F_2^{\ep}, g^{\ep}=\p_{\bn}\sigma^{\ep}|_{\p\Omega}.$
%F_2^{\ep}\cdot\bn.$
%From the uniform estimates
By the virtue of Remark \ref{rmkptna2theta}, we see that by assuming \eqref{additionasp-theta}, $(\ep\kpa)^{\f{1}{2}}\pt\na^2\theta^{\ep}$ are uniformly bounded in $L_{T_1}^2H_{co}^1,$ which, combined with the uniform estimate \eqref{unies-fixmukap}, yields that:
\begin{align*}
    \|(\ep\pt F_1^{\ep}, F_2^{\ep}, \div F_2^{\ep} )\|_{L_{T_1}^2H_{co}^1}\xlongrightarrow{\ep\rightarrow 0} 0.
\end{align*}
In light of the uniform boundedness of 
$\il(\theta^{\ep},\na\theta^{\ep})\il_{1,\infty,T_1},$ we find that, 
\begin{align*}
    G^{\ep}=\colon \ep\pt F_1^{\ep}+\div\big(R\beta(\theta^{\ep})F_2^{\ep}\big)\xlongrightarrow{\ep\rightarrow 0} 0 \ \text{ in } L_{T_1}^2H_{co}^1.
\end{align*}
For the boundary term, it has been justified in \eqref{Es-tsigma3} that:
\begin{align*}
(\Id,\ep\pt)\p_{\bn}\sigma^{\ep}|_{\p\Omega}\xlongrightarrow{\ep\rightarrow 0}  0  \, \text{ in } L^2([0,T_1]\times\p\Omega).
    %|(\Id,\ep\pt)\p_{\bn}\sigma^{\ep}|_{L^2([0,T_1]\times\p\Omega)} \lesssim \ep \lab \cN_{m,T_1}\big)\xlongrightarrow{\ep\rightarrow 0}  0.
\end{align*}
Furthermore,  on can verify by using the uniform estimate \eqref{unies-fixmukap} that:
\begin{align*}
   \ep \div \big(\pt(R\beta(\theta^{\ep}))\na\tsigma^{\ep}\big)\xlongrightarrow{\ep\rightarrow 0} 0 \ \text{ in } L^2([0,T_1]\times\Omega).
\end{align*}
The previous four estimates enable us to apply Theorem \ref{thm-conv1} to obtain that 
\begin{align*}
    \tsigma^{\ep}, \ep\pt\tsigma^{\ep}\xlongrightarrow{\ep\rightarrow 0} 0 \ \text{ in } L^2([0,T_1], L^2_{loc}(\Omega)),
\end{align*}
 which implies readily \eqref{conv-compressible}.
 
%\eqref{\label{sec4:eq12}} and 
\end{proof}
\section{Appendix}
In this appendix, we gather some useful estimates/facts and prove some technical lemmas. %which are frequently used in the main texts. 
We begin with the product  and commutator estimates 
  which are used throughout the paper:
\begin{prop}\label{prop-prdcom}
 For each $0\leq t\leq T,$ and for any integer $k\geq 2,$ one has the (rough) product estimates: 
   \beq\label{roughproduct1}
 \|(fg)(t)\|_{H_{co}^k}\lesssim \|f(t)\|_{H_{co}^k}\il g \il_{[\f{k}{2}]-1,\infty,t}+\|g(t)\|_{H_{co}^k}\il f \il_{[\f{k+1}{2}],\infty,t},
 \eeq
    \beq\label{roughproduct2}
 \|(fg)(t)\|_{H_{co}^k}\lesssim \|(f,g)(t)\|_{H_{co}^k}\il (f,g) \il_{[\f{k}{2}]-1,\infty,t},
 \eeq
 and commutator estimates:
 \beq\label{roughcom}
 \|[Z^{I},f]g(t)\|_{L^2}\lesssim 
 \|Z f(t)\|_{H_{co}^{k-1}}
 \il g\il_{[\f{k}{2}]-1,\infty,t}+\|g(t)\|_{H_{co}^{k-1}}\il Z f\il_{[\f{k-1}{2}],\infty,t}, \quad\, |I|=k,
 \eeq
 or 
  \beq\label{roughcom1}
 \|[Z^{I},f]g(t)\|_{L^2}\lesssim 
 \|Z f(t)\|_{H_{co}^{k-1}}
 \il g\il_{[\f{k}{2}],\infty,t}+\|g(t)\|_{H_{co}^{k-1}}\il Z f\il_{[\f{k}{2}]-1,\infty,t}, \quad\, |I|=k.
 \eeq
% \|[(\ep\pt)^k,f]g(t)\|_{L^2}\lesssim\| (\ep\pt f)(t)\|_{\cH^{k-1}} \il g\il_{[\f{k}{2}]-1,\infty,t}+\|g(t)\|_{\cH^{k-1}}\il \ep\pt f\il_{[\f{k-1}{2}],\infty,t}. \eeq
Moreover, on the boundary the following estimate  also holds
\begin{align}\label{product-bd}
  |(fg)(t)|_{\tilde{H}^k}\lesssim  
  |(f,g)(t)|_{\tilde{H}^k} \il(f,g)\il_{[\f{k}{2}],\infty,t}.
\end{align}
 \end{prop}
 \begin{proof}
 This proposition follows from  simply  counting the derivatives hitting on $f$ or $g.$ For instance, to prove the product estimate \eqref{roughproduct1} and the commutator estimate \eqref{roughcom},
 one can use the following expansion:
 %For $|I|=k,$ one has:
 \beq\label{id-product-com}
 \begin{aligned}
 Z^I (fg)&%=\big(\sum_{|J|\leq [(k-1)/2]}+\sum_{|I-J|\leq [k/2]}\big)(C_{I,J}Z^J g Z^{I-J}f)\\
 =\big(\sum_{1\leq |J|\leq [k/2]-1}+\sum_{1\leq |I-J|\leq [(k+1)/2]}\big)(C_{I,J}Z^J g Z^{I-J}f)+f Z^I g+g Z^I f , \quad |I|=k.
 \end{aligned}
 \eeq
 We can indeed have the more precise product estimate: for $j+l\leq k,$
 \beq\label{roughproduct-2}
  \|(f g)(t)\|_{\cH^{j,l}}\lesssim \|g(t)\|_{\cH^{j,l}}\il f \il_{0, \infty,t} 
  +%(\il g(t)\il_{\cH^{j-1,l}}+\il g(t)\il_{\cH^{j,l-1}}) 
  \|g(t)\|_{H_{co}^{m-1}}\il  Z f\il_{[\f{k-1}{2}],\infty,t}
+\|f(t)\|_{\cH^{j,l}} \il  g\il_{[\f{k}{2}]-1,\infty,t}.
  \eeq
 \end{proof}

 As a corollary of Proposition \ref{prop-prdcom},
 the following composition estimates hold:
 \begin{cor}
 Suppose that $h\in C^0([0,T]\times\Omega)\cap L_t^pH_{co}^m \,(p=2,+\infty)$ with 
 $$A_1\leq  h(t,x)\leq A_2, \quad \forall(t,x)\in [0,T]\times\Omega.$$
 Let
 $F(\cdot):[A_1,A_2]\rightarrow \mathbb{R}$ be a smooth function satisfying 
 $$\sup_{s\in [A_1,A_2]}|F^{(m)}|(s)\leq B.$$
 Then we have the composition estimate:
 \beqs
 \|F(h(\cdot,\cdot))-F(0)\|_{L_t^pH_{co}^m}\leq  \Lambda(B,\il h\il_{[\f{m}{2}],\infty,t})\|h\|_{L_t^pH_{co}^m},
 \eeqs
 where $\Lambda(B,\il h\il_{[\f{m}{2}],\infty,t})$ is a polynomial with respect to
 $B$ and $\il h\il_{[\f{m}{2}],\infty,t}.$ 
 \end{cor}

 This Corollary, combined with Proposition \ref{prop-prdcom}, leads to the following estimates:
 \begin{cor}\label{cor-gb}
 Let $\Gamma(\ep\sigma),\beta(\theta)$ be defined in \eqref{defgamma-beta}, 
 $r_0, r_1$ be defined as:
$$r_0=\f{1}{R\beta(0)}\exp(-\theta+\ep R\tilde{\sigma}/C_v\gamma),\quad r_1= \ep^{-1}(\exp (\ep R\tilde{\sigma}/C_v\gamma)-1),$$
where $ \tilde{\sigma}=\sigma-\ep\mu(2\lambda_1+\lambda_2)\Gamma\div u.$
Assume that   \eqref{preasption},  \eqref{preasption1} hold.
 Then one has the following 
 estimates: for $j=1,2,\, p=2,+\infty, X^m=H_{co}^m \text{ or } \underline{H}_{co}^m ,$ $\tilde{Z}= Z \text{ or } \na$ or $\pt,$ %{\color{red}write in a more compact way and give also the estimate of second derivatives of beta, r0, r1}
 \begin{align}
&\| \tilde{Z} \beta\|_{L_t^p X^{m-1}}\lesssim \|(\Id, \tilde{Z})\theta\|_{L_t^p X^{m-1}}\Lambda\big(\f{1}{c_0},\il \theta\il_{[\f{m}{2}],\infty,t}+\il \tilde{Z}\theta\il_{[\f{m}{2}]-1,\infty,t}\big),\label{esofbeta}\\
&\kappa\|\na^2 \beta\|_{L_t^PX^{m-1}} \lesssim \|(\Id, \kappa^{\f{1}{2}}\na, \kappa\na^2)\theta \|_{L_t^PX^{m-1}}\notag
 \\
&\qquad \cdot \Lambda\big (\f{1}{c_0}, %\il\theta \il_{[\f{m}{2}],\infty,t}+ 
\il (\Id, \kappa^{\f{1}{2}}\na ) \theta \il_{[\f{m-1}{2}],\infty,t}+\il\kappa\na^2\theta \il_{[\f{m}{2}]-1,\infty,t}\big), \label{esbeta-sec} \\
&\|\Gamma(\ep\sigma)-\Gamma(0)\|_{L_t^p X^m}\lesssim \ep \|\sigma\|
 _{L_t^p X^m}\Lambda\big (\f{1}{c_0}, \il \sigma\il_{[\f{m}{2}],\infty,t}\big),\label{esofGamma-1}\\
 &\|\tilde{Z}\Gamma\|_{L_t^p X^{m-1}}\lesssim \ep\|(\Id, \tilde{Z})\sigma\|_{L_t^p X^{m-1}} \Lambda\big(\f{1}{c_0},\il \sigma\il_{[\f{m}{2}],\infty,t}+\il \tilde{Z}\sigma\il_{[\f{m}{2}]-1,\infty,t}\big),\label{esofGamma-2} \\
 &\|\tilde{Z} (r_0, r_0^{-1} )\|_{L_t^PX^{m-1}}\notag
 \\
& \lesssim  \|(\Id, \tilde{Z})(\theta, \ep\tsigma)\|_{L_t^PX^{m-1}}\Lambda\big (\f{1}{c_0}, \il (\theta, \ep\tsigma)\il_{[\f{m}{2}],\infty,t}+\il\tilde{Z}(\theta, \ep\tsigma)\il_{[\f{m}{2}]-1,\infty,t}\big),\label{esr0}\\
&\mu^{\f{1}{2}}\|\na^2(r_0, r_0^{-1})\|_{L_t^PX^{m-2}} \lesssim \|(\Id, \na, \mu^{\f{1}{2}}\na^2)(\theta, \ep\tsigma)\|_{L_t^PX^{m-2}}\notag
 \\
&\qquad \cdot \Lambda\big (\f{1}{c_0}, \il(\theta, \ep\tsigma)\il_{[\f{m}{2}],\infty,t}+ \il \na  (\theta, \ep\tsigma)\il_{[\f{m}{2}]-1,\infty,t}+\il\mu^{\f{1}{2}}\na^2(\theta, \ep\tsigma)\il_{[\f{m}{2}]-2,\infty,t}\big),\label{esr0-sec}\\
&\|\tilde{Z} r_1\|_{L_t^PX^{m-1}}\lesssim \| (\Id ,\tilde{Z})\tsigma\|_{L_t^PX^{m-1}}\Lambda\big (\f{1}{c_0}, \il  \tsigma\il_{[\f{m}{2}],\infty,t}+
 \il \tilde{Z} \tsigma\il_{[\f{m}{2}]-1,\infty,t}) \big), \label{esr1}\\
 &\mu^{\f{1}{2}}\|\na^2 r_1\|_{L_t^PX^{m-2}} \lesssim \ep\|(\Id, \na, \mu^{\f{1}{2}}\na^2) \tsigma\|_{L_t^PX^{m-2}}\notag
 \\
&\qquad \cdot \Lambda\big (\f{1}{c_0}, \il \tsigma\il_{[\f{m}{2}],\infty,t}+ \il \na   \tsigma\il_{[\f{m}{2}]-1,\infty,t}+\il\na^2\tsigma\il_{[\f{m}{2}]-2,\infty,t}\big).\label{esr1-sec}
 \end{align}
 \end{cor}
When there is no time derivatives involved in the norm, we have also the following  inequalities proved in \cite{MR1070840},
 \begin{lem}
 Using the notation \eqref{newvector}-\eqref{newinfty}, %\eqref{newnotation}, , 
 we have the following product estimate and commutator estimate 
 \begin{align}
 &\|(fg)(t)\|_{m}\lesssim \|f(t)\|_{m}\il g \il_{0,\infty,t}+ \|g(t)\|_{m}\il f \il_{0,\infty,t}, \label{GN1}\\
& \|[\cZ^{\gamma}, f]g (t)\|_{0}\lesssim \|g(t)\|_{m-1}\il f\il_{1,\infty,t}+\|f(t)\|_{m}\il f\il_{0,\infty,t} \qquad \text{ for } |\gamma|=m.  \label{GN2}
  \end{align}
 \end{lem}
 
 We will use often the following Sobolev embedding inequality whose proof is sketched in \cite{MR4403626}.
 \begin{prop}\label{propembeddding}
Let  $\Omega$ be a  smooth  bounded domain or an exterior domain, we have the following Sobolev embedding inequality: by using the notation \eqref{normfixt},
\begin{equation}\label{sobebd}
 \| f (t)\|_{k,\infty}\lesssim 
 \|\nabla f(t)\|_{H_{co}^{k+1}}^{\frac{1}{2}}\|f(t)\|_{H_{co}^{k+2}}^{\frac{1}{2}}+\|f(t)\|_{H_{co}^{k+2}}.
\end{equation}
 \end{prop}
 The following trace inequalities are also used:
  \begin{lem}
  For multi-index $I=(I_0,\cdots, I_{M})$ with
$|I|=k,$ we have the following trace inequalities:
  \begin{align}
   &  |Z^{I}f(t)|_{L^2(\p\Omega)}^2 \lesssim \|\nabla f(t)\|_{H_{co}^{k}}\|f(t)\|_{H_{co}^{k}}+\|f(t)\|_{H_{co}^{k}}^2,  \label{traceLinfty}\\
   & \int_0^t |Z^{I}f(s)|_{L^2(\p\Omega)}^2\,\d s\lesssim \|\nabla f\|_{L_t^2H_{co}^{k}}\|f\|_{L_t^2H_{co}^{k}}+\|f\|_{L_t^2H_{co}^{k}}^2, \label{traceL2}\\
  & \int_0^t |Z^{I}f(s)|_{H^{\f{1}{2}}(\p\Omega)}^2\d s\lesssim \|\nabla f\|_{L_t^2H_{co}^{k}}^2+\|f\|_{L_t^2H_{co}^{k}}^2. \label{normaltraceineq}
  \end{align}
  \end{lem}
  In the next proposition, we state some elliptic estimates which are used frequently.
  \begin{prop}\label{propneumann}
 Given a bounded or an exterior domain $\Omega$ with $C^{k+2}$ boundary.
Consider the following elliptic equation with Neumann boundary condition:
\begin{equation}\label{Neumann problem}
  \left\{
  \begin{array}{l}
  \Delta q=\div f \quad \text{in}\quad  \Omega,\\
  \partial_{\bn} q=f\cdot \bn+g \quad \text{on} \quad  \partial\Omega,\\
  \int_{\Omega} q\, \d x=0, \text{ if } \Omega \text{ is bounded, } \quad %q\stackrel{|x|\rightarrow+\infty}{\longrightarrow}
  q\xlongrightarrow{|x|\rightarrow+\infty} 0 \text{ if }  \Omega \text{ is an exterior domain. }
  \end{array}
  \right.
\end{equation}
The system  \eqref{Neumann problem} has  a unique solution in $H^1(\Omega)$ which satisfies the following gradient estimate:
 \begin{equation}\label{gradient estimate}
     \|\nabla q(t)\|_{L^2(\Omega)}\lesssim \|f(t)\|_{L^2(\Omega)}
     +|g(t)|_{H^{-\frac{1}{2}}(\partial\Omega)}.
 \end{equation}
Moreover, for $j+l=k,$
\begin{align}
     &\|\nabla q(t)\|_{\cH^{j,l}(\Omega)}\lesssim \|f(t)\|_{\cH^{j,l}(\Omega)}
     +|g(t)|_{\tilde{H}^{k-\frac{1}{2}}(\partial\Omega)}, \label{highconormal}\\
   &  \|\nabla^2 q(t)\|_{\cH^{j,l}(\Omega)}\lesssim \|f(t)\|_{\cH^{j,l+1}}
 +\|\div f(t)\|_{\cH^{j,l}(\Omega)}
     +|g(t)|_{\tilde{H}^{k+\frac{1}{2}}(\partial\Omega)}, \label{secderelliptic}\\
   &  \|\nabla^2 q(t)\|_{\cH^{j,l}(\Omega)}\lesssim
 \|\div f(t)\|_{\cH^{j,l}(\Omega)}
     +|(f\cdot\bn, g)(t)|_{\tilde{H}^{k+\frac{1}{2}}(\partial\Omega)}.\label{secderelliptic-1}
 \end{align}
% Finally, consider $\theta \in L^2(\Omega)$ satisfies the elliptic equation, 
 %\begin{equation*} \left\{ \begin{array}{l}\Delta \theta=h\quad \text{in}\quad  \Omega,\\ \partial_{\bn} \theta=0 \quad \text{on} \quad  \partial\Omega. \end{array}  \right.\end{equation*}
 %Then one has the estimates:
% \beq \|\na^3\theta(t)\|_{H_{co}^{k}}\lesssim \|\na h\|_{H_{co}^k} \eeq
\end{prop}
\begin{lem}\label{pre-zeta1}
 Let  $Y$ and $K_{\pm}$ be defined as: 
\beq\label{defY-0}
%b=b(t,y)=\big(\f{\tilde{\chi}}{r_0}\big)^{\Psi}|_{z=0},\quad
Y=Y(t,t',y)=\mu \lambda_1\overline{\Gamma} \int_{t'}^t b(\tau,y)\, \d\tau,\quad K_{\pm}=\f{1}{(4\pi {Y})^{\f{1}{2}}}e^{-\f{|z\pm z'|^2}{4{Y}}}
\eeq
where $b=b(t,y)$ is such that
\beq\label{def-b-prop}
c_0<b(t, y)<\f{1}{c_0}  \qquad \forall (t,y)\in [0, T]\times \mR^2.
\eeq
Recall the definition of the vector fields in $\mR_{+}^3:$ 
$$\cZ_1=\p_{y^1}, \,\cZ_2=\p_{y^2}, \,\cZ_3=\f{z}{1+z}\p_z.$$
Let multi-index $\gamma=(\gamma_1,\gamma_2,\gamma_3),$ there exists  polynomial  $P_{2|\gamma|+1} $ %P_{2\gamma_3+1}^{\gamma_1,\gamma_2}$,
with degree $2|\gamma|+1,$
%and $ P_{\gamma,2}$, 
such that if $|\gamma|\leq m-3,$
\beq\label{appen-4}
\bigg|\cZ_1^{\gamma_1}\cZ_2^{\gamma_2}\cZ_3^{\gamma_3}
\bigg(\f{1}{(\pi Y)^{\f{1}{2}}}\f{z}{2Y} e^{-\f{z^2}{4Y}}b(t',y) \bigg)\bigg|\lesssim \f{1}{Y} %P_{\gamma,1}P_{2\gamma_3+1}^{\gamma_1,\gamma_2}
P_{2|\gamma|+1}\bigg(\f{z}{Y^{\f{1}{2}}}\bigg)e^{-\f{z^2}{4Y}} \lab | b|_{\gamma_1+\gamma_2,\infty,t}\big),
\eeq
and if $m-2\leq |\gamma|%=\gamma_1+\gamma_2
\leq m-1,$
\begin{align}\label{appen-5}
&\qquad\bigg|\cZ_1^{\gamma_1}\cZ_2^{\gamma_2}\cZ_3^{\gamma_3}
\bigg(\f{1}{(\pi Y)^{\f{1}{2}}}\f{z}{2Y} e^{-\f{z^2}{4Y}}b(t',y) \bigg)\bigg|\\
&\lesssim\sum_{m-2\leq k\leq |\gamma|} \f{1}{Y} %P_{2\gamma_3+1}^{\gamma_1,\gamma_2}
P_{2|\gamma|+1}\bigg(\f{z}{Y^{\f{1}{2}}}\bigg)e^{-\f{z^2}{4Y}}\big(1+|\p_{y}^k%^{\gamma_1}\p_{y^2}^{\gamma_2}
b(t',y)|+\big|\f{1}{t-t'}\int_{t'}^t \p_{y}^k b(\tau, y)\d\tau\big|\big)\lab  | b|_{m-3,\infty,t}\big).\notag
\end{align}
Moreover, %using the definition  \eqref{defY1} for $K_{\pm}$, 
we have the following property for $K_{-}-K_{+}:$ there exist  smooth functions $F^{\gamma_3,j}_{\gamma_1,\gamma_2}=F_{\gamma_1,\gamma_2}^{\gamma_3,j}(t,t', y,\cdot), j=0,\cdots \gamma_3$ such that:
\beq\label{appen-6}
\cZ_1^{\gamma_1}\cZ_2^{\gamma_2}(z\p_z)^{\gamma_3}(K_{-}-K_{+})=\sum_{i=0}^{\gamma_3} (z'\p_{z'})^j \big(F_{\gamma_1,\gamma_2}^{\gamma_3,j}(t,t',y,z-z')-F_{\gamma_1,\gamma_2}^{\gamma_3,j}(t,t', y, z+z')\big).
\eeq
In addition, $F^{\gamma_3,j}_{\gamma_1,\gamma_2}$ admits the following properties:
%and the smooth functions admitting the following properties:
if $\gamma_1+\gamma_2\leq m-3,$
\beq\label{appen-8}\sup_{0\leq t'\leq t\leq T, y\in \mR^2}\|F_{\gamma_1,\gamma_2}^{\gamma_3,k} (t,t',y,\cdot)\|_{L_z^1(\mR_{+})}\lesssim \lab  | b|_{\gamma_1+\gamma_2,\infty,t}\big),
\eeq
\beq\label{appen-7.5}
\sup_{y\in \mR^2}\|\mu^{\f{1}{2}}\p_z F_{\gamma_1,\gamma_2}^{\gamma_3,k} (t,t',y,\cdot)\|_{L_z^1(\mR_{+})}\lesssim\f{1}{(t-t')^{\f{1}{2}}}  \lab  | b|_{\gamma_1+\gamma_2,\infty,t}\big), \forall \, 0\leq t'< t\leq T
\eeq
%\beq\label{appen-8}\sup_{0\leq t'\leq t\leq T} \| F_{\gamma_1,\gamma_2}^{\gamma_3,k} (t,t',y,\cdot)\|_{L_y^2L_z^1(\mR_{+}^3)}\lesssim \lab  | b|_{m-3,\infty,t}\big) |b|_{L_t^{\infty}\tilde{H}^{m-2}}, \quad \forall\, \gamma_1+\gamma_2= m-2.\eeq
and if $m-2\leq |\gamma|\leq  m-1,$
\beq\label{appen-7}
\|\mu^{\f{1}{2}}\p_z F_{\gamma_1,\gamma_2}^{\gamma_3,k} (t,t',y,\cdot)\|_{L_y^2L_z^1(\mR_{+}^3)}\lesssim\f{1}{(t-t')^{\f{1}{2}}}  \lab  | b|_{m-3,\infty,t}\big)(1+|\p_y b|_{L_t^{\infty}\tilde{H}^{|\gamma|-1}}), \forall \, 0\leq t'< t\leq T.
\eeq
%and if $\gamma_1+\gamma_2=m-1, \gamma_3=0,$ then 
%\begin{align} F_{\gamma_1,\gamma_2}^{0,0}(t,t', y, \cdot)= e^{\f{|\cdot|^2}{4Y}}\f{1}{4(\pi Y)^{\f{1}{2}}}\big(1+\f{|\cdot|^2}{2Y}\big)\f{ \cZ^{\gamma_1}\cZ^{\gamma_2}Y}{Y}+  \tilde{F}_{\gamma_1,\gamma_2}^{0,0}(t,t', y, \cdot)\end{align}

\end{lem}
\begin{proof}
By induction, we can find a polynomial with degree $2\gamma_3+1,$
such that:
\beq\label{appen-1}
\p_z^{\gamma_3}\bigg( \f{1}{(\pi Y)^{\f{1}{2}}}\f{z}{2Y} e^{-\f{z^2}{4Y}}\bigg)=\f{1}{Y} P_{2\gamma_3+1}\bigg(\f{z}{Y^{\f{1}{2}}}\bigg)
 e^{-\f{z^2}{4Y}} z^{-\gamma_3}.
\eeq
Moreover, it follows from the basic facts:
\beq\label{appen-1.5}
\p_{y^j}\big(\f{1}{Y} \big)=-\f{1}{Y} \f{\p_{y^j}Y}{Y},  \quad \p_{y^j}\bigg(\f{z}{Y^{\f{1}{2}}}\bigg)=-\f{\p_{y^j}Y}{2Y}\f{z}{Y^{\f{1}{2}}}, \quad j=1,2
\eeq
and induction arguments that, there are some (multi-variable) polynomials $F_{l,J_{1l},J_{2l}}^{\gamma_3}, l=0,\cdots (\gamma_1+\gamma_2), J_{jl}=(\gamma_{j1},\cdots \gamma_{jl}), j=1,2$, %J_l=(\gamma_{1,1},\gamma_{1,2},\cdots \gamma_{1,l},\gamma_{2,l})$ |J_{j,l}|= \gamma_{j, 1}+\cdots \gamma_{j,l}
such that:
\begin{align}\label{appen-2}
  &\p_{y^1}^{\gamma_1}  \p_{y^2}^{\gamma_2} \big(\f{1}{Y} P_{2\gamma_3+1}\big(\f{z}{Y^{\f{1}{2}}}\big)
 e^{-\f{z^2}{4Y}} \big)\notag\\
&= \sum_{l=0}^{\gamma_1+\gamma_2}\sum_{%\substack{\gamma_{1, 1}+\cdots \gamma_{1,l}=\gamma_1\\\gamma_{2, 1}+\cdots \gamma_{2,l}=\gamma_2 }
{ |J_{1l}|=\gamma_1,  |J_{2l}|=\gamma_2}}\f{1}{Y} F_{l,J_{1l},J_{2l}}^{\gamma_3}\bigg(\f{z}{Y^{\f{1}{2}}}, \f{\p_{y^1}^{\gamma_{11}}\p_{y^2}^{\gamma_{21}}Y}{Y},\cdots \f{\p_{y^1}^{\gamma_{1l}}\p_{y^2}^{\gamma_{2l}}Y}{Y}\bigg) e^{-\f{z^2}{4Y}}.
\end{align}
By the definition $Y=\mu\lambda_1\overline{\Gamma}\int_{t'}^t b(\tau, y)\d \tau,$ we see that:
\beqs 
  \f{\p_{y^1}^{\gamma_{1k}}\p_{y^2}^{\gamma_{2k}}Y}{Y}=\f{\int_{t'}^t \p_{y^1}^{\gamma_{1k}}\p_{y^2}^{\gamma_{2k}}b(\tau, y)\d y}{\int_{t'}^t b(\tau, y)\d \tau}, \quad k=1,\cdots l,
\eeqs
which, by noticing \eqref{def-b-prop},
can be bounded from above by $\Lambda\big(\f{1}{c_0}, \il b \il_{\gamma_1+\gamma_2,\infty,t}\big)$ if $\gamma_{1}+\gamma_{2}\leq m-3$ and by $$\Lambda\big(\f{1}{c_0}\big)\sum_{m-2\leq k\leq |\gamma_1+\gamma_2|}\big|\f{1}{t-t'}\int_{t'}^t %\p_{y^1}^{\gamma_1}  \p_{y^2}^{\gamma_2}
\p_y^{k}b(\tau, y)\d\tau\big|$$ if $m-2\leq \gamma_{1k}+\gamma_{2k}\leq m-1.$
We thus finish the proof of \eqref{appen-4} and \eqref{appen-5}
by gluing \eqref{appen-1} and \eqref{appen-2}. 

Let us switch to the proof of \eqref{appen-6}. By induction, one can verify that there exist  polynomials $P_{2l}, l=0,\cdots \gamma_3$ with degree $2j,$ such that: 
\beq
\begin{aligned}
    (z\p_z)^{\gamma_3}(K_{-}-K_{+})&= \f{1}{(4\pi Y)^{\f{1}{2}}}
    (z\p_z)^{\gamma_3}\big(e^{-\f{|z-z'|^2}{4Y}}-e^{-\f{|z+z'|^2}{4Y}}\big) \\
    &= \f{1}{(4\pi Y)^{\f{1}{2}}} \sum_{j=0}^{\gamma_3} (z'\p_{z'})^j \bigg(P_{2(\gamma_3-j)}\big(\f{z-z'}{Y^{\f{1}{2}}}\big)e^{-\f{|z-z'|^2}{4Y}}- P_{2(\gamma_3-j)}\big(\f{z+z'}{Y^{\f{1}{2}}}\big)e^{-\f{|z+z'|^2}{4Y}}\bigg).
\end{aligned}
\eeq
Using again the fact \eqref{appen-1.5}, on can find some (multi-variable) polynomials $F_{l, J_{1l}, J_{2l}}^{\gamma_3, j},$ such that:
\beq\label{appen-useful}
\begin{aligned}
&\cZ_1^{\gamma_1}\cZ_2^{\gamma_2} \bigg( \f{1}{(4\pi Y)^{\f{1}{2}}} P_{2(\gamma_3-j)}\big(\f{z\pm z'}{Y^{\f{1}{2}}}\big)e^{-\f{|z\pm z'|^2}{4Y}}\bigg)\\
&=\f{1}{ Y^{\f{1}{2}}} \sum_{l=0}^{\gamma_1+\gamma_2}\sum_{{ |J_{1l}|=\gamma_1,  |J_{2l}|=\gamma_2}} F_{l,J_{1l},J_{2l}}^{\gamma_3,j}\bigg(\f{z\pm z'}{Y^{\f{1}{2}}}, \f{\p_{y^1}^{\gamma_{11}}\p_{y^2}^{\gamma_{21}}Y}{Y},\cdots \f{\p_{y^1}^{\gamma_{1l}}\p_{y^2}^{\gamma_{2l}}Y}{Y}\bigg) e^{-\f{|z\pm z'|^2}{4Y}}\\
&=\colon F_{\gamma_1,\gamma_2}^{\gamma_3,j}(t,t',y, z\pm z').
\end{aligned}
\eeq
We thus finish the proof of \eqref{appen-6}. The properties \eqref{appen-8}- \eqref{appen-7} then follow from the direct computations.
\end{proof}

\begin{lem}\label{lemheatgreen}
Let $f: \mR_{+}^3\rightarrow \mR$ be the solution to the system:
\beq\label{eqheattransport}
\left\{
\begin{array}{l}
   \p_t f + a_1\p_{y^1}f+a_2\p_{y^2}f+z a_3\p_z f-\nu b\p_z^2 f=F\quad\, (t,y,z)\in [0,T]\times \mR^2\times\mR_{+}, \\[2.5pt]
f|_{z=0}=\underline{f}(t,y), \quad f|_{t=0}=f_0,
\end{array}
\right.
\eeq
where $(a,b)=(a_1, a_2, a_3, b)=(a_1, a_2, a_3, b)(t,y),$  $b(t,y)>c_0,$ for all $(t,y)\in [0,T]\times \mR^2$ and the constant $\nu>0$ is a placeholder of $\mu$ or $\kpa.$ Using the notation \eqref{newinfty-time},
the following uniform in $\nu$ estimates hold: for any $0<t\leq T,$
\begin{align}\label{infty1}
%\sum_{0\leq j\leq k}(\ep\pt)^j
\il  f\il_{k,\infty,t}\lesssim \Lambda_{a,b}^k \bigg(\|f_0\|_{k,\infty,t}+|\underline{f}|_{k,\infty,t}+
T {\Lambda}_{a, \pt b}^k \il f\il_{k,\infty,t}+ T^{\f{1}{2}}\big(\izt \|  F(s)\|_{k,\infty}^2\d s\big)^{\f{1}{2}}\bigg), 
\end{align}
where we denote for simplicity 
\beq\label{deflambda-ab}
\Lambda_{a,b}^k=\lab |(T(a_1,a_2, a_3), b)|_{k,\infty,t}\big) e^{4(k+1) T|a_3|_{0,\infty,t}}, \quad 
{\Lambda}_{a, \pt b}^k=\lab |(a_1, a_2,a_3,\p_t b|_{k,\infty,t} \big) .
\eeq
Moreover, if $F=F_1+\nu \p_z F_2,$ then for any $2<p\leq +\infty,$
\beq\label{infty2}
\begin{aligned}
\il f\il_{k,\infty,t}\lesssim_p \Lambda_{a,b}^k &\bigg(\|f_0\|_{k,\infty}+|\underline{f}|_{k,\infty,t}+
+T{\Lambda}_{a,\pt b}^k \il f\il_{k,\infty,t}+
T^{\f{1}{2}}\big(\izt \|  F_1(s)\|_{k,\infty}^2\d s\big)^{\f{1}{2}}\\
%T{\Lambda}_{b}^k\il (f, F_1)\il_{k,\infty,t}\\
&\qquad\qquad\qquad\qquad+%\big(\f{2p-2}{p-2}\big)^{\f{p-1}{p}}
T^{\f{p-2}{2p}}\big( \izt \| \nu^{\f{1}{2}}F_2(s)\|_{k,\infty}^p\d s\big)^{\f{1}{p}}   \bigg).
\end{aligned}
\eeq
Finally, if $\underline{f}=0,$ it holds also that for $\alpha\geq \f{1}{2},$
\begin{align}\label{infty3}
    \nu^{\alpha}\il\p_z f\il_{k,\infty,t}\lesssim 
    \Lambda_{a,b}^k  \bigg(\nu^{\alpha}\|f_0\|_{k,\infty}+T^{\f{1}{2}}{\Lambda}_{a, \pt b}^k\nu^{\alpha-\f{1}{2}} \il (f, F)\il_{k,\infty,t}\bigg),
\end{align}
and for $2\leq q<+\infty,$
\begin{align}\label{infty4}
 \big( \int_0^t \|\nu^{\f{1}{2}}\p_z f\|_{k,\infty}^q\d s\big)^{\f{1}{q}}\lesssim_q T^{\f{1}{q}} \Lambda_{a,b}^k 
 \bigg(\|\nu^{\f{1}{2}} f_0\|_{k,\infty}+%(\f{2+q}{2})^{\f{1}{2}+\f{1}{q}}
 {\Lambda}_{a, \pt b}^k\big(\izt \| (f, F)(s)\|_{k,\infty}^2\d s\big)^{\f{1}{2}} \bigg).
\end{align}
\end{lem}
\begin{proof}
We first show that the solution to \eqref{eqheattransport} can be expressed in the following way:
\beq\label{green-heat}
\begin{aligned}
f(t,y,z)&=\int_{\mR_{+}} ({K}_{-}-{K}_{+})(t, 0, y, z, z') f_0\big(X(0,t,y), z' e^{-\Theta(t,0, y)}\big)\,\d z'\\
&\quad + 2 \int_0^t 
\p_{z}{K}_{+}(t,t',y, z, 0) (b\underline{f})\big(t', X(t',t, y)\big) e^{2\Theta(t,t',y)} \d t'\\
&\,+\izt \int_{\mR_{+}} ({K}_{-}-{K}_{+})(t, t', y, z, z')F(t', X(t',t, y), z'e^{-\Theta(t,t',y)})\d z' \d t'
\end{aligned}
\eeq
where 
\begin{align}\label{defY-1}
 & K_{\pm}(t,t', y, z,z' )=\f{1}{(4\pi {Y})^{\f{1}{2}}}e^{-\f{|z\pm z'|^2}{4{Y}}}\notag\\
 \Theta(t,t',y)=\int_{\tau}^t a_3 &(\tau, X(\tau, t, y))\d \tau,\quad Y(t,t',y)=\nu\int_{t'}^t  b(\tau, X(\tau, t, y))e^{2\Theta(t,\tau,y)}\d\tau,
\end{align}
and $X:\mR\times\mR\times\mR^2\rightarrow \mR^2$ is  the characteristic associated to the vector $(a_1,a_2)^t:$
\beqs 
\pt X(t,\tau, y) =(a_1, a_2)(t,X(t,\tau, y)), \quad X(\tau,\tau,y )=y.
\eeqs
Let 
$$ \tilde{b}(\tt,y)=b(\tt,X(\tt,0, y)), \quad \tilde{\Theta}(\tt,y)= \int_0^{\tt} a_3(\tau, X(\tau,0, y))\,\d \tau, \quad z=\tz \exp (-{\tilde{\Theta}(\tt,y)}), \quad t=\nu\int_0^{\tilde{t}}\tilde{b} e^{-2\tilde{\Theta}}(\tau, y)\d \tau$$
It follows from the direct computation that
the function $\tilde{f}(t,y, z)=\colon f(\tt, X(\tt, 0, y), \tz)$
solves the heat equation with constant variable in the half line ($y$ is now considered to be a parameter):
%Let $$ and  $f_2(t, y, z)=f_1(\tilde{t}, y, z)$ (in other words, $f_2(t,y, z)= f(\tt, X(\tt, y), \tz(\tt, ,y, z))$), then $f_2$ solves 
    \beqs
\left\{
\begin{array}{l}
   \p_t \tilde{f} -\p_z^2 \tilde{f}=\big((\nu\tilde{b})^{-1}e^{2\tilde{\Theta}}\big)(\tt,y)F(\tt, X(\tt, 0, y), \tz), 
   \\[2.5pt]
\tilde{f}|_{z=0}(t,y)=\underline{f}(\tt,X(\tt,0, y)), \quad \tilde{f}|_{t=0}(y,z)=f_0(y, z),
\end{array}
\right.
\eeqs
%where we denote $\underline{\tilde{f}}(\tt,y)=.$ 
which can be solved explicitly:
\begin{align*}
    \tilde{f}(t, y, z)&=\int_{\mR_{+}} (E_{-}-E_{+})(t,0 ,y,z,z')f_0(y, z') \d z'+2\izt \p_{z'}E_{+}(t, t',y, z, 0) \underline{f}(\tt', X(\tt',0, y)) \d t'\\
 & \quad + \izt\int_{\mR_{+}} (E_{-}-E_{+})  (t,t' ,y,z,z')F(\tt', X(\tt', 0, y), \tz')\d z'\d t'
\end{align*}
where $E_{\pm}(t,t',y,z,z')=\f{1}{(4\pi(t-t'))^{\f{1}{2}}}e^{-\f{|z\pm z'|^2}{4(t-t')}}.$
Changing back to the variable $(\tt, X(\tt,0, y), \tz)$ and 
still denoting it by $(t, X(t,0, y), z),$ we find that:
\beq\label{green-interm}
\begin{aligned}
&f(t,X(t,0, y), z)=\int_{\mR_{+}} (\tilde{K}_{-}-\tilde{K}_{+})(t, 0, y, z, z') f_0(y, z' e^{-\tilde{\Theta}(t,0,y)})\,\d z'\\
&\quad + 2 \nu\int_0^t 
\p_{z} \tilde{K}_{+}(t,t',y, z, 0) \underline{f}(t', X(t',0, y)) \tb(t',y) e^{2\tilde{\Theta}(t,t',y)} \d t'\\
&\, +\izt \int_{\mR_{+}} (\tilde{K}_{-}-\tilde{K}_{+})(t, t', y, z, z') F(t', X(t',0, y), z' e^{-\tilde{\Theta}(t,t',y)}) %e^{2\tilde{\Theta}(t,t',y)}
\,\d z'\d t'
\end{aligned}
\eeq
where 
\beqs 
\tilde{\Theta}(t,t',y)=\tilde{\Theta}(t,y)-\tilde{\Theta}(t',y)
\eeqs
and
\begin{align*}
\tilde{K}_{\pm}(t, t', y, z, z')=\f{1}{(4\pi \tilde{Y})^{\f{1}{2}}}e^{-\f{|z\pm z'|^2}{4\tilde{Y}}}, \text{ with }
\tilde{Y}(t, t',y)=\nu\int_{t'}^t \tb(\tau,y) e^{2\tilde{\Theta}(t,\tau,y)}\d \tau.
\end{align*}
The explicit formula \eqref{green-heat} then follows readily from \eqref{green-interm} by changing variable $X(t,0,y)$ to $y.$ 

We are now in position to sketch the proof of  \eqref{infty1}-\eqref{infty3}. Changing the definition of $Y$ in \eqref{defY-0} by \eqref{defY-1}, we can check that the properties \eqref{appen-4}, \eqref{appen-6}-\eqref{appen-7.5} still hold after minor revision-just replacing $\lab | b|_{\gamma_1+\gamma_2,\infty,t}\big)$ by $\Lambda_{a,b}^{\gamma_1+\gamma_2}$ defined in \eqref{deflambda-ab}. Indeed, the identities \eqref{appen-1}-\eqref{appen-useful} used in the derivation of these properties remain formally unchanged. However, the  properties for $Y$ used to derive \eqref{appen-4}, \eqref{appen-6}-\eqref{appen-7.5} from \eqref{appen-1}-\eqref{appen-useful}  are now:
\begin{align*}
    Y%(t,t', y, z, z')
    \gtrsim \f{\nu}{c_0} e^{-2T|a_3|_{0,\infty,t}}, \quad  \p_{y^1}^{\gamma_1}\p_{y^2}^{\gamma_1} Y \lesssim 
\lab |T(a_1,a_2,a_3),b|_{\gamma_1+\gamma_2,\infty,t}\big) e^{2 T|a_3|_{0,\infty,t}}.    
\end{align*}
%Without much ambigiuty, we will still denote \eqref{appen-4}, \eqref{appen-6}-\eqref{appen-7.5} as the updated properties here. 
These properties in hand, one can verify that, by using 
the convolution inequality in $z$ and $t$ variables, 
for $\cZ^{\gamma}=\p_{y^1}^{\gamma_1}\p_{y^2}^{\gamma_2}(z\p_z)^{\gamma_3},$ %\eqref{infty1}-\eqref{infty3} by substituting the 
\begin{align}\label{infty1-1}
\sum_{|\gamma|\leq k}\il \cZ^{\gamma} f\il_{0,\infty,t}\lesssim \Lambda_{a,b}^k \bigg(\|f_0\|_{k,\infty}+|\underline{f}|_{k,\infty,t}+
T^{\f{1}{2}}\big(\izt \|  F(s)\|_{k,\infty}^2\d s\big)^{\f{1}{2}} \bigg)
%T\il  F\il_{k,\infty,t}\big),
\end{align}
and if $F=F_1+\nu \p_z F_2,$
\beq
\begin{aligned}%\label{infty2}
\sum_{|\gamma|\leq k}\il \cZ^{\gamma} f\il_{0,\infty}\lesssim_p \Lambda_{a,b}^k &\bigg(\|f_0\|_{k,\infty,t}+|\underline{f}|_{k,\infty,t}+ T^{\f{1}{2}}\big(\izt \|  F_1(s)\|_{k,\infty}^2\d s\big)^{\f{1}{2}} %T\il F_1\il_{k,\infty,t}%T^{\f{1}{2}}\il F_2\il_{k,\infty,t}\big).
\\
&\qquad+%\f{2p-2}{p-2}
T^{\f{p-2}{2p}} \big(\izt \| \nu^{\f{1}{2}}F_2(s)\|_{k,\infty}^p\d s\big)^{\f{1}{p}}   \bigg), \quad \forall\, 2< p\leq +\infty,
\end{aligned}
\eeq
and if $\underline{f}=0,$ 
\begin{align}\label{infty3-1}
   \sum_{|\gamma|\leq k} \nu^{\alpha}\il\cZ^{\gamma}\p_z f\il_{k,\infty,t}\lesssim 
    \Lambda_{a,b}^k  \bigg(\nu^{\alpha}\|f_0\|_{k,\infty}+T^{\f{1}{2}}\nu^{\alpha-\f{1}{2}} \il F\il_{k,\infty,t}\bigg), \, \alpha\geq \f{1}{2},
\end{align}
\begin{align}\label{infty4-1}
 \big( \int_0^t \big(\sum_{|\gamma|\leq k}\|\nu^{\f{1}{2}}\cZ^{\gamma}\p_z f\|_{0,\infty}\big)^q\d s\big)^{\f{1}{q}}\lesssim_q T^{\f{1}{q}} \Lambda_{a,b}^k 
 \bigg(\|\nu^{\f{1}{2}} f_0\|_{k,\infty}+
 \big(\izt \|  F(s)\|_{k,\infty}^2\d s\big)^{\f{1}{2}} \bigg),\, 2\leq q <+\infty.
\end{align}
To show \eqref{infty1}-\eqref{infty3}, it remains to control the weighted time derivatives. To do so, we 
apply $(\ep\pt)^j, 1\leq j\leq k$ on the equation \eqref{eqheattransport} and obtain the following equation for $f^j=\colon (\ep\pt)^j f:$
\begin{align*}
\left\{
\begin{array}{l}
   \p_t f^j + a_1\p_{y^1}f^j+a_2\p_{y^2}f^j+z a_3\p_z f^j-\nu b\p_z^2 f^j=F^j\quad\, (t,y,z)\in [0,T]\times \mR^2\times\mR_{+}, \\[2.5pt]
f|_{z=0}=\underline{f}(t,y), \quad f|_{t=0}=f_0,
\end{array}
\right.
\end{align*}
where
\begin{align*}
  F^j=(\ep\pt)^j F-\sum_{l=1}^2 [(\ep\pt)^j,a_l]\p_{y^l}f-[(\ep\pt)^j,a_3](z\p_{z} f)+\nu[(\ep\pt)^j, b]\p_{z}^2 f.
\end{align*}
The term $\ep\nu\p_z^2 f$ can be expressed by using the equation \eqref{eqheattransport} as:
\beqs 
\f{1}{b}(\ep\pt f)+\ep (a_1\p_{y^1}f+a_2\p_{y^2}f+z a_3\p_z f),
\eeqs
which allows us to bound the term $F^j$ as:
\begin{align*}
  \sum_{|\gamma|\leq k-j}\|\cZ^{\gamma}F^j(s)\|_{0,\infty}\lesssim  \Lambda_{a,b}^k \big(\Lambda_{a, \pt b}^k  \| f(s)\|_{k,\infty}+\| F(s)\|_{k,\infty}  \big), \, \forall \,
  0<s\leq t.
\end{align*}
Applying \eqref{infty1-1}-\eqref{infty4-1}, we achieve eventually \eqref{infty1}-\eqref{infty4}.
\end{proof}
%\section*{Data Availability Statement}
%\qquad \qquad \qquad \qquad \qquad \qquad\qquad \qquad 
\textbf{Data Availability Statement}
No datasets were generated or analyzed during the current study. 
 
\textbf{Conflicts of Interest Statement:}
The author declares that he has no conflicts of interest.
\section*{Acknowledgement}
The work of the author is supported by the ANR LabEx CIMI (grant ANR-11-LABX-0040) within the French State Programme “Investissement d’Avenir”. He wishes to thank Professors Nader Masmoudi and Fr\'ed\'eric  Rousset for the valuable discussions and suggestions.
\bibliographystyle{plain}
\nocite{*}
\bibliography{referencencns}

\end{document}